\PassOptionsToPackage{unicode,linktocpage}{hyperref}
\PassOptionsToPackage{hyphens}{url}
\PassOptionsToPackage{dvipsnames,svgnames*,x11names*}{xcolor}
\DeclareSymbolFont{AMSb}{U}{msb}{m}{n}
\documentclass[
  12pt,
  sumlimits,
  namelimits]{amsart}
\usepackage{lmodern}
\usepackage{amssymb}
\usepackage{mathtools}
\usepackage{ifxetex,ifluatex}
\ifnum 0\ifxetex 1\fi\ifluatex 1\fi=0 
  \usepackage[T1]{fontenc}
  \usepackage[utf8]{inputenc}
  \usepackage{textcomp} 
\else 
  \ifxetex
    \usepackage[MnSymbol]{mathspec}
  \else
    \usepackage{unicode-math}
  \fi
  \defaultfontfeatures{Scale=MatchLowercase}
  \defaultfontfeatures[\rmfamily]{Ligatures=TeX,Scale=1}
\fi
\IfFileExists{upquote.sty}{\usepackage{upquote}}{}
\IfFileExists{microtype.sty}{
  \usepackage[]{microtype}
  \UseMicrotypeSet[protrusion]{basicmath} 
}{}
\usepackage{xcolor}
\IfFileExists{xurl.sty}{\usepackage{xurl}}{} 
\IfFileExists{bookmark.sty}{\usepackage{bookmark}}{\usepackage{hyperref}}
\hypersetup{
  pdftitle={On a quantum-classical correspondence: from graphs to manifolds},
  pdfauthor={Akshat Kumar},
  colorlinks=true,
  linkcolor=Maroon,
  filecolor=Maroon,
  citecolor=Blue,
  urlcolor=Blue,
  pdfcreator={LaTeX via pandoc}}
\usepackage[margin=1in,,paper=a4paper]{geometry}
\usepackage{longtable,booktabs}
\usepackage{etoolbox}
\makeatletter
\patchcmd\longtable{\par}{\if@noskipsec\mbox{}\fi\par}{}{}
\makeatother
\IfFileExists{footnotehyper.sty}{\usepackage{footnotehyper}}{\usepackage{footnote}}
\makesavenoteenv{longtable}
\usepackage{graphicx}
\makeatletter
\def\maxwidth{\ifdim\Gin@nat@width>\linewidth\linewidth\else\Gin@nat@width\fi}
\def\maxheight{\ifdim\Gin@nat@height>\textheight\textheight\else\Gin@nat@height\fi}
\makeatother
\setkeys{Gin}{width=\maxwidth,height=\maxheight,keepaspectratio}
\makeatletter
\def\fps@figure{htbp}
\makeatother
\providecommand{\tightlist}{%
  \setlength{\itemsep}{0pt}\setlength{\parskip}{0pt}}
\usepackage{filecontents}
\usepackage{amsthm}
\usepackage{array}
\usepackage{endnotes}
\usepackage{hyperendnotes}
\usepackage{mathrsfs}
\usepackage{calc}
\makeatletter
\@ifclassloaded{beamer}{}{
\usepackage{float}
\let\origfigure\figure
\let\endorigfigure\endfigure
\renewenvironment{figure}[1][2] {
   \expandafter\origfigure\expandafter[H]
} {
    \endorigfigure
}
}
\makeatother

\makeatletter
\@ifclassloaded{beamer}{
\makeatletter
\g@addto@macro\normalsize{%
    \setlength\belowdisplayskip{-0pt}
}
}{}
\makeatother

\usepackage[hang]{footmisc}
\usepackage{nomencl}
\makenomenclature

\makeatletter
\@ifclassloaded{beamer}{}{
\theoremstyle{plain}
\newtheorem{lemma}{Lemma}
\newtheorem{theorem}{Theorem}
\newtheorem*{theorem*}{Theorem}
\newtheorem*{corollary}{Corollary}

\newtheorem*{fact*}{Fact}

\newtheorem*{assumptions*}{Assumptions}

\theoremstyle{definition}
\newtheorem*{definition}{Definition}
\newtheorem*{notation}{Notation}

\newtheorem*{example*}{Example}

\newtheoremstyle{break}
  {}
  {}
  {}
  {}
  {\bfseries}
  {.}
  {\newline}
  {}
\theoremstyle{break}
\newtheorem{algorithm}{Algorithm}
\newtheorem*{algorithm*}{Algorithm}

\theoremstyle{remark}
\newtheorem*{remark}{Remark}
}
\makeatother

\theoremstyle{plain}
\newtheorem{proposition}{Proposition}
\newtheorem*{proposition*}{Proposition}


\IfFileExists{\jobname.ent}{
   \AtEndDocument{\theendnotes}
}{}

\AtBeginDocument{%
  \def\MR#1{}
}

\makeatletter
\@ifclassloaded{beamer}{}{
\DeclareRobustCommand{\SkipTocEntry}[5]{}
\setcounter{tocdepth}{1}
}
\makeatother

\allowdisplaybreaks
\newcommand{\isomto}{\stackrel{\sim}{\smash{\longrightarrow}\rule{0pt}{0.4ex}}}
\newcommand{\embedM}{\Lambda}
\newcommand{\mdim}{n}
\newcommand{\rdim}{D}
\newcommand{\vol}[1]{\nu_g(#1)}
\newcommand{\GLap}{\Delta}
\newcommand{\GAve}{A}
\newcommand{\AvOp}{\mathscr{A}}

\newcommand{\Kconst}{K_{\mathcal{M},k,p}}
\newcommand{\supp}{\operatorname{supp}}
\newcommand{\Op}{\operatorname{Op}}
\newcommand{\PDO}{{\Psi\text{DO}}}

\newcommand{\leftmath}[1]{\mathmakebox[0pt][l]{#1}}
\newcommand{\pwb}{\gamma^{\operatorname{pw}}}

\newcommand{\CSeps}[2]{\psi_{h,#2,#1}}
\newcommand{\CSepsN}[2]{\psi_{h,#2,#1,N}}
\newcommand{\Ueps}[2]{U_{\lambda,\epsilon}^{(#1) #2}}
\newcommand{\UepsN}[2]{U_{\lambda,\epsilon,N}^{(#1) #2}}
\newcommand{\Flow}{\Theta}
\usepackage[square,numbers]{natbib}

\title[On a quantum-classical correspondence]{On a quantum-classical correspondence: from graphs to manifolds}
\author{Akshat Kumar}
\date{Updated: December 13, 2022}

\address{Department of Mathematics, Clarkson University, Potsdam, NY 13699 USA}
\address{Instituto de Telecomunicações, Lisbon, Portugal}

\email{akumar@clarkson.edu}

\usepackage{import}

\begin{document}
\begin{abstract}
We establish conditions for which graph Laplacians \(\Delta_{\lambda,\epsilon}\) on compact, boundaryless, smooth submanifolds \(\mathcal{M}\) of Euclidean space are semiclassical pseudodifferential operators (\(\Psi\)DOs): essentially, that the graph Laplacian's kernel bandwidth (\emph{bias term}) \(\sqrt{\epsilon}\) decays faster than the semiclassical parameter \(h\), \emph{i.e.}, \(h \gg \sqrt{\epsilon}\) and we compute the symbol. Coupling this with Egorov's theorem and coherent states \(\psi_h\) localized at \((x_0, \xi_0) \in T^*\mathcal{M}\), we show that with \(U_{\lambda,\epsilon}^t := e^{-i t \sqrt{\Delta}_{\lambda,\epsilon}}\) spectrally defined, the (co-)geodesic flow \(\Gamma^t\) on \(T^*\mathcal{M}\) is approximated by \(\langle U_{\lambda,\epsilon}^{-t} \operatorname{Op}_h(a) U_{\lambda,\epsilon}^t \psi_h, \psi_h \rangle = a \circ \Gamma^t(x_0, \xi_0) + O(h)\). Then, we turn to the discrete setting: for \(\Delta_{\lambda,\epsilon,N}\) a normalized graph Laplacian defined on a set of \(N\) points \(x_1, \ldots, x_N\) sampled \emph{i.i.d.} from a probability distribution with smooth density, we establish Bernstein-type lower bounds on the probability that \(||U_{\lambda,\epsilon,N}^t[u] - U_{\lambda,\epsilon}^t[u]||_{L^{\infty}} \leq \delta\) with \(U_{\lambda,\epsilon,N}^t := e^{-i t \sqrt{\Delta}_{\lambda,\epsilon,N}}\). We apply this to coherent states to show that the geodesic flow on \(\mathcal{M}\) can be approximated by matrix dynamics on the discrete sample set, namely that \emph{with high probability}, \(c_{t,N}^{-1} \sum_{j=1}^N |U_{\lambda,\epsilon,N}^t[\psi_h](x_j)|^2 u(x_j) = u(x_t) + O(h)\) for \(c_{t,N} := \sum_{j=1}^N |U_{\lambda,\epsilon,N}^t[\psi_h](x_j)|^2\) and \(x_t\) the projection of \(\Gamma^t(x_0, \xi_0)\) onto \(\mathcal{M}\).
\end{abstract}
\maketitle

\tableofcontents
\hypertarget{introduction}{%
\section{Introduction}\label{introduction}}

Given a pointset \(\embedM_N := \{ X_1, \ldots, X_N \} \subset \mathbb{R}^{\rdim}\) that is assumed to lie on a submanifold \(\embedM \supset \embedM_N\), a natural inverse problem asks to approximate the \emph{intrinsic} geometry of \(\embedM\), using only this data. Implicitly, there is an isometric embedding \(\iota : \mathcal{M} \to \embedM \subset \mathbb{R}^{\rdim}\) of an abstract manifold \(\mathcal{M}\), both of which are unknown. Granted sufficient smoothness, such a reconstruction may consist of: an atlas that is prescribed by intrinsic data, \emph{e.g.}, a system of normal coordinates, or operators that are intimately tied to the intrinsic geometry, as in the Laplace-Beltrami\footnote{In the following, we define the Laplace-Beltrami operator to be signed so it is \emph{positive semi-definite}: that is, in local coordinates, \(\Delta_{\mathcal{M}} := -\frac{1}{\sqrt{|g|}} \sum_{j,\ell} \partial_j \sqrt{|g|} g^{j,\ell} \partial_{\ell}\), wherein \(g\) is the Riemannian metric on \(\mathcal{M}\) and \(g^{j,\ell}\) are the \((j,\ell)\) components of the inverse, \(g^{-1}\).} operator \(\Delta_{\mathcal{M}}\), to name a couple of examples related to the study in this paper. Further, the restriction of functions on \(\embedM\) to the pointset may be modeled by identifying \(\embedM_N\) with the canonical basis in \(\mathbb{C}^N\), which gives a complex vector space \(\mathcal{H}_N \cong \mathbb{C}^N\). Then, for example, a vector in \(\mathcal{H}_N\) may be seen as the restriction to \(\embedM_N\) of a linear combination of \(N\) bump functions \(u_1, \ldots, u_N\) with \(u_j(X_j) = 1\) or \(0\) and pairwise disjoint supports, \(\supp u_j \cap \supp u_{\ell} = \emptyset\) for all \(j \neq \ell\). Hence, the question arises as to how the dynamics of functions on \(\mathcal{M}\) can be approximated by dynamics in \(\mathcal{H}_N\). In many ways the geometry of manifolds is intimately tied with dynamics on them and so these two kinds of problems are interconnected.

Over the last two decades, the theory of graph Laplacians has advanced, giving convergence rates for the recovery of \(\Delta_{\mathcal{M}}\) from samples \(\embedM_N\) drawn from a probability distribution \(P\) supported on \(\embedM\), as \(N \to \infty\), under various assumptions on \(\mathcal{M}\), \(\embedM\), \(P\) and the kernels used to construct the graph Laplacians \citep{hein2005graphs, singer2006, von2008consistency, calder2019improved}. A common schema is described in \citep{nadler2006diffusion} as follows: using a positive, \emph{essentially} locally supported\footnote{This means that if \(k \in C^m(\mathbb{R}_+)\), then there are \(R_1, \ldots, R_m > 0\) such that \(\partial^j k\) decays exponentially fast outside of \([0, R_j]\).} smooth function \(k : \mathbb{R}_+ \to \mathbb{R}\), an \emph{averaging} matrix \((\AvOp_{\epsilon,N})_{j,\ell} := k(||X_j - X_{\ell}||_{\mathbb{R}^{\rdim}}^2/\epsilon)\) is defined that after row normalisation represents a Markov chain process in discrete time and space. This normalized matrix \(\GAve_{\epsilon,N}\) converges in such settings, with high probability as \(N \to \infty\), to an operator \(\GAve_{\epsilon}\) that represents a Markov process in continuous space \(\mathcal{M}\) and discrete time with steps of size \(\epsilon\); its transition probability from \(x \in \mathcal{M}\) to \(y \in \mathcal{M}\) is proportional to \(k_{\epsilon}(x, y) := k(||\iota(x) - \iota(y)||_{\mathbb{R}^{\rdim}}^2/\epsilon)\). The infinitesimal generator of this Markov process is \(\Delta_{\mathcal{M}} + O(\partial^1)\) with \(O(\partial^1)\) denoting terms with differential operators of order at most one. Further, in the limit \(\epsilon \to 0\), this Markov process converges to a diffusion process with the same infinitesimal generator. The (\emph{random walk}) \emph{discrete} and \emph{continuum} graph Laplacians here are equal, up to a constant factor, to \((I - \GAve_{\epsilon,N})/\epsilon\) and and \((I - \GAve_{\epsilon})/\epsilon\), respectively, and they give the concrete approximations to the infinitesimal generator.

In this work, we study for smooth, compact, boundaryless manifolds \(\mathcal{M}\), with some basic assumptions on \(\embedM\) and \(P\), the inclusion of graph Laplacians in the framework of semiclassical analysis and switch perspective from diffusion processes to quantum dynamics. The geometric data we are interested in recovering are the geodesics on \(\mathcal{M}\), which can be cast as projections of the intrinsic \emph{(co)-geodesic} flow \(\Gamma^t : T^*\mathcal{M} \to T^*\mathcal{M}\) governed by the Hamiltonian \(|\xi|_{g_x} := \langle g_x^{-1} \xi, \xi \rangle^{\frac{1}{2}}\), wherein \(g_x\) is the Riemannian metric at \(x \in \mathcal{M}\). This Hamiltonian flow is \emph{quantized} by \(e^{-i t \sqrt{\Delta_{\mathcal{M}}}}\), in that for times \(|t|\) smaller than the injectivity radius at \(x_0\), it moves the support of wavepackets localized at \((x_0, \xi_0)\) in \emph{phase space} \(T^*\mathcal{M}\) to those localized at \(\Gamma^t(x_0, \xi_0)\), which ultimately leads to a linear PDEs approach to accessing geodesics. We use the framework of semiclassical pseudodifferential operators (\(\PDO\)s) to recover this behaviour with continuum graph Laplacians in place of \(\Delta_{\mathcal{M}}\) and a manifold generalization of coherent states for the localized wavepackets. Then, we find probabilistic convergence rates so geodesics can be recovered with matrix dynamics on \(\mathcal{H}_N\) using discrete graph Laplacians in place of their continuum counterparts and the restrictions of coherent states to the sample points \(\embedM_N\). Along the way, we establish consistency between solutions to discrete wave equations on \(\embedM_N\) and solutions to wave equations on \(\mathcal{M}\).

The dichotomy in the continuum and discrete settings requires completely different considerations. In \protect\hyperlink{from-graph-laplacians-to-geodesic-flows}{Section \ref{from-graph-laplacians-to-geodesic-flows}}, we realise continuum graph Laplacians, denoted \(\GLap_{\lambda,\epsilon}\) with \(\lambda \geq 0\) a normalization parameter (such that \(\lambda = 0\) gives the random walk continuum graph Laplacian) --- following \citep{hein2005graphs}, we define these in \protect\hyperlink{laplacians-from-graphs-to-manifolds}{Section \ref{laplacians-from-graphs-to-manifolds}} ---, as semiclassical pseudodifferential operators (\(\PDO\)s). Then, the quantum dynamical perspective is driven by Egorov's theorem, which tells more generally that if for \(h \in [0, h_0)\) with \(h_0 > 0\), \(Q\) is an order one \(\PDO\) on \(\mathcal{M}\) with real-valued principal symbol \(q \in C^{\infty}(T^*\mathcal{M} \times [0, h_0)_h)\) that generates a Hamiltonian flow \(\Theta_q^t : T^*\mathcal{M} \to T^*\mathcal{M}\) on \(T^*\mathcal{M}\) and \(\hat{O}\) is a \(\PDO\) with symbol \(a \in C^{\infty}(T^*\mathcal{M} \times [0, h_0)_h)\), then the operator dynamics \(\hat{O}_t := e^{\frac{i}{h} t Q} \hat{O} e^{-\frac{i}{h} t Q}\) remains pseudodifferential and is, up to lower order, the quantization of \(a \circ \Theta_q^t\). In this sense, we say that \(e^{-\frac{i}{h} t Q}\) \emph{quantizes} the flow \(\Theta_q^t\). To retrieve the flow of symbols from their quantized operators, we make use of \emph{coherent states}, which are wavepackets of the form \(\tilde{\psi}_h := e^{-\frac{i}{h} \bar{\phi}(\cdot ; x_0, \xi_0)}\) localized at \((x_0, \xi_0) \in T^*\mathcal{M}\) with a complex phase \(\phi\) satisfying certain \emph{admissibility conditions} as defined in \protect\hyperlink{fbi-transform}{Section \ref{fbi-transform}} that say roughly that \(\tilde{\psi}_h\) locally resembles \(e^{-\frac{||\cdot - x_0||^2}{2h}} e^{\frac{i}{h} \langle (\cdot - x_0), \xi_0 \rangle}\). The coherent states are essentially Schwartz kernels of an FBI transform defined on manifolds \citep{wunsch2001fbi} and since the latter are well-known to essentially \emph{diagonalize} \(\PDO\)s over \(T^*\mathcal{M}\), we can recover the flow of symbols through the diagonal matrix elements of \(\PDO\)s in the basis of coherent states. That is, if \(\psi_h := \tilde{\psi}_h / ||\tilde{\psi}_h||_{L^2}\) is a normalized coherent state, then the flow is approximated by \(\langle \hat{O}_t \psi_h, \psi_h\rangle_{L^2}\) in the sense that \(|\langle \hat{O}_t \psi_h, \psi_h\rangle_{L^2} - a \circ \Theta_q^t(x_0, \xi_0)| \leq C h\) for some constant \(C > 0\) and \(h \in [0, h_0)\) (see \protect\hyperlink{semi-classical-measures-of-coherent-states}{Section \ref{semi-classical-measures-of-coherent-states}}). This relationship between operator dynamics in the space of \(\PDO\)s and Hamiltonian mechanics on \emph{phase space} \(T^*\mathcal{M}\) is an instance of \emph{quantum-classical correspondence}.

The order one \(\PDO\) \(h \sqrt{\Delta}_{\mathcal{M}}\) quantizes \(|\xi|_{g_x}\) and therefore, the geodesic flow is quantized by \(U^t := e^{-i t \sqrt{\Delta_{\mathcal{M}}}}\), which can be defined by spectral theory. While previous works (\emph{e.g.}, \citep{hein2005graphs}) have shown that \(\GLap_{\lambda,\epsilon} = \Delta_{\mathcal{M}} + O(\partial^1) + O(\epsilon)\) and it is certainly true that Egorov's theorem is agnostic to the \(O(\partial^1)\) terms since the semiclassical symbol of \(h \sqrt{\GLap_{\mathcal{M}} + O(\partial^1)}\) is just \(|\xi|_{g_x}\), the method of approximation goes through a Taylor series expansion so that the \(O(\epsilon)\) term consists of higher order operators. This indeed raises an issue in both, including \(h^2 \GLap_{\lambda,\epsilon}\) in the class of \(\PDO\)s and quantizing the geodesic flow through it, which can briefly be put as: the semiclassical parameter \(h > 0\) sets the wavelength at which the pseudodiffferential calculus resolves functions, so if the diffusion scale \(\sqrt{\epsilon}\) is greater than the semiclassical wavelength \(h\), then the corresponding symbol is concentrated \emph{below} this wavelength and hence, its content goes undetected by the quantization procedure. This perspective is expounded in more detail and drives the analysis in \protect\hyperlink{from-graph-laplacians-to-geodesic-flows}{Section \ref{from-graph-laplacians-to-geodesic-flows}}.

A perhaps more direct way to see the issue in approximating \(U^t\) with \(U_{\lambda,\epsilon}^t := e^{-i t \sqrt{\GLap}_{\lambda,\epsilon}}\) is the following: it is well-known from microlocal analysis, or indeed discernible from Egorov's theorem that \(U^t\) propagates wavepackets localized in phase space along geodesics. That is, if the FBI transform (or physically, the \emph{Husimi function}, which is its squared modulus) at wavelength \(h\) of \(u_h \in C^{\infty}(\mathcal{M} \times (0, h_0)_h)\) is localized to a \(\sqrt{h}\) neighbourhood of \((x_0, \xi_0) \in T^*\mathcal{M}\), then for \(t\) smaller than the injectivity radius at \(x_0\), the Husimi function of \(U^t[u_h]\) is localized to roughly a \(\sqrt{h}\) neighbourhood of \(\Gamma^t(x_0, \xi_0)\). On the other hand, by functional calculus, we find that
\begin{equation} \label{eq:prop-plwp-glap-decomp}
U_{\lambda,\epsilon}^t[u_h] = f(0) u_h + A_{\lambda,\epsilon}[Df(A_{\lambda,\epsilon})[u_h]],
\end{equation}
wherein \(f(z) := e^{-i t c \sqrt{(1 - z)/\epsilon}}\) for a constant \(c > 0\) and \(Df := (f(z) - f(0))/z\). Recall that for \(\lambda = 0\), \(A_{0,\epsilon} = A_{\epsilon}\) represents a Markov process with transition probabilities being close to one in roughly \(\sqrt{\epsilon}\)-balls and decaying exponentially outside of them; in fact, this is true for all \(\lambda \geq 0\). Therefore, the second term of the right-hand side of \(\eqref{eq:prop-plwp-glap-decomp}\) will be diffused to at least a \(\sqrt{\epsilon}\)-ball and we can calculate that the Husimi function at wavelength \(h\) of the Schwartz kernel for \(\GAve_{\lambda,\epsilon} : u(\cdot) \mapsto \int_{\mathcal{M}} k_{\lambda,\epsilon}(\cdot,y) u(y) ~ p(y) d\nu_g(y)\) (we assume \(P\) has smooth density \(p\) with respect to the volume form \(\nu_g\) on \(\mathcal{M}\)) at a fixed \(y \in \mathcal{M}\) is localized in phase space to balls of radius roughly \(h/\sqrt{\epsilon}\) about \(0 \in T_{\alpha_x}^*\mathcal{M}\) for all \(\alpha_x\) in a neighbourhood about \(y\). Hence, if \(\epsilon \gtrsim h^2\), then for \(h\) sufficiently small, the Husimi function for \(U_{\lambda,\epsilon}^t[u_h]\) will retain the support of its initial condition at all times, which is in stark contrast to the behaviour of \(U^t[u_h]\). On the other hand, if \(\epsilon \ll h^2\), then the Husimi functions of the two terms on the right-hand side of \(\eqref{eq:prop-plwp-glap-decomp}\) can interact and there is chance for cancellations to occur.

In \protect\hyperlink{from-graph-laplacians-to-geodesic-flows}{Section \ref{from-graph-laplacians-to-geodesic-flows}}, we find that essentially \(\epsilon \ll h^2\) is also sufficient for \(h^2 \GLap_{\lambda,\epsilon}\) to be a \(\PDO\). Its symbol approximates \(|\xi|_{g_x}^2\) within roughly a ball of radius \(h/\sqrt{\epsilon}\) about the zero section in \(T^*\mathcal{M}\) and outside of this, it is of order zero for any fixed \(h\). Still, as a \emph{semiclassical} \(\PDO\), this makes the graph Laplacian an order \emph{two} \(\PDO\), but clearly not elliptic. We find that on application to coherent states \(\psi_h\), due to their localizing of \(\PDO\)s to \(\sqrt{h}\)-balls in phase space, we can approximately quantize \(\Gamma^t\) with \(U_{\lambda,\epsilon}^t\). Indeed, in \protect\hyperlink{thm:sym-cs-glap-psido}{Theorem \ref{thm:sym-cs-glap-psido}} we show under the assumptions

\begin{enumerate}
\def\labelenumi{\arabic{enumi}.}
\tightlist
\item
  \(\mathcal{M}\) is a compact, boundaryless \(C^{\infty}\) manifold and
\item
  \(P\) is a probability distribution on \(\embedM\) so that with respect to the volume form \(\nu_g\) on \(\mathcal{M}\), there is a positive density \(p := dP \circ \iota/d\nu_g \in C^{\infty}\)
\end{enumerate}

\noindent that the following quantum-classical correspondence holds:

\begin{theorem*}[{\emph{Quantum-Classical Correspondence} for the graph Laplacian}] Let \(\lambda \geq 0\) and \(h \in (0, 1]\). Then, for all \(\alpha \geq 1\), \(h^2 \GLap_{\lambda,h^{2 + \alpha}}\) is a semiclassical \(\PDO\) in \(h^0 \Psi^2\) as defined in \protect\hyperlink{quantization-and-symbol-classes}{Section \ref{quantization-and-symbol-classes}}, whose symbol in a fixed neighbourhood of \(0 \in T^*\mathcal{M}\) is \(|\xi|_{g_x}^2 + O(h)\). Furthermore, let \(\psi_h\) be an \(L^2\) normalized coherent state localized at \((x_0, \xi_0) \in T^*\mathcal{M} \setminus 0\) and denote by \(\Gamma^t : T^*\mathcal{M} \to T^*\mathcal{M}\) the (co-)geodesic flow. Then, there exists \(h_0 > 0\) such that given \(a \in C^{\infty}(T^*\mathcal{M})\) belonging to the symbol class \(h^0 S^0\) with its \emph{quantization} \(A := \Op_h(a) \in h^0 \Psi^0\) as per \protect\hyperlink{quantization-and-symbol-classes}{Section \ref{quantization-and-symbol-classes}} and \(|t|\) smaller than the injectivity radius at \(x_0\), we have for all \(h \in (0, h_0]\),
\begin{equation}\begin{aligned}
\langle U_{\lambda,h^{2 + \alpha}}^{-t} A U_{\lambda,h^{2 + \alpha}}^t \psi_h(\cdot ; x_0, \xi_0), \psi_h(\cdot ; x_0, \xi_0) \rangle_{L^2(\mathcal{M})} = a \circ \Gamma^t(x_0, \xi_0) + O(h).
\end{aligned}  \nonumber  \end{equation}

\end{theorem*}

As an immediate application of this, we study in \protect\hyperlink{propagation-of-coherent-states}{Section \ref{propagation-of-coherent-states}} the localization properties of the propagation \(U_{\lambda,\epsilon}^t[\psi_h]\) of a coherent state. Namely, we let \(A\) be the multiplication operator by a smooth approximation to a point mass at \(x \in \mathcal{M}\), say \(\rho_{x,\varepsilon}\) and consider the \emph{density} \(|U_{\lambda,\epsilon}^t[\psi_h]|^2(x) = \lim_{\varepsilon \to 0} \langle \rho_{x,\varepsilon} U_{\lambda,\epsilon}^t \psi_h, U_{\lambda,\epsilon}^t \psi_h \rangle\). By appealing to FBI transforms, we arrive at a characterization in \protect\hyperlink{prop:glap-prop-cs-localized}{Proposition \ref{prop:glap-prop-cs-localized}} that gives a rather explicit form of this function, up to a term that is \(L^{\infty}\) bounded by \(Ch\) for some \(C > 0\). It further tells that the propagated state is localized to roughly a \(\sqrt{h}\)-ball about \(x_t\), the projection onto \(\mathcal{M}\) of \(\Gamma^t(x_0, \xi_0)\).

These properties are put to use in \protect\hyperlink{from-graphs-to-manifolds}{Section \ref{from-graphs-to-manifolds}}, wherein we study the convergence of the discrete counterparts to the preceding constructions in the continuum setting. Namely, the propagator we use is \(U_{\lambda,\epsilon,N}^t := e^{-i t \sqrt{\GLap}_{\lambda,\epsilon,N}}\), which is defined by functional calculus and acts on \(\mathcal{H}_N\), which is equipped with the inner product \(\langle u, v \rangle_N := \frac{1}{N} \sum_{j=1}^N u(x_j) \overline{v(x_j)}\) that gives the norm \(||\cdot||_N := \langle \cdot, \cdot \rangle_N^{\frac{1}{2}}\). This vector space structure limits to \(L^2(\mathcal{M}, p d\nu_g)\) as \(N \to \infty\), so it becomes important to normalize vectors to reduce the effect of the sampling density \(p\). The matrix \(U_{\lambda,\epsilon,N}^t\) is not unitary, so the initial state is a \emph{time-dependent} normalization of the coherent state \(\psi_h\), that is, we set \(\psi_{h,t,N} := \psi_h / ||U_{\lambda,\epsilon,N}^t [\psi_h]||_N\). Then, under the additional assumption that \(\embedM_N\) is a set of \emph{i.i.d.} random vectors on \(\mathbb{R}^{\rdim}\) with law \(P\) and \(\iota\) has a bounded second fundamental form, in \protect\hyperlink{thm:mean-prop-geoflow-consistency}{Theorem \ref{thm:mean-prop-geoflow-consistency}} we arrive at the following consistency result:

\begin{theorem*}[{Quantum-Classical Correspondence for the discrete graph Laplacian}] Let \(\lambda \geq 0\), \(\alpha \geq 1\), \((x_0, \xi_0) \in T^*\mathcal{M}\) and \(|t|\) be smaller than the injectivity radius at \(x_0\). Then, for \(\psi_h\) a coherent state localized about \((x_0, \xi_0)\) and given \(u \in C^{\infty}\), there are constants \(h_0, C > 0\) such that we have for all \(h \in (0, h_0]\) and \(\epsilon := h^{2 + \alpha}\),
\begin{equation}\begin{aligned}
\Pr [| \langle |U_{\lambda,\epsilon,N}^t[\psi_{h,t,N}](x_j)|^2, u\rangle_N - u(x_t) | > h ] \leq e^{-\frac{N h^{2(\mdim + 2)} \epsilon^{\frac{5}{2}\mdim + 4}}{C}} .
\end{aligned}  \nonumber  \end{equation}

\end{theorem*}

As an application of this, we recover in \(\mathcal{X}_N := \iota^{-1}[\embedM_N]\) the geodesic flow on \(\mathcal{M}\). That is, using \emph{coordinate functions} with either local coordinates or the extrinsic coordinates \(\iota_1, \ldots, \iota_{\rdim}\), we find in \protect\hyperlink{prop:mean-geodesic-recover-consistency}{Proposition \ref{prop:mean-geodesic-recover-consistency}} that,

\begin{proposition*}[{Observing Geodesics: the \(\operatorname{mean}\) case}] Let \(\psi_h\) be a coherent state localized at \((x_0, \xi_0) \in T^*\mathcal{M}\), \(|t|\) less than the injectivity radius at \(x_0\), \(\lambda \geq 0\) and \(\alpha \geq 1\).

\noindent
\emph{Extrinsic case}. Define the \emph{global extrinsic sample mean} \(\bar{x}_{N,\iota,t}\) to be the closest point in \(\embedM_N\) to
\begin{flalign*}
&&& \bar{\iota}^t_N(x_0, \xi_0) := (\bar{\iota}_{N,1}^t(x_0, \xi_0), \ldots, \bar{\iota}_{N,\rdim}^t(x_0, \xi_0)), &  \\
& \text{with } && \bar{\iota}_{N,j}^t(x_0, \xi_0) := \langle \, |U_{\lambda,\epsilon,N}^t[\psi_{h,t,N}]|^2 \,, \iota_j \rangle_N . &
\end{flalign*}
That is, \(\bar{x}_{N,\iota,t} := \arg\min_{X \in \Lambda_N} ||X - \bar{\iota}^t_N(x_0, \xi_0)||_{\mathbb{R}^{\rdim}}\). Then, there are constants \(h_{\iota,\max} > 0\) and \(C \geq 1\) such that for all \(h \in (0, h_{\iota, \max}]\) and \(\epsilon := h^{2 + \alpha}\),
\begin{equation}\begin{aligned}
\Pr[d_g(\iota^{-1}(\bar{x}_{N,\iota,t}),x_t) > h] \leq e^{-\frac{N h^{2(\mdim + 2)} \epsilon^{\frac{5}{2}\mdim + 4}}{C}} .
\end{aligned}  \nonumber  \end{equation}
\emph{Local case}. Let \(\mathscr{O}_t \subset \mathcal{M}\) be an open neighbourhood about a maxmizer \(\hat{x}_{N,t}\) of \(|U_{\lambda,\epsilon,N}^t[\psi_h]|^2\) over \(\mathcal{X}_N\) and \(u : \mathscr{O}_t \to V_t \subset \mathbb{R}^{\mdim}\) its \(C^{\infty}\) diffeomorphic coordinate mapping. Given a smooth cut-off \(\chi \in C_c^{\infty}(\mathbb{R}^{\mdim}, [0, 1])\) with \(\operatorname{supp} \chi \subset V_t\), define the \emph{local sample mean} \(\bar{x}_{N,u,t}\) to be the closest point in \(\mathscr{V}_N := u[\mathcal{X}_N \cap \mathscr{O}_t]\) to
\begin{flalign*}
&&& \bar{u}^t_N(x_0, \xi_0) := (\bar{u}_{N,1}^t(x_0, \xi_0), \ldots, \bar{u}_{N,\mdim}^t(x_0, \xi_0)),    & \\
& \text{with } && \bar{u}_{N,j}^t(x_0, \xi_0) := \langle \,|U_{\lambda,\epsilon,N}^t[\psi_{h,t,N}]|^2 , (\chi \circ u) \, u_j \, \rangle_N . &
\end{flalign*}
That is, \(\bar{x}_{N,u,t} := \arg\min_{X \in \mathscr{V}_N} ||X - \bar{u}_N^t(x_0, \xi_0)||_{\mathbb{R}^{\mdim}}\). Then, there are constants \(h_{u,\max} > 0\) and \(C \geq 1\) such that for all \(h \in (0, h_{u,\max}]\), if \(\overline{\mathscr{B}}_t := \{||u(\hat{x}_{N,t}) - v||_{\mathbb{R}^{\mdim}} \leq \sqrt{h} \} \subset V_t\), then for any smooth cut-off \(\chi \in C_c^{\infty}(\mathbb{R}^{\mdim}, [0, 1])\) with \(\operatorname{supp} \chi \subset V_t\) that is \(\chi \equiv 1\) on \(\overline{\mathscr{B}}_t\), we have with \(\epsilon := h^{2 + \alpha}\),
\begin{equation}\begin{aligned}
\Pr[d_g(u^{-1}(\bar{x}_{N,u,t}),x_t) > h] \leq e^{-\frac{N h^{2(\mdim + 2)} \epsilon^{\frac{5}{2}\mdim + 4}}{C}} .
\end{aligned}  \nonumber  \end{equation}

\end{proposition*}

On the way to these consistency results, we have studied in \protect\hyperlink{from-graphs-to-manifolds}{Section \ref{from-graphs-to-manifolds}} the consistency between certain matrix dynamics on \(\mathcal{H}_N\) and operator dynamics on \(C^{\infty}(\mathcal{M})\). An intermediary result of independent interest is that for \(u \in C^{\infty}\), a natural \emph{extension} of \(U_{\lambda,\epsilon,N}^t[u]\) to \(\mathcal{M}\) is, with high probability, close to \(U^t_{\lambda,\epsilon}[u]\) in \(L^{\infty}\) norm. The extension is realized by a discrete version of \(\eqref{eq:prop-plwp-glap-decomp}\), namely, for any \(x \in \mathcal{M}\) we have,
\begin{equation} \label{eq:prop-plwp-glapN-decomp}
U_{\lambda,\epsilon,N}^t[u](x) = f(0) u(x) + A_{\lambda,\epsilon,N}[Df(A_{\lambda,\epsilon,N})[u_h]](x).
\end{equation}
Here, \(Df(A_{\lambda,\epsilon,N})\) is an \(N \times N\) matrix over \(\mathcal{H}_N\) as defined by spectral calculus, while \(\GAve_{\lambda,\epsilon,N}\) naturally extends to a finite-rank operator on \(L^2(\mathcal{M})\), for example in the \(\lambda = 0\) case by the fact that \(\GAve_{\epsilon,N}[v](x) = \langle v, k_{\epsilon}(x, \cdot) \rangle_N/\langle k_{\epsilon}(x, \cdot), 1 \rangle_N\). Thus, \(\eqref{eq:prop-plwp-glapN-decomp}\) is essentially a Nyström extension to the manifold. With this formalism at hand, we find in \protect\hyperlink{thm:halfwave-soln}{Theorem \ref{thm:halfwave-soln}} that,

\begin{theorem*}[{Consistency of Half-Wave Propagations}] Let \(\lambda \geq 0\) and \(\epsilon \in (0, 1]\). Then, we have for any \(t \in \mathbb{R}\) and \(u \in L^{\infty}(\mathcal{M})\) that if there is \(K_u > 0\) such that for all \(|s| \leq |t|\), \(||U_{\lambda,\epsilon}^s[u]||_{\infty} \leq K_u\), then there are constants \(C, C_0 \geq 1\) such that for all \(\frac{K_u^{\frac{1}{2}} \epsilon^{-(\frac{5}{8} \mdim + 1)}}{C_0 N^{\frac{1}{4}}} \leq \delta < C_0\),
\begin{align*}
\Pr&[||U_{\lambda,\epsilon,N}^t[u] - U_{\lambda,\epsilon}^t[u]||_{\infty} > \delta] \leq \exp{\left( -\frac{N \delta^4 \epsilon^{{\frac{5}{2}\mdim + 4}}}{C K_u^2 \, |t|^8} \right)} .
\end{align*}

\end{theorem*}

\noindent In connection to the semiclassical nature of \(\GLap_{\lambda,\epsilon}\) for \(\epsilon \ll h^2\), it may be possible to have a generic bound \(K_u\) by appealing to semiclassical Strichartz-type estimates to arrive at an \(L^1 \to L^{\infty}\) bound on \(U^t_{\lambda,\epsilon}\), as used for example in the proof of \citep[Theorem 10.8]{zworski2012}. Alternatively, a bound is attainable directly from inspection of \(\eqref{eq:prop-plwp-glap-decomp}\) combined with spectral theory and we give that in the Corollary to \protect\hyperlink{thm:halfwave-soln}{Theorem \ref{thm:halfwave-soln}}.

The discrete results we present justify the use of spectral manipulations of the discrete graph Laplacian for approximating solutions to wave equations and ultimately, geodesic flows, on \(\mathcal{M}\) through linear computations on the dataset \(\embedM_N\). In \protect\hyperlink{observing-geodesics-on-graphs}{Section \ref{observing-geodesics-on-graphs}} we indicate an algorithm of this sort for approximating geodesics of \(\mathcal{M}\) in \(\embedM_N\). More careful constructions, especially with regards to the coherent states, are deployed in \citep{qml}, where several examples are given, both for model cases of a sphere and torus as well as for real-world datasets and convergence rates are shown. The consistency analysis we study is independent of arguments from the spectral convergence of graph Laplacians to \(\Delta_{\mathcal{M}}\), which is currently a very active topic, both in the probabilistic \citep{trillos2020error, trillos2018variational, wormell2021spectral, cheng2020spectral} and deterministic \citep{burago2015graph} settings; an interesting direction would be to apply the consistency of half-wave propagations in \protect\hyperlink{thm:halfwave-soln}{Theorem \ref{thm:halfwave-soln}} to the study of these spectral problems.

\textbf{Outline}. We proceed as follows: in \protect\hyperlink{preliminaries}{Section \ref{preliminaries}} we state the assumptions on \(\mathcal{M}, \iota, P\) and the \emph{kernel function} \(k\) used to construct the graph Laplacians. Then, we state some basic geometric lemmas that will be used later and come to the definitions of the graph Laplacians. We give some of their basic spectral properties and in \protect\hyperlink{consistency-bounds-redux}{Section \ref{consistency-bounds-redux}} we re-prove a well-known result on the pointwise consistency of the application of discrete averaging operators \(A_{\lambda,\epsilon,N}\) to \(L^{\infty}\) functions, following \citep{hein2005geometrical}, but give more explicit details on the dependence on the \(L^{\infty}\) norm of the function. This allows further convergence results to be stated in terms of these norms, which may depend on \(\epsilon\), as is the case for example with the coherent states we use later. Then, we provide some background on \(\PDO\)s, mainly to state the forms of quantization we consider and set some notation and in \protect\hyperlink{fbi-transform}{Section \ref{fbi-transform}} we recall the notion of an FBI transform on a smooth, compact manifold as explicated in \citep{wunsch2001fbi}. In \protect\hyperlink{semi-classical-measures-of-coherent-states}{Section \ref{semi-classical-measures-of-coherent-states}} we give an explicit form to the Husimi function of a coherent state; this simply follows from the Schwartz kernel of the operator that projects functions in \(L^2(T^*\mathcal{M})\) onto the range of the FBI transform and indeed, we follow the discussion in \citep{wunsch2001fbi}. The results in this section are likely well-known, but to the best of the author's knowledge, not explicitly written down in literature and perhaps not directly accessible to a wider audience interested in the topics of this study, so we provide them in some detail. We briefly discuss in \protect\hyperlink{state-preparation}{Section \ref{state-preparation}} some practical considerations in the construction of coherent states from the data, \(\embedM_N\); this program is carried out in more detail in \citep{qml}.

We study the connection between semiclassical analysis and graph Laplacians in \protect\hyperlink{from-graph-laplacians-to-geodesic-flows}{Section \ref{from-graph-laplacians-to-geodesic-flows}}. Then, we study the propagation of coherent states in \protect\hyperlink{propagation-of-coherent-states}{Section \ref{propagation-of-coherent-states}}; the main discussion follows general considerations of properties of \(|e^{-\frac{i}{h} t Q} [\psi_h]|^2\) with \(Q\) a \(\PDO\) whose symbol locally approximates \(|\xi|_{g_x}\). This is again a topic that is likely part of folklore, but to the author's knowledge, not explicated in literature. Since by the discussion in \protect\hyperlink{from-graph-laplacians-to-geodesic-flows}{Section \ref{from-graph-laplacians-to-geodesic-flows}}, the propagation \(|U_{\lambda,\epsilon}^t[\psi_h]|^2\) is essentially reduced to this situation, we quickly conclude the localization properties of this propagation. In \protect\hyperlink{from-graphs-to-manifolds}{Section \ref{from-graphs-to-manifolds}} we study the discrete problems over \(\embedM_N\) and establish probabilistic convergence rates to bridge the gap to the continuum, \(N \to \infty\) setting.

\hypertarget{preliminaries}{%
\section{Preliminaries}\label{preliminaries}}

We will discuss the necessary results for manifold learning and semiclassical analysis. To set the stage, there is an \emph{unknown manifold} \(\mathcal{M}\), on which we wish to implement operators and dynamics and whose geometry we wish to learn; we always operate under the following assumptions:

\begin{assumptions*}	\hypertarget{assumptions}{\label{assumptions}} The manifold \(\mathcal{M}\) is \(C^{\infty}\), compact, boundaryless, of dimension \(\mdim := \dim \mathcal{M}\) with Riemannian metric \(g\) giving the geodesic distance function \(d_g : \mathcal{M} \times \mathcal{M} \to \mathbb{R}_+\) and embedded as \(\embedM \subset \mathbb{R}^{\rdim}\) through a smooth isometry \(\iota : \mathcal{M} \isomto \embedM \subset \mathbb{R}^D\) with \(\rdim \geq \mdim+1\). The second fundamental form of \(\embedM\) (see \citep[Definition 2.11]{hein2005geometrical}) has bounded norm and
\begin{equation}\begin{aligned}
\kappa := \inf_{x \in \mathcal{M}} \inf_{y \in B_{\pi \rho}(x,g)} |\iota(x) - \iota(y)| > 0,
\end{aligned}  \nonumber  \end{equation}
wherein \(\rho := \inf \{ \gamma \text{ unit-speed geodesic in $\mathcal{M}$} ~|~ |\partial_t^2 [\iota \circ \gamma]| \}\), termed the \emph{minimum radius of curvature}, is positive due to \(\embedM\) having bounded second fundamental form (see \citep[\(\S 2.2.1\)]{hein2005geometrical}). Samples \(X_1, \ldots, X_N\) are drawn \emph{i.i.d.} from a fixed probability distribution \(P\) concentrated on \(\Lambda\) such that \(P \circ \iota^{-1}\) is absolutely continuous with respect to the volume measure on \(\mathcal{M}\) and has smooth, positive density \(p \in C^{\infty}(\mathcal{M})\).

\end{assumptions*}

\begin{remark} The conditions on \(\Lambda\) give control on its extrinsic geometry through the intrinsic geometry of \(\mathcal{M}\): in particular, \citep[Lemma 2.22]{hein2005geometrical} tells that for all \(x, y \in \mathcal{M}\) such that \(|\iota(x) - \iota(y)| \leq \kappa/2\), we have \(d_g(x,y)/2 \leq |\iota(x) - \iota(y)| \leq d_g(x,y) \leq \kappa\).

\end{remark}

A common way to think of the samples \(X_1, \ldots, X_N\) is through a \emph{graph structure} with vertices \(\mathcal{X}_N := \{ x_1, \ldots, x_N \}\), wherein \(x_j := \iota^{-1}(X_j)\) for each \(j \in [N]\). In this setting, the graph connectivity is given by an adjacency matrix that uses the embedding \(\iota\) and is defined as one of the \emph{discrete averaging operators} \(A_{\lambda,\epsilon,N}\) defined in \protect\hyperlink{laplacians-from-graphs-to-manifolds}{Section \ref{laplacians-from-graphs-to-manifolds}}. We can also view \(\mathcal{X}_N\) as a set of basis elements giving rise to a vector space structure and endow it with an inner product in which these are orthogonal. Then, the restriction of \(C^{\infty}\) functions to \(\mathcal{X}_N\) can be seen as vectors in this space and with appropriate weighting schema, the discrete structures that emerge, begin to parallel certain weighted \(L^2\) spaces. Such representations form the basis from which discrete constructions --- particularly functions of averaging operators that encode certain dynamics --- are extended to the full, \emph{continuum} space \(\mathcal{M}\). The basic principles of this point of view are treated in \citep{hein2005geometrical} and briefly in \citep{hein2005graphs}; we will utilise it primarily in \protect\hyperlink{from-graphs-to-manifolds}{Section \ref{from-graphs-to-manifolds}}.

The general notation used throughout the following sections is summarized in the Appendix and more specific notation is listed there for quick reference; when some notation is introduced for the time, it will be defined in context.

\hypertarget{geometric-properties}{%
\subsection{Geometric properties}\label{geometric-properties}}

We will often need to switch to normal coordinates and expand them in Taylor series about a point of interest. Affecting this, we have a classical theorem of Riemann:

\begin{lemma}[{Expansion of Normal Coordinates}]	\hypertarget{lemma:expansion-normal-coords}{\label{lemma:expansion-normal-coords}} Let \(x \in \mathcal{M}\) and \(V_x \subset \mathbb{R}^{\mdim}\) be a neighbourhood about the origin providing normal coordinates about \(x\). Then, for \(\exp_x : V_x \to \mathcal{M}\) defined via the identification \(T_x\mathcal{M} \cong \mathbb{R}^{\mdim}\), we have \(D \exp_x(v) = I_n + O(|v|^2)\) and \((g \circ \exp_x)(v) = I_n + O(|v|^2)\).

\qed

\end{lemma}

In practice, direct access to normal coordinates is seldom achievable. However, projection onto the tangent space at a point on \(\embedM\) also provides local coordinates and for \emph{shrinking} neighbourhoods, this approximates normal coordinates reasonably well:

\begin{lemma}[{Normal to Projection Coordinates}]	\hypertarget{lem:proj-normal-coords}{\label{lem:proj-normal-coords}} Let \(\mathcal{U} \subset \mathcal{M}\) be a neighbourhood and \(V_x \subset \mathbb{R}^{\mdim} \cong T_{\iota(x)}\Lambda\) for some \(x \in \mathcal{U}\). If \(\gamma : \mathcal{U} \to V_x\) are local coordinates provided by the orthogonal projection \(\iota(\mathcal{U}) \to V_x\) and \(s_x(y) := \exp^{-1}_x(y)\), then \(s_x \circ \gamma^{-1}(v) = v + O(|v|^3)\) and \(D[s_x \circ \gamma^{-1}](v) = I + O(|v|^3)\).

\end{lemma}

\begin{proof}

This is \citep[Lemma 5]{lafon2004thesis}.
\end{proof}

Another option is to work directly with extrinsic coordinates provided by the embedding map \(\iota : \mathcal{M} \to \embedM\). To this end, we have the following expansion, useful for integration formulae:

\begin{lemma}[{Extrinsic to Normal Coordinates}]	\hypertarget{lem:ext-normal-coords}{\label{lem:ext-normal-coords}} Let \(\mathcal{U} \subset \mathcal{M}\) be a neighbourhood about \(x \in \mathcal{M}\) contained inside the normal neighbourhood of \(x\) and \(s_x(y) := \exp^{-1}_x(y)\). Then,

\begin{enumerate}
\def\labelenumi{\arabic{enumi}.}
\tightlist
\item
  \(D[\iota \circ s_x^{-1}](0) = D[\iota](x)\) and
\item
  for all \(y \in \mathcal{U}\), \(|\iota(x) - \iota(y)|^2 = d_g(x,y)^2 + O(d_g(x,y)^4)\).
\end{enumerate}

\end{lemma}

\begin{proof}

These are both shown in the \citep[Proof of Lemma 2.1]{antil2018fractional}.
\end{proof}

We will also need the following information on the growth of covering numbers on compact manifolds using balls of decaying radii:

\begin{lemma}[{Bishop-Günther inequality}]	\hypertarget{lem:covering-number}{\label{lem:covering-number}} Let \(\rho > 0\) be fixed and define \(\mathcal{N}(\rho)\) to be the minimal number of (geodesic) balls of radius \(\rho\) that cover \(\mathcal{M}\). Then, there exist positive constants \(C_{\mathcal{M}}\) and \(C'_{\mathcal{M}}\) depending only on the dimension, volume, sectional curvature and injectivity radius of \(\mathcal{M}\) such that if \(0 < \rho \leq C'_{\mathcal{M}}\), then \(\mathcal{N}(\rho) \leq C_{\mathcal{M}} \rho^{-\mdim}\).

\end{lemma}

\begin{proof}

The volume of \(B(x, \rho)\) for \(x \in \mathcal{M}\) is \(V(x) \geq \rho^{\mdim}/C''_{\mathcal{M}}\) when \(\rho \leq C'_{\mathcal{M}}\) for some constants \(C'_{\mathcal{M}}, C''_{\mathcal{M}} > 0\) independent of \(x\) that are provided by the Bishop-Günther inequality and depend only on the sectional curvature, injectivity radius and dimension of \(\mathcal{M}\). Therefore, \(\mathcal{N}(\rho) \leq \operatorname{vol}(\mathcal{M})/(\inf V) \leq C_{\mathcal{M}} \rho^{-\mdim}\) with \(C_{\mathcal{M}} := \operatorname{vol}(\mathcal{M}) C''_{\mathcal{M}}\).
\end{proof}

\hypertarget{laplacians-from-graphs-to-manifolds}{%
\subsection{Laplacians: from graphs to manifolds}\label{laplacians-from-graphs-to-manifolds}}

As alluded to in the introduction, there are now well-established schema for approximating Laplace-Beltrami operators, modulo lower-order terms, from the Euclidean distances of samples \(X_1, \ldots, X_N\) of a smooth isometric embedding of \(\mathcal{M}\) satisfying the \protect\hyperlink{assumptions}{Assumptions}. Namely, the objects of study are adjacency matrices of weighted graphs on these samples and taking on a certain form, which we call \emph{averaging operators}. We now fix the terminology:

\begin{definition}[{Kernel functions}] A \emph{kernel function} is a monotonically decreasing function \(k: \mathbb{R}_+ \to \mathbb{R}_+\) that is smooth on \((0,\infty)\) with all derivatives having exponential decay: that is, for each \(m \geq 0\), there are constants \(R_{k,m}, A_{k,m} > 0\) so that for all \(t > R_{k,m}\),
\begin{equation}\begin{aligned}
|\partial_t^m k(t)| \leq e^{-A_{k,m} t}
\end{aligned}  \nonumber  \end{equation}
and that satisfies \(k(t) > ||k||_{\infty}/2\) on \([0, r_k)\) for some \(r_k > 0\). We denote \(R_k := R_{k,0}\).

\end{definition}

Kernel functions are the basic building blocks for the graph structures leading to the operators of interest. Using them, we can make,

\begin{definition}[{Averaging operators}] Let there be \(N > 0\) random vectors \(X_1, \ldots, X_N \subset \Lambda\) with law \(P\) as per the \protect\hyperlink{assumptions}{Assumptions} and for each \(j \in [N]\), denote \(x_j := \iota^{-1}(X_j)\). let \(k: \mathbb{R}_+ \to \mathbb{R}_+\) be a kernel function. Then, fixing \(k : \mathbb{R}_+ \to \mathbb{R}_+\) a kernel function and letting \(k_{\epsilon} : \mathcal{M} \times \mathcal{M} \to \mathbb{R}\) be given by \(k_{\epsilon} : (x,y) \mapsto \epsilon^{-\frac{\mdim}{2}} k(|\iota(x) - \iota(y)|^2/\epsilon)\) for \(\epsilon > 0\), we define the \emph{discrete} and \emph{continuum averaging operators} as
\begin{equation} \label{def:averaging-op}
(\AvOp_{\epsilon,N})_{j,j'} := \frac{1}{N} k_{\epsilon}(x_j,x_{j'}),
\quad
\AvOp_{\epsilon} : u \mapsto \int_{\mathcal{M}} k_{\epsilon}(x,y) u(y) p(y) ~ d\vol{y} ,
\end{equation}
respectively. Further, with \(p_{\epsilon,N}(x) := \AvOp_{\epsilon,N}[1](x)\) and \(p_{\epsilon}(x) := \AvOp_{\epsilon}[1](x)\) we define
\begin{equation}\begin{aligned}
\GAve_{\epsilon,N} := \operatorname{diag}(p_{\epsilon,N}^{-1}) \AvOp_{\epsilon,N},    \quad
\GAve_{\epsilon} := \frac{1}{p_{\epsilon}} \AvOp_{\epsilon},
\end{aligned}  \nonumber  \end{equation}
called the \emph{discrete} and \emph{continuum} \emph{renormalized averaging operators}, respectively. More generally, let \(\lambda \geq 0\) and define \(k_{\lambda,\epsilon,N}(x,y) := k_{\epsilon}(x,y)/[p_{\epsilon,N}(x) p_{\epsilon,N}(y)]^{\lambda}\), \(k_{\lambda,\epsilon}(x,y) := k_{\epsilon}(x,y)/[p_{\epsilon}(x) p_{\epsilon}(y)]^{\lambda}\). With this,
\begin{equation}\begin{aligned}
(\AvOp_{\lambda,\epsilon,N})_{j,j'} := k_{\lambda,\epsilon,N}(x_j,x_{j'}), \quad
\AvOp_{\lambda,\epsilon} : u \mapsto \int_{\embedM} k_{\lambda,\epsilon}(x,y) u(y) ~ p(y)d\vol{y}
\end{aligned}  \nonumber  \end{equation}
are the \emph{discrete} and \emph{continuum} \(\lambda\)-\emph{averaging operators}, respectively and
\begin{equation}\begin{aligned}
\GAve_{\lambda,\epsilon,N} := \operatorname{diag}(\AvOp_{\lambda,\epsilon,N}[1]^{-1}) \, \AvOp_{\epsilon,\lambda,N},  \quad
\GAve_{\lambda,\epsilon} : u \mapsto (\AvOp_{\lambda,\epsilon}[1])^{-1} \AvOp_{\lambda,\epsilon}[u]
\end{aligned}  \nonumber  \end{equation}
are the \emph{discrete} and \emph{continuum} \(\lambda\)-\emph{renormalized averaging operators}, respectively.

\end{definition}

\begin{notation} We will denote for every \(\lambda \geq 0\), \(p_{\lambda,\epsilon} := \AvOp_{\lambda,\epsilon}[1]\).

\end{notation}

The main use of these operators is to approximate the Laplace-Beltrami operator on the manifold. Therefore, we also make the

\begin{definition}[{Graph Laplacians}] Let \(c_2\) and \(c_0\) be the second and zeroth order moments of \(k(||\cdot||^2)\) on \(\mathbb{R}^{\mdim}\). We define
\begin{equation}\begin{aligned}
\GLap_{\lambda,\epsilon,N} := \frac{2 c_0}{c_2} \frac{I - \GAve_{\lambda,\epsilon,N}}{\epsilon}, \quad\quad
\GLap_{\lambda,\epsilon} := \frac{2 c_0}{c_2} \frac{I - \GAve_{\lambda,\epsilon}}{\epsilon},
\end{aligned}  \nonumber  \end{equation}
which we call the \emph{discrete} and \emph{continuum} \(\lambda\)-\emph{renormalized} \emph{graph Laplacians}, respectively. Note that when \(\lambda = 0\), \(\GAve_{\lambda,\epsilon,N} = \GAve_{\epsilon,N}\) and \(\GAve_{\lambda,\epsilon} = \GAve_{\epsilon}\).

\end{definition}

\begin{remark} In using the above terminologies, we will often omit the \emph{discrete} and \emph{continuum} specifications as well as \emph{\(\lambda\)-renormalized} when the symbols and context make the regime sufficiently clear. The symbolic expressions are adapted from \citep{hein2005graphs}, while the terminology is non-standard.

\end{remark}

We will use the spectral properties of \(A_{\lambda,\epsilon,N}\) and \(A_{\lambda,\epsilon}\) and the simplest way to access those is to see that they are symmetrized by conjugation via \(\sqrt{p_{\lambda,\epsilon,N}}\) and \(\sqrt{p_{\lambda,\epsilon}}\) respectively. That is, we have,

\begin{lemma}	\hypertarget{lem:lap-symm}{\label{lem:lap-symm}} Let for all \(\lambda \geq 0\) and \(\epsilon > 0\),
\begin{equation}\begin{aligned}
A^{(s)}_{\lambda,\epsilon,N} := \sqrt{p_{\lambda,\epsilon,N}} A_{\lambda,\epsilon,N} \frac{1}{\sqrt{p_{\lambda,\epsilon,N}}}, \quad A^{(s)}_{\lambda,\epsilon} := \sqrt{p_{\lambda,\epsilon}} A_{\lambda,\epsilon} \frac{1}{\sqrt{p_{\lambda,\epsilon}}} .
\end{aligned}  \nonumber  \end{equation}
Then,

\begin{enumerate}
\def\labelenumi{\arabic{enumi}.}
\item
  \(A_{\lambda,\epsilon,N}^{(s)}\) and \(A_{\lambda,\epsilon}^{(s)}\) are symmetric and their spectra coincide with those of \(A_{\lambda,\epsilon,N}\) and \(A_{\lambda,\epsilon}\), respectively,
\item
  the spectrum of \(A_{\lambda,\epsilon}\) is contained in \([-1,1]\) with \(\lambda = 1\) a simple eigenvalue and all other eigenvalues lying strictly in \((-1, 1)\) and there are constants \(C_1, C_2 > 0\) such that for all \(\epsilon \in (0, C_1]\), the same holds for \(A_{\lambda,\epsilon,N}\) with probability at least \(1 - C_2 \epsilon^{-\frac{\mdim}{2}} e^{-N \epsilon^{\frac{\mdim}{2}}/C_2}\),
\item
  \(A_{\lambda,\epsilon,N}^{(s)}\) and \(A_{\lambda,\epsilon}^{(s)}\) fix \(\sqrt{p_{\lambda,\epsilon,N}}\) and \(\sqrt{p_{\lambda,\epsilon}}\), respectively and
\item
  for any connected open subset \(D \subset \mathbb{C}\) with \([-1, 1] \subset \bar{D}\) and \(f : \bar{D} \to \mathbb{C}\) analytic on \(D\) with an absolutely convergent Taylor series on \([-1,1]\),
  \begin{equation}\begin{aligned}
  f(A_{\lambda,\epsilon,N}) = \frac{1}{\sqrt{p_{\lambda,\epsilon,N}}} f(A^s_{\lambda,\epsilon,N}) \sqrt{p_{\lambda,\epsilon,N}}, \quad f(A_{\lambda,\epsilon}) = \frac{1}{\sqrt{p_{\lambda,\epsilon}}} f(A^s_{\lambda,\epsilon}) \sqrt{p_{\lambda,\epsilon}}
  \end{aligned}  \nonumber  \end{equation}
  with \(\epsilon \in (0, C_1]\) and probability at least \(1 - C_2 \epsilon^{-\frac{\mdim}{2}} e^{-N \epsilon^{\frac{\mdim}{2}}/C_2}\) in the former case of application to \(A_{\lambda,\epsilon,N}\).
\end{enumerate}

\end{lemma}

\begin{proof}

The properties (1) and (3) follow immediately from the definitions of \(A_{\lambda,\epsilon,N}^{(s)}\) and \(A_{\lambda,\epsilon}^{(s)}\). The spectral property (2) for \(A_{\lambda,\epsilon}\) follows from the Krein-Rutman theorem along with the fact that it is a compact operator that maps \(1 \mapsto 1\) and non-negative, non-zero \(C^{\infty}\) functions to postive \(C^{\infty}\) functions: indeed, then for example \citep[Theorem 1.2, Chap. 1,][]{du2006order} applies.

In the case of \(A_{\lambda,\epsilon,N}\), the Perron-Frobenius theorem applies to give (2) as soon as the matrix gives a connected graph. This follows from covering \(\mathcal{M}\) by sufficiently small balls such that for any ball in the cover, \(k_{\epsilon}\) is bounded below on all nearest neighbour balls. Then, any two points can be reached by a path of balls in this cover such that \(k_{\epsilon} \gg 0\) on adjacent balls and with high probability, there is a sample point in every ball, which gives a connected graph. The details are as follows. Let \(x_1^*, \ldots, x_M^* \in \mathcal{M}\) be points such that \(\bigcup_{j \in [M]} B_j = \mathcal{M}\) with \(B_j := B_{\mathcal{M}}(x_j^*, \sqrt{r_k \epsilon}/8)\). Now for each \(j \in [M]\), let \(Y_j : \mathcal{M} \to \{ 0, 1 \}\) map \(Y_j[B_j] = 1\) and \(Y_j[\mathcal{M} \setminus B_j] = 0\). Each \(Y_j\) gives \(N\) \emph{i.i.d.} random variables when applied to the sample points \(x_1, \ldots, x_N\) and since \(\vol{B_{\mathcal{M}}(x_j, \sqrt{r_k \epsilon/8})} \approx (r_k \epsilon)^{\frac{\mdim}{2}}/8^{\mdim}\), this is a Bernoulli process with mean \(N \epsilon^{\frac{n}{2}}/C_{\mathcal{M}}\) for some constant \(C_{\mathcal{M}} > 0\). Therefore, by a Chernoff bound, \(\Pr[\#(\{ x_1, \ldots, x_N \} \cap B_j) \geq 1] > 1 - e^{-N \epsilon^{\frac{\mdim}{2}}/(2 C_{\mathcal{M}})}\) for each \(j \in [M]\). Hence, by a union bound, there is at least one sample point in each \(B_j\) with probability at least \(1 - M e^{-N \epsilon^{\frac{\mdim}{2}}/(2 C_{\mathcal{M}})}\). The Bishop-Günther inequality then tells that there is a constant \(C'_{\mathcal{M}} > 0\) such that for \(0 < \epsilon \leq C'_{\mathcal{M}}\), \(M \lesssim \epsilon^{-\frac{\mdim}{2}}\). Assuming this event it follows that for any \(\ell, \ell' \in [N]\), there are \(j, j' \in [M]\) such that \(d_g(x_{\ell}, x_{j}^*), d_g(x_{\ell'}, x_{j'}^*) \leq \sqrt{r_k \epsilon}/8\) and there are \(1 \leq j_1 \leq \cdots \leq j_m \leq M\) with \(m \leq M-2\) such that \(d_g(x_{j_{\gamma_1}}^*, x_{j_{\gamma_2}}^*) \leq \sqrt{r_k \epsilon}/4\) for all \(1 \leq \gamma_1 \leq \gamma_2 \leq m\) and \(d_g(x_j^*, x_{j_1}^*), d_g(x_{j'}^*, x_{j_m}^*) \leq \sqrt{r_k \epsilon}/4\). Furthermore, for each \(1 \leq \gamma \leq m\), there is \(\ell_{\gamma}\) such that \(d_g(x_{\ell_{\gamma}}, x_{j_{\gamma}}^*) \leq \sqrt{r_k \epsilon}/8\). Therefore, for all \(1 \leq \gamma_1 \leq \gamma_2 \leq m\), \(d_g(x_{\ell_{\gamma_1}}, x_{\ell_{\gamma_2}}) \leq \sqrt{r_k \epsilon}/2\) and \(d_g(x_{\ell}, x_{\ell_1}), d_g(x_{\ell'}, x_{\ell_m}) \leq \sqrt{r_k \epsilon}/2\). By the assumption on the kernel function \(k\) and the Assumptions on \(\mathcal{M}\) and \(\embedM\), whenever \(\epsilon < \kappa\) we have that \(k(|\iota(x_{\ell_1}) - \iota(x_{\ell_2})|^2/\epsilon) > ||k||_{\infty}/2\). Hence altogether we have a path \(x_{\ell} \to x_{\ell_1} \to \cdots \to x_{\ell_m} \to x_{\ell'}\) such the entries of \(\GAve_{\lambda,\epsilon,N}\) corresponding to neighbouring points along this path are bounded below by \(||k||/2 > 0\), which gives that the underlying graph is connected.

Now together with the holomorphic functional calculus, properties (1)-(3) imply property (4).
\end{proof}

A summary of the notation introduced here is provided in \protect\hyperlink{notation-related-to-graph-laplacians}{Appendix}.

\hypertarget{consistency-bounds-redux}{%
\subsubsection{Consistency bounds redux}\label{consistency-bounds-redux}}

The relationship among the discrete and continuum operators is that the discrete averaging operators are naturally accessed from samples of \(\embedM\) and these are shown to converge, in the \(N \to \infty\) limit to the continuum averaging operators; the residual coming from the difference of the two operators is deemed the \emph{variance term}. Then, a sequence of Taylor expansions shows that the (\(\lambda\)-renormalized) continuum graph Laplacians, derived from the continuum averaging operators, converge to \(\Delta_{\mathcal{M}} + \text{(lower order differential terms)}\) in the \(\epsilon \to 0\) limit and the corresponding residuals are deemed the \emph{bias terms}. A general treatment of the convergence of both terms, simultaneously, is given in \citep{hein2005graphs}. We will draw upon these results throughout the forthcoming analyses, so we record the precise forms of the particular theorems here.

In regards to the bias term, our requirements are modest: for the most part, we will need a semiclassical expansion of the operator, which in this case is similar to the Taylor series type expansions given, for example in \citep{hein2005graphs}, but requires an essentially different approach. These expansions will be carried out in \protect\hyperlink{symbol-of-a-graph-laplacian}{Section \ref{symbol-of-a-graph-laplacian}}. Nevertheless, the \emph{degree functions} \(p_{\lambda,\epsilon}\) have an important role in making the approximation procedures agnostic to \emph{a priori} knowledge of the intrinsic dimensionality \(\mdim\) of the manifold and controlling the effect of lower-order error terms in the approximation to the Laplace-Beltrami operator, primarily in as far as the effect of the non-uniform sampling density \(p\) is concerned. This is carried out essentially through their Taylor expansions about a given point \(x \in \mathcal{M}\); the following is recorded here for use in later sections:

\begin{lemma}	\hypertarget{lem:taylor-expand-deg-func}{\label{lem:taylor-expand-deg-func}} Let \(\lambda \geq 0\) and \(\epsilon > 0\). Then, there exists \(q_{\lambda} \in C^{\infty}(\mathcal{M})\) depending only on \(k, p, \lambda\) and \(\mathcal{M}\) such that for all \(x \in \mathcal{M}\),
\begin{equation}\begin{aligned}
p_{\lambda,\epsilon}(x) = (c_0 p(x))^{1 - 2\lambda} + \epsilon \frac{c_2}{2} p(x)^{1 - 2\lambda} q_{\lambda}(x) + O(\epsilon^2) .
\end{aligned}  \nonumber  \end{equation}

\end{lemma}

\begin{proof}

The proof of this expansion is contained in the proof of \citep[Proposition 2.33]{hein2005geometrical}.
\end{proof}

As for the variance term, we now recount some of the prior results, with small modifications that we will use for the convergence of quantum dynamics from graphs to manifolds. In particular, prior works have left out the explicit dependence on \(||u||_{\infty}\) in the probabilistic convergence bounds for consistency. Since we ultimately wish to drive localized, \(L^2\) bounded states according to Schrödinger-type dynamics, their \(L^{\infty}\) norms depend on \(\epsilon\), hence it is useful for the convergence theory to make this dependence explicit. Therefore, we proceed with the essential convergence results from the thesis \citep{hein2005geometrical}, while making the small adjustments necessary to \emph{unpack} the uniform norm dependence from the constants. The primary cases that concern us are for the averaging operators.

Since constants bounding the kernel, probability and geometric quantities appear abundantly in these considerations, we record the useful

\begin{notation} \emph{Bounds on degree functions} (see \citep[Lemma 2.32]{hein2005geometrical} for more detailed bounds):

\begin{itemize}
\tightlist
\item
  \(\underline{C}_p := \inf p_{\epsilon} \gtrsim ||k||_{\infty} (\inf p) \inf_x \operatorname{vol}(B(x, \epsilon^{\frac{1}{2}} R_k))\)
\item
  \(\overline{C}_p := \sup p_{\epsilon} \lesssim_{\mathcal{M}} R_k^{\mdim} ||k||_{\infty} ||p||_{\infty}\)
\item
  \(\underline{C}_{p,\lambda} := \inf p_{\lambda,\epsilon} \geq \underline{C}_p/\overline{C}_p^{2\lambda}\), \(\overline{C}_{p,\lambda} := ||p_{\lambda,\epsilon}||_{\infty} \leq \overline{C}_p/\underline{C}_p^{2\lambda}\)
\item
  \(\underline{C}_{k,p} := ||k||_{\infty}/(\inf p)\).
\end{itemize}

\end{notation}

Let the subscripted variables \(X_1, \ldots, X_j, \ldots, X_N\) denote \emph{i.i.d.} random vectors in \(\mathbb{R}^{\rdim}\) with law \(P\) that has positive \(C^{\infty}(\embedM)\) density \(\tilde{p}\) with respect to the volume element \(dV\) on \(\Lambda = \iota(\mathcal{M}) \subset \mathbb{R}^{\rdim}\); we denote by \(x_j := \iota^{-1}(X_j)\) the pull-back of \(X_j\) and by \(p\) the pull-back of \(\tilde{p}\) to \(\mathcal{M}\), which gives a density with respect to \(d\nu_g\). For simplicity, we assume \(k\) has compact support on \([0, R_k^2]\) (but this can be generalized to \emph{essential support} so that \(t^N k(t) < C_N\) for all \(t > R_k^2, N \geq 0\)). Then, Hein has shown,

\begin{lemma}[{\cite[Lemma 2.30]{hein2005geometrical}}]	\hypertarget{lem:kern-consistency}{\label{lem:kern-consistency}} There exists a constant \(K_{\mathcal{M}, k, p}\) depending on the curvature and injectivity radius of \(\mathcal{M}\), \(||p||_{\infty}\), \(||k||_{\infty}\) and \(R_k\) such that for all \(u \in L^{\infty}(\mathcal{M})\) and \(x \in \mathcal{M}\), given any \(\delta > 0\),
\begin{align*}
\Pr\left[ \left|\frac{1}{N} \sum_{j=1}^N k_{\epsilon}(x,x_j) u(x_j) - \int_{\mathcal{M}} k_{\epsilon}(x,y) u(y) p(y) ~ d\nu_g(y) \right| > \delta \right] \\
\leq 4 \exp\left(- \frac{N \epsilon^{\frac{n}{2}} \delta^2}{2||u||_{\infty}(\Kconst ||u||_{\infty} + ||k||_{\infty}\delta/3)} \right).
\end{align*}

\end{lemma}

\begin{proof}

The proof is exactly as in \citep[Lemma 2.30]{hein2005geometrical}, with the only change that we allow \(u\) to be complex-valued, which gives another factor of \(2\) for the probability, due to either applying the real-valued version in two parts, or the two-dimensional Matrix Bernstein inequality via the matrix representation of complex numbers.
\end{proof}

On extending this to the consistency of normalized averaging operators, one must account for the factors of \(p_{\epsilon,N}, p_{\epsilon}\) appearing in \(\GAve_{\epsilon,N}[u], A_{\epsilon}[u]\) respectively. As random variables, the \(k_{\epsilon}(x,x_j)/p_{\epsilon,N}(x)\) are now coupled, so the application of Bernstein's inequality requires a decoupling step. This comes in the form of the assumption that \(p_{\epsilon,N}(x_j) - p_{\epsilon}(x_j)\) is small for every \(j \in [N]\), which upon taking a union bound introduces a factor of \(N\) in the probability. Since this gives the same bound for all \(\lambda \geq 0\), Hein has shown altogether an Bernstein-type bound for the probabilistic consistency between \(A_{\lambda,\epsilon,N}[u](x)\) and \(A_{\lambda,\epsilon}[u](x)\) for all \(\lambda \geq 0\), \(u \in L^{\infty}\) and \(x \in \mathcal{M}\). However, the bounds have implicitly rolled up the dependence on \(||u||_{\infty}\) into constants, while in \protect\hyperlink{from-graphs-to-manifolds}{Section \ref{from-graphs-to-manifolds}} we will use this dependence explicitly so that we have a clear view of the convergence rates when applied to localized states. Therefore, we reproduce the proof of consistency while pulling out the relevant pieces of the constants and record it as,

\begin{lemma}[{\cite[Proposition 2.38]{hein2005geometrical}}]	\hypertarget{lem:avgop-consistent}{\label{lem:avgop-consistent}} Given \(\lambda \geq 0\), there is a constant \(K_{\lambda} > 1\) depending only on \(||k||_{\infty}\), \(\underline{C}_p\), \(\overline{C}_p\) and \(\lambda\) such that given \(u \in L^{\infty}\) and \(x \in \mathcal{M}\), we have for all \(\epsilon > 0\), and \(2 ||k||_{\infty} \epsilon^{-\frac{n}{2}}/N < \delta \leq \underline{C}_p/2\),
\begin{align}
\begin{split} \label{eq:thm-avgop-consistent-bound}
\Pr&[|A_{\lambda,\epsilon,N}[u](x) - A_{\lambda,\epsilon}[u](x)| > \delta]    \\
  &\quad\leq (4 + 2N) \exp\left( -\frac{N \epsilon^{\frac{n}{2}} \delta^2}{2 K_{\lambda} \underline{C}_p^{-\lambda} ||u||_{\infty}(\Kconst  K_{\lambda} ||u||_{\infty} \underline{C}_p^{-\lambda} + ||k||_{\infty}\delta/3)} \right).
\end{split}
\end{align}

\end{lemma}

\begin{remark} The choice of \(\delta < \underline{C}_p / 2\) was made for simplicity; a different choice of \(\delta < \underline{C}_p\) changes only \(K_{\lambda}\) (in the proof, \(K_{\lambda,k,p}\)) as a rational function of \(\delta\). For states that are localized in tandem with the operators --- as will be our case --- only the first bound is useful, since \(||u||_{\infty}\) will grow as \(\epsilon\) decreases and the condition would then have \(\delta\) growing as well.

\end{remark}

\begin{proof}

We follow the proof of \citep[Proposition 2.38]{hein2005geometrical}, paying close attention to the dependency of the constants on \(||u||_{\infty}\). As shown there, using that for all \(\lambda \geq 0\) and \(a, b \geq \beta \geq 0\), \(|1/a^{\lambda} - 1/b^{\lambda}| \leq \lambda |a - b|/\beta^{\lambda + 1}\) gives,
\begin{align*}
|\AvOp_{\epsilon,N}&[u/p_{\epsilon}^{\lambda}](x) - \AvOp_{\epsilon}[u/p_{\epsilon,N}^{\lambda}](x)|    \\
    &\leq   |\AvOp_{\epsilon,N}[u(1/p_{\epsilon,N}^{\lambda} - 1/p_{\epsilon}^{\lambda})]| + |\AvOp_{\epsilon,N}[u/p_{\epsilon}^{\lambda}] - \AvOp_{\epsilon}[u/p_{\epsilon}^{\lambda}]|    \\
    &\leq||u||_{\infty} \lambda \, \AvOp_{\epsilon,N}\left[ \frac{|p_{\epsilon} - p_{\epsilon,N}|}{(\min\{\inf p_{\epsilon,N}, \inf p_{\epsilon} \})^{\lambda + 1}} \right] + |(\AvOp_{\epsilon} - \AvOp_{\epsilon,N})[u/p_{\epsilon}^{\lambda}]|   \\
    &=: \mathcal{E}_{\lambda,1}[u] .
\end{align*}
and using that for all \(a, b\), \(|a - b|^{\lambda} \leq \lambda \max \{ |a|, |b| \}^{\lambda - 1} |a - b|\), we have,
\begin{align*}
\left| \frac{\AvOp_{\epsilon,N}[u/p_{\epsilon,N}^{\lambda}]}{p_{\epsilon,N}^{\lambda}} - \frac{\AvOp_{\epsilon}[u/p_{\epsilon}^{\lambda}]}{p_{\epsilon}^{\lambda}} \right|
    &\leq \frac{|\AvOp_{\epsilon,N}[u/p_{\epsilon,N}^{\lambda}] - \AvOp_{\epsilon}[u/p_{\epsilon}^{\lambda}]|}{p_{\epsilon,N}^{\lambda}} + |\AvOp_{\epsilon}[u/p_{\epsilon}^{\lambda}]|\frac{|p_{\epsilon,N}^{\lambda} - p_{\epsilon}^{\lambda}|}{p_{\epsilon}^{\lambda} p_{\epsilon,N}^{\lambda}}  \\
    &\leq \mathcal{E}_{\lambda,1}[u]/p_{\epsilon,N}^{\lambda} + \\
    &\quad + ||u||_{\infty} \lambda \max \{ \sup p_{\epsilon,N}, \sup p_{\epsilon} \}^{\lambda - 1} \frac{\overline{C}_p}{\underline{C}_p^{\lambda}} \frac{|p_{\epsilon,N} - p_{\epsilon}|}{(p_{\epsilon} p_{\epsilon,N})^{\lambda}}    \\
    &=: \mathcal{E}_{\lambda,2}[u].
\end{align*}
Taking the normalizations by the \(\lambda\)-degree functions into account, we have
\begin{align*}
|A_{\lambda,\epsilon,N}[u] - A_{\lambda,\epsilon}[u]|
    &= \left| \frac{\mathcal{B}_N}{p_{\lambda,\epsilon,N}} - \frac{\mathcal{B}}{p_{\lambda,\epsilon}} \right|   \\
    &\leq \frac{|\mathcal{B}_N - \mathcal{B}|}{p_{\lambda,\epsilon,N}} + |\mathcal{B}|\frac{|p_{\lambda,\epsilon,N} - p_{\lambda,\epsilon}|}{p_{\lambda,\epsilon,N} \, p_{\lambda,\epsilon}}    \\
  &\leq \frac{\mathcal{E}_{\lambda,2}[u]}{p_{\lambda,\epsilon,N}} + |\mathcal{B}| \frac{\mathcal{E}_{\lambda,2}[1]}{p_{\lambda,\epsilon,N} \, p_{\lambda,\epsilon}} ,
\end{align*}
with
\begin{equation}\begin{aligned}
\mathcal{B}_N := \frac{\AvOp_{\epsilon,N}[u/p_{\epsilon,N}^{\lambda}]}{p_{\epsilon,N}^{\lambda}}, \quad \mathcal{B} := \frac{\AvOp_{\epsilon}[u/p_{\epsilon}^{\lambda}]}{p_{\epsilon}^{\lambda}} .
\end{aligned}  \nonumber  \end{equation}
In the event that \(|p_{\epsilon,N}(y) - p_{\epsilon}(y)| \leq \delta\), for the deterministic choice \(y = x\) along with all random variables \(y = x_j\) for all \(j \in [N]\), we have that \(\underline{C}_p - \delta < p_{\epsilon,N}(y) < \overline{C_p} + \delta\). Call this event \(\mathcal{A}_1\) and call \(\mathcal{A}_2\) the event that \(|\AvOp_{\epsilon,N}[u/p_{\epsilon}^{\lambda}] - \AvOp_{\epsilon}[u/p_{\epsilon}^{\lambda}]| \leq \delta ||u||_{\infty}\). Then,
\begin{align*}
&\mathcal{E}_{\lambda,1}[u] \leq \delta ||u||_{\infty}\{ \lambda (\overline{C}_p + \delta)/(\underline{C}_p - \delta)^{\lambda + 1} + 1 \}, \\
&\begin{aligned}[t]
(\underline{C}_p - \delta)/(\overline{C}_p + \delta)^{2\lambda}
    &\leq p_{\lambda,\epsilon,N} = \AvOp_{\epsilon,N}[1/p_{\epsilon,N}^{\lambda}]/p_{\epsilon,N}^{\lambda}  \\
    &\leq (\overline{C}_p + \delta)/(\underline{C}_p - \delta)^{2\lambda} ,
\end{aligned}   \\
&|\AvOp_{\epsilon}[u/p_{\epsilon}^{\lambda}]| \leq ||u||_{\infty}  \overline{C}_p/\underline{C}_p^{\lambda} ,   \\
&\underline{C}_p /\overline{C}_p^{2\lambda} \leq p_{\lambda,\epsilon} \leq \overline{C}_p/\underline{C}_p^{2\lambda} ,
\end{align*}
wherein we've used the occurrence of the joint event \(\mathcal{A} := \mathcal{A}_1 \wedge \mathcal{A}_2\) for the first inequality and that of \(\mathcal{A}_1\) for the second inequality. The event \(\mathcal{A}_1\) further supplies,
\begin{equation}\begin{aligned}
\mathcal{E}_{\lambda,2}[u] \leq \frac{\mathcal{E}_{\lambda,1}[u]}{(\underline{C}_p - \delta)^{\lambda}} + \delta\lambda ||u||_{\infty} \frac{(\overline{C}_p + \delta)^{\lambda - 1} \overline{C}_p}{\underline{C}_p^{2\lambda} (\underline{C}_p - \delta)^{\lambda}} .
\end{aligned}  \nonumber  \end{equation}
Now, working in event \(\mathcal{A}\) together with the assumption that \(\delta \leq \underline{C}_p/2\) we find,
\begin{align*}
|A_{\lambda,\epsilon,N}[u](x) - A_{\lambda,\epsilon}[u](x)| &\leq \mathcal{E}_{\lambda,2}[u] \frac{(\overline{C}_p + \delta)^{2\lambda}}{\underline{C}_p - \delta} + ||u||_{\infty} \mathcal{E}_{\lambda,2}[1]\frac{\overline{C}_p^{2\lambda + 1}}{\underline{C}_p^{2\lambda + 1}} \frac{(\overline{C}_p + \delta)^{2\lambda}}{\underline{C}_p - \delta}    \\
    &\leq \delta ||u||_{\infty}[(\overline{C}_p + \delta)^{2\lambda}(\underline{C}_p - \delta)^{-\lambda-1} + \tilde{C}_{\lambda,k,p}(\delta) + \lambda C_{k,p}(\lambda,\delta)]    \\
    &\leq 2(9/2)^{\lambda}\delta ||u||_{\infty}(\overline{C}_p^{2\lambda}\underline{C}_p^{-\lambda-1} + 2^{-\lambda}\overline{C}_p^{4\lambda + 1}\underline{C}_p^{-2(\lambda + 1)} + \lambda C_{\lambda,k,p}),  \\
    &:= K_{\lambda,k,p} ||u||_{\infty} \delta ,
\end{align*}
wherein \(\tilde{C}_{\lambda,k,p}(\delta) := \overline{C}_p^{2\lambda + 1} (\overline{C}_p + \delta)^{2\lambda} \underline{C}_p^{-2\lambda - 1} (\underline{C}_p - \delta)^{-1}\), \(C_{k,p}(\lambda,\delta) > 0\) is monotonically increasing in \(\delta, \lambda \geq 0\) that is rational and bounded in \(\delta < \underline{C}_p\) and has coefficients that depend only on \(\underline{C}_{p}, \overline{C}_p\). The final inequality follows upon assuming that \(\delta \leq \underline{C}_p/2\) and setting \(C_{\lambda,k,p} := C_{k,p}(\lambda,\underline{C}_p/2)\).

Concerning the probabilities governing the event \(\mathcal{A}\), since \(p_{\epsilon} = \AvOp_{\epsilon}[1]\) and \(p_{\epsilon,N} = \AvOp_{\epsilon,N}[1]\), an application of \protect\hyperlink{lem:kern-consistency}{Lemma \ref{lem:kern-consistency}} along with a union bound gives that whenever \(\delta > 2 ||k||_{\infty} \epsilon^{-\frac{\mdim}{2}}/N\), we have that
\begin{align*} 
& \Pr[(\forall) j \in [N], |p_{\epsilon,N}(x_j) - p_{\epsilon}(x_j)| \leq \delta]   \\
&\quad\quad > 1 - 2N \exp\left( -\frac{(N - 1)\epsilon^{\frac{\mdim}{2}} \delta^2}{4(2 \Kconst + ||k||_{\infty}\delta/3)} \right).
\end{align*}
Another application of \protect\hyperlink{lem:kern-consistency}{Lemma \ref{lem:kern-consistency}} gives,
\begin{align*}
&\Pr[|\AvOp_{\epsilon,N}[u/p_{\epsilon}^{\lambda}](x) - \AvOp_{\epsilon}[u/p_{\epsilon}^{\lambda}](x)| \leq ||u||_{\infty} \delta]  \\
&\quad\quad > 1 - 4 \exp\left( -\frac{N \epsilon^{\frac{n}{2}}\delta^2}{2 \underline{C}_p^{-\lambda}(\Kconst \underline{C}_p^{-\lambda} +  ||k||_{\infty}\delta/3)} \right).
\end{align*}
Therefore, a union bound gives the bound in the statement of the Lemma.
\end{proof}

On the other hand, Hein has given a simpler argument for the \(\lambda = 0\) case that is related to the \emph{random walk} graph Laplacian:

\begin{lemma}[{\cite[Theorem 2.37]{hein2005geometrical}}]	\hypertarget{lem:rwlap-conv}{\label{lem:rwlap-conv}} For all \(u \in C^{\infty}\) and \(x \in \mathcal{M}\), given any \(\delta > 0\),
\begin{equation} \label{eq:thm-rwlap-conv-bound}
\Pr[|A_{\epsilon,N}[u](x) - A_{\epsilon}[u](x)| > \delta] \leq 6 \exp\left( -\frac{N \epsilon^{\frac{n}{2}}\delta^2}{8||u||_{\infty} (\Kconst ||u||_{\infty}  + ||k||_{\infty} \delta/6)} \right) .
\end{equation}

\end{lemma}

\begin{proof}

The proof in \citep[Theorem 2.37]{hein2005geometrical} and \citep[Theorem 2]{hein2005graphs} is based on a rewriting of \(A_{\epsilon,N}[u]\) following \citep{greblicki1984distribution} as,
\begin{align*}
\GAve_{\epsilon,N}[u](x) &= \frac{\GAve_{\epsilon}[u](x) - \mathcal{B}_{1,N}}{1 + \mathcal{B}_{2,N}},   \\
\mathcal{B}_{1,N} &:= \frac{1}{p_{\epsilon}(x)}\AvOp_{\epsilon,N}[u](x) - \GAve_{\epsilon}[u](x),   \\
\mathcal{B}_{2,N} &:= \frac{p_{\epsilon,N}(x)}{p_{\epsilon}(x)} - 1.
\end{align*}
On applying \protect\hyperlink{lem:kern-consistency}{Lemma \ref{lem:kern-consistency}} twice while noting \(p_{\epsilon}(x) = \AvOp_{\epsilon}[1](x)\) and \(p_{\epsilon,N}(x) = \AvOp_{\epsilon,N}[1](x)\), we have that
\begin{equation}\begin{aligned}
\Pr\left[ \mathcal{B}_{1,N} > \delta/2 \right] \leq 4 \, \exp\left( -\frac{N \epsilon^{\frac{n}{2}} \delta^2}{8||u||_{\infty}(\Kconst ||u||_{\infty} + ||k||_{\infty} \delta/6)} \right) ,  \\
\Pr\left[ \mathcal{B}_{2,N} > \frac{\delta}{2||u||_{\infty}} \right] \leq 2 \, \exp\left( -\frac{N \epsilon^{\frac{n}{2}} \delta^2}{8 ||u||_{\infty} (\Kconst ||u||_{\infty} + ||k||_{\infty} \delta/6)} \right) .
\end{aligned}  \nonumber  \end{equation}
We also see,
\begin{equation}\begin{aligned}
|\GAve_{\epsilon,N}[u](x) - \GAve_{\epsilon}[u](x)| \leq \left( \left| \frac{\mathcal{B}_{1,N}}{1 + \mathcal{B}_{2,N}} \right| + ||u||_{\infty} \left| \frac{\mathcal{B}_{2,N}}{1 + \mathcal{B}_{2,N}} \right| \right).
\end{aligned}  \nonumber  \end{equation}
Thus taking a union bound gives,
\begin{equation}\begin{aligned}
\Pr\left[ |\GAve_{\epsilon,N}[u](x) - \GAve_{\epsilon}[u](x)| > \delta \right] \leq 6 \exp\left( -\frac{N \epsilon^{\frac{n}{2}}\delta^2}{8||u||_{\infty} (\Kconst ||u||_{\infty}  + ||k||_{\infty} \delta/6)} \right).
\end{aligned}  \nonumber  \end{equation}
\end{proof}

\begin{remark} In our application, when \(u = \psi_h\) is a state localized to an \(O(\sqrt{h})\)-ball with unit \(L^2\) norm, we will have that \(||\psi_h||_{\infty} \geq h^{-\frac{n}{2}}\), so this will give a convergence rate of \(e^{-\Omega(N \epsilon^{\frac{\mdim}{2}} h^{\mdim} \delta^2)}\) for \(\GAve_{\epsilon,N}[\psi_h](x) \to \GAve_{\epsilon}[\psi_h](x)\) as \(N \to \infty\) and \(\epsilon, h \to 0\).

\end{remark}

\hypertarget{quantization-and-symbol-classes}{%
\subsection{Quantization and Symbol classes}\label{quantization-and-symbol-classes}}

The study of pseudodifferential calculus is rooted in the relationships between certain classes of smooth functions on phase space and linear operators acting on smooth functions on configuration space, called \emph{pseudodifferential operators} (\(\PDO\)s) and in relation to the phase space functions, often called their \emph{quantizations}. The classes of functions under study are called \emph{symbols} and they are typically bounded on configuration space with well-behaved (\emph{e.g.}, polynomially bounded) momentum space asymptotics. Then, of interest are the relationships between \(\PDO\)s and the symbols they quantize. An additional \emph{small parameter} \(h \in (0, 1]\) enters the picture when we work in the \emph{semiclassical} regime and serves essentially to rescale the unit length in momentum space, so the asymptotics we are concerned with happen simultaneously at scales of vanishing \(h\) and increasing momentum variable.

The formal classes of functions that we quantize, as well as the exact quantization procedures, are standard and described, for example, in \citep[\(\S 14.2\)]{zworski2012}. These procedures amount to locally pushing the symbols onto Euclidean space by charts and seeing these local pieces as quantizations on open balls in \(\mathbb{R}^{\mdim}\), then gluing them together through an atlas. Hence, the key definitions are specified on \(\mathbb{R}^{\mdim}\). To summarize the idea briefly, we make

\begin{definition}[{Symbols \& quantization on \(\mathbb{R}^{\mdim}\)
\cite[$\S 4$]{zworski2012}}] Let \(h_0 \in (0, 1]\). A smooth function \(a \in C^{\infty}(\mathbb{R}^{2\mdim} \times (0, h_0])\) belongs to the \(h\)-\emph{symbol class of order} \((\ell, m) \in \mathbb{N}_0 \times \mathbb{Z}\) on \(\mathbb{R}^{\mdim}\), \(h^{\ell} S^m(\mathbb{R}^{2n})\) if for all \(\alpha, \beta \in \mathbb{N}_0^{\mdim}\), \(|\partial_x^{\alpha} \partial_{\xi}^{\beta} a(x,\xi ; h)| \lesssim_{\alpha,\beta} h^{\ell} \langle \xi \rangle^{m - \beta}\). Then, a \emph{quantization} of \(a \in h^{\ell} S^m\) is an operator acting on \(u \in \mathscr{S}(\mathbb{R}^{\mdim})\) given by,
\begin{equation}\begin{aligned}
\Op_h^t(a) := \frac{1}{(2 \pi h)^{\mdim}} \int_{\mathbb{R}^{2\mdim}} e^{\frac{i}{h} \langle x - y, \xi \rangle} a(tx + (1 - t)y, \xi) u(y) ~ d\xi dy ,
\end{aligned}  \nonumber  \end{equation}
for any \(t \in [0, 1]\). We call the quantization with \(t = 1\) the \emph{adjoint} quantization, which we will primarily use and hence we denote it simply as \(\Op_h := \Op_h^1\), while the case \(t = 0\) is called the \emph{Kohn-Nirenberg} quantization and we will denote it \(\Op_h^{KN} := \Op_h^0\). A highly used case is \(t = 1/2\) and called the \emph{Weyl} quantization; while it is very useful in proving various properties of pseudodifferential operators, we will not make explicit use of this.

An operator of the form \(\Op_h^t(a)\) is called a \emph{(semiclassical) pseudodifferential operator of order} \((\ell,m)\) and sometimes (to evoke a physical interpretation) a \emph{quantum observable}, while its symbol \(a\) is sometimes called a \emph{classical observable}. We also say that \(a \in h^{\ell} S^m\) is a \emph{symbol of order} \((\ell,m)\).

\end{definition}

A key property in the definition of the symbol class is that under a \emph{symplectic} change of variables \(\tilde{\gamma} := (\gamma, [D \gamma^{-1}|_{\gamma}]^T)\) on \(\mathbb{R}^{2\mdim}\) for any diffeomorphism \(\gamma : \mathbb{R}^{\mdim} \to \mathbb{R}^{\mdim}\), the symbol class \(h^{\ell} S^m\) is invariant \emph{to top order}, meaning that if \(a \in h^{\ell} S^m\) then \(a \circ \tilde{\gamma} \in h^{\ell} S^m\) and \(u \mapsto \Op^t_h(a)[ u \circ \gamma] \circ \gamma^{-1} = \Op^t_h(a_{\gamma})[u]\) with \(a_{\gamma}(\gamma(x), \xi) = a(x, [D \gamma|_x]^T \xi) + b(x,\xi)\) for \(b \in h^{\ell + 1} S^{m-1}\) (see \citep[\(\S 9\)]{zworski2012}). While the asymptotic behaviour required in \(\xi\) is simple to understand, we offer the notion about the joint asymptotics with \(h\) that a symbol should have, asymptotically, at most polynomial variation in cubes of side-length \(h\) in momentum space.

The basic calculus of pseudodifferential operators and symbols on \(\mathbb{R}^{\mdim}\) is exposed in detail in \citep{zworski2012} and \citep{dimassi1999spectral}. Now using this, we can give meaning to quantization on a compact manifold \(\mathcal{M}\), following \citep{zworski2012}:

\begin{definition}[{Symbols \& quantization on \(\mathcal{M}\)
\cite[$\S 14.2.2$]{zworski2012}}] A linear operator \(A : C^{\infty}(\mathcal{M}) \to C^{\infty}(\mathcal{M})\) is a \emph{pseudodifferential operator} on \(\mathcal{M}\) if

\begin{enumerate}
\def\labelenumi{\arabic{enumi}.}
\tightlist
\item
  given an atlas \(\{ (\gamma : \mathcal{M} \supset U_{\gamma} \to V_{\gamma} \subset \mathbb{R}, U_{\gamma}) ~|~ \gamma \in \mathcal{F} \}\), there is \((\ell,m) \in \mathbb{N}_0 \times \mathbb{Z}\) and \(t \in \mathbb{R}\) such that for each \(\gamma \in \mathcal{F}\), there is \(a_{\gamma} \in h^{\ell} S^m\) so that we may write, given any \(\chi_1, \chi_2 \in C_c^{\infty}(U_{\gamma})\) and \(u \in C^{\infty}(\mathcal{M})\),
  \begin{equation}\begin{aligned}
  \chi_1 A[\chi_2 u] = \chi_1 \gamma^* \Op_h^t[a_{\gamma}] \circ (\gamma^{-1})^*[\chi_2 u] ,
  \end{aligned}  \nonumber  \end{equation}
  wherein \((\gamma^{-1})^* : C^{\infty}(U_{\gamma}) \to C^{\infty}(V_{\gamma})\) denotes the \emph{pullback}, \((\gamma^{-1})^* : u \mapsto u \circ \gamma^{-1}\) and similarly \(\gamma^*\) the pullback from \(V_{\gamma}\) to \(U_{\gamma}\) and
\item
  \(A\) is \emph{pseudolocal}, meaning: given any \(\chi_1, \chi_2 \in C^{\infty}(\mathcal{M})\) with \(\supp \chi_1 \cap \chi_2 = \emptyset\) and any \(N \in \mathbb{N}\), we have \(||\chi_1 A \chi_2||_{H^{-N}(\mathcal{M}) \to H^N(\mathcal{M})} = O(h^{\infty})\), wherein \(H^{\pm N}(\mathcal{M})\) Sobolev spaces as defined in \citep[\(\S 14.2.1\)]{zworski2012}.
\end{enumerate}

We say then that the pseudodifferential operator \(A\) is of \emph{order} \((\ell,m)\) and belongs to \(h^{\ell} \Psi^m\) and for brevity, we call it a \(\PDO\).

A function \(a \in C^{\infty}(T^*\mathcal{M})\) belongs to the \(h\)-\emph{symbol class of order} \((\ell,m) \in \mathbb{N}_0 \times \mathbb{Z}\) on \(\mathcal{M}\), \(h^{\ell} S^m(T^*\mathcal{M})\) if given an atlas \(\mathcal{F}\), for each \(\gamma \in \mathcal{F}\) and \(\chi \in C_c^{\infty}(U_{\gamma})\) the pullback \(a_{\gamma} \in C^{\infty}(V_{\gamma} \times \mathbb{R}^{\mdim})\) of \(a\chi\) under the induced change of coordinates \(V_{\gamma} \times \mathbb{R}^n \to T^*(U_{\gamma})\) belongs to \(h^{\ell} S^m(\mathbb{R}^{2\mdim})\). In short, we say that \(a\) is a \emph{symbol of order} \((\ell, m)\).

\end{definition}

The fact that \(h^{\ell} \Psi^{m}\) forms a bi-filtered (in the order \((\ell,m)\)) algebra respecting the properties of a \emph{quantization map} and moreover, that quantized operators extend to Sobolev spaces, mapping \(A : H_h^s \to H_h^{s - m}\) continuously whenever \(A \in h^0 \Psi^m\) and \(A : L^2 \to L^2\) whenever \(A \in h^0 \Psi^0\) are standard results from pseudodifferential calculus \citep{zworski2012, dimassi1999spectral}. We will make liberal use of these properties along with the following basic \emph{symbol calculus}:

\begin{fact*}[{\cite[Theorem 14.1]{zworski2012}}] There are linear maps
\begin{equation}\begin{aligned}
\operatorname{Sym} : h^{\ell} \Psi^m \to h^{\ell} S^m(T^*\mathcal{M}) / h^{\ell + 1} S^{m - 1}(T^*\mathcal{M})
\end{aligned}  \nonumber  \end{equation}
and for all \(t \in [0, 1]\),
\begin{equation}\begin{aligned}
\Op_h^t : h^{\ell} S^m \to h^{\ell} \Psi^m
\end{aligned}  \nonumber  \end{equation}
such that
\begin{equation}\begin{aligned}
\operatorname{Sym}(A_1 A_2) = \operatorname{Sym}(A_1) \operatorname{Sym}(A_2)
\end{aligned}  \nonumber  \end{equation}
and
\begin{equation}\begin{aligned}
\operatorname{Sym}[\Op^t_h(a)] = [a] \in h^{\ell} S^m(T^*\mathcal{M}) / h^{\ell + 1} S^{m - 1}(T^* \mathcal{M}),
\end{aligned}  \nonumber  \end{equation}
wherein \([a]\) denotes the equivalence class of \(a\) in the quotient space \(h^{\ell} S^m / h^{\ell + 1} S^{m - 1}\).

\end{fact*}

We often call \(\operatorname{Sym}(A)\) the \emph{principal symbol} of \(A\) and note that by the invariance properties of \(h^{\ell} S^m\) under changes of coordinates as discussed above, this equivalence class is defined invariantly of the choice of coordinates. In fact, this is independent of the particular quantization \(t \in [0, 1]\) as well (see \citep[Theorem 9.10]{zworski2012}).

A useful sub-class of \(h^{\ell} S^m\) is the class of semiclassical \emph{polyhomogeneous} symbols. The quantizations of these symbols are more direct generalizations of partial differential operators, whose symbols can be seen to behave like polynomials in the momentum \(\xi\) variable. The polyhomogeneous symbols enjoy many nice properties through their asymptotic expansions, including that we can specify an exact principal symbol. We restrict to this class primarily when dealing with the phases and amplitudes appearing in the FBI transform, following \citep{wunsch2001fbi} (discussed below, in \protect\hyperlink{semi-classical-measures-of-coherent-states}{Section \ref{semi-classical-measures-of-coherent-states}}). In particular, we will use such symbols in the phases defining coherent states, for then we have a good interaction between them and the FBI transform. Assuming the definitions of \emph{classical} polyhomogeneous symbols, as given in \citep{hormander_vol_III}, the definitions of their semiclassical counterparts are as follows:

\begin{definition}[{following \cite{wunsch2001fbi}}]	\hypertarget{def:phg-psidos}{\label{def:phg-psidos}} A smooth function \(a \in C^\infty(T^*\mathcal{M} \times \mathcal{M} \times [0,h_0))\) for some \(h_0 > 0\) belongs to the \(h\)-\emph{polyhomogeneous symbol class of order} \((\ell,m)\), \(h^{\ell} S^m_{\text{phg}}(T^*\mathcal{M} \times \mathcal{M}) \subset C^\infty(T^*\mathcal{M} \times \mathcal{M} \times [0, h_0))\), if for every \(j \geq 0\), there exists a polyhomogeneous symbol \(a_j(x,\xi,y)\) of degree \(j\) in \(\xi\) and a constant \(C_j > 0\) such that for all \(|\xi| > 1\) and \(h \in [0, h_0)\), the following asymptotic expansion holds:
\begin{equation} \label{eq:asymp-h-symbol-phg}
|a(x,\xi,y;h) - h^{\ell}(a_m + h\, a_{m-1} + \cdots + h^{\ell}\, a_{m-j})| \leq C_j h^{\ell+j+1} |\xi|^{m-j-1} .
\end{equation}
We abbreviate this as,
\begin{equation}\begin{aligned}
a(x, \xi, y ; h) \sim h^{\ell}(a_m(x,\xi,y) + h \, a_{m-1}(x,\xi,y) + \cdots )
\end{aligned}  \nonumber  \end{equation}
and call \(a_m\) the \emph{principal part} of \(a\). Similarly, we define the \(h\)-symbol class \(h^{\ell} S^m_{\text{phg}}(T^*\mathcal{M})\) as the subset of the above class without dependence on the \(y\) variable. When ambiguity is not an issue, we drop the explicit dependence on the space and use simply the notation \(h^{\ell} S^m_{\text{phg}}\) to denote either of the two classes of symbols.

The \emph{quantization} of \(a \in h^{\ell} S^m_{\text{phg}}(T^* \mathcal{M} \times \mathcal{M})\) is the operator \(\operatorname{Op}_h(a) : C^\infty(\mathcal{M}) \to C^\infty(\mathcal{M})\) given by
\begin{equation} \label{def:op-quantization}
\operatorname{Op}_h(a)[u](x) := \frac{1}{(2\pi h)^n} \int_{T^*\mathcal{M}} e^{\frac{i}{h} \langle \exp_y^{-1}(x), \xi \rangle_{g_y}} a(y,\xi,x;h) \chi(x,y) u(y) ~ dy d\xi,
\end{equation}
wherein \(\chi\) is a smooth cut-off near \(x = y\). We denote by \(h^{\ell} \Psi_{\text{phg}}^m(\mathcal{M})\) the algebra of all operators quantizing symbols \emph{modulo} \(h^\infty\), viz., \(A \in h^{\ell} \Psi_{\text{phg}}^m(\mathcal{M})\) iff. \(A = \operatorname{Op}_h(a) + R\) for some \(a \in h^{\ell} S^m_{\text{phg}}(T^*\mathcal{M} \times \mathcal{M})\) and \(R : C^{-\infty}(\mathcal{M}) \to C^\infty(\mathcal{M})\) of order \(O(h^{\infty})\). When the polyhomogeneous and space contexts are clear, we will shorten the notation to just \(h^{\ell} \Psi_{\text{phg}}^m\).

The \emph{symbol map} of order \((\ell,m)\) is the linear operator \(\operatorname{Sym}_{\ell,m} : h^{\ell} \Psi_{\text{phg}}^m \to h^{\ell} S^m_{\text{phg}} / h^{\ell-1} S^{m-1}_{\text{phg}}\) such that the \emph{principal symbol} of an operator \(A = \operatorname{Op}_h(a) \in h^{\ell} \Psi_{\text{phg}}^m\) is given by \(\operatorname{Sym}_{\ell,m}(A) = [a]\). This gives a short exact sequence,
\begin{equation}\begin{aligned}
0 \to h^{\ell-1}\Psi_{\text{phg}}^{m-1} \hookrightarrow h^{\ell} \Psi_{\text{phg}}^m \xrightarrow{\operatorname{Sym}_{\ell,m}} h^{\ell} S^m_{\text{phg}}/h^{\ell-1}S^{m-1}_{\text{phg}} \to0.
\end{aligned}  \nonumber  \end{equation}
An \emph{elliptic symbol} is a symbol \(a \in h^{\ell}{S^m}_{\text{phg}}\) whose principal part follows \(|a_k| \sim \langle \xi \rangle\) uniformly in all other variables.

\end{definition}

Symbols that are functions of both, the integrand as well as the output parameter with respect to quantization, such as \(a \in h^{\ell} S_{\text{phg}}^m(T^*\mathcal{M} \times \mathcal{M})\) are sometimes called \emph{amplitudes} (see \emph{e.g.}, \citep{taylor81}). The phases appearing in coherent states are naturally of this sort.

We aim to bring graph Laplacians into the context of pseudodifferential operators and by their basic characteristics, we find that they must belong to a more general class than the polyhomogeneous. Indeed, the polyhomogeneous class is too restrictive to support symbols localized in phase-space \(T^*\mathcal{M}\), which is a stepping-stone since we go through the diffusive averaging operators \(\AvOp_{\lambda,\epsilon}\); nevertheless, we will have solace in the class \(h^{\ell} S^m\), which is able to support this while remaining invariant to coordinate transformations.

\hypertarget{fbi-transform}{%
\subsection{FBI Transform}\label{fbi-transform}}

We follow \citep{wunsch2001fbi} to state the fundamental results about FBI transforms on smooth, compact manifolds. They define the following:

\begin{definition}[{Admissible Phase \& FBI Transform}]	\hypertarget{def:FBI-adm-phase}{\label{def:FBI-adm-phase}} A smooth function \(\phi : T^* \mathcal{M} \times \mathcal{M} \to \mathbb{C}\) is an \emph{admissible phase function} if it satisfies all of the properties:

\begin{enumerate}
\def\labelenumi{\arabic{enumi}.}
\tightlist
\item
  \(\phi(x,\xi,y)\) is an elliptic polyhomogeneous symbol of order one in \(\xi\),
\item
  \(\Im(\phi) \geq 0\),
\item
  \(D_y \phi|_{(x,\xi,x)} = -\xi \, dy\),
\item
  \(D_y^2 \Im(\phi)|_{(x,\xi,x)} \sim \langle \xi \rangle\),
\item
  and \(\phi|_{(x,\xi,x)} = 0\);
\end{enumerate}

\noindent herein we have denoted by \(\varphi|_{(x,\xi,x)}\) the restriction of \(\varphi:T^*\mathcal{M} \times \mathcal{M} \to \mathbb{C}\) to the set \(\{(x,\xi,y) ~|~ x = y \}\).

An \emph{FBI transform} is a map \(T_h[\cdot ; \phi,a] : C^\infty(\mathcal{M}) \to C^\infty(T^* \mathcal{M})\) given by,
\begin{equation}\begin{aligned}
T_h[u;\phi,a](x) := \int_{\mathcal{M}} e^{\frac{i}{h}\phi(x,\xi,y)} a(x,\xi,y;h) \chi(x,\xi,y) u(y) \;dy,
\end{aligned}  \nonumber  \end{equation}
with \(\phi\) a fixed admissible phase function, \(a \in h^{-\frac{3n}{4}} S^{\frac{n}{4}}_{\text{phg}}\) is elliptic and \(\chi\) is a cut-off function near \(\{(x,\xi,y) ~|~ x = y \}\).

\end{definition}

We will need the following basic properties of the FBI transform, as it relates to quantization:

\begin{theorem}[{\cite{wunsch2001fbi}}]	\hypertarget{thm:FBI-basic}{\label{thm:FBI-basic}} Let \(q \in h^{\ell} S^m(T^*\mathcal{M})\) and fix both, an admissible phase function \(\phi\) and an elliptic symbol \(a \in h^{-\frac{3n}{4}} S^{\frac{n}{4}}_{\text{phg}}(T^* \mathcal{M} \times \mathcal{M})\). Denote \(T_h := T_h[\cdot ; \phi, a]\). Then, there exists \(b_0 \in h^{\frac{3n}{2}} S_{\text{phg}}^{-\frac{n}{2}}(T^*\mathcal{M} \times \mathcal{M})\) positive, elliptic and depending only on \(\phi\), such that \(T_h^* q T_h \in h^{\ell} \Psi^m\) and
\begin{equation}\begin{aligned}
T_h^* q T_h - \operatorname{Op}_h(|a|^2 b_0 \, q) \in h^{\ell+1} \Psi^{m-1}.
\end{aligned}  \nonumber  \end{equation}
Furthermore, specializing to \(a = b_0^{-\frac{1}{2}}\), there exists an elliptic \(b \in h^{\frac{3n}{2}} S_{\text{phg}}^{-\frac{n}{2}}\) depending only on \(\phi\) with a positive principal symbol such that setting \(T_h := T_h[\cdot; \phi, b^{-\frac{1}{2}}]\) gives
\begin{equation}\begin{aligned}
T_h^* q T_h - \operatorname{Op}(q) \in h^{\ell+1}\Psi^{m-1},  \\
T_h^*T_h - I \in h^{\infty}\Psi_{\text{phg}}^{-\infty} .
\end{aligned}  \nonumber  \end{equation}
Whenever \(q\) is polyhomogeneous, then so is the symbol being quantized by the FBI transform, above. The operator \(T_h : L^2(\mathcal{M}) \to L^2(T^*\mathcal{M})\) is bounded for all \(h \in [0,h_0)\).

\end{theorem}

\begin{remark} The proof given in \citep[Proposition 3.1]{wunsch2001fbi} goes through for \(q\) taken to be a symbol belonging to the more general class we consider here, \(S^m\).

\end{remark}

\hypertarget{semi-classical-measures-of-coherent-states}{%
\section{Semi-classical measures of coherent states}\label{semi-classical-measures-of-coherent-states}}

The basic properties of FBI transforms, as in \protect\hyperlink{thm:FBI-basic}{Theorem \ref{thm:FBI-basic}} show that much like the Fourier transform in flat space, they give a \emph{diagonal} representation of pseudodifferential operators through their principal symbols on phase space. On the other hand, as a map, \(T_h\) also \emph{lifts} an \(L^2(\mathcal{M})\) element to \(L^2(T^*\mathcal{M})\). These lifts allow us to access certain \emph{statistics} on phase space through inner products of the original (configuration-space) functions with quantized operators. To give some concreteness, let \(u_h \in L^2(\mathcal{M})\) for every \(h \in (0, h_0)\). Then, the \emph{Husimi function} for \(u_h\) is \(|T_h[u_h]|^2 \in L^2(T^*\mathcal{M})\) and clearly, \(\mu_h[u_h] := |T_h[u_h]|^2 \; dxd\xi\) defines a sequence of measures on the cotangent bundle in the parameter \(h \in (0, h_0)\). If \(\sup_{0 \leq h < h_0} ||u_h||_{L^2}^2 < \infty\), then also \(\sup _{0 \leq h < h_0}||T_h[u_h]||_{L^2(T^*\mathcal{M})} < \infty\) so by a diagonal argument, given any sequence \(h_j \to 0\), there exists a subsequence \(h_{j_k} \to 0\) and a non-negative finite Borel measure \(\mu\) on \(L^2(T^*\mathcal{M})\) such that \(\mu_{h_{j_k}}[u_{h_{j_k}}] \to \mu\) weakly; this is often called a \emph{semi-classical (defect) measure} \citep{zworski2012}. Therefore, \protect\hyperlink{thm:FBI-basic}{Theorem \ref{thm:FBI-basic}} allows us to take expectations of symbols with respect to Husimi functions in a simple way:
\begin{flalign}
&&\langle u_h |\operatorname{Op}_h(q)|u_h \rangle_{L^2(\mathcal{M})} &= \langle u_h | T_h^*[\cdot;\phi,b^{-\frac{1}{2}}] \, q \, T_h[\cdot;\phi,b^{-\frac{1}{2}}] | u_h \rangle + O(h) &&   \nonumber \\
&&  &= \langle T_h[u_h] | \;q\; | T_h[u_h] \rangle_{L^2(T^*\mathcal{M})} + O(h) \nonumber \\
&&  &= \int_{T^*\mathcal{M}} q(x,\xi) \; |T_h[u_h]|^2(x,\xi) ~ dxd\xi + O(h), \label{eq:expectation-FBI}\\
&\leftmath{\text{hence, }} & \langle u_{h_{j_k}} |\operatorname{Op}_h(q)|u_{h_{j_k}} \rangle_{L^2(\mathcal{M})} &\xrightarrow[h_{j_k} \to \; 0]{} \int_{T^*\mathcal{M}} q(x,\xi) ~d\mu(x,\xi). \nonumber
\end{flalign}
Now, of particular interest as a step towards \emph{quantum-classical correspondence} is if a state \(u_h\) could have a Dirac mass as semi-classical measure, for then we could \emph{access} the \emph{classical} space of symbols by probing pseudodifferential operators via inner products. There are in fact many such candidate states and we can construct and understand them through stationary phase analysis of the FBI transform. Indeed, a particularly simple case is essentially to use the kernel itself: that is, to consider the interaction between the state
\begin{equation}\begin{aligned}
\psi_h(x;x_0,\xi_0) := e^{-\frac{i}{h} \bar{\phi}(x_0,\xi_0,x)}
\end{aligned}  \nonumber  \end{equation}
with a fixed admissible phase \(\phi\) and \emph{localization} \((x_0, \xi_0) \in T^*\mathcal{M}\) and the FBI transform \(T_h := T_h[\cdot ; \phi, a]\) with the same phase and some fixed elliptic \(a \in h^{-\frac{3n}{4}} S^{\frac{n}{4}}\). Then, we would like to know the Husimi function for \(\psi_h\) and the \emph{ansatz} from looking at the flat case is that this is indeed a Dirac mass at \((x_0, \xi_0)\), weighted by an order zero symbol. In fact, the Husimi function is just the modulus square of the Schwartz kernel for the projection operator \(T_h T_h^*\), up to an amplitude factor: indeed,
\begin{gather*}
T_h T_h^*(x,\xi, x',\eta) ~ dx'd\eta = \\
\int_{\mathcal{M}}e^{\frac{i}{h}(\phi(x,\xi;y) - \bar{\phi}(x',\eta;y))} a(x,\xi,y;h) \bar{a}(x',\eta,y;h) \chi(x,\xi,y) \bar{\chi}(x',\eta,y) ~ dy ~ dx'd\eta ,
\end{gather*}
while
\begin{equation}\begin{aligned}
T_h[\psi_h(\cdot ; x_0,\xi_0)](x,\xi) = \int_{\mathcal{M}} e^{\frac{i}{h}(\phi(x,\xi; y) - \bar{\phi}(x_0,\xi_0; y))} a(x,\xi,y;h) \chi(x,\xi,y) ~ dy.
\end{aligned}  \nonumber  \end{equation}
The computation of the Schwartz kernel for \(T_h T_h^*\) is carried out in \citep[Theorem 4.4]{wunsch2001fbi} and the corresponding stationary phase analysis yields also \(T_h[\psi_h]\). The proof of \citep[Theorem 4.4]{wunsch2001fbi} goes on to elucidate the real (oscillatory) part of the resulting phase, which is more than what's necessary to understand the Husimi function \(|T_h[\psi_h]|^2\); therefore, although the changes are straight-forward, for completeness we will recount their proof, with the necessary modifications, stopping just at the relevant outcome for our context and record the resulting characterization as,

\begin{lemma}[{\cite[Theorem 4.4.]{wunsch2001fbi}}]	\hypertarget{thm:wunsch-zworski}{\label{thm:wunsch-zworski}} Let \(\phi\) be an admissible phase function and \(a \in h^{-\frac{3n}{4}} S_{\text{phg}}^{\frac{n}{4}}\) be an elliptic amplitude such that \(h^{\frac{3\mdim}{4}} a(x_0,\xi_0,x_0;h)|_{h=0} \neq 0\). Then, there exist \(\Phi \in C^\infty(T^*\mathcal{M} \times T^* \mathcal{M}_{(x_0,\xi_0)})\) and \(c \in C^\infty(T^*\mathcal{M} \times T^* \mathcal{M}_{(x_0,\xi_0)} \times [0,h_0))\) such that

\begin{enumerate}
\def\labelenumi{\arabic{enumi}.}
\item
  \(c(x,\xi;x_0,\xi_0;h)\) is order zero in \(h\), supported in a neighbourhood of \((x_0,\xi_0)\) in the first variable and \(c(x_0,\xi_0;x_0,\xi_0;0) \neq 0\),
\item
  there is a positive-definite matrix \(H\) depending only on \(x_0,\xi_0\) such that
  \begin{equation}\begin{aligned}
  \Im \Phi = \frac{1}{2}\langle (x - x_0, \xi - \xi_0) | H | (x - x_0, \xi - \xi_0)\rangle +O(|x - x_0|^4 + |\xi - \xi_0|^4)
  \end{aligned}  \nonumber  \end{equation}
  and
  \begin{equation}\begin{aligned}
  T_h[\psi_h](x,\xi) = h^{-\frac{n}{4}} c(x,\xi;x_0,\xi_0;h) e^{\frac{i}{h}\Phi(x,\xi ; x_0,\xi_0)} + O(h^{\infty}).
  \end{aligned}  \nonumber  \end{equation}
\end{enumerate}

\end{lemma}

\begin{proof}

We follow closely the proof of \citep[Theorem 4.4]{wunsch2001fbi} to compute
\begin{equation}\begin{aligned}
T_h[\psi_h(\cdot;x_0,\xi_0)](x,\xi) = \int_{\mathcal{M}} e^{\frac{i}{h} [\phi(x,\xi;y) - \bar{\phi}(x_0,\xi_0;y)]} a(x,\xi,y;h) ~ dy,
\end{aligned}  \nonumber  \end{equation}
wherein we have absorbed the cut-off \(\chi\) into the symbol \(a\). The conditions of an admissible phase allow us to write locally, by Taylor expansion in \(y\) near \(\{(x, \xi, y) \in T^*\mathcal{M} \times \mathcal{M} ~|~ x = y \}\),
\begin{equation}\begin{aligned}
\phi(x,\xi;y) = \xi \cdot (y-x) + \frac{1}{2} \langle Q(y;x,\xi) (x-y), x-y\rangle + O(|x-y|^3),
\end{aligned}  \nonumber  \end{equation}
wherein \(Q\) is a complex-symmetric matrix symbol of degree one in \(\xi\) such that \(\Im Q|_{x = y} \sim \langle\xi\rangle I\) and \(\langle \cdot , \cdot \rangle\) denotes the \emph{real} inner product. Since \(T_h\) localizes to a neighbourhood of \(x\), we can therefore write (up to \(O(|x-y|^3)\) terms),
\begin{alignat*}{4}
\varphi_0(x,\xi,x_0,\xi_0;y) &:= \phi(x,\xi;y) - \bar{\phi}(x_0,\xi_0;y) \\
    &= (x_0-y)\cdot\xi_0 - (x-y)\cdot\xi + \\
    &\quad + \frac{1}{2}\langle Q(y;x_0,\xi_0) (x_0-y),(x_0-y) \rangle - \frac{1}{2} \langle \bar{Q}(y;x,\xi)(x-y),(x-y)\rangle .
\end{alignat*}
The integrand is localized about \(x = x_0 = y\), and since the integral decays rapidly as \(O(h^{\infty})\) away from the neighbourhood of points where the phase is real and stationary, we also reduce to localization about \(\xi = \xi_0\). Hence, we switch to the coordinates,
\begin{equation}\begin{aligned}
\vartheta := x_0 - y, \quad \omega := x_0 - x, \\
\zeta:= \xi_0 - \xi ,
\end{aligned}  \nonumber  \end{equation}
giving
\begin{alignat*}{4}
\tilde{\varphi}_0(\vartheta,\omega,\zeta; x_0,\xi_0) & := && \; \varphi_0(x_0 - \omega, \xi_0 - \zeta,x_0, \xi_0;x_0 - \vartheta)   \\
    &= && \; \vartheta \cdot \xi_0 - (\vartheta - \omega)\cdot(\xi_0 - \zeta) + \\
    &\; && + \frac{1}{2}\langle Q(x_0 - \vartheta;x_0,\xi_0)\vartheta,\vartheta\rangle - \\
&   &&- \frac{1}{2}\langle \bar{Q}(x_0 - \vartheta;x_0 - \omega,\xi_0 - \zeta)(\vartheta - \omega),\vartheta - \omega \rangle .
\end{alignat*}
We wish to apply the method of complex stationary phase, as per \citep[Theorem 7.7.12]{hormander_vol_I}. To see that this can be done, just note that the change of variables \(y \mapsto \vartheta\) along with the admissibility conditions on \(\varphi\) and \(\bar{\varphi}\) imply that \(\tilde{\varphi}_0\) satisfies the hypotheses of that theorem around \(\{\vartheta = \omega = \zeta = 0\}\). The result is determined up to the ideal \(I_{\tilde{\varphi}_0} := \langle \partial_{\vartheta_1} \tilde{\varphi}_0, \ldots, \partial_{\vartheta_n} \tilde{\varphi}_0 \rangle\) --- the generators of which, if we expand in Taylor series about \(\vartheta = \omega = \zeta = 0\), are given by the components of the vector,
\begin{equation}\begin{aligned}
D_{\vartheta} \tilde{\varphi}_0 = \zeta + Q(x_0;x_0,\xi_0) \cdot \vartheta - \bar{Q}(x_0;x_0,\xi_0) \cdot (\vartheta - \omega) + O(|\vartheta|^2 + |\omega|^2).
\end{aligned}  \nonumber  \end{equation}
Thus, we can proceed to apply the Malgrange preparation theorem, which allows us to divide \(\partial_{\vartheta_1} \tilde{\varphi}_0, \ldots, \partial_{\vartheta_n} \tilde{\varphi}_0\) by each of the coordinates \(\vartheta_1, \ldots, \vartheta_n\) to yield (complex-valued) smooth remainder functions \(X_{\vartheta_1}, \ldots, X_{\vartheta_n}(\omega,\zeta,x_0,\xi_0)\) in a neighbourhood of \(\omega = \zeta = 0\); \emph{viz}.,
\begin{equation}\begin{aligned}
\vartheta \equiv X_{\vartheta} \pmod{ I_{\tilde{\phi}_0}} .
\end{aligned}  \nonumber  \end{equation}
We now work modulo the ideal \(I_{\tilde{\varphi}_0}\). We can resolve \(X_{\vartheta}\) via the Taylor expansion of \(D_{\vartheta} \tilde{\varphi}_0\) in a neighbourhood of \(\vartheta = \omega = \zeta = 0\): set \(Q_0 := Q(x_0;x_0,\xi_0)\), then
\begin{align*}
&   && 0 \equiv D_{\vartheta}\tilde{\varphi}_0 \equiv \zeta + 2i \Im Q_0 \cdot X_\vartheta + \bar{Q}_0 \cdot \omega + O(|X_{\vartheta}|^2 + |\omega|^2),    \\
\text{hence, } & && X_{\vartheta}(\omega,\zeta;x_0,\xi_0) \equiv \frac{i}{2}(\Im Q_0)^{-1}(\zeta + \bar{Q}_0 \cdot \omega + O(|\zeta|^2 + |\omega|^2))
\end{align*}
(wherein, the \(O(|\zeta|^2 + |\omega|^2)\) term comes from recursively substituting the expression for \(X_{\vartheta}\) in the higher order term). Now, we may expand \(\tilde{\varphi}_0\) in Taylor series about \(\{\vartheta = \omega = \zeta = 0 \}\) and substitute \(X_{\vartheta}\) for \(\vartheta\), to find the phase resulting from the integration:
\begin{align*}
\Phi_0(\omega,\zeta;x_0,\xi_0) := &\; X_{\vartheta} \cdot \zeta + \omega \cdot (\xi_0 - \zeta) + \frac{1}{2}\langle Q_0 X_{\vartheta}, X_{\vartheta} \rangle - \frac{1}{2}\langle \bar{Q}_0 (X_{\vartheta} - \omega), X_{\vartheta} - \omega \rangle    \\
    &+ O(|\omega|^2 + |\zeta|^2)^2 \equiv \tilde{\varphi}_0 \pmod{I_{\tilde{\varphi}_0}}
\end{align*}
(wherein, \(|X_{\vartheta}|^2 \lesssim |\omega|^2 + |\zeta|^2\) gives size of the higher order terms).

We elucidate this phase a bit: note that
\begin{equation}\begin{aligned}
X_{\vartheta} - \omega = \frac{i}{2}(\Im Q_0)^{-1}(\zeta + \Re Q_0 \cdot \omega)  - \frac{1}{2} \omega = \bar{X}_{\vartheta} .
\end{aligned}  \nonumber  \end{equation}
Letting \(S := \Re Q_0\) and \(T := \Im Q_0\), this gives
\begin{align*}
\Im\Phi_0(\omega,\zeta;x_0,\xi_0) = &\; \frac{1}{2}\langle T^{-1}\zeta,\zeta \rangle + \frac{1}{2}\langle T^{-1}S\, \omega, \zeta\rangle    \\
    &- \frac{1}{4} \Im\langle T^{-1}Q_0 T^{-1}(\zeta + \bar{Q}_0 \omega), \zeta + \bar{Q}_0 \omega \rangle  \\
    &+ O(|\omega|^2 + |\zeta|^2)^2
\end{align*}
and after expanding fully, we find that the Hessian of \(\Im \Phi_0\) at \((\omega,\zeta)=(0,0)\) is,
\begin{equation}\begin{aligned}
\frac{1}{4}\begin{pmatrix}
T^{-1} & T^{-1} S   \\
S T^{-1} & T + ST^{-1}S
\end{pmatrix} =
\frac{1}{4}\begin{pmatrix}
I & 0   \\
S   & I
\end{pmatrix}
\begin{pmatrix}
T^{-1} & 0  \\
0   & T
\end{pmatrix}
\begin{pmatrix}
I & S   \\
0   & I
\end{pmatrix} .
\end{aligned}  \nonumber  \end{equation}
By the admissibile phase assumption, \(T \sim \langle \xi_0 \rangle I\) at \((\omega,\zeta) = (0,0)\), therefore this decomposition shows that this Hessian \(H_0\) of \(\Im \Phi_0\) is also positive definite, whence, there exists a constant \(C > 1\) such that
\begin{equation}\begin{aligned}
C^{-1} (|\omega|^2 + |\zeta|^2) \leq \Im \Phi_0 = \langle(\omega,\zeta) | H_0 |(\omega,\zeta)\rangle + O(|\omega|^2 + |\zeta|^2)^2 \leq C(|\omega|^2 + |\zeta|^2) .
\end{aligned}  \nonumber  \end{equation}
This proves part (2) with \(H := 2H_0\).

By \citep[Theorem 7.7.12]{hormander_vol_I}, there are differential operators \(L_j\) of order \(2j\) acting in the \(\omega,\zeta\) variables, with coefficients that are dependent only on \(\tilde{\varphi}_0\) such that
\begin{equation}\begin{aligned}
\left|T_h[\psi_h] - e^{\frac{i}{h} \Phi_0(\omega,\zeta;x_0,\xi_0)} \sum_{j=0}^{N-1} \tilde{b}^0 (L_j[a ])^0(\omega,\zeta;x_0,\xi_0;h) h^j \right| \lesssim_N h^{N + n/2}.
\end{aligned}  \nonumber  \end{equation}
Here, \(\tilde{b} := (\det[\partial_{\vartheta}^2 \tilde{\varphi}_0/(2\pi i h)])^{-\frac{1}{2}}\) and (following Hörmander) we denote for a function \(G(x,\xi,y;x_0,\xi_0)\), its reduction modulo \(I_{\tilde{\varphi}_0}\) by \(G^0 \equiv G \pmod{I_{\tilde{\varphi}_0}}\), which can be determined by a change of variables to \(G(\vartheta,\omega,\zeta;x_0,\xi_0;h) := G(x_0 - \omega, \xi_0 - \zeta,x_0-\vartheta;x_0, \xi_0;h)\), followed by Taylor expansion about \((\vartheta,\omega,\zeta) = 0\) and substitution of \(X_{\vartheta}\) for \(\vartheta\). We note that \(\tilde{b}^0 \equiv (h/\pi)^{\frac{n}{2}} [(\det H_0)^{-\frac{1}{2}} + O(|\omega|^2 + |\zeta|^2)]\). Then, the series can be summed, upon taking \(N \to \infty\), to \(h^{-\frac{n}{4}}\) times a smooth function \(c_0(\omega,\zeta;x_0,\xi_0;h)\) of order zero in \(h\) in a neighbourhood of \((\omega,\zeta) = (0,0)\), which is non-vanishing at \(\{(\omega, \zeta,x,\xi, h) ~|~ \omega = \zeta = 0, h = 0 \}\) due to the non-vanishing of \(a\) and the Hessian of \(\tilde{\varphi}_0\) there. After re-labelling the variables to give \(\Phi(x,\xi;x_0,\xi_0) := \Phi_0(\omega,\zeta;x,\xi)\) and \(c(x,\xi,x_0,\xi_0;h) := c_0(\omega,\zeta;x_0,\xi_0;h)\), altogether this brings us to
\begin{equation}\begin{aligned}
T_h[\psi_h] = h^{-\frac{n}{4}} c(x,\xi;x_0,\xi_0;h) e^{\frac{i}{h}\Phi(x,\xi;x_0,\xi_0)} + O(h^{\infty}) .
\end{aligned}  \nonumber  \end{equation}
\end{proof}

\begin{remark} The statement of the theorem expresses the phase and symbol of the result of \(T_h[\psi_h]\) in terms of \((x,\xi;x_0,\xi_0)\), but since we know that the functions are localized to a neighbourhood of \((x_0,\xi_0)\) in the variable \((x,\xi)\), we can just as well switch to local coordinates and write these as functions of \((\omega,\zeta;x_0,\xi_0)\) localized about \((\omega,\zeta) = (0,0)\) as in the proof. This will be useful at times in the forthcoming calculations.

\end{remark}

We now have quite precise information about the \emph{localization} of \(T_h[\psi_h]\), but we are left with the symbol \(c\) as a factor, which persists when trying to access the value of the symbol of a \(\PDO\) at \((x_0, \xi_0) \in T^*\mathcal{M}\) through \(\eqref{eq:expectation-FBI}\); this would later be obstructive to locating geodesic points through operator expectations. To alleviate this, it is sufficient to use \(L^2\)-normalized states \(\psi_h/||\psi_h||\): by \protect\hyperlink{thm:FBI-basic}{Theorem \ref{thm:FBI-basic}} and the Lemma just proved, using \(T_h := T_h[\cdot;\phi,b^{-\frac{1}{2}}]\) gives,
\begin{equation}    \label{eq:norm-coherent-FBI-xi-dep-2}
\langle \psi_h | \psi_h \rangle = \langle\psi_h| T_h^* T_h | \psi_h \rangle + O(h^\infty) = h^{-\frac{n}{2}}||c e^{\frac{i}{h}\Phi}||^2_{(x,\xi)}(x_0,\xi_0;h) + O(h^\infty),
\end{equation}
where \(|| \cdot ||_{(x,\xi)}\) denotes the norm in \(L^2(T^*\mathcal{M})\) in the variables \((x,\xi)\) and the form on the right-hand side leads to the following,

\begin{theorem}[{Symbol from C-S expectation}]	\hypertarget{thm:sym-cs-psido}{\label{thm:sym-cs-psido}} Let \(\phi\) be an admissible phase, and \(\psi_h(x;x_0,\xi_0) := C_h(x_0,\xi_0) e^{-\frac{i}{h} \bar{\phi}(x_0,\xi_0 ; x)}\) be normalized to \(||\psi_h||_{L^2(\mathcal{M})} = 1\). Then, for all \(m \in \mathbb{Z}\) and \(a \in h^0 S^m(T^*\mathcal{M})\),
\begin{equation}\begin{aligned}
\langle \psi_h | \operatorname{Op}_h(a) | \psi_h\rangle = a(x_0,\xi_0;h) + O(h).
\end{aligned}  \nonumber  \end{equation}

\end{theorem}

\begin{proof}

The preceding theorem tells that on Taylor expanding in a neighbourhood of \(x = x_0\), \(\xi = \xi_0\),
\begin{align*}
T_h[\psi_h]^2 &= e^{-\frac{2}{h} \Im \Phi}  \\
    &= e^{-\frac{1}{h}||(x - x_0,\xi - \xi_0)||_H^2}(1 + \sum_{j=1}^{\infty}O(|x - x_0|^4 + |\xi - \xi_0|^4)^j/h^j ).
\end{align*}
This, together with \(\eqref{eq:norm-coherent-FBI-xi-dep-2}\), Taylor expansion of \(|c|^2\) about \((x,\xi) = (x_0,\xi_0)\) and \(d[{(s_{x_0}^{-1})}^* \nu_g](v) = \sqrt{|g_{s_{x_0}^{-1}(v)}|} dv\) (see \citep[Prop C.III.2]{berger1971spectre}) along with \(d\xi \, dx = (d\xi/\sqrt{|g_x|}) \, d\nu_g(x)\) gives,
\begin{align*}
h^{\frac{n}{2}} C_h(x_0,\xi_0)^{-2} &= ||ce^{\frac{i}{h}\Phi}||^2(x_0,\xi_0)    \\
    &= \int |c|^2(x,\xi;x_0,\xi_0;h) e^{-\frac{2}{h}\Im \Phi (x-x_0,\xi - \xi_0)} ~ dxd\xi \\
    &= \int e^{-\frac{1}{h} ||(x - x_0,\xi-\xi_0)||_H^2}[1 + \sum_{j=1}^{\infty} O(|x - x_0|^4 + |\xi - \xi_0|^4)^j/h^j]    \times  \\
    &\quad\quad \times [|c|^2(x_0,\xi_0;x_0,\xi_0;h) + O(|x - x_0| + |\xi - \xi_0|)] ~ dxd\xi   \\
    & = \int_{\mathbb{R}^n} \int_{V_{x_0}} e^{-\frac{1}{h} ||(v,\zeta)||_H^2}[1 + \sum_{j=1}^{\infty} O(|v|^4 + |\zeta|^4)^j/h^j]   \times  \\
    &\quad\quad \times [|c|^2(x_0,\xi_0;x_0,\xi_0;h) + O(|v| + |\zeta|)] \, ~ dv d\zeta + O(h^{\infty}) \\
    &= h^{n} [m_{\Phi}|c|^2(x_0,\xi_0;x_0,\xi_0;h) + O(h)], \\
m_{\Phi} &:= \int_{\mathbb{R}^{2n}} e^{-||(v,\zeta)||^2_H} dvd\zeta .
\end{align*}
Therefore, another application of the preceding theorem and the same method of Taylor expansions gives, for any \(a \in h^0 S^m\),
\begin{align*}
\langle \psi_h|T_h^* a T_h|\psi_h\rangle &= h^{-\frac{n}{2}}\int_{T^*\mathcal{M}} a(x,\xi) \, C_h(x_0,\xi_0)^2 |c|^2(x,\xi;x_0,\xi_0;h) e^{-\frac{2}{h} \Im \Phi(x-x_0,\xi-\xi_0)} ~ dxd\xi \\
    &= h^{-\frac{n}{2}} \int a\; e^{-\frac{2}{h}\Im\Phi} \frac{|c|^2(x_0,\xi_0,x_0,\xi_0;h) + O(|x-x_0| + |\xi - \xi_0|)}{h^{\frac{n}{2}} [m_{\Phi}|c|^2(x_0,\xi_0,x_0,\xi_0;h) + O(h)]} ~ dxd\xi   \\
    &= h^{-n}\int a\; e^{-\frac{2}{h}\Im\Phi}[1/m_{\Phi} + O(|x - x_0| + |\xi - \xi_0|) + O(h)] ~ dxd\xi    \\
    &= h^{-n}\int e^{-\frac{1}{h}||(x - x_0), (\xi - \xi_0)||_H^2}[a(x_0,\xi_0)/m_{\Phi} + O(|x - x_0| + |\xi - \xi_0|) + O(h)]\times   \\
    &   \quad\quad \times \; [1 + \sum_{j=1}^{\infty} O(|x - x_0|^4 + |\xi - \xi_0|^4)^j/h^j] ~ dx d\xi     \\
    &= h^{-n}\int_{\mathbb{R}^{2n}} [a(x_0,\xi_0)/m_{\Phi} + O(|v| + |\zeta|)] \times   \\
        &\quad\quad \times \; e^{-\frac{1}{h}||v,\zeta||_H^2}[1 + \sum_{j=1}^{\infty} h^{-j} O(|v|^4 + |\zeta|^4)^j][1 + O(|v| + |\zeta|) + O(h)]   \, dv d\zeta    \\
    &= a(x_0,\xi_0) + O(h).
\end{align*}

The domains of integration have been localized and expanded (as in the fourth equality) at will, due to the fact that a symbol \(a \in h^0 S^m\) has at most polynomial growth in the cotangent fibres, so \(|T_h[\psi_h]|^2\) localizes the integrand to an \(O(\sqrt{h})\) ball about \((x_0, \xi_0) \in T^*\mathcal{M}\). On application of \protect\hyperlink{thm:FBI-basic}{Theorem \ref{thm:FBI-basic}}, we have the last part of the statement of the Theorem.
\end{proof}

Having seen the utility of the function \(\psi_h\), we now dignify it with its own,

\begin{definition}[{Coherent States}]	\hypertarget{def:coherent-state}{\label{def:coherent-state}} With \(\phi\) an admissible phase and \(h > 0\), we call \(\psi_h(x ; x_0, \xi_0) := C_h(x_0,\xi_0) e^{-\frac{i}{h} \bar{\phi}(x_0,\xi_0 ; x)}\), with \(C_h := ||e^{\frac{i}{h} \phi}||_{L^2(\mathcal{M})}\), a \emph{coherent state} \emph{localized at} \((x_0,\xi_0) \in T^*\mathcal{M}\), or simply a \emph{coherent state} for brevity. At times, we will also use the \emph{un-normalized coherent state}, \(\tilde{\psi}_h := e^{-\frac{i}{h} \bar{\phi}(x_0, \xi_0 ; x)}\).

\end{definition}

\hypertarget{state-preparation}{%
\subsection{State preparation}\label{state-preparation}}

We discuss briefly some considerations in practically constructing coherent states. At the outset, when the isometry \(\iota : \mathcal{M} \to \mathbb{R}^{\rdim}\) is available, a simple prescription for a phase is \(\phi_{\iota} := \langle \Pi_{\iota}^T(x) \cdot \xi, \iota(x) - \iota(y) \rangle_{\mathbb{R}^{\rdim}} + \frac{i}{2} \frac{\langle \xi \rangle}{\langle \xi_0 \rangle} |\iota(y) - \iota(x)|_{\mathbb{R}^{\rdim}}^2\), with \(\Pi_{\iota}(x) \cong D\iota|_x^T : \mathbb{R}^{\rdim} \supset T^*_{\iota(x)}\embedM \to T^*_x \mathcal{M} \cong \mathbb{R}^{\mdim}\) and a fixed \((x_0, \xi_0) \in T^*\mathcal{M}\). Then,
\begin{equation}\begin{aligned}
D\phi_{\iota}|_{y = x} = -\xi \, dx, \quad\quad D^2 \Im(\phi)|_{y = x} = \frac{\langle \xi \rangle}{\langle \xi_0 \rangle} D\iota|_x^T D\iota|_x \sim \langle \xi \rangle ,
\end{aligned}  \nonumber  \end{equation}
so \(\phi_{\iota}\) is an admissible phase and \(\phi_{\iota}(x_0, \xi_0 ; y) = \langle \Pi_{\iota}^T(x_0) \cdot \xi_0, \iota(x_0) - \iota(y) \rangle_{\mathbb{R}^{\rdim}} + \frac{i}{2} |\iota(y) - \iota(x_0)|^2_{\mathbb{R}^{\rdim}}\). Hence, if \(\psi_{\iota,h}\) is a coherent state localized at \((x_0, \xi_0)\) with phase \(\phi_{\iota}\), then for any \(a \in h^0 S^{m}\) we have by \protect\hyperlink{thm:sym-cs-psido}{Theorem \ref{thm:sym-cs-psido}},
\begin{equation}\begin{aligned}
\langle \psi_{\iota,h} | \Op_h(a) | \psi_{\iota,h} \rangle = a(x_0, \xi_0 ; h) + O(h).
\end{aligned}  \nonumber  \end{equation}
The vector \(\xi_{\iota,0} := \Pi_{\iota}^T(x_0) \cdot \xi_0\) is just the coordinate representation of \(\xi_0\) in the hyperplane tangent to \(\embedM\) with center \(\iota(x_0)\). If normal coordinates are placed at \(x_0\) and \(x_*\) is in a normal neighbourhood of \(x_0\), then we understand that \(\xi_0 := s_{x_0}(x_*)/|s_{x_0}(x_*)|_{\mathbb{R}^{\mdim}}\) is the direction in which a geodesic originating at \(x_0\) must emanate, to reach \(x_*\). Then, given \(\iota(x_*) \in \embedM\) in a small neighbourhood of \(\iota(x_0)\), a reasonable approximation to \(\xi_{\iota,0}\) is given by
\begin{align*}
\xi_* := \frac{\iota(x_*) - \iota(x_0)}{|\iota(x_*) - \iota(x_0)|_{\mathbb{R}^{\rdim}}} &= D\iota|_{x_0} \cdot \frac{s_{x_0}(x_*)}{|s_{x_0}(x_*)|_{\mathbb{R}^{\mdim}}} + O(|\iota(x_*) - \iota(x_0)|_{\mathbb{R}^{\rdim}}) \\
    &= \Pi_{\iota}^T(x_0) \cdot \xi_0 + O(|\iota(x_*) - \iota(x_0)|_{\mathbb{R}^{\rdim}}) ,
\end{align*}
wherein the first equality follows from \protect\hyperlink{lem:ext-normal-coords}{Lemma \ref{lem:ext-normal-coords}} and Taylor's theorem upon expanding \(\iota \circ s_{x_0}^{-1}\) in a neighbourhood of the origin. Now, supposing \(|\iota(x_*) - \iota(x_0)|_{\mathbb{R}^{\rdim}} \lesssim h^{\frac{\mdim}{4} + 2}\) and setting \(\hat{\phi}_{\iota} := \langle \xi, \iota(x) - \iota(y) \rangle_{\mathbb{R}^{\rdim}} + \frac{i}{2} |\iota(y) - \iota(x)|_{\mathbb{R}^{\rdim}}^2\), another application of Taylor's theorem gives,
\begin{align*}
e^{-\frac{i}{h} \bar{\hat{\phi}}_{\iota}(x_0, \xi_* ; x)}
    &= e^{-\frac{i}{h} [\Re(\hat{\phi}_{\iota})(x_0,\xi_{\iota,0} + O(h^2); x) - i\Im(\phi_{\iota})(x_0,\xi_0;x)]}  \\
    &= e^{-\frac{i}{h} \bar{\phi}_{\iota}(x_0, \xi_{\iota,0} ; x)} + O_{L^{\infty}}(h^{\frac{\mdim}{4} + 1}) ,
\end{align*}
wherein we've expanded in a neighbourhood of \(\xi_{\iota,0}\) in increments of the error term coming from the Taylor expansion of \(\xi_*\). Then, \(\hat{\psi}_{\iota,h} := e^{-\frac{i}{h} \bar{\hat{\phi}}_{\iota}(x_0, \xi_* ; x)} / ||e^{\frac{i}{h} \hat{\phi}_{\iota}(x_0, \xi_* ; \cdot)}||_{L^2}\) satisfies \(||\hat{\psi}_{\iota,h} - \psi_{\iota,h}||_{\infty} = O(h)\), from which we recover,
\begin{equation}\begin{aligned}
\langle \hat{\psi}_{\iota,h} | \Op_h(a) | \hat{\psi}_{\iota,h} \rangle = a(x_0, \xi_0 ; h) + O(h) .
\end{aligned}  \nonumber  \end{equation}
This means that we can use points in a sufficiently small neighbourhood of \(\iota(x_0)\) to determine unit momentum vectors at which to recover, to order \(O(h)\), the value of the symbol of a given \(\PDO\).

The specifications of an admissible phase make coherent states particularly amenable to prescription in local coordinate patches. Let \(u : \mathcal{M} \supset \mathscr{O} \to V \subset \mathbb{R}^{\mdim}\) be a diffeomorphism providing local coordinates in a neighbourhood \(\mathscr{O}\) about \(x_0 \in \mathcal{M}\). Then, for a fixed \(\xi_0 \in T^*_{x_0}\mathcal{M}\), a simple construction of a phase is of the sort,
\begin{equation}\begin{aligned}
\phi_u(x, \xi ; y) := \langle Du|_x^{-T}\xi, u(x) - u(y) \rangle_{\mathbb{R}^{\mdim}} + \frac{i}{2} \frac{\langle \xi \rangle}{\langle \xi_0 \rangle}|u(y) - u(x)|^2_{\mathbb{R}^{\mdim}}
\end{aligned}  \nonumber  \end{equation}
and we have,
\begin{equation}\begin{aligned}
D_y \phi_u|_{y = x} = -\xi, \quad\quad D_y^2 \Im(\phi_u)|_{y = x} = \frac{\langle \xi \rangle}{\langle \xi_0 \rangle} Du|_x^T Du|_x ,
\end{aligned}  \nonumber  \end{equation}
so \(\phi_u\) is admissible in \(\mathscr{O}\). Also suppose \(\chi \in C^{\infty}\) is a cut-off in \(\supp \chi \subset \mathscr{O}\) such that \(\chi \equiv 1\) on \(\overline{\mathscr{O}}_0 \subset \mathscr{O}\) with \(\mathscr{O}_0\) an open neighbourhood of \(x_0\). Letting \(\psi_{u,h}\) be the coherent state localized at \((x_0, \xi_0)\) with phase \(\phi_u\), we find that \protect\hyperlink{thm:wunsch-zworski}{Lemma \ref{thm:wunsch-zworski}} holds also for \(T_h[\chi \psi_{u,h}]\) since \(\chi\) cuts off the integration domain, with a slightly different symbol \(c\) that still satisfies all of the properties stated in the Lemma. Further, \(||\psi_{u,h} \chi||_{L^2} = 1 + O(h^{\infty})\), so by \protect\hyperlink{thm:sym-cs-psido}{Theorem \ref{thm:sym-cs-psido}}, we have,
\begin{equation} \label{eq:cs-local-coord-sym}
\langle \psi_{u,h} \chi| \Op_h(a) | \chi \psi_{u,h} \rangle = a(x_0,\xi_0) + O(h).
\end{equation}

\hypertarget{from-graph-laplacians-to-geodesic-flows}{%
\section{From graph Laplacians to geodesic flows}\label{from-graph-laplacians-to-geodesic-flows}}

As discussed in the \protect\hyperlink{introduction}{Introduction}, there is now a wide collection of results on the convergence of graph Laplacians to Laplace-Beltrami operators (modulo lower order terms) on manifolds. These take the following form: starting from a collection of \(N\) samples, they approximate \(\GLap_{\lambda,\epsilon}\) in the \(N \to \infty\) limit and in turn, \(\GLap_{\lambda,\epsilon}\) approximates \(\Delta_{\mathcal{M}} + O(\partial^1)\) in the sense that for each \(u \in C^{\infty}\), \(||\GLap_{\epsilon}u - [\Delta_{\mathcal{M}} + O(\partial^1)]u||_{\infty} = O(\epsilon)\) (we denote by \(O(\partial^1)\) a partial differential operator of order at most one). These are results along the lines of those displayed in \protect\hyperlink{laplacians-from-graphs-to-manifolds}{Section \ref{laplacians-from-graphs-to-manifolds}} and they are closely related to diffusion processes \citep{nadler2006diffusion}. We now wish to bring a \emph{quantum} perspective to graph Laplacians. Methods of semi-classical analysis enable us to instantiate a quantum-classical correspondence via semiclassical \(\Psi\)DOs. Namely, as discussed in \protect\hyperlink{quantization-and-symbol-classes}{Section \ref{quantization-and-symbol-classes}}, a \emph{quantization procedure} assigns a linear operator \(\operatorname{Op}_h(a) : C^{\infty}(\mathcal{M}) \to C^{\infty}(\mathcal{M})\) to a function \(a \in C^{\infty}(T^*\mathcal{M} \times (0, h_0])\) (for some \(h_0 < 1\)) belonging to a class of \emph{symbols} (\emph{classical observables}); these are, roughly speaking, functions that have controlled variation in regions of phase space \(T^*\mathcal{M}\) of unit volume, uniformly with scaling the unit length in each fibre \(T_x^*\mathcal{M}\) as \(1/h\) for \(h \in (0, h_0]\). If a symbol \(q\) is real-valued, then we may treat it as a Hamiltonian that generates a flow \(\Flow_q^t\) with respect to time \(|t| \leq T\) on \(T^*\mathcal{M}\). By Egorov's theorem, we can relate, or \emph{quantize}, the Hamiltonian flow \(\Flow_q^t\) with respect to a real-valued symbols \(q\) to \emph{Heisenberg dynamics} \(A(t) := U_q^{-t} A U_q^t\) of \(A := \Op_h(a)\) for \(U_q^t := e^{-\frac{i}{h} t \operatorname{Op}_h(q)}\), in the sense that the theorem gives conditions on \(q\) and \(a\) such that \(A(t)\) has principal symbol \(a \circ \Flow_q^t\). In particular, using \(q := |\xi|_x\) gives the quantization of the geodesic flow \(\Gamma^t\). Our first aim is to give light to the conjugate flow on \(\Psi\)DOs by \(U_{\lambda,\epsilon}^t := e^{-i t \sqrt{\Delta_{\lambda,\epsilon}}}\) as the quantization of a flow on symbols \(a \circ \Flow^t_q\) for an appropriate \(q\). To this end, we will find that when \(\epsilon = h^{2 + \alpha}\) with \(\alpha > 0\), we have \(h^2 \GLap_{\lambda,\epsilon} = \operatorname{Op}_h(q_{\lambda,\alpha})\) with \(q_{\lambda,\alpha}\) a symbol of \emph{order two} and to order\footnote{We write for \(a \in C^{\infty}\) that \(a = O_{\mathscr{S}}(h)\) whenever there is a constant \(h_0 > 0\) such that for each \(\beta\), there is a constant \(C_{\beta} > 0\) such that for \(h \in [0, h_0)\), \(|\partial^{\beta} a| \leq C_{\beta} h\).} \(O_{\mathscr{S}}(h)\), this agrees with \(|\xi|_x^2\) compact neighbourhoods of the origin in \(T^*\mathcal{M}\) whenever \(\alpha > 1\). In terms of quantization, the scaling \(\epsilon = h^{2 + \alpha}\) has a simple interpretation: since the kernel \(k(x,y;\epsilon)\) is localized to \(\sqrt{\epsilon}\)-balls on \(\mathcal{M}\), the corresponding symbol must be localized to an \(h/\sqrt{\epsilon}\)-ball in phase space, which is above unit volume (\emph{viz.}, the uncertainty principle) when \(\epsilon \in o(h^2)\).

The next task is to extract the geodesic flow: once we have quantized it in a region of phase space about a cosphere bundle \(S_r^*\mathcal{M}\) for \(r > 0\), we can use coherent states localized at \((x_0, \xi_0) \in S_r^*\mathcal{M}\) to retreive \(a \circ \Gamma^t(x_0, \xi_0)\) up to \(O(h)\) error. This works because the localization of a coherent state restricts the action of \(\operatorname{Op}_h(a)\) to that of another whose symbol agrees with \(a\) on an \(O(\sqrt{h})\) neighbourhood of \((x_0, \xi_0)\) and vanishes outside this. Since, as we will see, \(q_{\lambda,\alpha} = |\xi|_x^2 + O_{\mathscr{S}}(h)\) on a fixed size neighbourhood of \(S_r^*{\mathcal{M}}\), this allows us to reduce to \(U_{\lambda,\epsilon}^{-t} \Op_h(a) U_{\lambda,\epsilon}^t | \psi_h \rangle = U_{|\xi|\chi}^{-t} \Op_h(a) U_{|\xi|\chi}^t |\psi_h \rangle + O_{L^2}(h)\) with \(\chi\) an appropriate cut-off equal to one in part of the neighbourhood. Tying this together with Egorov's theorem and the results of \protect\hyperlink{semi-classical-measures-of-coherent-states}{Section \ref{semi-classical-measures-of-coherent-states}} leads to the desired propagation result: \(\langle \psi_h | U_{\lambda,\epsilon}^{-t} \operatorname{Op}_h(a) U_{\lambda,\epsilon}^t | \psi_h \rangle = a \circ \Gamma^t(x_0, \xi_0) + O(h)\).

\hypertarget{symbol-of-a-graph-laplacian}{%
\subsection{Symbol of a graph Laplacian}\label{symbol-of-a-graph-laplacian}}

We now compute the symbol of a graph Laplacian operator. An important aspect of this is that in order to \emph{see} the symbol of a graph Laplacian, the quantization must be done at lengths \(h\) \emph{above} the \emph{sampling density} \(\sqrt{\epsilon}\). We proceed with an \emph{intrinsic} version of the averaging operator, which readily exposes this basic facet of the application of semiclassical analysis: namely, define
\begin{equation} \label{eq:intrinsic-diffop}
\AvOp_{g,\epsilon} : C^{\infty} \ni u \mapsto \epsilon^{-\frac{n}{2}} \int_{\mathcal{M}} k(d_g(\cdot,y)^2/\epsilon) u(y) ~ p(y) d\nu_g(y) \in C^{\infty} .
\end{equation}
This operator takes on the form \(\eqref{def:op-quantization}\) in the following way,
\begin{align*}
\AvOp_{g,\epsilon}&[u](x)   \\
    &\equiv \int_{\mathbb{R}^{\mdim}} \int_{\mathbb{R}^{\mdim}} e^{-\frac{i}{h} v \cdot \xi} \mathcal{F}_h^{-1}[k(|z|^2/\epsilon)]_{z \to \xi}(\xi) \, (u \cdot p \cdot \chi_x) \circ s_x^{-1}(v) \sqrt{|g_{s_x^{-1}(v)}|} ~ d\xi dv    \\
    &\equiv \frac{1}{(2 \pi h)^{\mdim}} \int_{\mathcal{M}} \int_{T_x^* \mathcal{M}} e^{-\frac{i}{h} \langle s_x(y), \xi \rangle_{g_x}} \left( \epsilon^{-\frac{\mdim}{2}} \int_{T_x^*\mathcal{M}} e^{\frac{i}{h} \langle z, \xi \rangle_{g_x}} k(|z|^2_{g_x}/\epsilon) \frac{dz}{\sqrt{|g_x|}} p(y) \right) \\
    &\quad\quad \times u(y) \chi_x(y) ~ \frac{d\xi}{\sqrt{|g_x|}} d\nu_g(y) \\
    &\equiv \frac{1}{(2 \pi h)^{\mdim}} \int_{T^*\mathcal{M}} e^{\frac{i}{h} \langle s_y(x), \eta \rangle_{g_y}} \left( \epsilon^{-\frac{\mdim}{2}} \int_{T_y^*\mathcal{M}} e^{\frac{i}{h} \langle \zeta, \eta \rangle_{g_y}} k(|\zeta|^2_{g_y}/\epsilon) \frac{d\zeta}{\sqrt{|g_y|}} p(y) \right) u(y) \chi_x(y) dyd\eta   \\
  &\equiv \frac{1}{(2 \pi h)^{\mdim}} \int_{T^*\mathcal{M}} e^{\frac{i}{h} \langle s_y(x), \eta \rangle_{g_y}} H_{g,h,\epsilon}(y,\eta) u(y) \chi_x(y) ~ d\eta dy
\end{align*}
with
\begin{align} \label{eq:sym-intrinsic-diff-op}
\begin{split}
H_{g,h,\epsilon}(x,\xi) &:= \epsilon^{-\frac{\mdim}{2}} \int_{T_x^*\mathcal{M}} e^{\frac{i}{h} \langle \zeta,\xi \rangle_{g_x}} k(|\zeta|^2_{g_x}/\epsilon) ~ \frac{d\zeta}{\sqrt{|g_x|}} ~ p(x)    \\
    & = \epsilon^{-\frac{\mdim}{2}} \int_{\mathbb{R}^{\mdim}} e^{\frac{i}{h}\langle \zeta, \xi \rangle} k(|g_x^{\frac{1}{2}} \zeta|^2/\epsilon) ~ d\zeta \sqrt{|g_x|} \, p(x)
\end{split}
\end{align}
and wherein \(\equiv\) denotes equality up to a term of order \(O(\epsilon^{\infty}) \in C^{\infty}\), \(\chi_x(y)\) is a smooth cut-off within the normal neighbourhood of \(x\) such that \(\chi_x(y) = 1\) on a fixed, open neighbourhood of \(x\), \(\mathcal{F}_h^{-1}[\cdot(z)](\xi)\) is the inverse \(h\)-Fourier transform in \(\mathbb{R}^{\mdim}\) with kernel \((2\pi h)^{-\mdim} e^{\frac{i}{h} z \cdot \xi}\) and we have used the parallel transport operator \(\mathcal{T}_{y \to x} : T^*_y\mathcal{M} \to T^*_x\mathcal{M}\) along the unit speed geodesic from \(y\) to \(x\) to make the changes of variables \(\eta := \mathcal{T}_{y \to x} \xi\) and \(\zeta := \mathcal{T}_{y \to x} z\). The features of parallel transport used here are that it is an isometry with \(\mathcal{T}_{y \to x}^* = \mathcal{T}_{x \to y}\), \(\mathcal{T}_{y \to x} s_y(x) = -s_x(y)\) and \(d\mathcal{T}_{x \to y}(\cdot) = (|g_x|/|g_y|)^{\frac{1}{2}} d(\cdot)\). The fact that \(d[{(s_x^{-1})}^* \nu_g](v) = \sqrt{|g_{s_x^{-1}(v)}|} dv\) is another consequence of parallel transport we've used and its justification can be found in \citep[Prop C.III.2]{berger1971spectre}.

The formal symbol \(H_{g,h,\epsilon}\) has
\begin{align*}
\langle \xi \rangle^{|\beta|-m} & |\partial_x^{\gamma}\partial_{\xi}^{\beta} H_{g,h,\epsilon}|  \\
    &= \langle \xi \rangle^{|\beta|-m} \left|\partial_x^{\gamma} \; \int_{\mathbb{R}^{\mdim}} (i \epsilon^{\frac{1}{2}}/h)^{|\beta|} z^{\beta} e^{i \frac{\sqrt{\epsilon}}{h}\langle z, \xi \rangle} k(|g_x^{\frac{1}{2}}z|^2) ~ dz ~ |g_x^{\frac{1}{2}}| p(x) \right|  \\
    &= \left( \frac{\epsilon}{h^2} \right)^{\frac{m}{2}}\left| \int_{\mathbb{R}^{\mdim}} e^{i\frac{\sqrt{\epsilon}}{h}\langle z, \xi \rangle} (\epsilon/h^2 + \Delta_z)^{\frac{|\beta|-m}{2}}\left( z^{\beta} \, \partial_x^{\gamma}[k(|g_x^{\frac{1}{2}} z|^2) |g_x^{\frac{1}{2}}| p(x)] \right) dz \right|,
\end{align*}
which is uniformly bounded in \((0, 1]_h \times T^*\mathcal{M}\) for \(\epsilon = h^{2 + \alpha}\) with \(\alpha \geq 0\) and each \(m, |\gamma|, |\beta| \geq 0\), while taking \(\beta = 0\) with \(m < 0\) and evaluating the right-hand side at \(\xi = 0\) shows that this quantity grows as \(\Omega(h^{m\alpha/2})\) whenever \(\alpha > 0\) and is uniformly bounded when \(\alpha = 0\). Thus, \(H_{g,h^{\alpha}} := H_{g,h,h^{2 + \alpha}} \in h^0 S^0\) for \(\alpha > 0\) and \(H_{g,h^0} \in h^0 S^{-\infty}\). On the other hand if \(\alpha < 0\), then for any \(m > 0, \delta < 0\), taking \(|\beta| = m\) gives
\begin{align*}
\sup_{(x,\xi) \in T^*\mathcal{M}} & h^{\delta} |\partial_{\xi}^{\beta} H_{g,h,h^{2+\alpha}}| \gtrsim_p h^{\alpha \frac{m}{2} + \delta} \left|\int_{\mathbb{R}^n} z^{\beta} k(|g_x^{\frac{1}{2}} z|^2) ~ dz\right| \gtrsim_{p, \beta} h^{\alpha \frac{m}{2} + \delta} ,
\end{align*}
hence \(H_{g,h,h^{2 + \alpha}} \not\in h^k S^m\) for any \(k, m \in \mathbb{Z}\). This shows,

\begin{lemma}	\hypertarget{lem:intrinsic-diffusion-is-psido}{\label{lem:intrinsic-diffusion-is-psido}} Let \(\epsilon, h \in (0,1]\), \(\AvOp_{g,\epsilon}\) be given by \(\eqref{eq:intrinsic-diffop}\) and \(H_{g,h,\epsilon}\) be given by \(\eqref{eq:sym-intrinsic-diff-op}\). Then, \(\AvOp_{g,\epsilon}\) with \(\epsilon = h^{2 + \alpha}\) for \(\alpha \in \mathbb{R}\) is a \(\Psi\text{DO}\) if and only if \(\alpha \geq 0\), in which case \(\AvOp_{g,h^{2 + \alpha}} = \operatorname{Op}_h(H_{g,h^\alpha})\) with \(H_{g,h^{\alpha}} := H_{g,h,h^{2 + \alpha}} \in h^0 S^0\) for \(\alpha > 0\) and \(H_{g,h^{\alpha}} \in h^0 S^{-\infty}\) whenever \(\alpha = 0\).

\qed

\end{lemma}

The \emph{extrinsically defined} averaging operator \(\AvOp_{\epsilon}\) given by \(\eqref{def:averaging-op}\) has small deformations from an isotropic kernel in normal coordinates that must be accounted for in its description as a \(\PDO\). Even so, the regime for \(\epsilon\) and \(h\) giving pseudodifferentiality is the same as in the intrinsic case and the extrinsic deviations only give lower-order terms in the symbol: this we now see in,

\begin{lemma}	\hypertarget{lem:averaging-op-is-psido}{\label{lem:averaging-op-is-psido}} Let \(\epsilon, h \in (0, 1]\). Then, \(\AvOp_{\epsilon}\) with \(\epsilon = h^{2 + \alpha}\) for \(\alpha \in \mathbb{R}\) is a \(\PDO\) if and only if \(\alpha \geq 0\), in which case \(\AvOp_{h^{2 + \alpha}} \equiv \Op_h(H_{g,h^{\alpha}}) \pmod{h^{\ell} \Psi^{-m}}\) with \((\ell, m) = (2 + \alpha(1 - m/2), m)\).

\end{lemma}

\begin{remark} We note the following cases of the order of the sub-principal symbol of \(\AvOp_{h^{2 + \alpha}}\): we have the extremes \((\ell, m) = (1, 2(1 + 1/\alpha))\) and \((\ell, m) = (2 + \alpha/2, 1)\) and the \emph{balanced case}, \((\ell, m) = (2, 2)\).

\end{remark}

\begin{proof}

We begin with a series expansion of the kernel of \(\AvOp_{\epsilon}\): by \protect\hyperlink{lem:ext-normal-coords}{Lemma \ref{lem:ext-normal-coords}} and Taylor's formula,
\begin{gather*}
|\iota(x) - \iota \circ s_x^{-1}(v)|^2 = |v|^2 + \sum_{|\beta|=4} v^{\beta} \tilde{E}_{\iota,\beta}(x,v) ,  \\
\tilde{E}_{\iota,\beta}(x,v) := \frac{1}{6} \int_0^1 (1 - t)^3 D_v^{\beta}[|\iota(x) - \iota \circ s_x^{-1}(\cdot)|^2]|_{\cdot = tv} ~ dt ,
\end{gather*}
hence denoting \(k^{(1)}(t) := \partial_t k(t)\) and \(E_{\iota}(x,y) := \sum_{|\beta| = 4} s_x(y)^{\beta} \tilde{E}_{\iota,\beta}(x, s_x(y))\) we have,
\begin{gather*}
k(|\iota(x) - \iota(y)|^2/\epsilon) - k(d_g(x,y)^2/\epsilon) = k_{\iota,\epsilon}(x,y), \\
\quad\quad k_{\iota,\epsilon}(x,y) := \frac{E_{\iota}(x,y)}{\epsilon} \int_0^1 k^{(1)}(d_g(x,y)^2/\epsilon + t E_{\iota}(x,y)/\epsilon) ~ dt .
\end{gather*}
There is a constant \(C_{\mathcal{M},\iota} > 0\) depending only on the geometry of \(\mathcal{M}\) and embedding \(\iota\) so that for all pairs \((x,y) \in \mathcal{M}^2\) in a geodesically convex neighbourhood, \(|E_{\iota}(x,y)| \leq C_{\mathcal{M},\iota} d_g(x,y)^4\) and by the defining assumptions on \(k\), there is \(R_k > 0\) such that \(k\) and \(k^{(1)}\) decay exponentially outside of \([0,R_k]\). Therefore, if \(\chi : \mathbb{R} \to \mathbb{R}\) is an even, non-negative smooth cut-off supported in \([-2 R_k, 2 R_k]\) and such that \(\chi(t) = 1\) on \(|t| \leq R_k\), then
\begin{equation}\begin{aligned}
k(|\iota(x) - \iota(y)|^2/\epsilon) - k(d_g(x,y)^2/\epsilon) = \tilde{k}_{\iota,\epsilon}(x,y) + O_{\mathscr{S}}(\epsilon^{\infty}),    \\
\tilde{k}_{\iota,\epsilon}(x,y) := k_{\iota,\epsilon}(x,y) \, \chi(d_g(x,y)^2/\epsilon - C_{\mathcal{M},\iota} d_g(x,y)^4/\epsilon)
\end{aligned}  \nonumber  \end{equation}
and there is a constant \(C > 0\) depending only on \(C_{\mathcal{M}, \iota}\) such that \(\supp[\tilde{k}_{\iota,\epsilon}(x, \cdot)] \subseteq B_{C \sqrt{\epsilon}}(x) =: \tilde{B}_{\sqrt{\epsilon}}(x) \subset \mathcal{M}\).

Now consider the operator
\begin{align*}
\AvOp_{\epsilon} & - \AvOp_{g,\epsilon} \\
    &\quad \equiv \AvOp^{(1)}_{\epsilon} : C^{\infty} \ni u \mapsto \epsilon^{-\frac{n}{2}} \int_{\mathcal{M}} \tilde{k}_{\iota,\epsilon}(\cdot,y) \, u(y) ~ p(y) d\nu_g(y) \in C^{\infty} \pmod{h^{\infty}\Psi^{-\infty}} .
\end{align*}
We wish to see that this is a \(\PDO\) with symbol belonging to \(h^{\ell} S^{-m}\) for some \(\ell, m \geq 1\), which would establish \(\AvOp_{\epsilon}\) as a \(\PDO\) with \(h\)-principal symbol \(H_{g,\epsilon}\). By the definition of \(\PDO\)s on a manifold, we must thus establish that \(\AvOp^{(1)}_{\epsilon}\) is \emph{pseudolocal} and locally resembles a \(\PDO\) on an open subset of Euclidean space. We begin with the former property: suppose \(\chi_1, \chi_2 \in C^{\infty}\) with \(\supp \chi_1 \cap \supp \chi_2 = \emptyset\). Since \(\tilde{k}_{\iota,\epsilon}(x,\cdot) \chi_1(\cdot) = 0\) whenever \(\supp \chi_2 \cap \tilde{B}_{\sqrt{\epsilon}}(x) = \emptyset\), in order for \(\chi_1(x) \tilde{k}_{\iota,\epsilon}(x,\cdot) \chi_1(\cdot) \neq 0\) we must have \(\supp \chi_2 \cap \tilde{B}_{\sqrt{\epsilon}}(x) \neq \emptyset\) and \(x \in \supp \chi_1 \cap \tilde{B}_{\sqrt{\epsilon}}(x) \neq \emptyset\), but when \(\epsilon > 0\) is sufficiently small, \(\tilde{B}_{\sqrt{\epsilon}}(x) \subset \supp \chi_1\), so uniformly over \((0, 1]_{\epsilon} \times \mathcal{M}^2\) we have \(|\chi_1(x) \tilde{k}_{\iota,\epsilon}(x,y) \chi_2(y)| \in O(\epsilon^{\infty})\). The same holds for all derivatives of \(\chi_1(x) \tilde{k}_{\iota,\epsilon}(x,y) \chi_2(y)\) since the support is stable under differentiation; furthermore, due to smoothness of this kernel, we have that \(\chi_1 \AvOp^{(1)}_{\epsilon}[\chi_2 \cdot] \in h^{\infty} \Psi^{-\infty}\), \emph{viz}., \(\AvOp_{\epsilon}^{(1)}\) is \emph{pseudolocal}.

We wish to see now that \(\AvOp_{\epsilon}^{(1)}\) gives, under a change of coordinates within a given patch, a \(\PDO\) on the corresponding open set in \(\mathbb{R}^{\mdim}\). So fix an atlas for \(\mathcal{M}\) with \(\gamma : \mathcal{M} \supset U \to V \subset \mathbb{R}^{\mdim}\) providing local coordinates on the patch \(U\) and let \(\chi_1, \chi_2 \in C^{\infty}\) with \(\supp \chi_1, \supp \chi_2 \subset U\). Then,
\begin{align*}
\AvOp^{(1)}_{\epsilon}[\chi_2 u](x) &\equiv
    \epsilon^{-\frac{\mdim}{2}} \int_{V} \tilde{k}_{\iota,\epsilon}(x,\gamma^{-1}(v)) [(\chi_2 \cdot p \cdot u) \circ \gamma^{-1}](v) \, |\det D\gamma^{-1}(v)| \sqrt{|g_{\gamma^{-1}(v)}|} ~ dv    \\
    &\equiv \frac{1}{(2\pi h)^{\mdim}} \int_{\mathbb{R}^{\mdim}}\int_{\mathbb{R}^{\mdim}} e^{\frac{i}{h}\langle \gamma(x) - v, \xi \rangle} a_{\gamma}(\gamma(x),\xi) \, (\chi_2 \cdot u) \circ \gamma^{-1}(v) ~ dvd\xi ,
\end{align*}
with
\begin{gather*}
a_{\gamma}(\tilde{w},\xi) := \epsilon^{-\frac{\mdim}{2}} \int_{\mathbb{R}^{\mdim}} e^{\frac{i}{h} \langle \xi, w - \tilde{w} \rangle} b_{\gamma,\epsilon}(\tilde{w}, w) ~ dw ,  \\
b_{\gamma,\epsilon}(\tilde{w},w) := \tilde{k}_{\iota,\epsilon}(\gamma^{-1}(\tilde{w}), \gamma^{-1}(w)) \, p \circ \gamma^{-1}(w) \, |\det D\gamma^{-1}(w)| \sqrt{|g_{\gamma^{-1}(w)}|}
\end{gather*}
and wherein \(\equiv\) denotes equality up to a term \(O_{\mathscr{S}}(\epsilon^{\infty})\) that is given by the action of a smoothing operator. The definition of \(a_{\gamma}\) is sensible because whenever \(x \in U\), with \(\epsilon\) sufficiently small we have \(\supp[\tilde{k}_{\iota,\epsilon}(x, \cdot)] \subset \tilde{B}_{\sqrt{\epsilon}}(x) \subset U\).

We now have \(\chi_1 \AvOp_{\epsilon}^{(1)}[\chi_2 u] = \chi_1 \gamma^* \Op_h^{KN}(a_{\gamma}) (\gamma^{-1})^* [\chi_2 u]\) in the formal sense of quantizing a smooth scalar-valued function; we wish to see that for a certain regime of \(\epsilon\) with respect to \(h\), \(a_{\gamma}\) indeed belongs to an appropriate symbol class. We first switch to normal coordinates, wherein homothetic rescalings apply cleanly so that contracting by \(\epsilon^{\frac{1}{2}}\) in these coordinates exposes the decay of \(a_{\gamma}\) in \(\epsilon\). On a change of variables \(w = \gamma \circ s_x^{-1}(v)\) with \(x = \gamma^{-1}(\tilde{w})\), we have
\begin{gather*}
a_{\gamma}(\tilde{w}, \xi) = \epsilon^{-\frac{\mdim}{2}} \int_{\mathbb{R}^{\mdim}} e^{\frac{i}{h} \langle \xi, \gamma \circ s_x^{-1}(v) - \gamma \circ s_x^{-1}(0) \rangle} b_{\iota,\epsilon}(\tilde{w}, v) ~ dv , \\
b_{\iota,\epsilon}(\tilde{w},v) := \tilde{k}_{\iota,\epsilon}(s_x^{-1}(0),s_x^{-1}(v)) \, p \circ s_x^{-1}(v) \, |\det Ds_x^{-1}(v)| \sqrt{|g_{s_x^{-1}(v)}|} .
\end{gather*}
We may now apply the Kuranishi trick to the phase appearing in the representation of \(a_{\gamma}\) by way of the Taylor expansion
\begin{gather*}
\gamma \circ s_x^{-1}(v) - \gamma \circ s_x^{-1}(0) = D[\gamma \circ s_x^{-1}]|_{v = 0} \cdot v + E_{\gamma}(\tilde{w},v) , \\
E_{\gamma}(\tilde{w},v) := \sum_{|\beta| = 2} v^{\beta} \int_0^1 D_v^2[\gamma \circ s_x^{-1}](tv) ~ dt
\end{gather*}
with \(F(\tilde{w}) := D[\gamma \circ s_x^{-1}]|_{v = 0}\) a smooth matrix-valued function, invertible at all \(\tilde{w} \in V\). Since
\begin{gather*}
E_{\gamma}(\tilde{w}, \epsilon^{\frac{1}{2}}v) = \epsilon E_{\gamma,\epsilon}(\tilde{w}, v),    \\
E_{\gamma,\epsilon}(\tilde{w},v) := \sum_{|\beta| = 2} v^{\beta} \int_0^1 D_v^2[\gamma \circ s_x^{-1}](t \epsilon^{\frac{1}{2}} v) ~ dt ,
\end{gather*}
another Taylor expansion then gives,
\begin{gather*}
e^{\frac{i}{h} \langle \xi, F(\tilde{w}) \cdot \epsilon^{\frac{1}{2}} v + E_{\gamma}(\tilde{w}, \epsilon^{\frac{1}{2}} v) \rangle}
    =   e^{i \frac{\sqrt{\epsilon}}{h} \langle \xi, F(\tilde{w}) \cdot v \rangle}\left( 1 + \epsilon^{\frac{1}{2}} \Theta_h(\tilde{w},v,\xi,\epsilon) \right),  \\
\Theta_h(\tilde{w},v,\xi,\epsilon) := \sum_{j=1}^{\infty} \frac{\epsilon^{\frac{j-1}{2}}}{j!} \langle E_{\gamma,\epsilon}(\tilde{w},v), (\epsilon^{\frac{1}{2}}/h) \, i \xi \rangle^{j} .
\end{gather*}
We also have,
\begin{gather*}
\tilde{k}_{\iota,\epsilon}(x, s_x^{-1}(\epsilon^{\frac{1}{2}} v)) = \epsilon \tilde{b}_{\iota,\epsilon}(x,v) ,  \\
\tilde{b}_{\iota,\epsilon}(x,v) := \chi(|v|^2 - \epsilon C_{\mathcal{M},\iota} |v|^4) E_{\iota}(x,s_x^{-1}(v)) \int_0^1 k^{(1)}\left( |v|^2 + \epsilon t E_{\iota}(x, s_x^{-1}(v)) \right) dt .
\end{gather*}
Therefore,
\begin{align*}
a_{\gamma}(\tilde{w},\xi) &= \epsilon \int_{\mathbb{R}^{\mdim}} e^{i \frac{\sqrt{\epsilon}}{h}\langle \xi, v \rangle}\left( 1 + \epsilon^{\frac{1}{2}} \Theta_h(\tilde{w}, \tilde{v}, \xi, \epsilon) \right) \tilde{b}_{\iota,\epsilon}(\gamma^{-1}(\tilde{w}), \tilde{v})  \\
    &\quad\quad \times p \circ s_x^{-1}(\epsilon^{\frac{1}{2}} \tilde{v}) \, |\det D s_x^{-1}(\epsilon^{\frac{1}{2}} \tilde{v})| \, |g_{s_x^{-1}(\epsilon^{\frac{1}{2}}\tilde{v})}|^{\frac{1}{2}} ~ d\tilde{v} ,    \\
\tilde{v} &:= F(\tilde{w})^{-1} \cdot v
\end{align*}
and an integration by parts gives,
\begin{gather*}
a_{\gamma}(\tilde{w},\xi) = \epsilon \int_{\mathbb{R}^{\mdim}} e^{i \frac{\sqrt{\epsilon}}{h} \langle \xi, v \rangle} (B_{\iota,\epsilon}(\tilde{w}, \tilde{v}) + \epsilon^{\frac{1}{2}} \tilde{B}_{\iota,\epsilon}(\tilde{w}, \tilde{v})) ~ d\tilde{v} ,   \\
B_{\iota,\epsilon}(\tilde{w},\tilde{v}) := \tilde{b}_{\iota,\epsilon}(x,\tilde{v}) \, p \circ s_x^{-1}(\epsilon^{\frac{1}{2}} \tilde{v}) \, |\det D s_{x}^{-1}(\epsilon^{\frac{1}{2}} \tilde{v})| \, |g_{s_x^{-1}(\sqrt{\epsilon} \tilde{v})}|^{\frac{1}{2}} ,
\end{gather*}
wherein \(\supp B_{\iota,\epsilon}(\tilde{w}, \cdot) \subset \{ |\tilde{v}| \leq C \}\) and
\begin{equation}\begin{aligned}
\tilde{B}_{\iota,\epsilon}(\tilde{w},\tilde{v}) := \sum_{j=1}^{\infty} \frac{\epsilon^{\frac{j-1}{2}}}{j!} \langle E_{\gamma,\epsilon}(\tilde{w}, \tilde{v}), \nabla_v \rangle^j[B_{\iota, \epsilon}(\tilde{w}, \cdot)](\tilde{v})
\end{aligned}  \nonumber  \end{equation}
is a convergent Taylor series that defines a smooth function, \(\tilde{B}_{\iota,\epsilon}(\tilde{w}, \tilde{v}) \in C_c^{\infty}(V \times \{|\tilde{v}| \leq C \})\). Now we proceed as we've done before:
\begin{align*}
\langle \xi & \rangle^{|\beta|-m} |\partial_{\tilde{w}}^{\vartheta}\partial_{\xi}^{\beta} a_{\gamma}|   \\
    &= \epsilon \langle \xi \rangle^{|\beta|-m} \left|\partial_{\tilde{w}}^{\vartheta} \; \int_{\mathbb{R}^{\mdim}} (i \epsilon^{\frac{1}{2}}/h)^{|\beta|} v^{\beta} e^{i \frac{\sqrt{\epsilon}}{h}\langle \xi, v \rangle} (B_{\iota,\epsilon}(\tilde{w}, \tilde{v}) + \epsilon^{\frac{1}{2}} \tilde{B}_{\iota,\epsilon}(\tilde{w},\tilde{v})) ~ d\tilde{v} \right|     \\
    &= \frac{\epsilon^{1 + \frac{m}{2}}}{h^m} \left| \int_{\mathbb{R}^{\mdim}} e^{i\frac{\sqrt{\epsilon}}{h}\langle \xi, v \rangle} (\epsilon/h^2 + \Delta_v)^{\frac{|\beta|-m}{2}}\left( v^{\beta} \, \partial_{\tilde{w}}^{\vartheta} \left[ (B_{\iota,\epsilon}(\tilde{w}, \tilde{v}) + \epsilon^{\frac{1}{2}} \tilde{B}_{\iota,\epsilon}(\tilde{w},\tilde{v}))/|\det F(\tilde{w})| \right] \right) dv \right|
\end{align*}
so that using \(\epsilon = h^{2 + \alpha}\) gives \(\epsilon^{1 + \frac{m}{2}}/h^m = h^{2 + \alpha(1 + m/2)}\). The Fourier transform is of a smooth, compactly supported function, hence it is uniformly bounded in \((0, 1]_h \times T^*\mathcal{M}\). Thus, if \(\alpha > 0\) then \(a_{\gamma} \in h^{\ell} S^{-m}\) for \((\ell, m) = (2 + \alpha(1 - m/2), m)\) with \(m \in \mathbb{N}\), \(m \leq 2(1 + 1/\alpha)\) and we may change quantizations to the adjoint form by a transformation of the symbol \citep[Theorem 4.13]{zworski2012}, which retains its order.
\end{proof}

\begin{example*} The simplest case is to take for \(\AvOp_{g,\epsilon}\) the operator with Schwartz kernel \(k_{g,\epsilon}(x,y) := \epsilon^{-\frac{\mdim}{2}} e^{-d_g(x,y)^2/\epsilon}\) and uniform density \(p \equiv 1\). Then, an \(h\)-Fourier transformation shows that \(H_{g,h,\epsilon} := \pi^{\frac{3\mdim}{2}} e^{-\frac{\epsilon}{4 h^2} |\xi|_{g_x}^2} |g_x|^{-\frac{1}{2}}\). Recall that due to the uncertainty principle, a semiclassical \(\Psi\)DO is defined for symbols with small variation in regions of phase space with unit volume. This symbol \(H_{g,h,\epsilon}\) varies polynomially within a ball of radius \(O(h/\sqrt{\epsilon})\) in each cotangent fibre \(T^*_x\mathcal{M}\) and decays exponentially outside of this ball. When \(\epsilon = h^{2 + \alpha}\), this symbol has exponential decay within a ball of radius \(h^{-\alpha/2 - \delta}\) for \(\delta > 0\), so for \(\alpha < 0\) this is amounts to large variation in sub-unit volume regions of phase space. Another (\emph{dual}) perspective is that in order for \(\AvOp_{g,\epsilon}\) to be regarded as a semiclassical \(\Psi\)DO in quantization scale \(h\), by pseudolocality it must have Schwartz kernel \(k_{g,\epsilon}\) that decays, along with all derivatives, as \(O(h^{\infty} |x - y|^{-\infty})\) outside of the diagonal. The critical rate here is \(\sqrt{\epsilon}\), in the sense that for points with \(0 < d_g(x,y) < \sqrt{\epsilon}\), the Schwartz kernel \(k_{g,\epsilon}\) follows this decay rate if and only if \(h > \sqrt{\epsilon}\). Effectively, if \(h\) remains below this rate, then all of the \emph{content} of the symbol is lost and hence the quantum-classical correspondence breaks down: quantization forces such a symbol to be essentially zero (an impulse at the zero section of the cotangent bundle), while \(\AvOp_{g,\epsilon}\) behaves as \(\epsilon \to 0\) like an operator with symbol having support in all of the cotangent bundle and moreover, going to a constant as \(h \to 0\).

\end{example*}

We may now employ symbol calculus to express graph Laplacians as \(\PDO\)s:

\begin{theorem}	\hypertarget{thm:sym-renorm-graph-lap}{\label{thm:sym-renorm-graph-lap}} Let \(\lambda \geq 0\), \(\epsilon, h \in (0, 1]\) with \(\epsilon = h^{2 + \alpha}\) and \(\alpha > 0\). Then,
\begin{equation}\begin{aligned}
\mathcal{L}_{g,\lambda,h^{\alpha}} := \frac{2 c_0}{c_2} \frac{1 - H_{g,h^{\alpha}}/(p_{\epsilon}^{2\lambda} p_{\lambda,\epsilon})}{h^{\alpha}} \in h^0 S^2
\end{aligned}  \nonumber  \end{equation}
and \(h^2 \GLap_{\lambda,\epsilon} \in h^0 \Psi^2\) with
\begin{equation}\begin{aligned}
h^2 \GLap_{\lambda,\epsilon} \equiv \Op_h(\mathcal{L}_{g,\lambda,h^{\alpha}}) \pmod{h \Psi} .
\end{aligned}  \nonumber  \end{equation}
Moreover, if we fix \(0 \leq \delta \leq 1\) and \(C > 0\) and let \(\chi_{C h^{-\delta}} : \mathbb{R} \to \mathbb{R}\) be a smooth function with \(\supp \chi_{C h^{-\delta}} \subseteq [-C h^{-2\delta}, C h^{-2\delta}]\), then for all pairs \((m,\alpha) \in \mathbb{Z} \times [1, \infty)\) that satisfy \(\alpha \geq 1 + \delta(4 + m)\), we have
\begin{equation}\begin{aligned}
\Op_h(\mathcal{L}_{g,\lambda,h^{\alpha}} \chi_{C h^{-\delta}}(|\xi|_{g_x}^2)) \equiv \Op_h(|\xi|_{g_x}^2 \chi_{C h^{-\delta}}(|\xi|_{g_x}^2)) \pmod{h \Psi^{-m}}
\end{aligned}  \nonumber  \end{equation}
and in particular, when \(\delta = 0\), this holds with \(m = \infty\) for all \(\alpha \geq 1\).

\end{theorem}

\begin{remark} The picture to keep in mind for the symbol is, for example, given in Figure \(\ref{fig:glaps-symbol}\) for the kernel \(k(d_g(x,y)^2/\epsilon) := \exp(-d_g(x,y)^2/\epsilon)\) with \(\epsilon = h^{2 + \alpha}\). The function \(\mathcal{L}_{g,\lambda,h^{\alpha}}\) approximates \(|\xi|_{g_x}^2\) within a neighbourhood of \(0 \in T^*\mathcal{M}\) that grows as \(h \to 0\). The rate of this approximation governs simultaneously, the symbol class and the momentum magnitude within which we can recover \(|\xi|_{g_x}^2\), set by \(m\) and \(\delta\), respectively. As the figure shows, \(\mathcal{L}_{g,\lambda,h^{\alpha}}\) is clearly not principal in \(h^0 S^2\); hence, as we will see later, dynamics generated by the graph Laplacian connect well to Hamiltonian dynamics on \(T^*\mathcal{M}\) when we can take \(m = \infty\).

\end{remark}

\begin{proof}

Since \protect\hyperlink{lem:averaging-op-is-psido}{Lemma \ref{lem:averaging-op-is-psido}} gives \(\AvOp_{\epsilon}\) as a \(\PDO\) and the operators of multiplication by \(p_{\lambda,\epsilon}^{-1}, p_{\epsilon}^{-\lambda} \in C^{\infty}\) are readily seen to be \(\PDO\)s with symbols \(p_{\lambda,\epsilon}^{-1}, p_{\epsilon}^{-\lambda} \in h^0 S^0\) respectively, we may employ the symbol calculus upon writing \(A_{\lambda,\epsilon}\) as a product of these operators to get:
\begin{gather*}
A^{(1)}_{\lambda,\epsilon} := A_{\lambda,\epsilon} - A_{g,\lambda,\epsilon} = p_{\lambda,\epsilon}^{-1} p_{\epsilon}^{-\lambda} \AvOp^{(1)}_{\epsilon} p_{\epsilon}^{-\lambda} \in h^{\ell} \Psi^{-m} , \\
A_{g,\lambda,\epsilon}[\cdot] := p_{\lambda,\epsilon}^{-1} p_{\epsilon}^{-\lambda} \AvOp_{g,\epsilon}[p_{\epsilon}^{-\lambda} \cdot] \in h^0 \Psi^0
\end{gather*}
with \((\ell, m) \in \mathbb{R} \times \mathbb{Z}\) as in \protect\hyperlink{lem:averaging-op-is-psido}{Lemma \ref{lem:averaging-op-is-psido}} and
\begin{equation}\begin{aligned}
\tilde{\GLap}_{\lambda,h,\alpha} := h^2 \frac{c_2}{2 c_0} \GLap_{\lambda,\epsilon} = \frac{I - A_{\lambda,\epsilon}}{h^{\alpha}} = \frac{I - A_{g,\lambda,\epsilon}}{h^{\alpha}} - h^{-\alpha} A^{(1)}_{\lambda,\epsilon} .
\end{aligned}  \nonumber  \end{equation}
Since \(h^{-\alpha} \GAve^{(1)}_{\lambda,\epsilon} \in h^{\ell'} \Psi^{-m}\) with \((\ell', m) = (2 - \alpha m/2, m)\) and
\begin{equation}\begin{aligned}
\Op_h(H_{g,h^{\alpha}}/(p_{\epsilon}^{2\lambda} p_{\lambda,\epsilon})) \equiv \AvOp_{g,\epsilon} \circ \frac{1}{p_{\epsilon}^{2\lambda} p_{\lambda,\epsilon}} \pmod{h^{\infty} \Psi^{-\infty}},
\end{aligned}  \nonumber  \end{equation}
we have
\begin{gather}
\tilde{\GLap}_{\lambda,h,\alpha} \equiv (p_{\epsilon}^{\lambda} p_{\lambda,\epsilon})^{-1} \tilde{\GLap}_{g,\lambda,h^{\alpha}} \circ p_{\epsilon}^{\lambda} p_{\lambda,\epsilon} - h^{-\alpha} \GAve^{(1)}_{\lambda,\epsilon} \pmod{h^{\infty} \Psi^{-\infty}},    \label{eq:glap-intrinsic-conj-renorm} \\
\tilde{\GLap}_{g,\lambda,h^{\alpha}} := \frac{I - \Op_h(H_{g,h^{\alpha}}/(p_{\epsilon}^{2\lambda} p_{\lambda,\epsilon}))}{h^{\alpha}} , \nonumber
\end{gather}
hence to see the first part of the Theorem, it suffices to show that \(L_{g,\lambda,h^{\alpha}} := (1 - H_{g,h^{\alpha}}/(p_{\epsilon}^{2\lambda} p_{\lambda,\epsilon}))/h^{\alpha} \in h^0 S^2\).

A change of variables in the representation \(\eqref{eq:sym-intrinsic-diff-op}\) shows the form,
\begin{equation}\begin{aligned}
H_{g,h^{\alpha}}(x,\xi) = \int_{\mathbb{R}^{\mdim}} e^{i h^{\frac{\alpha}{2}} \langle z, g_x^{-\frac{1}{2}} \xi \rangle} k(|z|^2) ~ dz ~ p(x)
\end{aligned}  \nonumber  \end{equation}
and combining this with the defining assumption that \(k\) decays at least as fast as an exponential implies that \(H_{g,h^{\alpha}}\) is analytic in \(\xi\). Thus, by a Taylor series at \(\xi = 0\) we have the expansion
\begin{gather*}
(H_{g,h^{\alpha}}/(p_{\epsilon}^{2\lambda} p_{\lambda,\epsilon}))(x,\xi) = \frac{p(x)}{p_{\epsilon}^{2\lambda} p_{\lambda,\epsilon}(x)}\sum_{|\kappa| \geq 0} (-1)^{|\kappa|} h^{\alpha |\kappa|} (g_x^{-\frac{1}{2}}\xi)^{2\kappa} c_{2\kappa}/(2\kappa)! ,    \\
c_{2\kappa} := \int z^{2\kappa} k(|z|^2) ~ dz .
\end{gather*}
By \protect\hyperlink{lem:taylor-expand-deg-func}{Lemma \ref{lem:taylor-expand-deg-func}}, there is a function \(q_{\lambda} \in C^{\infty}\) depending only on \(p\), \(\mathcal{M}\), \(c_0\) and \(\lambda\) such that
\begin{equation}\begin{aligned}
p_{\lambda,\epsilon} = (c_0 p)^{1 - 2\lambda} + \epsilon \, c_2 p^{1 - 2\lambda} q_{\lambda} + O(\epsilon^2).
\end{aligned}  \nonumber  \end{equation}
Applying this to \(p_{\epsilon} = p_{0, \epsilon}\) and taking a Taylor series shows
\begin{gather*}
p_{\epsilon}^{2\lambda} = (c_0 p)^{2\lambda} + \epsilon \, (c_0 p)^{2\lambda - 1} \tilde{q}_{\lambda} + O(\epsilon^2),  \\
\tilde{q}_{\lambda} := 2\lambda c_2 p \, q_0 .
\end{gather*}
Therefore,
\begin{gather*}
p_{\lambda,\epsilon} p_{\epsilon}^{2\lambda} = c_0 p + \epsilon \tilde{q}_{\lambda,0}  + O(\epsilon^2), \\
\tilde{q}_{\lambda,0} := \tilde{q}_{\lambda} + c_0^{2\lambda} c_2 p \, q_{\lambda}
\end{gather*}
and upon taking a geometric series expansion this leads to
\begin{equation}\begin{aligned}
p/(p_{\lambda,\epsilon} p_{\epsilon}^{2\lambda}) = c_0^{-1}[1 - \epsilon \, \tilde{q}_{\lambda,0}/(c_0 p) + O(\epsilon^2)].
\end{aligned}  \nonumber  \end{equation}
Employing this expansion we now have,
\begin{align}   \label{eq:glap-intrinsic-sym-expansion}
\begin{split}
&\frac{1 - H_{g,h^{\alpha}}/(p_{\epsilon}^{2 \lambda} p_{\lambda,\epsilon})}{h^{\alpha}}    \\
    &\quad= h^{-\alpha} - h^{-\alpha} \frac{c_0 p}{p_{\epsilon}^{2 \lambda} p_{\lambda,\epsilon}} + \frac{p}{p_{\epsilon}^{2 \lambda} p_{\lambda,\epsilon}}\sum_{|\kappa| \geq 1} \frac{(-1)^{|\kappa|-1}}{(2\kappa)!} h^{\alpha(|\kappa| - 1)} (g_x^{-\frac{1}{2}} \xi)^{2\kappa} c_{2\kappa} \\
    &\quad= \frac{c_2}{2}\frac{p}{p_{\epsilon}^{2 \lambda} p_{\lambda,\epsilon}}|\xi|^2_{g_x} + h^2 \tilde{q}_{\lambda,0}/(c_0 p) + O_x(h^{4 + \alpha}) \\
  & \quad\quad + \frac{p}{p_{\epsilon}^{2 \lambda} p_{\lambda,\epsilon}}\sum_{|\kappa| \geq 2} (-1)^{|\kappa| - 1} h^{\alpha(|\kappa| - 1)} (g_x^{-\frac{1}{2}} \xi)^{2\kappa} c_{2\kappa}/(2\kappa)! \\
  &\quad= \frac{c_2}{2 c_0} |\xi|^2_{g_x} + O_x(h^2) + \frac{p}{p_{\epsilon}^{2 \lambda} p_{\lambda,\epsilon}}\sum_{|\kappa| \geq 2} \frac{(-1)^{|\kappa| - 1}}{(2\kappa)!} h^{\alpha(|\kappa| - 1)} (g_x^{-\frac{1}{2}} \xi)^{2\kappa} c_{2\kappa} ,
\end{split}
\end{align}
wherein the \(O_x(h^2)\) term denotes that it is a function only of \(x\) and \(h\) and is uniformly bounded (due to compactness of \(\mathcal{M}\)) by a constant multiple of \(h^2\). Denoting \(L_{g,\lambda, h^{\alpha}} := (1 - H_{g,h^{\alpha}}/(p_{\epsilon}^{2\lambda} p_{\lambda,\epsilon}))/h^{\alpha}\), we see from this, that for small \(|\xi| \leq h^{-\frac{\alpha}{2}}\), \(m \geq 2\), \(0 \leq |\beta| \leq m\) and all \(\gamma \geq 0\),
\begin{align*}
\langle \xi \rangle^{|\beta|-m} & |\partial_x^{\gamma} \partial_{\xi}^{\beta} L_{g,\lambda,h^{\alpha}}| \\
    &\lesssim_{\gamma} \langle \xi \rangle^{2-m} + h^2 \langle \xi \rangle^{-m} + \langle \xi \rangle^{|\beta| - m}\sum_{\substack{|\kappa| \geq \\ \max\{ 2, |\beta|/2 \}}} h^{\alpha( |\kappa| - 1)}  |\xi^{2\kappa - \beta}| \, c_{2\kappa}/(2\kappa - \beta)!   \\
    &\lesssim_{\gamma,\beta} (1 + h^{-\alpha})^{\frac{|\beta| - 2 - (m -2)}{2}} (h^\alpha)^{\frac{|\beta|}{2} - 1} \sum_{|\kappa| \geq 2} c_{2\kappa}/(2\kappa - \beta)!    \\
    &\lesssim_{\gamma,\beta} (h^{\alpha} + 1)^{\frac{|\beta|}{2} - 1} (1 + h^{-\alpha})^{-\frac{m-2}{2}}    \\
    &\lesssim_{\gamma,\beta} 1,
\end{align*}
while for large \(|\xi| > h^{-\frac{\alpha}{2}}\),
\begin{align*}
\langle \xi \rangle^{|\beta|-m} &|\partial_x^{\gamma} \partial_{\xi}^{\beta} L_{g,\lambda,h^\alpha}|    \\
        &\lesssim (1 + h^{-\alpha})^{\frac{|\beta| - m}{2}} (h^{\alpha})^{\frac{|\beta|}{2} - 1} \left|\int_{\mathbb{R}^{\mdim}} e^{i h^{\frac{\alpha}{2}}\langle z, \xi \rangle}  z^{\beta} \partial_x^{\gamma} \left( \delta_0(z) - k(|g_x^{\frac{1}{2}} z|^2) |g_x|^{\frac{1}{2}} \frac{p}{p_{\epsilon}^{2\lambda} p_{\lambda,\epsilon}}(x) \right)  ~ dz \right|    \\
        &\lesssim_{\gamma,\beta} (h^{\alpha} + 1)^{\frac{|\beta|}{2}-1} (1 + h^{-\alpha})^{-\frac{m-2}{2}}
\end{align*}
so together this implies \(\langle \xi \rangle^{-m} |\partial_x^{\gamma} \partial_{\xi}^{\beta} L_{g,\lambda,h^\alpha}|\) is uniformly bounded in \((0, 1]_h \times T^*\mathcal{M}\). Further, for \(|\beta| \geq m\) we have
\begin{align}
\langle \xi &\rangle^{|\beta|-m} |\partial_x^{\gamma} \partial_{\xi}^{\beta}L_{g,\lambda,h^{\alpha}}|   \nonumber\\
        &= \left|\int e^{i h^{\frac{\alpha}{2}}\langle z, \xi \rangle}(h^{\alpha})^{\frac{m}{2} - 1} \; (h^{\alpha} + \Delta_z)^{\frac{|\beta|-m}{2}} \left(z^{\beta}  \partial_x^{\gamma} \left[ k(|g_x^{\frac{1}{2}} z|^2) |g_x|^{\frac{1}{2}} \frac{p}{p_{\epsilon}^{2\lambda} p_{\lambda,\epsilon}}(x) \right] \right) ~ dz \right| \label{eq:glap-is-symbol}\\
        &< \infty,  \nonumber
\end{align}
so altogether we find that \(L_{g,\lambda,h^{\alpha}} \in h^0 S^2\). To see that \(L_{g,\lambda,h^{\alpha}}\) has order exactly \((0,2)\) when \(\alpha > 0\), simply evaluate the right-hand side of \(\eqref{eq:glap-is-symbol}\) at \(\xi = 0\) with \(|\beta| \geq 2 > m\).

Now an application of symbol calculus to \(\eqref{eq:glap-intrinsic-conj-renorm}\) shows that for \(\alpha > 0\),
\begin{equation}\begin{aligned}
h^2 \GLap_{\lambda,\epsilon} \equiv \Op_h((2c_0/c_2) \, L_{g,\lambda,h^{\alpha}}) \pmod{h \Psi} .
\end{aligned}  \nonumber  \end{equation}
If we introduce a smooth cut-off \(\chi_{C h^{-\delta}} : \mathbb{R} \to \mathbb{R}\) with \(\supp \chi_{h^{-\delta}} \subseteq C [-h^{-2\delta}, h^{-2\delta}]\) for \(0 < \delta < 1\) and \(C > 0\), then from \(\eqref{eq:glap-intrinsic-sym-expansion}\) we see that for each \(m \in \mathbb{R}\),
\begin{align*}
h^{-1}\langle \xi \rangle^{|\beta| + m} & |\chi_{C h^{-\delta}}(|\xi|_{g_x}^2) \, \partial_x^{\gamma} \partial_{\xi}^{\beta}[L_{g,\lambda,h^{\alpha}} - (c_2/(2 c_0)) |\xi|_{g_x}^2]|   \\
    &\lesssim_{\gamma} h^{-1} (1 + C^2h^{-2\delta})^{\frac{|\beta| + m}{2}} \sum_{\substack{|\kappa| \geq \\ \max\{ 2, |\beta|/2 \}}} h^{\alpha(|\kappa| - 1) - \delta (2|\kappa| - |\beta|)} c_{2\kappa}/(2\kappa - \beta)!    \\
    &\lesssim_{\gamma,\beta} \sum_{\substack{|\kappa| \geq \\ \max\{ 2, |\beta|/2 \}}} h^{-\delta(m + |\beta|)} h^{-1} h^{\alpha(|\kappa| - 1) - \delta (2|\kappa| - |\beta|)} c_{2\kappa}/(2\kappa - \beta)!   \\
    &\lesssim_{\gamma,\beta} \sum_{|\kappa| \geq 2} h^{(|\kappa| - 1)\left( \alpha - \frac{\delta(2|\kappa| + m) + 1}{|\kappa| - 1} \right)} c_{2\kappa}/(2\kappa)! \\
    &\lesssim_{\gamma,\beta} h^{\alpha - (1 + \delta(4 + m))} .
\end{align*}
Since the partial derivatives of \(\chi_{C h^{-\delta}}(|\xi|^2_{g_x})\) are also cut-offs of \(|\xi|^2_{g_x}\) in the same region and the above bound holds with the factor \(\chi_{C h^{-\delta}}\) replaced by any such cut-off, we find that upon summing up via the triangle inequality, we have the same bound up to (new) constants depending on \(\gamma\) and \(\beta\). Therefore,
\begin{equation}\begin{aligned}
\Op_h((2 c_0/c_2) L_{g,\lambda,h^{\alpha}} \chi_{C h^{-\delta}}(|\xi|_{g_x}^2)) \equiv \Op_h(|\xi|_{g_x}^2 \chi_{C h^{-\delta}}(|\xi|_{g_x}^2)) \pmod{h \Psi^{-m}}
\end{aligned}  \nonumber  \end{equation}
whenever \(\alpha \geq 1 + \delta(4 + m)\) so in particular, we may take \(m = \infty\) when \(\delta = 0\).
\end{proof}

\begin{remark} The second part of the Theorem can be generalized to give a lower bound \(\rho + \delta(4 + m)\) on \(\alpha\) if instead of using the unit step size for orders of \(h\) in the pseudodifferential calculus, we take steps of size \(\rho > 0\). In any case, while \(\alpha = 0\) is applicable to the first part of the Theorem --- and then we actually have \(h^2 \GLap_{\lambda,\epsilon} \in h^0 \Psi^{-\infty}\) --- as the second part shows, this would not be able to extract the \emph{kinetic term} \(|\xi|_{g_x}^2\) (for any step size \(\rho > 0\)).

\end{remark}

\begin{figure}[h]
\centering
\includegraphics[width=0.4\textwidth,keepaspectratio]{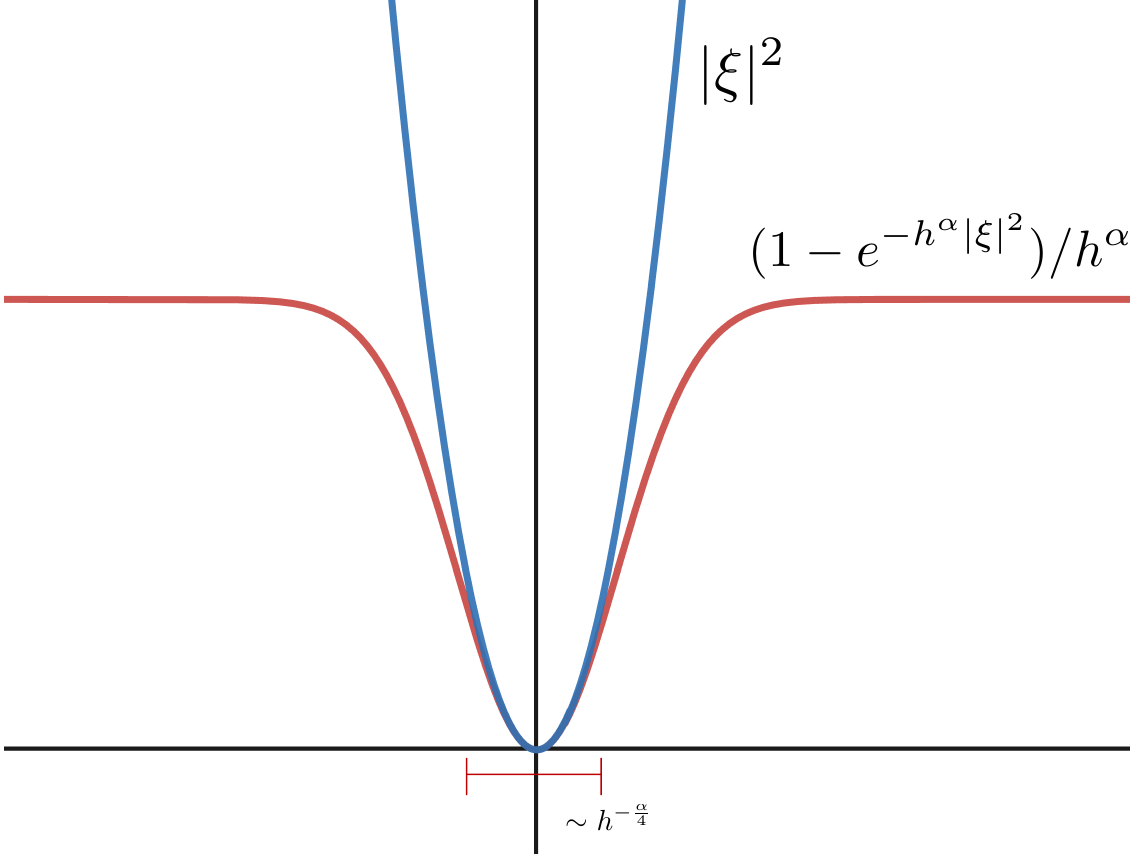}
\caption{The symbol $\mathcal{L}_{g,\lambda,h^{\alpha}}$ for $k(d_g(x,y)^2/h^{2+\alpha}) := \exp(-d_g(x,y)^2/h^{2+\alpha})$ (modulo normalization factors) in comparison with the symbol $|\xi|_{g_x}^2$, depicted in terms of momentum magnitudes $|\xi|$ along the horizontal axis. Note that the modulus of the symbol of the graph Laplacian grows as $\Theta(h^{-\alpha})$ outside of a sufficiently large neighbourhood of $0 \in T^*\mathcal{M}$, which means it is not principal in $h^0 S^2$, but cutting off to a neighbourhood is sufficient for connecting with Hamiltonian mechanics via Egorov's theorem. \label{fig:glaps-symbol}}
\end{figure}

\hypertarget{geodesic-flows-of-symbols}{%
\subsection{Geodesic flows of symbols}\label{geodesic-flows-of-symbols}}

We now extend the relationships between classical observables and their quantized counterparts via coherent states as displayed in \protect\hyperlink{semi-classical-measures-of-coherent-states}{Section \ref{semi-classical-measures-of-coherent-states}}, to Hamiltonian dynamics on the observables and operator dynamics on their quantizations. The basic idea is facilitated by Egorov's theorem, which states roughly that if \(Q\) is a \(\PDO\) with principal symbol \(q_0\) then the \emph{operator dynamics} \(A(t) = e^{-\frac{i}{h} t Q} \Op_h(a) e^{\frac{i}{h} t Q}\) preserves pseudodifferentiality and order, meaning that up to a given time \(T\) constant in \(h\), \(A(0) \in \Psi^m \implies A(t) \in \Psi^m\) for all \(|t| \leq T\) and furthermore, \(A(t) \equiv \Op_h(a \circ \Phi_{q_0}^t) \mod h \Psi^{m-1}\). Therefore, combining with \protect\hyperlink{lem:coherent-localization}{Lemma} leads to \(\langle \psi_h | A(t) | \psi_h \rangle = a \circ \Phi_{q_0}^t + O(h)\). As we have seen, the symbol of a graph Laplacian is locally equal to \(|\xi|_{g_x}^2\). So, for a fixed \((x_0, \xi_0) \in T^*\mathcal{M}\) and \(h\) sufficiently small, we expect the that the Hamiltonian flow given by the square root of (the principal part of) the symbol of a graph Laplacian coincides with the co-geodesic flow in a neighbourhood of this point. An issue, however, is that graph Laplacians are not elliptic, which obstructs directly utilising their square roots. Our way out goes back to the application of coherent states: since \(\psi_h\) is localized about \((x_0, \xi_0)\) in phase space, by going through the FBI transform as in \protect\hyperlink{thm:wunsch-zworski}{Lemma \ref{thm:wunsch-zworski}}, we can approximate it with the application of the quantization of a phase space localized symbol. Upon choosing an appropriate form of this symbol in concert with the spectral gaps in the graph Laplacian, we get approximately a spectral cut-off of the graph Laplacian. Finally, we use this to bring a cut-off into the generator \(Q\), for which we have the following version of Egorov's theorem:

\begin{theorem}[{Egorov}]	\hypertarget{thm:egorov}{\label{thm:egorov}} Let \(q \in C_c^{\infty}(T^*\mathcal{M} \times [0, h_0)) \cap h^0 S^{-\infty}\) for some \(h_0 > 0\) have a real principal symbol \(q_0\) that generates the Hamiltonian flow \(\Phi_{q_0}^t\) on \(T^*\mathcal{M}\). Then, given \(a \in h^{\ell} S^{m}\) and \(T > 0\), for all \(|t| \leq T\), \(a \circ \Phi_{q_0}^t \in h^{\ell} S^{m}\) and
\begin{equation}\begin{aligned}
e^{\frac{i}{h} t Q} \Op_h(a) e^{-\frac{i}{h} t Q} \equiv \Op_h(a \circ \Phi_{q_0}^t) \pmod{h^{\ell + 1} \Psi^{-\infty}}
\end{aligned}  \nonumber  \end{equation}
with \(Q := \Op_h(q)\).

\end{theorem}

\begin{proof}

The main differences from standard treatments are that we use the full symbol to generate the operator dynamics, while we give the correspondence with the flow generated by the principal part and we are using the \emph{standard} quantization as opposed to the Weyl quantization for \(Q\). Thus, while the propagator \(e^{\frac{i}{h} t Q}\) solves an operator equation of the same form, it need not be unitary. Another minor difference is that typically Egorov's theorem is given for polyhomogeneous symbols when the operator dynamics uses the full symbol and the correspondence is to a classical flow with only the principal symbol. Nevertheless, the statement holds following the arguments in the proof of \citep[\emph{Egorov's Theorem} in \(\S 1\) of Ch. 8]{taylor81} and importing them into the semiclassical calculus; we carry this out presently.

Let \(A := \Op_h(a)\) and \(A(t) := U^{-t} \Op_h(a) U^t\) for \(U^t := e^{-\frac{i}{h} t Q}\). Then, \(A(t)\) is a solution to the system,
\begin{equation}\begin{aligned}
\partial_t A(t) = \frac{i}{h} [Q, A(t)], \quad A(0) = A .
\end{aligned}  \nonumber  \end{equation}
We wish to construct a solution \(\tilde{A}(t)\) to this system with principal symbol \(a \circ \Phi^t_{q_0}\) such that \(\tilde{A}(t) - A(t) \in h^{\infty} \Psi^{-\infty}\). That is, we look for \(\tilde{A}(t)\) that satisfies
\begin{equation}\begin{aligned}
\partial_t \tilde{A}(t) = \frac{i}{h}[Q, \tilde{A}(t)] + R(t) , \quad
\tilde{A}(0) = A
\end{aligned}  \nonumber  \end{equation}
with \(R(t) \in h^{\infty} \Psi^{-\infty}\). On the symbol side, we look for \(\tilde{a}(t, x, \xi ; h)\) with asymptotic expansion
\begin{equation}\begin{aligned}
\tilde{a}(t,x,\xi ; h) \sim \sum_{j=0}^{\infty} \tilde{a}_j(t, x, \xi ; h)
\end{aligned}  \nonumber  \end{equation}
such that \(\tilde{a}_j \in h^{\ell + j} S^{m - j}\). Then by \citep[Theorem 9.5]{zworski2012}, the symbol of \(\frac{i}{h}[Q, \tilde{A}(t)]\) must have an asymptotic expansion in \(h^{\ell} S^{m}\) given in local coordinates by,
\begin{align*}
\operatorname{Sym}\left[ \frac{i}{h}[Q, \tilde{A}(t)] \right] &\sim \sum_{|\beta| \geq 0} h^{|\beta|-1} \frac{i^{-|\beta|+1}}{\beta!} (\partial_{\xi}^{\beta} \tilde{a}(t,x,\xi ; h) \partial_x^{\beta} q(x,\xi ; h) - \partial_{\xi}^{\beta} q(x,\xi ; h) \partial_x^{\beta}\tilde{a}(t,x,\xi ; h))    \\
    &= \{ q_0, \tilde{a} \} + \{ \tilde{q}_0, \tilde{a} \} + \sum_{|\beta| \geq 2} h^{|\beta| - 1} \frac{i^{-|\beta| + 1}}{\beta!} (\partial_{\xi}^{\beta} \tilde{a} \, \partial_x^{\beta} q - \partial_{\xi}^{\beta} q \, \partial_x^{\beta} \tilde{a})
\end{align*}
with \(\tilde{q}_0 := q - q_0 \in h S^{-\infty}\). Since \(\partial_t \tilde{A}_t = \Op_h(\partial_t \tilde{a}_t)\), it follows that \(\operatorname{Sym}[i h^{-1}[Q, \tilde{A}(t)]] = \partial_t \tilde{a}_t\), hence the asymptotic expansions would coincide. So let \(\tilde{a}_0(t, x, \xi ; h)\) be defined by the transport equation
\begin{equation}\begin{aligned}
\partial_t \tilde{a}_0 = \{ q_0, \tilde{a}_0 \}, \quad\quad \tilde{a}_0(0, x, \xi) = a(x, \xi),
\end{aligned}  \nonumber  \end{equation}
which is sensible as it is solved by \(\tilde{a}_0 = a \circ \Phi_{q_0}^t\) that is a symbol of order \((\ell, m)\) since \(\Phi_{q_0}^t\) is smooth with compact support. The operator \(\tilde{A}_0(t) := \Op_h(\tilde{a}_0)\) has
\begin{align*}
\operatorname{Sym}\left[ \frac{i}{h} [Q, \tilde{A}_0(t)] - \partial_t \tilde{A}_0(t) \right] &\sim \{ \tilde{q}_0, \tilde{a}_0 \} + \sum_{|\beta| \geq 2} h^{|\beta| - 1} \frac{i^{-|\beta| + 1}}{\beta!} (\partial_{\xi}^{\beta} \tilde{a}_0 \partial_x^{\beta} q - \partial_{\xi}^{\beta} q \partial_x^{\beta} \tilde{a}_0)   \\
    &=: r_0(t,x,\xi) \in h^{\ell + 1} S^{-\infty}.
\end{align*}
Now, \(\tilde{a}_1(t, x, \xi ; h)\) defined by the transport equation
\begin{equation}\begin{aligned}
\partial_t \tilde{a}_1 = \{ q_0, \tilde{a}_1 \} - r_0, \quad\quad \tilde{a}_1(0, x, \xi) = 0
\end{aligned}  \nonumber  \end{equation}
is a symbol of order \((\ell + 1, -\infty)\) due to Duhamel's formula and the fact that \(r_0\) and \(\Phi_{q_0}^t\) are smooth and compactly supported. This leads to a recursive procedure to determine \(\tilde{a}_j\) for \(j \geq 1\) via transport equations: suppose that for all \(0 \leq j \leq J-1\) we have \(\tilde{a}_j \in h^{\ell + j} S^{-\infty}\) such that \(\tilde{A}_{J-1}(t) := \sum_{j=0}^{J-1} \Op_h(\tilde{a}_j)\) has \(\sum_{j=0}^{J-1} \tilde{a}_j(0,x,\xi) = a(x,\xi)\) and
\begin{gather*}
\operatorname{Sym}\left[ \frac{i}{h} [Q, \tilde{A}_{J-1}(t)] - \partial_t \tilde{A}_{J-1}(t) \right] \sim r_{J-1} , \\
r_{J-1} := \{ \tilde{q}_0, \tilde{a}_{J-1} \} + \sum_{|\beta| \geq 2} h^{|\beta| - 1} \frac{i^{-|\beta| + 1}}{\beta!} (\partial_{\xi}^{\beta} \tilde{a}_{J-1} \partial_x^{\beta} q - \partial_{\xi}^{\beta} q \partial_x^{\beta} \tilde{a}_{J-1}) \in h^{\ell + J} S^{-\infty}.
\end{gather*}
Then, let \(\tilde{a}_J(t,x,\xi ; h)\) be defined by the transport equation
\begin{equation}\begin{aligned}
\partial_t \tilde{a}_{J} = \{ q_0, \tilde{a}_J \} - r_{J-1}, \quad\quad \tilde{a}_J(0, x, \xi) = 0.
\end{aligned}  \nonumber  \end{equation}
Since \(r_{J-1}\) and \(\Phi_{q_0}^t\) are smooth and compactly supported, by Duhamels formula this is solved with \(\tilde{a}_J\) also of order \((\ell + J, -\infty)\). Moreover, \(A_J(t) := A_{J-1}(t) + \Op_h(\tilde{a}_{J})\) has \(\sum_{j=0}^J \tilde{a}_j(0,x,\xi) = a(x,\xi)\) and
\begin{gather*}
\operatorname{Sym}\left[ \frac{i}{h} [Q, \tilde{A}_{J}(t)] - \partial_t \tilde{A}_{J}(t) \right] \sim r_{J} ,   \\
r_{J} := \{ \tilde{q}_0, \tilde{a}_{J} \} + \sum_{|\beta| \geq 2} h^{|\beta| - 1} \frac{i^{-|\beta| + 1}}{\beta!} (\partial_{\xi}^{\beta} \tilde{a}_{J} \partial_x^{\beta} q - \partial_{\xi}^{\beta} q \partial_x^{\beta} \tilde{a}_{J}) \in h^{\ell + J + 1} S^{-\infty} .
\end{gather*}
Continuing like this, we have for all \(j \geq 1\), \(\tilde{a}_j \in h^{\ell + j} S^{-\infty}\) so by Borel's theorem there is \(\tilde{a}^{(1)} \in h^{\ell + 1} S^{-\infty}\) such that
\begin{equation}\begin{aligned}
\tilde{a}^{(1)} \sim \sum_{j=1}^{\infty} \tilde{a}_j, \quad\quad \tilde{a} := \tilde{a}_0 + \tilde{a}^{(1)} \in h^{\ell} S^m
\end{aligned}  \nonumber  \end{equation}
and therefore, \(\tilde{A}(t) := \Op_h(\tilde{a}(t,x,\xi ; h))\) satisfies
\begin{equation}\begin{aligned}
\partial_t \tilde{A}(t) = \frac{i}{h} [Q, \tilde{A}] + R(t), \quad\quad \tilde{A}(0) = A ,
\end{aligned}  \nonumber  \end{equation}
with \(R(t) \in h^{\infty} \Psi^{-\infty}\).

Now we wish to see that \(\tilde{R}(t) := A(t) - \tilde{A}(t) \in h^{\infty} \Psi^{-\infty}\). Since \(U^{-t} : H_h^s \to H_h^s\) continuously for all \(s \in \mathbb{R}\), it suffices to show that \(U^{-t} A - \tilde{A}(t) U^{-t} = \tilde{R}(t) U^{-t} \in h^{\infty} \Psi^{-\infty}\). Let \(u \in H_h^{-N}\) for \(N \in \mathbb{N}\) and denote \(v(t, x) := \tilde{A}(t) U^{-t}[u]\). Then \(v\) satisfies \(v(0, x) = A[u]\) and
\begin{align*}
\partial_t v &= (\partial_t[\tilde{A}(t)] U^{-t} + \tilde{A}(t) \partial_t[U^{-t}])[u]  \\
    &= \left( \frac{i}{h} [Q, \tilde{A}(t)] + R(t) \right)U^{-t}[u] + \frac{i}{h} \tilde{A}(t) Q U^{-t}[u]  \\
    &= \frac{i}{h} Q \tilde{A}(t)U^{-t}[u]  + R(t) U^{-t}[u]    \\
    &= \frac{i}{h} Q[v] + R(t) U^{-t}[u].
\end{align*}
Thus, \(\tilde{v}(t, x) := v(t,x) - U^{-t} A[u](x)\) satisfies
\begin{equation}\begin{aligned}
\partial_t \tilde{v} = \frac{i}{h} Q[\tilde{v}] + R(t)U^{-t}[u], \quad\quad \tilde{v}(0, x) = 0.
\end{aligned}  \nonumber  \end{equation}
Since \(w(t) := R(t) U^{-t}[u] \in C^{\infty}\) with \(||w(t)||_{H_h^N} = O(h^{\infty})\) and \(U^s : C^{\infty} \to C^{\infty}\) for all \(|s| \leq T\), by Duhamel's formula it follows that \(\tilde{v} = \tilde{R}(t) U^{-t}[u] \in C^{\infty}\) with \(||\tilde{v}||_{H_h^N} = O(h^{\infty})\). Thus, \(||\tilde{R}(t) U^{-t}||_{H_h^{-N} \to H_h^N} = O(h^{\infty})\), whence \(\tilde{R}(t) U^{-t} \in h^{\infty} \Psi^{-\infty}\).
\end{proof}

The equivalence between the operator dynamics given by conjugation with \(e^{-\frac{i}{h} t Q}\) and the dynamics on symbols given by composition with \(\Phi_{q_0}^t\) as guaranteed by Egorov's theorem is a precise form of the physical notion of \emph{quantum-classical correspondence}. This establishes an assurance of the correspondence for a Liouvillian flow on symbols when we use a smoothing generator \(Q\) the operator side, whose real principal symbol serves as a local Hamiltonian for the symbol dynamics. On one hand, the second part of \protect\hyperlink{thm:sym-renorm-graph-lap}{Theorem \ref{thm:sym-renorm-graph-lap}} shows that smoothly cutting-off the principal symbol of \(h^2 \GLap_{\lambda,\epsilon}\) in a sufficiently small region of phase space gives a symbol whose Hamiltonian flow is geodesic for initial points in that region. On the other hand, the first part of that Theorem gives that this operator has order \((0,2)\). Thus, to apply Egorov's theorem we must maintain the former property and simultaneously localize the symbol of our generator to an appropriately small part of the phase space.

We proceed as follows: we first see that when applied to a coherent state, the operator dynamics given by \(U_{\lambda,\epsilon}^t\) is, up to \(O(h)\) error, equal to the dynamics given by a semigroup that is generated by the square root of a phase space localized form of the graph Laplacian. Then, we apply the second part of \protect\hyperlink{thm:sym-renorm-graph-lap}{Theorem \ref{thm:sym-renorm-graph-lap}} and semiclassical functional calculus to see that this generator is pseudodifferential and has a principal symbol, whose Hamiltonian flow is geodesic about the initial point \((x_0, \xi_0)\) where the coherent state is localized. This puts us in a setting to apply Egorov's theorem, which gives the equivalence to the quantization of a symbol propagated along this flow. Finally, taking an inner product with the same coherent state recovers, up to \(O(h)\) error, the symbol propagated along the geodesic, thus giving the desired result, namely:

\begin{theorem}	\hypertarget{thm:sym-cs-glap-psido}{\label{thm:sym-cs-glap-psido}} Let \(\lambda \geq 0\), \(\alpha \geq 1\), \((x_0, \xi_0) \in T^*\mathcal{M} \setminus 0\) and \(|t| \leq \operatorname{inj}(x_0)\). Then, there exists \(h_0 > 0\) such that given \(a \in h^0 S^0\), for all \(h \in (0, h_0]\) and with \(\epsilon := h^{2 + \alpha}\),
\begin{equation} \label{eq:glap-prop-coherent-symbol}
\langle \psi_h(\cdot ; x_0, \xi_0) | U_{\lambda,\epsilon}^{-t} \Op_h(a) U_{\lambda,\epsilon}^t | \psi_h(\cdot ; x_0, \xi_0) \rangle = a \circ \Gamma^t(x_0, \xi_0) + O(h).
\end{equation}
In fact, there is a cut-off \(\chi \in C_c^{\infty}(\mathbb{R}, [0, 1])\) with \(\chi \equiv 1\) on \([-r, r]\) for some \(r > 0\) such that for all \(h \in (0, h_0]\) and with \(\epsilon := h^{2 + \alpha}\),
\begin{equation} \label{eq:glap-prop-microlocal-cut}
U_{\lambda,\epsilon}^{-t} \Op_h(a) U_{\lambda,\epsilon}^t[\psi_h] = U_{\chi}^{-t} \Op_h(a) U_{\chi}^t[\psi_h] + O_{L^2}(h) .
\end{equation}
Here, \(U_{\chi} := e^{\frac{i}{h} t Q}\) with \(Q = \Op_h(q)\) and \(q \in C_c^{\infty}(T^*\mathcal{M} \times [0, h_0)) \cap h^0 S^{-\infty}\) such that \(\operatorname{Sym}[Q] (x,\xi)= |\xi|_{g_x} \, \chi(|\xi|_{g_x}^2 - r_0)\) for \(r_0 := |\xi_0|_{g_{x_0}}^2\).

The above statements hold with \(U_{\lambda,\epsilon}^{\pm (\varepsilon) t} := e^{\mp i t (\GLap_{\lambda,\epsilon} + (2 c_0/c_2) (\varepsilon/\epsilon) I)^{\frac{1}{2}}}\) in place of \(U_{\lambda,\epsilon}^{\pm t}\), whenever \(\varepsilon \in O(h^{1 + \alpha})\).

\end{theorem}

\begin{proof}

The idea is to show first that \(\eqref{eq:glap-prop-microlocal-cut}\) holds so that the form of Egorov's theorem in \protect\hyperlink{thm:egorov}{Theorem \ref{thm:egorov}} can be applied to give \(\eqref{eq:glap-prop-coherent-symbol}\). In order to see \(\eqref{eq:glap-prop-microlocal-cut}\), we use the fact that there exists \(r > 0\) so that the cut-off \(\chi(|\xi|_{g_x}^2 - r_0)\) contains the \emph{essential support} of \(T_h[\psi_h]\), while \(\chi \equiv 0\) on the spectrum of \(\GLap_{\lambda,\epsilon}\) outside of \([r_0 - r, r_0 + r]\), as depicted in Figure \(\ref{fig:glaps-cutoff}\). This ensures that \(\chi(\GLap_{\lambda,\epsilon})\) is a spectral projector and for any operator \(A : L^2 \to L^2\), \(A[\psi_h] = A T_h^* \chi T_h[\psi_h] + O_{L^2}(h^{\infty})\). The Helffer-Sjöstrand formula enables combining these two facts by ensuring that the spectral projector is in \(h^0 \Psi^{-\infty}\), so that \(A[\psi_h] = A \chi(\GLap_{\lambda,\epsilon})[\psi_h] + O_{L^2}(h)\). Then, when we set \(A = U_{\lambda,\epsilon}^{-t} \Op_h(a) U_{\lambda,\epsilon}^t\), we can \emph{commute the spectral projector} into the exponent of the propagator \(U_{\lambda,\epsilon}^t\) and by pseudodifferential calculus achieve \(\eqref{eq:glap-prop-microlocal-cut}\).

We start with the phase-space localization property of coherent states: using the special form of the FBI transform with the adapted phase coinciding with \(\psi_h\) and appropriate amplitude, we have by \protect\hyperlink{thm:FBI-basic}{Theorem \ref{thm:FBI-basic}} and \protect\hyperlink{thm:wunsch-zworski}{Lemma \ref{thm:wunsch-zworski}} that there is \(C > 0\) depending only on \((x_0, \xi_0)\) and \(\mathcal{M}\) such that given any \(\gamma > 0\), if \(\chi \in C_c^{\infty}(\mathbb{R}, [0,1])\) with \(\chi \equiv 1\) on \(\mathcal{I}_{\chi} \supset \mathcal{I}_{\psi_h} := |\xi_0|_{g_{x_0}}^2 + h^{1 - \gamma}[-C^2, C^2]\), then
\begin{align*}
\psi_h &= T_h^* T_h[\psi_h] + O_{L^2}(h^{\infty})   \\
    &= T_h^* \chi(|\xi|_{g_x}^2) T_h[\psi_h] + O_{L^2}(h^{\infty}) .
\end{align*}
On the other hand, by straight-forward modifications to the proof of the first part of \protect\hyperlink{thm:sym-renorm-graph-lap}{Theorem \ref{thm:sym-renorm-graph-lap}} it follows that \(h^{\alpha} \mathcal{L}_{\lambda,\epsilon} \in h^0 S^0\), so we may apply the Helffer-Sjöstrand formula as in \citep[\(\S 8\)]{dimassi1999spectral} to find that upon contracting the support \(\chi_{h^{-\alpha}}(\cdot) := \chi(h^{-\alpha} \cdot)\), \(\Pi_{h,\alpha,\chi} := \chi_{h^{-\alpha}}(\epsilon \GLap_{\lambda,\epsilon}) = \Op_h(q_{\chi,\lambda,\epsilon})\) with \(q_{\chi,\lambda,\epsilon} \in h^0 S^{-\infty}\) and \(q_{\chi,\lambda,\epsilon} \equiv \chi \circ \mathcal{L}_{g,\lambda,h^{\alpha}} \pmod{h S^{-\infty}}\). An application of Taylor's theorem shows that \(\chi \circ \mathcal{L}_{g,\lambda,h^{\alpha}} \equiv \chi(|\xi|_{g_x}^2) \pmod{h^{\alpha} S^{-\infty}}\), so by \protect\hyperlink{thm:FBI-basic}{Theorem \ref{thm:FBI-basic}}, \(\Pi_{h,\alpha,\chi} \equiv T_h^* (\chi \circ |\xi|_{g_x}^2) T_h \pmod{h \Psi^{-\infty}}\). Therefore,
\begin{equation} \label{eq:cs-phase-space-localization}
\psi_h = \Pi_{h,\alpha,\chi}[\psi_h] + O_{L^2}(h).
\end{equation}
Now, as recorded in \protect\hyperlink{lem:lap-symm}{Lemma \ref{lem:lap-symm}}, \(\GAve_{\lambda,\epsilon}\) has a discrete spectrum contained in \((-1, 1]\), so the spectrum of \(h^2 \GLap_{\lambda,\epsilon}\) is also discrete and contained in \([0, 2\tilde{c}_{2,0} \, h^{-\alpha})\), werein we set \(\tilde{c}_{2,0} := c_2/(2 c_0)\), with all eigenvalues isolated, except for \(\tilde{c}_{2,0} h^{-\alpha}\) where there is an accumulation point. Hence, with \(0 < |\xi_0|_{g_{x_0}}^2 \leq \tilde{c}_{2,0} h^{-\alpha}\) and supposing \(\mathcal{I}_{\chi} \subset [0, \tilde{c}_{2,0} h^{-\alpha} + C^2 h^{1 - \gamma}]\), we have \(\underline{\sigma}_{\chi}, \overline{\sigma}_{\chi} \in \operatorname{Spec}(h^2 \Delta_{\lambda,\epsilon}) \cup \{ \tilde{c}_{2,0} h^{-\alpha} + C^2 h^{1 - \gamma} \}\) such that \(\underline{\sigma}_{\chi} = \sup\{ \sigma \in \operatorname{Spec}(h^2 \Delta_{\lambda,\epsilon}) ~|~ \sigma \leq \min \mathcal{I}_{\chi} \}\) and \(\overline{\sigma}_{\chi} = \inf\{ \sigma \in \operatorname{Spec}(h^2 \Delta_{\lambda,\epsilon}) \cup \{ \tilde{c}_{2,0} h^{-\alpha} + C^2 h^{1 - \gamma} \} ~|~ \sigma \geq \max \mathcal{I}_{\chi} \}\). After possibly perturbing \(\chi\) so that \(\underline{\sigma}_{\chi}, \overline{\sigma}_{\chi} \neq \tilde{c}_{2,0} h^{-\alpha}\), we can modify \(\chi\) while keeping \(\chi^{-1}\{ 1 \} = \mathcal{I}_{\chi}\) fixed such that \(\mathcal{I}_{\chi} \subsetneq \supp \chi \subset (\underline{\sigma}_{\chi}, \overline{\sigma}_{\chi})\). Then, denoting by \(\Pi : L^2 \to L^2\) the spectral projector onto the eigenspaces of \(h^2 \GLap_{\lambda,\epsilon}\) given by the eigenvalues in \([\min \mathcal{I}_{\chi}, \max \mathcal{I}_{\chi}]\), we have
\begin{equation}\begin{aligned}
\Pi_{h,\alpha,\chi} = \Pi .
\end{aligned}  \nonumber  \end{equation}
Hence with the short-hand \(A := \Op_h(a)\) we have by the calculus of \(\PDO\)s, including that \(A \in h^0 \Psi^0\) implies \(A\) is bounded on \(L^2(\mathcal{M})\) and by the foregoing considerations that
\begin{align*}
U_{\lambda,\epsilon}^{-t} A U_{\lambda,\epsilon}^t[\psi_h] &= U_{\lambda,\epsilon}^{-t} A \, \Pi^2 \, U_{\lambda,\epsilon}^{t}[\psi_h] + O_{L^2}(h) \\
    &= U_{\lambda,\epsilon}^{-t} (\Pi \, A - [\Pi, A])\Pi \, U_{\lambda,\epsilon}^{t}[\psi_h] + O_{L^2}(h)  \\
    &= U_{\lambda,\epsilon}^{-t} \Pi \, A U_{\lambda,\epsilon}^{t} \Pi [\psi_h] + O_{L^2}(h)    \\
    &= e^{i t(\GLap_{\lambda,\epsilon} \Pi)^{\frac{1}{2}}} \Pi A \Pi \, e^{-i t (\GLap_{\lambda,\epsilon} \Pi)^{\frac{1}{2}}}[\psi_h] + O_{L^2}(h)  \\
    &= e^{i t(\GLap_{\lambda,\epsilon} \Pi)^{\frac{1}{2}}} (A\Pi - [A, \Pi])\Pi \, e^{-i t (\GLap_{\lambda,\epsilon} \Pi)^{\frac{1}{2}}}[\psi_h] + O_{L^2}(h)   \\
    &= e^{i t(\GLap_{\lambda,\epsilon} \Pi)^{\frac{1}{2}}} A \, e^{-i t (\GLap_{\lambda,\epsilon} \Pi)^{\frac{1}{2}}}[\psi_h] + O_{L^2}(h).
\end{align*}
Indeed, by \protect\hyperlink{lem:lap-symm}{Lemma \ref{lem:lap-symm}}, \(||U_{\lambda,\epsilon}^{\pm t}||_{L^2 \to L^2} \leq \overline{C}_{p,\lambda}/\underline{C}_{p,\lambda}\) and \([\Pi, A]\Pi \in h \Psi^{-\infty}\), so \(||U_{\lambda,\epsilon}^{-t} [\Pi, A] \Pi U_{\lambda,\epsilon}^{t}[\psi_h]||_{L^2} = O(h)\) and the remaining equalities use the commutativity among spectral functions of \(\GLap_{\lambda,\epsilon}\) and idempotency of \(\Pi\) as a spectral projector. The ultimate equality is due firstly to \(U^t := e^{i t (\GLap_{\lambda,\epsilon} \Pi)^{\frac{1}{2}}}\) having the bound, again by \protect\hyperlink{lem:lap-symm}{Lemma \ref{lem:lap-symm}}, \(||U^{\pm t}||_{L^2 \to L^2} \leq \overline{C}_{p,\lambda}/\underline{C}_{p,\lambda}\) and \([A,\Pi]\Pi \in h \Psi^{-\infty}\), which gives that the corresponding term is \(O_{L^2}(h)\) and then due to \(\eqref{eq:cs-phase-space-localization}\) so that \(||U^{-t} A U^t(\Pi - I)[\psi_h]||_{L^2} = O(h)\). Since \(|\xi_0|_{g_{x_0}} > 0\), we can choose \(\mathcal{I}_{\chi}\) so that \(\min \mathcal{I}_{\chi} > 0\) and then there is \(h_0 > 0\) such that for all \(h \in [0, h_0]\), \(\mathcal{I}_{\psi_h} \subset \mathcal{I}_{\chi}\). On this interval the square root function is smooth and since \(\Pi \in h^0 \Psi^{-\infty}\), we have that \(\GLap_{\lambda,\epsilon} \Pi = \Op_h(q^2) \in h^0 \Psi^{-\infty}\), wherein due to the second part of \protect\hyperlink{thm:sym-renorm-graph-lap}{Theorem \ref{thm:sym-renorm-graph-lap}}, \(q^2 \equiv |\xi|_{g_x}^2 \chi(|\xi|_{g_x}^2) \pmod{h S^{-\infty}}\). Thus, another application of the Helffer-Sjöstrand formula gives,
\begin{equation}\begin{aligned}
(h^2 \GLap_{\lambda,\epsilon} \Pi)^{\frac{1}{2}} = \sqrt{h^2 \GLap_{\lambda,\epsilon}} \, \Pi \equiv \Op_h(|\xi|_{g_x} \, \chi\circ |\xi|_{g_x}^2)  \pmod{h \Psi^{-\infty}} .
\end{aligned}  \nonumber  \end{equation}
Therefore, combining the above considerations with Egorov's theorem to time \(|t| < \operatorname{inj}(x_0)\) and \protect\hyperlink{thm:sym-cs-psido}{Theorem \ref{thm:sym-cs-psido}}, we have upon recalling \(\mathcal{I}_{\psi_h} \subset \mathcal{I}_{\chi}\) that
\begin{equation}\begin{aligned}
\langle \psi_h | U_{\lambda,\epsilon}^{-t} A U_{\lambda,\epsilon}^t | \psi_h \rangle = a \circ \Phi_{q_0}^t(x_0, \xi_0) + O(h), \\
q_0(x,\xi) = |\xi|_{g_x} \chi(|\xi|_{g_x}^2), \quad (x_t, \xi_t) := \Phi^t_{q_0}(x,\xi).
\end{aligned}  \nonumber  \end{equation}
Since for all \((x,\xi) \in T^*\mathcal{M}\) and all \(t \in \mathbb{R}\), \(q_0 \circ \Phi^t_{q_0}(x,\xi) = q_0(x,\xi)\), we find in particular that \(|\xi_t|_{g_{x_t}}^2 = |\xi|_{g_x}^2\) for all \((x,\xi) \in \mathcal{N}_{\xi_0} := \{|\xi|_{g_x}^2 \in \mathcal{I}_{\chi} \}\). Furthermore, \(q_0(x,\xi) = |\xi|_{g_x}\) in the neighbourhood \(\operatorname{int}(\mathcal{N}_{\xi_0})\) of \((x_0, \xi_0)\), whence we find upon integrating the Liouvillian flow that \(a \circ \Phi_{q_0}^t(x_0, \xi_0) = a \circ \Gamma^t(x_0, \xi_0)\).

The preceding argument applies to \(U_{\lambda,\epsilon}^{(\varepsilon) t}\) as well, since with \(\varepsilon \in O(h^{1 + \alpha})\), \(h^2(\GLap_{\lambda,\epsilon} + (2 c_0/c_2) (\varepsilon/\epsilon) I) = h^2 \GLap_{\lambda,\epsilon} + (2 c_0/c_2) h I\) has the same principal symbol as \(h^2 \GLap_{\lambda,\epsilon}\). This gives the second part of the statement of the Theorem.
\end{proof}

\begin{remark} The proof of the preceding Theorem shows that we may approximate the flow of classical observables along geodesics upon projecting onto any part of the spectrum of \(h^2 \GLap_{\lambda,\epsilon}\) that contains \(\{ \sigma \in \operatorname{Spec}(h^2 \GLap_{\lambda,\epsilon}) ~|~ |\sigma - |\xi_0|_{g_{x_0}}^2| \leq C^2 h^{1 - \gamma} \}\). As \(\epsilon \to 0\), we have spectral convergence of \(\GLap_{\lambda,\epsilon}\) to \(\Delta_{\mathcal{M}} + O(\partial^1)\) so by Weyl's law, roughly speaking, if we order eigenvalues increasingly and \(\sigma_{\epsilon,k} \in \operatorname{Spec}(\GLap_{\lambda,\epsilon})\) is the \(k\)-th eigenvalue then from the bottom, then with \(\epsilon \to 0\) we have \(\sigma_{\epsilon,k + 1} - \sigma_{\epsilon,k} \sim k\). If \(h = C^{\frac{2}{\gamma - 1}} k^{-1}\) and \(|\xi_0|_{g_{x_0}}^2 = 1\), then this is asking for at least \(\operatorname{Spec}(\GLap_{\lambda,\epsilon}) \cap [C' k^2 - k^{1 + \gamma}, C' k^2 + k^{1 + \gamma}]\), which will span an increasing spectral band as \(h \to 0\).

\end{remark}

\includegraphics[width=0.46\textwidth,height=0.46\textheight]{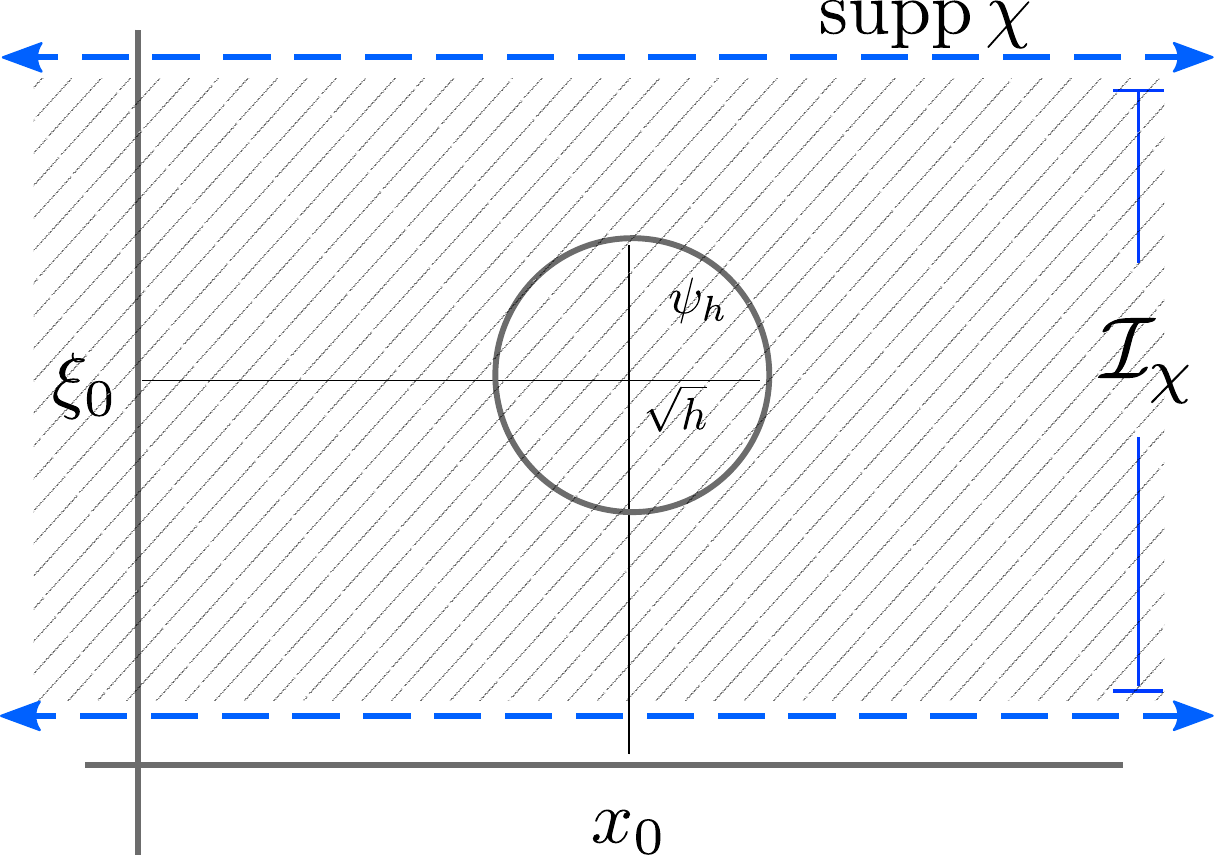}
\hfill
\includegraphics[width=0.46\textwidth,height=0.46\textheight]{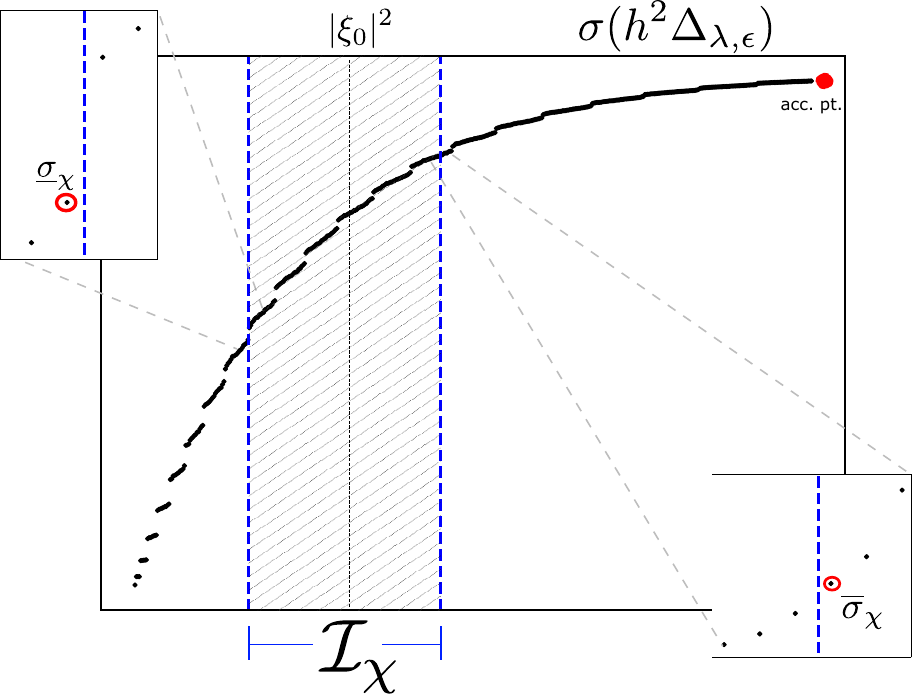}

\begin{figure}[!h]
\caption{The cut-off $\chi$ on phase-space (left) and on the spectrum\protect\footnotemark ~ (right). \label{fig:glaps-cutoff}}
\end{figure}

\footnotetext{The plot is actually of the first $480$ eigenvalues of $\GLap_{1,\epsilon,N}$ for $N = 6000$ uniformly random samples from the unit sphere in $\mathbb{R}^3$ with $\epsilon = e^{-3.4}$.}

\hypertarget{propagation-of-coherent-states}{%
\section{Propagation of coherent states}\label{propagation-of-coherent-states}}

We now discuss briefly the behaviour of a coherent state as it propagates according to \(U^t := e^{-\frac{i}{h} t Q_h}\) for \(Q_h : C^{\infty} \to C^{\infty}\) satisfying the conditions of \protect\hyperlink{thm:egorov}{Theorem \ref{thm:egorov}}, with a generator that locally looks like \(\sqrt{\Delta_{\mathcal{M}}}\), namely:

\begin{enumerate}
\def\labelenumi{\arabic{enumi}.}
\tightlist
\item
  \(Q_h = \Op_h(\tilde{q}) \in h^0 \Psi^{-\infty}\) such that
\item
  \(q = \operatorname{Sym}[Q_h]\) is real-valued and \(\tilde{q} \in C_c^{\infty}(T^*\mathcal{M} \times [0, h_0))\) for some \(h_0 > 0\),
\item
  there are \(r > 0\) and an open neighbourhood \(\mathcal{U} \subset T^*\mathcal{M} \setminus 0\) of the cosphere bundle \(S_r^*\mathcal{M}\) such that \(\mathcal{U} \subset \supp \tilde{q}(\cdot ; h)\) for all \(h \in [0, h_0)\) and
\item
  \(q(\cdot ; h) = |\xi|_{g_x}\) on \(\mathcal{U}\) for all \(h \in [0, h_0)\).
\end{enumerate}

\noindent The \emph{ansatz} is that when \(|\xi_0|_{x_0} = r\), the evolved state \(U^t |\psi_h(\cdot;x_0,\xi_0)\rangle\) to time \(|t| < \operatorname{inj}(x_0)\) is localized about \((x_t, \xi_t) := \Gamma^t(x_0,\xi_0)\) in phase space, and in particular, it is approximately another coherent state (with different amplitude), so a reasonable conjecture would be that it is roughly a Hermite distribution about \(x_t\). This is sensible\footnote{This is worked out explicitly in the case of the Schrödinger equation with \(Q = h^2 \Delta_{\mathcal{M}}\), in \citep{paul_uribe_semiclassical_measures}.}, since the evolution of coherent states under FIOs follows a more general consideration regarding the evolution of \emph{Lagrangian distributions} \citep{boutet_guillemin, paul_uribe_semiclassical_trace, zelditch2017eigenfunctions}. Such descriptions can be likened to the \emph{Schrödinger picture} of quantum dynamics, wherein one is concerned with the explicit description of propagated states \(| \psi_h(t) \rangle\) according to \(i h \partial_t | \psi_h \rangle = Q_h | \psi_h \rangle\). In the present context, we work in the \emph{quantum statistical} or \emph{Heisenberg} framework, using the correspondence between the evolution of symbols and the evolution of their corresponding quantized operators under conjugation by the dynamics \(U^t\), as driven by Egorov's theorem, in order to achieve a weaker description. Ultimately, we arrive at a middle-ground: the quantity \(|U^t[\psi_h]|^2(x)\) is physically interpreted --- as per the \emph{Born rule} --- as the probability that at time \(t\), the particle whose quantum state is described by \(|\psi_h(t)\rangle\) can be found at the position \(x\), while in the following we use quantum-classical correspondence in the Heisenberg picture to turn this quantity into a description about the location of \(x_t\). We will revisit this interpretation in \protect\hyperlink{observing-geodesics}{Section \ref{observing-geodesics}} when we use the \emph{mean} of this distribution as another way (with better error rate) to locate \(x_t\). Hence, presently the main concerns are that \(|U^t[\psi_h]|^2\) can be approximated by a local function and, for consistency with the formulation on graphs that its \(L^{\infty}\) norm can be bounded. These can essentially be resolved using only symbol calculus.

\hypertarget{localization-property}{%
\subsection{Localization property}\label{localization-property}}

We begin with the quantum statistics picture. Start by approximating point-evaluation with integration against a localized function: let \(\varepsilon > 0\) and define \(\rho_{\varepsilon}(x;p) := \exp(-d_g(x,p)^2/\varepsilon)\) for some \(p \in \mathcal{M}\) fixed. Then, given \(\varphi \in C^{\infty}\) and setting \(s_p : x \mapsto \exp_p^{-1}(x)\), we have
\begin{align}
\langle \varphi | \rho_{\varepsilon}(\cdot;p) \varphi\rangle &= \int_\mathcal{M} |\varphi|^2(y) e^{-\frac{d_g(y,p)^2}{\varepsilon}} ~ d\vol{y}  \nonumber \\
    &= \int_{\mathbb{R}^n} (|\varphi|^2\circ s_p^{-1})(u) \; e^{-\frac{||u||_p^2}{\varepsilon}} \sqrt{|g_{s_p^{-1}(u)}|} ~ du + O(\varepsilon^{\infty}) \label{eq:approx-ident-inner-prod} \\
    &= \varepsilon^{\frac{\mdim}{2}} m_0 |\varphi|^2(p) + O(\varepsilon^{\frac{\mdim}{2} + 1}) ,    \nonumber
\end{align}
with \(m_0 := \int_{\mathbb{R}^n} e^{-||u||^2} ~ du\). The second equality is due to \(d[{(s_x^{-1})}^* \nu_g](v) = \sqrt{|g_{s_x^{-1}(v)}|} dv\) (see \citep[Prop C.III.2]{berger1971spectre}) and the ultimate one is due to Taylor expansions and an application of \protect\hyperlink{lemma:expansion-normal-coords}{Lemma \ref{lemma:expansion-normal-coords}}. Now, we quantize \(\rho_{\varepsilon}(\cdot; p)\); this is sensible to do since for any \(u \in C^{\infty}(\mathcal{M})\), \(\partial_{\xi} u = 0\) and \(|u | \in O(1)\) together imply that that \(u \in h^0 S^{0}\). The quantization itself is straightforward:

\begin{align}
\operatorname{Op}_h(\rho_{\varepsilon}(\cdot;p))[\varphi](x) &= \frac{1}{(2\pi h)^n}\int_{T^*\mathcal{M}} e^{\frac{i}{h}s_x(y) \cdot \xi} \rho_{\varepsilon}(y;p) \varphi(y) \chi_x(y) ~ dy\,d\xi   \nonumber\\
    &= \frac{1}{(2\pi h)^n}\int_{\mathbb{R}^n} \int_{\mathbb{R}^n} e^{\frac{i}{h}v\cdot\xi} \rho_{\varepsilon}(s_x^{-1}(v);p)(\varphi \chi_x\circ s_x^{-1})(v) ~ (s_x^{-1})^*[ dy](v) ~ d\xi    \label{eq:quantization-functions} \\
    &= (2\pi)^{-n}\rho_{\varepsilon}(x;p) \varphi(x) ,  \nonumber
\end{align}

wherein \(\chi_x \in C^{\infty}\) is a cut-off supported in a normal neighbourhood of \(x\) and \(\chi_x \equiv 1\) in a smaller neighbourhood. The function \(\rho_{\varepsilon}(\cdot ; p)\) is a symbol of order \((0,0)\) for each \(\varepsilon > 0\) and upon defining \(\tilde{\rho}_{\varepsilon}(\cdot;p) := (2\pi)^{n} \rho_{\varepsilon}(\cdot;p)/||\rho_{\varepsilon}(\cdot;p)||_1\), if for each fixed \(h \geq 0\), we let \(\varepsilon \to 0^+\), then as a sequence of linear maps \(\mathscr{S} \to \mathscr{S}'\) from Schwartz-class functions to tempered distributions, \(\operatorname{Op}_h(\tilde{\rho}_{\varepsilon}(\cdot;p)) \to \hat{\delta}_p : \mathscr{S} \ni \phi \mapsto \delta_p \phi \in \mathscr{S}'\) in the weak-* sense. Therefore, we have that for each \(h \in [0, h_0)\),
\begin{equation}\begin{aligned}
\lim_{\varepsilon \to 0} \langle \varphi | \operatorname{Op}_h(\tilde{\rho}_\varepsilon(\cdot;p))|\varphi\rangle =|\varphi|^2(p) .
\end{aligned}  \nonumber  \end{equation}
This implies, in particular, that
\begin{equation} \label{eq:prop-cs-pos-rep}
\lim_{\varepsilon \to 0} \langle \psi_h(\cdot;x_0,\xi_0) | U^{-t} \operatorname{Op}_h(\tilde{\rho}_{\varepsilon}(\cdot;p))U^t | \psi_h(\cdot;x_0,\xi_0)\rangle = U^t[\psi_h](p) \, \overline{(U^{-t})^*[\psi_h](p)}
\end{equation}
and if \(U\) is unitary, then
\begin{equation}\begin{aligned}
\eqref{eq:prop-cs-pos-rep} = |U^t \psi_h|^2(p).
\end{aligned}  \nonumber  \end{equation}
Now, Egorov's theorem (in the present setting, \protect\hyperlink{thm:egorov}{Theorem \ref{thm:egorov}}) tells that the conjugation by \(U^t\) gives a \(\Psi\)DO with principal symbol that flows along the classical Hamiltonian vector field governed by the principal symbol of the generator of \(U^t\). Let \(\Flow_q^t\) be this flow corresponding to \(q\). Then, upon invoking the FBI transform, we have,
\begin{equation}\begin{aligned}
\langle \psi_h | U^{-t} \operatorname{Op}_h(\rho_{\varepsilon}) U^t|\psi_h\rangle = \langle \psi_h | T_h^* \tilde{\rho}_{\varepsilon}(\Flow_q^t;p)T_h|\psi_h\rangle + O(h,1/\varepsilon)
\end{aligned}  \nonumber  \end{equation}
with \(O(h,1/\varepsilon)\) denoting an error term that is of order \(1\) in \(h\) and order \(-1\) in \(\varepsilon\), due to derivatives of \(\rho_{\varepsilon}\) being taken in the process of converting the left-hand side to the right-hand side. Now we specify to \(|t| < T = \operatorname{inj}(x_0)\) and assume that \(|\xi_0|_{g_{x_0}} = r > 0\). Then, to see that we can still take the limit \(\varepsilon \to 0\), we compute the first term on the right: by \protect\hyperlink{thm:wunsch-zworski}{Lemma \ref{thm:wunsch-zworski}} we have,
\begin{align*}
h^{\frac{\mdim}{2}} \frac{\langle \psi_h | T_h^* \rho_{\varepsilon}(\Flow_q^t;p) T_h|\psi_h \rangle}{C_h(x_0,\xi_0)^2}
    &= \int_{T^*\mathcal{M}} e^{-d_g^2(\Flow_q^t(x,\xi),p)/\varepsilon} e^{-\frac{2}{h}\Im \Phi(s_{x_0}(x),\xi - \xi_0)} |c_h|^2 ~ dx d\xi  \\
    &=:(I),
\end{align*}
wherein the integrand is localized to a ball of radius \(O(h^{\frac{1}{2} - \gamma})\) for any \(\gamma > 0\) about \((x_0, \xi_0) \in T^*\mathcal{M}\) and by assumption, \(q(\cdot ; h) = |\xi|_{g_x}\) for all \(h \in [0, h_0)\), so locally \(\Flow_q^t = \Gamma^t\), hence
\begin{equation}\begin{aligned}
(I) = h^{\frac{\mdim}{2}} \frac{\langle \psi_h | T_h^* \rho_{\varepsilon}(\Gamma_q^t;p) T_h|\psi_h \rangle}{C_h(x_0,\xi_0)^2} + O(h^{\infty}).
\end{aligned}  \nonumber  \end{equation}
Now, working with the first term of the right-hand side we have,
\begin{align*}
h^{\frac{\mdim}{2}} \frac{\langle \psi_h | T_h^* \rho_{\varepsilon}(\Gamma^t;p) T_h|\psi_h \rangle}{C_h(x_0,\xi_0)^2}
    &= \int_{T^*\mathcal{M}} e^{-d_g^2(\Gamma^t(x,\xi),p)/\varepsilon} e^{-\frac{2}{h}\Im \Phi(s_{x_0}(x),\xi - \xi_0)} |c_h|^2 ~ dx d\xi   \\
    &= \int_{T^*\mathcal{M}} e^{-d_g^2(x,p)/\varepsilon} e^{-\frac{2}{h} \Im \Phi[s_{x_0}(\pi_{\mathcal{M}}\Gamma^{-t}(x,\xi)),\pi_{T_x^*\mathcal{M}}\Gamma^{-t}(x,\xi) - \xi_0]} |c_h^{-t}|^2 ~ dx d\xi    \\
    &= \int_{\mathbb{R}^n} \int_{V_p} e^{-||v||^2/\varepsilon} e^{-\frac{2}{h}\Im \Phi[s_{x_0} \circ x^{-t}(s_p^{-1}(v),\xi), \xi^{-t}(s_p^{-1}(v),\xi) - \xi_0]} ~ \times  \\
  & \quad\quad \times |c_h^{-t}|^2 dv d\xi + O(\varepsilon^{\infty}) \\
    &= \int_{\mathbb{R}^n} \left(e^{-\frac{2}{h} \Im \Phi[s_{x_0} \circ x^{-t}(p,\xi), \xi^{-t}(p,\xi) - \xi_0]} \right) \times \\
    & \quad\quad \times \int_{V_p} e^{-||v||^2/\varepsilon} [|c_h^{-t}|^2(p,\xi;x_0,\xi_0) + O(v)] (1 + O(v/h))     \; \times   \\
  & \quad\quad  \quad\quad  \times dv d\xi + O(\varepsilon^{\infty})    \\
    & = \varepsilon^{\frac{n}{2}} \int_{\mathbb{R}^n} e^{-\frac{2}{h}\Im\Phi[s_{x_0} \circ x^{-t}(p,\xi),\xi^{-t}(p,\xi) - \xi_0]} \times   \\
    &\quad\quad \times [|c_h^{-t}|^2(p,\xi;x_0,\xi_0)  + O(\varepsilon/h^2)] d\xi + O(\varepsilon^{\infty}),    \\
\end{align*}
wherein we have written \(x^{-t}(x,\xi) := \pi_{\mathcal{M}} \Gamma^{-t}(x,\xi)\) and \(\xi^{-t}(x,\xi) := \pi_{T_x^*\mathcal{M}} \Gamma^{-t}(x,\xi)\). To pass from the first to second line, we've used the invariance to the geodesic flow of the measure induced by the canonical symplectic form; then, we pass to normal coordinates about \(p\) and localize to its neighbourhood \(V_p \subset \mathbb{R}^n\) of size \(\omega(\sqrt{\varepsilon})\) since the integrand is decaying exponentially outside of that (the factors of the Gaussian in \(\varepsilon\) are bounded and independent of \(\varepsilon\)), which gives the \(O(\varepsilon^{\infty})\) error term. Finally, in the penultimate equality we have Taylor expanded \(|T_h \psi_h|^2 \circ \Gamma^{-t}\) and \(|c_h^{-t}|^2\) at \(x = p\) (in coordinates, \(v = 0\)) and in the last equality we have integrated the \(v\) variables after rescaling them by \(\sqrt{\varepsilon}\) and expanding the Gaussian into a Taylor series. Since now we have only non-negative powers of \(\varepsilon\), this allows to take the limit:
\begin{align}
|\psi_h^t|^2(p) &:= \lim_{\varepsilon \to 0} \langle \psi_h | T_h^* \tilde{\rho}_{\varepsilon}(\Gamma^t;p)T_h|\psi_h \rangle    \nonumber   \\
    &= C_h(x_0,\xi_0)^2 \, h^{-\frac{\mdim}{2}} \int_{\mathbb{R}^n} e^{-\frac{2}{h}\Im\Phi[\Gamma^{-t}(p,\xi) - (x_0,\xi_0)]} |c_h^{-t}|^2(p,\xi;x_0,\xi_0) ~ d\xi .    \label{eq:psi_h-t-full}
\end{align}
On the left-hand side, we have simply introduced new (but suggestive) notation defining a function \(|\psi^t_h|^2\) and on the right-hand side we have written the integrand in local coordinates: this is justified since \(c_h^{-t}\) is a symbol of order zero, which implies that the Gaussian localizes the integral to a \(C_0\sqrt{h}\)-ball \(\tilde{B}_p\) about \(\xi_t(p) := \arg\min_{\xi} ||\Gamma^{-t}(p,\xi) - (x_0,\xi_0)||_H^2\) with \(C_0 := C \sup_{x \in \mathcal{M}} \operatorname{vol} B[(x,0), H]\) since \(\Im\Phi(x,\xi) \leq C' \frac{1}{2} ||(x,\xi)||_H^2 \leq C (|x|^2 + |\xi|^2)\).

Now, we can relate this to the right-hand side of \(\eqref{eq:prop-cs-pos-rep}\) and in case that \(U\) is unitary, to \(|U^t \psi_h|^2\) via the computations above and a sequence of approximations:

\begin{enumerate}
\def\labelenumi{\arabic{enumi}.}
\item
  \(|U^t \psi_h|^2(p) = \langle \psi_h | U^{-t} \operatorname{Op}_h(\tilde{\rho}_{\varepsilon}(\cdot;p)) U^t | \psi_h \rangle + O(\varepsilon)\) due to \(\eqref{eq:approx-ident-inner-prod}\) and \(\eqref{eq:quantization-functions}\);
\item
  upon applying Egorov's theorem to the dominant term on the right-hand side of step (1),
  \begin{equation}\begin{aligned}
  |U^t\psi_h|^2(p) = \langle \psi_h | \operatorname{Op}_h(\tilde{\rho}_{\varepsilon}(\Flow_q^t ; p)) | \psi_h \rangle + O_{2,\varepsilon}(h) + O(\varepsilon),
  \end{aligned}  \nonumber  \end{equation}
  wherein \(O_{2,\varepsilon}(h)\) denotes a dependence of the error term on \(\varepsilon\) and tracks that the error comes from this step (2);
\item
  upon applying \protect\hyperlink{thm:FBI-basic}{Theorem \ref{thm:FBI-basic}} to the main term on the right-hand side of step (2),
  \begin{equation}\begin{aligned}
  |U^t\psi_h|^2(p) = \langle \psi_h | T_h^* \tilde{\rho}_{\varepsilon}(\Flow_q^t;p)T_h | \psi_h \rangle + O_{\varepsilon}(h) + O_{2,\varepsilon}(h) + O(\varepsilon),
  \end{aligned}  \nonumber  \end{equation}
  with the \(O_{\varepsilon}(h)\) error coming from the relation between a \(\Psi\)DO and its conjugation by the FBI transform;
\item
  the computation of the main term on the right-hand side of step (3) was carried out above, so if we pass to the limit \(\varepsilon \to 0\), then we have,
  \begin{align*}
  |U^t \psi_h|^2(p) - |\psi^t_h|^2(p)  &= \lim_{\varepsilon \to 0} [|U^t \psi_h|^2  - \langle \psi_h |T_h^* \tilde{\rho}_{\varepsilon}(\Gamma^t;p)T_h|\psi_h \rangle]  \\
   &= \lim_{\varepsilon \to 0}[O_{\varepsilon}(h) + O_{2,\varepsilon}(h)],
  \end{align*}
  wherein the convergence of the left-hand side gives the existence of the limit of the error terms on the right-hand side and since this only affects the coefficients of the terms dominated by \(h\), the over-all error is still \(O(h)\).
\end{enumerate}

This proves the first part of:

\begin{lemma}	\hypertarget{lem:prop-coherent-state-localized}{\label{lem:prop-coherent-state-localized}} With the notation above, assuming that \(\psi_h\) is localized about \((x_0, \xi_0) \in T^*{\mathcal{M}} \setminus 0\) with \(|\xi_0|_{g_{x_0}}^2 = r > 0\) and setting \(x_t := \pi_{\mathcal{M}}\Gamma^t(x_0,\xi_0)\), we have for all \(h \in (0, h_0]\) and \(|t| \leq \operatorname{inj}(x_0)\),

\begin{enumerate}
\def\labelenumi{\arabic{enumi}.}
\tightlist
\item
  \(|U^t \psi_h|^2(p) = |\psi_h^t|^2(p) + O(h)\),
\item
  there is \(h_1 \in (0, h_0]\) such that for all \(h \in (0, h_1]\), \(|U^t[\psi_h]|^2(x_t) \in \Theta(h^{-\frac{\mdim}{2}})\),
\item
  \(|\psi_h^t|^2(p)\) is localized to (meaning it is \(O(h^{\infty})\) outside of) a ball of radius \(O(h^{\frac{1}{2} - \gamma})\), for any \(\gamma > 0\), about \(x_t\) and
\item
  there is an \(h_{\max} \in (0, h_1]\) such that for all \(h \in [0, h_{\max})\), \(|U^t[\psi_h]|^2\) achieves its maximum value in a ball of radius \(O(\sqrt{h})\) about \(x_t\).
\end{enumerate}

\end{lemma}

\begin{proof}

We clarify parts (2)-(4). The exponential in the integral in \(\eqref{eq:psi_h-t-full}\) defining \(|\psi_h^t|^2\) is localized to an \(O(h^{\frac{1}{2} - \gamma})\) ball about \(x_t\) for any \(\gamma > 0\) and \(c_h\) is bounded, so \(|\psi_h^t|^2\) is itself localized to such a ball and \(|\psi_h^t|^2 \lesssim C_h(x_0,\xi_0)^2 ||c_h^{-t}||^2_{\infty} \lesssim h^{-\frac{\mdim}{2}}\). Moreover, since we have used the symbol \(b^{-\frac{1}{2}}\) from \protect\hyperlink{thm:FBI-basic}{Theorem \ref{thm:FBI-basic}} in the FBI transform defining \(|\psi_h^t|^2\), it follows from \protect\hyperlink{thm:wunsch-zworski}{Lemma \ref{thm:wunsch-zworski}} that there are constants \(h_1 > 0\) and \(C' \geq 1\) such that for all \(h \in (0, h_1]\), \(C' \geq |c_h^{-t}|^2(x,\xi ; x_0,\xi_0) \geq C'^{-1}\) in a fixed (in \(h\)) neighbourhood \(\mathscr{O}_t\) of \((x_t, \xi_t)\). This, combined with the proof of \protect\hyperlink{thm:sym-cs-psido}{Theorem \ref{thm:sym-cs-psido}} gives \(C_h(x_0, \xi_0)^2 \gtrsim h^{-\frac{\mdim}{2}}\), so taking \(h_1 \leq h_0\) along with part (1) gives part (2). Now, to see part (4), we use that by \protect\hyperlink{thm:wunsch-zworski}{Lemma \ref{thm:wunsch-zworski}}, there is a constant \(C \geq 1\) such that \(C(|x - x_0|^2 + |\xi - \xi_0|^2) \geq \Im \Phi \geq C^{-1}(|x - x_0|^2 + |\xi - \xi_0|^2)\), which gives \(0 < h_1 \leq h_0\) such that for all \(h \in [0, h_1]\), \(|\psi_h^t|^2 = O(h^{\infty})\) in \(\mathscr{O}_t^c\) and \(B_0 := B(x_t, g ; (\tilde{C} h)^{\frac{1}{2}}) \subset \mathscr{O}_t\) with \(\tilde{C} := 2 C(\log(C') + 1)\). Then, given \(p_0 \in \overline{B} := \overline{B}(x_t, g; (h/C)^{\frac{1}{2}})\), \(|\psi_h^t|^2(p_0) \geq C^{\frac{\mdim}{2}} C_{\mathscr{O}_t}/(C' e) + O(h^{\infty})\) for \(C_{\mathscr{O}_t} := \int_{\mathscr{O}_t} e^{-|\xi_t - \xi_0|^2} ~ d\xi\). Likewise, given \(p_1 \in \mathcal{M} \setminus B_0\), \(|\psi_h^t|^2(p_1) \leq C^{-\frac{\mdim}{2}} C_{\mathscr{O}_t}/(C' e^2) + O(h^{\infty})\). Since \(\overline{B} \subset B_0\), there is \(0 < h_{\max} \leq h_1\) such that for all \(h \in (0, h_{\max}]\), \(|U^t[\psi_h]|^2\) achieves its maximum in \(B_0\).
\end{proof}

\hypertarget{graph-laplacian-as-hamiltonian}{%
\subsubsection{Graph Laplacian as Hamiltonian}\label{graph-laplacian-as-hamiltonian}}

Connecting back to quantum dynamics generated by the graph Laplacian, \protect\hyperlink{thm:sym-cs-glap-psido}{Theorem \ref{thm:sym-cs-glap-psido}} tells that on application to a coherent state localized at \((x_0, \xi_0)\), there is \(h_0 > 0\) so that \((\Ueps{\tilde{\varepsilon}}{})^{-t} \Op_h(a) \Ueps{\tilde{\varepsilon}}{t}[\psi_h] = e^{-\frac{i}{h} t Q_h} \Op_h(a) e^{\frac{i}{h} t Q_h}[\psi_h] + O_{L^2}(h)\) whenever \(\epsilon = h^{2 + \alpha}\), \(\tilde{\varepsilon} \in O(h^{1 + \alpha})\) and \(\alpha \geq 1\), such that \(Q_h\) satisfies the conditions (1) - (4) from the beginning of \protect\hyperlink{propagation-of-coherent-states}{Section \ref{propagation-of-coherent-states}}. We wish to relate this to the density \(|\Ueps{\tilde{\varepsilon}}{t}[\psi_h]|^2\), for which we note that \protect\hyperlink{lem:lap-symm}{Lemma \ref{lem:lap-symm}} implies,
\begin{equation} \label{eq:glap-prop-adjoint-identity}
\left( \frac{1}{p_{\lambda,\epsilon}}(\Ueps{\tilde{\varepsilon}}{t})^* \circ p_{\lambda,\epsilon} \right) \Ueps{\varepsilon}{t} = I
\end{equation}
with the adjoint being taken on \(L^2(\mathcal{M}, p d\nu_g)\), so together with \(\eqref{eq:prop-cs-pos-rep}\), we have upon fixing \(x \in \mathcal{M}\),
\begin{equation}\begin{aligned}
|\Ueps{\tilde{\varepsilon}}{t}[\psi_h]|^2(x) = \lim_{\varepsilon \to 0} \langle p \,p_{\lambda,\epsilon} (\Ueps{\tilde{\varepsilon}}{})^{-t} \frac{\rho_{\varepsilon}(\cdot ; x)}{p\, p_{\lambda,\epsilon}} \Ueps{\tilde{\varepsilon}}{t}[\psi_h] |\psi_h \rangle .
\end{aligned}  \nonumber  \end{equation}
Denoting \(\tilde{p}_{\lambda,\epsilon} := p_{\lambda,\epsilon} p\), we can further write by \protect\hyperlink{thm:egorov}{Theorem \ref{thm:egorov}}, for all \(\varepsilon > 0\),
\begin{align*}
\langle & \psi_h | \tilde{p}_{\lambda,\epsilon} \, (\Ueps{\tilde{\varepsilon}}{})^{-t} \frac{\rho_{\varepsilon}(\cdot ; x)}{\tilde{p}_{\lambda,\epsilon}} \Ueps{\tilde{\varepsilon}}{t} | \psi_h \rangle    \\
    &\quad\quad = \langle \psi_h | \tilde{p}_{\lambda,\epsilon} \, e^{-\frac{i}{h} t Q_h} \frac{\rho_{\varepsilon}(\cdot ; x)}{\tilde{p}_{\lambda,\epsilon}} e^{\frac{i}{h} t Q_h} | \psi_h \rangle + O(h)  \\
    &\quad\quad = \langle \psi_h | \Op_h[\rho_{\varepsilon}(\Gamma^t ; x) \tilde{p}_{\lambda,\epsilon} /\tilde{p}_{\lambda,\epsilon}(\Gamma^t)] | \psi_h \rangle + O(h) \\
    &\quad\quad = \langle \psi_h | T_h^* \rho_{\varepsilon}(\Gamma^t ; x) \tilde{p}_{\lambda,\epsilon} /\tilde{p}_{\lambda,\epsilon}(\Gamma^t) T_h | \psi_h \rangle + O(h).
\end{align*}
Now we may proceed along the lines of the discussion leading up to \protect\hyperlink{lem:prop-coherent-state-localized}{Lemma \ref{lem:prop-coherent-state-localized}}, which upon absorbing \(\tilde{p}_{\lambda,\epsilon}/\tilde{p}_{\lambda,\epsilon}(\Gamma^t) \sim 1\) into the symbol \(|c_h|^2\), gives the corresponding \(|\psi_h^t|^2\) of \(\eqref{eq:psi_h-t-full}\) so that \(|\Ueps{\tilde{\varepsilon}}{t}[\psi_h]|^2(x) = |\psi_h^t|^2(x) + O(h)\). This leads to the localization properties for the propagation of a coherent state by the dynamics generated by a graph Laplacian as given by \protect\hyperlink{lem:prop-coherent-state-localized}{Lemma \ref{lem:prop-coherent-state-localized}}, with essentially the same reasoning as in its proof, which then gives us,

\begin{proposition}	\hypertarget{prop:glap-prop-cs-localized}{\label{prop:glap-prop-cs-localized}} Let \(\lambda \geq 0\), \(\alpha \geq 1\), \((x_0, \xi_0) \in T^*\mathcal{M} \setminus 0\) and \(|t| < \operatorname{inj}(x_0)\). Then, for \(\psi_h\) a coherent state localized at \((x_0, \xi_0)\) and \(\tilde{\varepsilon} \in O(h^{1 + \alpha})\), there are constants \(h_0 \geq h_1 \geq h_{\max} > 0\) such that the properties (1) - (4) of \protect\hyperlink{lem:prop-coherent-state-localized}{Lemma \ref{lem:prop-coherent-state-localized}} hold for the propagated coherent state \emph{density} \(|\Ueps{\tilde{\varepsilon}}{t}[\psi_h]|^2\).

\qed

\end{proposition}

\hypertarget{observing-geodesics}{%
\subsection{Observing geodesics}\label{observing-geodesics}}

We've seen that the location of the maximum of a propagated coherent state \emph{observes} the geodesic flow in the sense that it gives an \(O(\sqrt{h})\) ball in configuration space where it can be found. Now, we borrow from quantum mechanics the \emph{position operator}, which in the basis of Dirac mass distributions represents position statistics of a particle. In particular, taking the expectation of the position operator with \(|\psi_h(t)\rangle\) gives the \emph{mean position} of a particle with quantum state \(|\psi_h(t)\rangle\). When \(U\) is unitary and a position operator is given by a vector of multiplication operators by coordinate functions, then recalling the interpretation of \(|U^t[\psi_h]|^2\) via the Born rule, we understand the mean position to be the mean of this probability density. On a submanifold of Euclidean space, the extrinsic coordinates can be used, but from a local perspective, in light of \protect\hyperlink{lem:prop-coherent-state-localized}{Lemma \ref{lem:prop-coherent-state-localized}}, the concentration of \(|U^t[\psi_h]|^2\) in an \(O(\sqrt{h})\) ball about \(x_t\) can be used to identify a position operator given by local coordinates of intrinsic dimension. In either case, through Egorov's theorem and quantization of \(\PDO\)s by FBI transforms, the mean turns out to be located in an \(O(h)\) ball about \(x_t\). Thus, we have a quadratic improvement in the error term when going from the maximum to the mean. We now elaborate on this point of view.

In principle, working in coordinates given by fixing a smooth, injective map \(u = (u_1, \ldots, u_N) : \mathcal{M} \to \mathbb{R}^N\) with \(u_j \in C^{\infty}(\mathcal{M})\), for any \(N \in \mathbb{N}\), we'd like to compute
\begin{equation} \label{eq:prop-cs-mean-coord}
\bar{u}_j^t(x_0,\xi_0) :=\int_{\mathcal{M}} u_j(x) |U^t \psi_h|^2(x) ~ d\vol{x}
\end{equation}
for each \(j = 1, \ldots, N\). Quantizing \(u_j\) gives a multiplication operator in \(h^0 \Psi^0\) and therefore by Theorems  \protect\hyperlink{thm:egorov}{\ref{thm:egorov}} and  \protect\hyperlink{thm:sym-cs-psido}{\ref{thm:sym-cs-psido}}, we have,
\begin{equation} \label{eq:prop-cs-coord-qmean}
\langle \psi_h | U^{-t} \operatorname{Op}_h(u_j) U^t | \psi_h \rangle = u_j\circ\Gamma^t(x_0,\xi_0) + O(h) = u_j(x_t) + O(h)
\end{equation}
and when \(U\) is unitary, this gives,
\begin{equation} \label{eq:prop-cs-sym-coord}
\bar{u}_j^t(x_0,\xi_0) = u_j(x_t) + O(h).
\end{equation}
Hence, for the \emph{mean coordinates}, \(\bar{u}^t := (\bar{u}_1^t, \ldots, \bar{u}_N^t)\) we have in the unitary case, \(||\bar{u}^t(x_0, \xi_0) - u(x_t)||_{\mathbb{R}^N} \lesssim N^{\frac{1}{2}} h\). One goal is then to turn this into a bound on \(d_g(u^{-1}[\bar{u}^t(x_0, \xi_0)], x_t)\) so that we can be certain that taking the mean of \(|U^t[\psi_h]|^2\) brings us close to the geodesic neighbour of \(x_0\) at distance \(t\) in the direction \(\xi_0\), \emph{viz}., that we are \emph{observing} the geodesic flow.

Another aspect is to see that this can work \emph{locally}: we've seen already that for sufficiently small \(h\), the geodesic neighbour \(x_t\) is within an \(O(\sqrt{h})\) ball about the maximum of \(|U^t[\psi_h]|^2\). Therefore, by \(\eqref{eq:prop-cs-sym-coord}\), we can localize the mean of each coordinate to this neighbourhood. This is of practical importance because the computation, along with coefficients of error, can be reduced after setting up a local, dimensionally reduced coordinate system, as opposed to working in the extrinsic coordinates (which in some applications can be exceedingly large and might even grow as a function of \(h^{-1}\) !). Further, when needed to be done in extrinsic coordinates, it would be useful to compute the means within extrinsic balls about the maximum and then take a nearest embedded point.

This is all feasible and the proof is now rather simple. It is facilitated by the following,

\begin{definition}	\hypertarget{def:mean-max-coords-density}{\label{def:mean-max-coords-density}} Let \(\psi_h\) be a coherent state localized at \((x_0, \xi_0) \in T^*\mathcal{M}\) and \(|t| \leq \operatorname{inj}(x_0)\). The \emph{intrinsic maximizer} of the propagated coherent state is,
\begin{equation}\begin{aligned}
\hat{x}_t := \arg\max_{\mathcal{M}} |U^t[\psi_h]|^2 .
\end{aligned}  \nonumber  \end{equation}

\begin{enumerate}
\def\labelenumi{\arabic{enumi}.}
\item
  \emph{Extrinsic case}. The \emph{extrinsic maximizer} of the propagated coherent state is,
  \begin{equation}\begin{aligned}
  \hat{x}_{\iota,t} := \arg\max_{\embedM} |U^t[\psi_h]|^2 \circ \iota^{-1}
  \end{aligned}  \nonumber  \end{equation}
  and we call the \emph{extrinsic mean} of the propagated coherent state to be \(\bar{x}_{\chi_{\iota}, t} \in \embedM\), defined with respect to a cut-off \(\chi_{\iota} \in C_c^{\infty}(\mathbb{R}^{\rdim})\), as the closest point on \(\embedM\) to
  \begin{align*}
  &&\bar{\iota}^t(x_0, \xi_0) &:= (\bar{\iota}_1^t(x_0, \xi_0), \ldots, \bar{\iota}_{\rdim}^t(x_0, \xi_0)),  \\
  \quad\quad \text{with} && \bar{\iota}_j^t(x_0, \xi_0) &:= \int_{\mathcal{M}} |U^t[\psi_h]|^2(x) \, (\chi_{\iota} \circ \iota)(x) \, \iota_j(x) ~ d\vol{x} .
  \end{align*}
\item
  \emph{Local coordinates}. Let \(\mathscr{O}_t \subset \mathcal{M}\) be an open neighbourhood of \(\hat{x}_t\). Then, given a diffeomorphic coordinate mapping \(u : \mathscr{O}_t \to V_t \subset \mathbb{R}^{\mdim}\), and \(\chi \in C_c^{\infty}(\mathbb{R}^{\mdim})\) a cut-off with \(\supp \chi \subset V_t\), the \(u\)-\emph{maximizer}, or \emph{coordinate maximizer} of the propagated coherent state is,
  \begin{equation}\begin{aligned}
  \hat{x}_{u,t} := \arg\max_{\mathscr{O}_t} |U^t[\psi_h]|^2 \circ u^{-1}
  \end{aligned}  \nonumber  \end{equation}
  and we call
  \begin{align*}
  && \bar{x}_{u,\chi,t} &:= (\bar{u}_1^t(x_0, \xi_0), \ldots, \bar{u}_{\mdim}^t(x_0, \xi_0)),    \\
  \quad\quad\text{with} && \bar{u}_j^t(x_0, \xi_0) &:= \int_{\mathcal{M}} |U^t[\psi_h]|^2(x) \, (\chi \circ u)(x) \, u_j(x) ~ d\vol{x}
  \end{align*}
  the \(u\)-\emph{mean with respect to} \(\chi\).
\end{enumerate}

\end{definition}

It's clear that \(\iota(\hat{x}_t) = \hat{x}_{\iota,t}\) and \(u(\hat{x}_t) = \hat{x}_{u,t}\), so the notation is there just for brevity. With this at hand, assuming \(U\) is unitary, we may state,

\begin{proposition}	\hypertarget{prop:local-mean-geodesic-flow}{\label{prop:local-mean-geodesic-flow}} Let \(\psi_h\) be a coherent state localized at \((x_0, \xi_0) \in T^*\mathcal{M}\). Then, for all \(|t| < \operatorname{inj}(x_0)\), given any open neighbourhood \(\mathscr{O}_t \subset \mathcal{M}\) about \(\hat{x}_t\) with its diffeomorphic coordinate mapping \(u : \mathscr{O}_t \to V_t \subset \mathbb{R}^{\mdim}\), there are constants \(h_{u,\max}, C_{u,\max} > 0\) such that if \(h \in [0, h_{u,\max})\) and \(\overline{\mathscr{B}}_t := \overline{B}_{C_{u,\max} \sqrt{h}}(\hat{x}_{u,t}, ||\cdot||_{\mathbb{R}^{\mdim}}) \subset V_t\), then for any smooth cut-off \(\chi\) with \(\supp \chi \subset V_t\) that is \(\chi \equiv 1\) on \(\overline{\mathscr{B}}_t\), we have
\begin{equation}\begin{aligned}
d_g(u^{-1}(\bar{x}_{u,\chi,t}), x_t) \leq C_u h,
\end{aligned}  \nonumber  \end{equation}
for \(C_u > 0\) a constant.

Likewise, there are constants \(h_{\iota,\max}, C_{\iota,\max} > 0\) such that for all \(h \in [0, h_{\iota,\max})\), given any cut-off \(\chi_{\iota} \in C_c^{\infty}(\mathbb{R}^{\rdim})\) such that \(\chi_{\iota} \equiv 1\) on \(\overline{\mathscr{B}}_{\iota,t} := \overline{B}_{C_{\iota,\max}\sqrt{h}}(\hat{x}_{\iota,t}, ||\cdot||_{\rdim})\), we have
\begin{equation}\begin{aligned}
d_g(\iota^{-1}(\bar{x}_{\chi_\iota,t}),x_t) \leq C_{\iota} h,
\end{aligned}  \nonumber  \end{equation}
for \(C_{\iota} > 0\) a constant.

\end{proposition}

\begin{remark} In both cases, if the cut-off is supported on a sufficiently large region that is of size \(O(1)\) with respect to \(h\), then the \(O(h)\) proximity of the coordinate means to the geodesics holds for \(h \in [0, h_0)\) with \(h_0\) as in the first part of \protect\hyperlink{lem:prop-coherent-state-localized}{Lemma \ref{lem:prop-coherent-state-localized}}.

\end{remark}

\begin{proof}

The localization of \(|U^t[\psi_h]|^2\) as observed in \protect\hyperlink{lem:prop-coherent-state-localized}{Lemma \ref{lem:prop-coherent-state-localized}} gives constants \(h_{\max}, C_{\max} > 0\) such that \(d_g(u^{-1}(\hat{x}_{u,t}), x_t) \leq C_{\max} \sqrt{h}\) for \(h \in [0, h_{\max})\). We now give the steps to deduce a basic geometric fact for smooth, compact manifolds:
\begin{align}
\begin{split} \label{eq:mani-coord-ball-mapping}
&(\exists) h_{u,0} : (\forall)h \in [0,h_{u,0}],    \\
& [(\exists)C_1 > 0 : x_t \in B_{C_1 \sqrt{h}}(u^{-1}(\hat{x}_{u,t}), g)    \\
&\quad\quad\implies (\exists) C_2 > 0 : u(x_t) \in B_{C_2 \sqrt{h}}(\hat{x}_{u,t}, ||\cdot||_{\mathbb{R}^{\mdim}})].
\end{split}
\end{align}
Since \(u\) is a diffeomorphism, there is a constant \(h_u > 0\) such that for all for all \(h \in [0, h_u)\),
\begin{gather*}
|\det Du(x_t)| \operatorname{vol}(\tilde{\mathscr{B}}_{t}) /2 \leq \operatorname{vol}(\overline{\mathscr{B}}_{g,t}),    \\
\overline{\mathscr{B}}_{g,t} := \overline{B}_{C_{\max} \sqrt{h}}(u^{-1}(\hat{x}_{u,t}), g), \quad \tilde{\mathscr{B}}_{t} := u[\overline{\mathscr{B}}_{g,t}].
\end{gather*}
Moreover, there are constants \(h_{\mathcal{M}} > 0\) and \(S_2 > 0\) such that for all \(h \in [0, h_{\mathcal{M}})\), \(\operatorname{vol}(\overline{\mathscr{B}}_{g,t}) \leq S_2 C_{\max}^{\mdim} h^{\frac{\mdim}{2}}\). Since \(\inf_{\mathcal{M}} |\det Du| > 0\), there is thus a constant \(C_u > 0\) such that for all \(h \in [0, \min\{ h_{\max}, h_u, h_{\mathcal{M}} \}]\), we have \(\operatorname{vol}(\tilde{\mathscr{B}}_t) \leq 2 C_u \operatorname{vol}(\overline{\mathscr{B}}_{g,t}) \leq 2 C_u S_2 C_{\max}^{\mdim} h^{\frac{\mdim}{2}}\). Then, denoting \(C_{u,\max} := (2 C_u S_2)^{\frac{1}{\mdim}} C_{\mdim} C_{\max}\) for some constant \(C_{\mdim} > 0\) and \(\mathscr{B}_{t} := B_{C_{u,\max} \sqrt{h}}(\hat{x}_{u,t}, || \cdot ||_{\mathbb{R}^{\mdim}})\), we have \(\hat{x}_{u,t}, u(x_t) \in \mathscr{\tilde{B}}_t \subset \overline{\mathscr{B}}_{t}\). Now if \(\chi \in C_c^{\infty}(V_t)\) such that \(\chi \equiv 1\) on \(\overline{\mathscr{B}}_t\), then by \(\eqref{eq:prop-cs-sym-coord}\) we have,
\begin{equation}\begin{aligned}
\bar{u}_j^t(x_0, \xi_0) = u_j(x_t) \chi(x_t) + O(h) = u_j(x_t) + O(h).
\end{aligned}  \nonumber  \end{equation}
This means that \(||\bar{u}^t(x_0, \xi_0) - u(x_t)||_{\mathbb{R}^{\mdim}} \leq C' h\) for some \(C' > 0\). Going through essentially the same arguments as we just saw for \(\eqref{eq:mani-coord-ball-mapping}\), we have that there are \(C, h_1 > 0\) such that for all \(h \in [0, h_1]\), \(d_g(u^{-1}(\bar{x}_{u,\chi,t}), x_t) \leq C h\), so with \(h_{u,\max} := \min\{ h_{\max}, h_u, h_{\mathcal{M}}, h_1 \}\), we have the first part of the statement of the Propostion.

As for the second part, since \(\iota^{-1}(\hat{x}_{\iota,t}) = \hat{x}_t\), as in the previous discussion we have that \(d_g(\iota^{-1}(\hat{x}_{\iota,t}), x_t) \leq C_{\max} \sqrt{h}\) for \(h \in [0, h_{\max})\). Then, the \protect\hyperlink{assumptions}{Assumptions} imply (see the Remarks just following) that for \(0 \leq h < \min\{ h_{\max}, (\kappa/C_{\max})^2 \}\), \(||\hat{x}_{\iota,t} - \iota(x_t)||_{\mathbb{R}^{\rdim}} \leq C_{\max} \sqrt{h}\) so that \(\iota(x_t) \in \overline{\mathscr{B}}_{\iota,t}\). If \(\chi_{\iota} \equiv 1\) on \(\overline{\mathscr{B}}_{\iota,t}\), then we again have from \(\eqref{eq:prop-cs-sym-coord}\) that \(\bar{\iota}_j^t(x_0, \xi_0) = \iota_j(x_t) + O(h)\). This gives \(||\bar{\iota}^t(x_0, \xi_0) - \iota(x_t)||_{\mathbb{R}^{\rdim}} \leq C'_{\iota} D^{\frac{1}{2}} h\) for some \(C'_{\iota} > 0\). Now for any \(x_{\iota}^* \in B(\bar{\iota}^t(x_0, \xi_0), || \cdot ||_{\mathbb{R}^{\rdim}} ; C'_{\iota} D^{\frac{1}{2}} h) \cap \embedM\), we have \(||x_{\iota}^* - \iota(x_t)||_{\mathbb{R}^{\rdim}} \leq 2 C'_{\iota} D^{\frac{1}{2}} h\). Thus, by the \protect\hyperlink{assumptions}{Assumptions}, if \(0 \leq h < \min\{ h_{\max}, (\kappa/C_{\max})^2, \kappa/(2 C'_{\iota} D^{\frac{1}{2}}) \}\), then \(d_g(\iota^{-1}(x_{\iota}^*), x_t) \leq 4 C'_{\iota} D^{\frac{1}{2}} h\), which gives the second part of the statement of the Proposition.
\end{proof}

\hypertarget{from-graphs-to-manifolds}{%
\section{From graphs to manifolds}\label{from-graphs-to-manifolds}}

We wish to develop consistency between the discrete semigroup \(U_{\lambda,\epsilon,N}^t\) that solves a finite-dimensional linear differential equation and the propagator \(U_{\lambda,\epsilon}^t\). We proceed by first establishing this for short times, \(t \sim \sqrt{\epsilon}\) and then extending by splitting the time into an integer sum of smaller times. For the former part, we need two constructions: first, the consistency of the analytic functional calculus follows almost immediately from the consistency between the averaging operators \(A_{\lambda,\epsilon,N}\) and \(A_{\lambda,\epsilon}\). This supplies us with consistency for
\begin{equation}    \label{eq:wave-ops}
\begin{aligned}
\mathcal{C}_N := f_e(1 - A_{\lambda,\epsilon,N} ; c_{2,0} t\sqrt{\epsilon}) & & \xrightarrow[N \to \infty]{} & & \mathcal{C} := f_e(1 - A_{\lambda,\epsilon} ; c_{2,0} t\sqrt{\epsilon}),       \\
\mathcal{S}_N := f_o(1 - A_{\lambda,\epsilon,N} ; c_{2,0} t\sqrt{\epsilon}) & & \xrightarrow[N \to \infty]{} & &  \mathcal{S} := f_e(1 - A_{\lambda,\epsilon} ; c_{2,0} t\sqrt{\epsilon})\;
\end{aligned}
\end{equation}
(in the probabilistic sense that we will discuss shortly), wherein we define \(c_{2,0} := \sqrt{2 c_0/c_2}\) and
\begin{align*}
f_e(1 - z; c_{2,0} t\sqrt{\epsilon}) &:= \cos\left( c_{2,0}t \sqrt{\frac{1 - z}{\epsilon}} \right), \\
f_o(1 - z; c_{2,0} t\sqrt{\epsilon}) &:= \sin\left( c_{2,0} t \sqrt{\frac{1 - z}{\epsilon}} \right)\bigg/\sqrt{1 - z} .
\end{align*}
Thereafter, we construct \(\frac{\sqrt{\epsilon}}{c_{2,0}} \sqrt{\GLap}_{\lambda,\epsilon,N} = \sqrt{I - \GAve_{\lambda,\epsilon,N}}\) in order to recover \(U_{\lambda,\epsilon,N}^t = \mathcal{C}_N - i \frac{\sqrt{\epsilon}}{c_{2,0}} \sqrt{\GLap}_{\lambda,\epsilon,N} \mathcal{S}_N\). The square root needs more care due to the branch point of \(\sqrt{1 - z}\) at \(z = 1\) that poses an obstruction to the direct application of the consistency for the functional calculus.

Since we are concerned with the functional calculus of avergaing operators, we note that in the same way that \(A_{\lambda,\epsilon,N}\) extends to a finite-rank operator globably defined on \(L^2(\mathcal{M})\), we can also define functions of \(A_{\lambda,\epsilon,N}\) as operators on \(L^2(\mathcal{M})\) via essentially a form of \emph{Nyström extension}. The general situation is as follows: given \(f : \mathcal{D} \to \mathbb{C}\) analytic on a disk \(\mathcal{D} \subset \mathbb{C}\) with \([-1, 1] \subset \bar{\mathcal{D}}\) and having an absolutely convergent Taylor series on \(z \in [-1, 1]\), the \emph{derived function relative to} \(f\) given by
\begin{equation}\begin{aligned}
Df(z) := \frac{f(z) - f(0)}{z}
\end{aligned}  \nonumber  \end{equation}
is also analytic on \(\mathcal{D}\) with absolutely convergent Taylor series on \([-1, 1]\), hence we have \(Df(A_{\lambda,\epsilon,N})[u] \in L^{\infty}\) and
\begin{equation} \label{eq:derived-fun-calc-expansion}
(\forall) x \in \mathcal{M}, \quad f(A_{\lambda,\epsilon,N})[u](x) = f(0)u(x) + A_{\lambda,\epsilon,N}[Df(A_{\lambda,\epsilon,N})[u]](x);
\end{equation}
the case for \(A_{\lambda,\epsilon,N}\) replaced with \(A_{\lambda,\epsilon}\) above is clear. Since \(A_{\lambda,\epsilon,N}, A_{\lambda,\epsilon} : L^{\infty} \to C^{\infty}\) are smoothing, this further implies that \((f(A_{\lambda,\epsilon}) - f(A_{\lambda,\epsilon,N})): L^{\infty} \to C^{\infty}\) and \(f(A_{\lambda,\epsilon,N}), f(A_{\lambda,\epsilon}): C^{\infty} \to C^{\infty}\). A particularly useful consequence is that
\begin{align} \label{eq:fun-calc-derived-bound}
\begin{split}
|f(A_{\lambda,\epsilon,N})[u](x)| &\leq |f(0)| |u(x)| + \frac{||k_{\lambda,\epsilon,N}(x,\cdot)||_{N,\infty}}{p_{\lambda,\epsilon,N}(x)} \frac{1}{N} \sum_{j=1}^N |Df(A_{\lambda,\epsilon,N})[u](x_j)|  \\
    &\leq |f(0)| |u(x)| \\
        &\quad\quad + \frac{||k_{\lambda,\epsilon,N}(x,\cdot)||_{N,\infty}}{p_{\lambda,\epsilon,N}(x)} ||p_{\lambda,\epsilon,N}||_{N,\infty}^{\frac{1}{2}} ||p_{\lambda,\epsilon,N}^{-1}||_{N,2}^{\frac{1}{2}} \sup_{z \in [-1,1]}|Df(z)| \, ||u||_{N,2} ,
\end{split}
\end{align}
wherein \(||\cdot||_{N,q} := (\frac{1}{N} \langle |\cdot|^{\frac{q}{2}}, |\cdot|^{\frac{q}{2}} \rangle_{\mathcal{H}_N})^{\frac{1}{q}}\) is the generalized mean \(q\)-norm on the space \(\mathcal{H}_N := \{ u|_{\mathcal{X}_N} ~ : ~ u \in L^{\infty}(\mathcal{M}) \} \cong \mathbb{C}^{N}\) of the restriction of bounded functions on \(\mathcal{M}\) to the set of samples \(\mathcal{X}_N := \{ x_1, \ldots, x_N \}\). The second inequality follows from the symmetric and spectral relations given in \protect\hyperlink{lem:lap-symm}{Lemma \ref{lem:lap-symm}}.

In the following, we will be building on the basic probabilistic bounds from Lemmas  \protect\hyperlink{lem:avgop-consistent}{\ref{lem:avgop-consistent}} and  \protect\hyperlink{lem:rwlap-conv}{\ref{lem:rwlap-conv}} to develop bounds that give consistency of solutions to wave equations and ultimately, a quantum-classical correspondence. Thus, along the way, we will have several other probabilistic bounds that will depend on previous ones and the notations will pack several functions of \(N, \epsilon, \lambda\), bounds of \(p_{\lambda,\epsilon}\), norms of \(u\), etc., so we record here for quick reference the primary bounds and where they are defined or make their first appearance:

\begin{notation}

\emph{Probability bounds}:

\begin{itemize}
\tightlist
\item
  \(\pwb_N, \pwb_{\lambda,N}, \gamma_{\lambda,N}, \gamma_{\lambda,N}^*\) are introduced in the beginning of \protect\hyperlink{uniform-short-time-consistency}{Section \ref{uniform-short-time-consistency}} and the Notation following \protect\hyperlink{lem:avgop-uniform}{Lemma \ref{lem:avgop-uniform}},
\item
  \(\tilde{\gamma}_N\) is defined in the statement of \protect\hyperlink{thm:bdd-analytic-calc-conv}{Theorem \ref{thm:bdd-analytic-calc-conv}},
\item
  \(\gamma_{\lambda,N,\upsilon}, \gamma^*_{\lambda,N,\upsilon}\) are defined in the Notation following \protect\hyperlink{lem:sqrt-perturb-eps-conv}{Lemma \ref{lem:sqrt-perturb-eps-conv}},
\item
  \(\omega_{\lambda,N}, \omega^*_{\lambda,N}\) are defined in the Notation following \protect\hyperlink{lem:prop-short-time-conv}{Lemma \ref{lem:prop-short-time-conv}},
\item
  \(\Omega_{\lambda,t,N}, \Omega^*_{\lambda,t,N}\) are defined in the statement of \protect\hyperlink{thm:halfwave-soln}{Theorem \ref{thm:halfwave-soln}},
\item
  \(\rho_N, \rho_{N,2}\) are defined in the statement of \protect\hyperlink{lem:l2-consistency}{Lemma \ref{lem:l2-consistency}},
\item
  \(\Xi_{\lambda,N}\) is defined in the statement of \protect\hyperlink{lem:prop-cs-norm-consistency}{Lemma \ref{lem:prop-cs-norm-consistency}} and
\item
  \(\tilde{\Omega}_{\lambda,t,N}, \tilde{\Omega}^*_{\lambda,t,N}\) are defined in the statement of \protect\hyperlink{lem:prop-cs-funcexpect-consistency}{Lemma \ref{lem:prop-cs-funcexpect-consistency}}.
\end{itemize}

\end{notation}

We briefly remark on the sequence of convergence results to follow and the nature\footnote{The author thanks Nalini Anantharaman for bringing attention to the predominantly deterministic nature of the forthcoming arguments.} of the arguments. In \protect\hyperlink{uniform-short-time-consistency}{Section \ref{uniform-short-time-consistency}} we establish \emph{uniform} consistency of a functional calculus for \(A_{\lambda,\epsilon,N}\) with respect to holmorphic functions admitting a uniform bounded over derivatives of all orders, along with the consistency of \(\sqrt{I - A_{\lambda,\epsilon,N}}\) and combine these to give consistency of \(U_{\lambda,\epsilon,N}^t[u]\) for \(t \lesssim \sqrt{\epsilon}\) and \(u \in L^{\infty}\), in \protect\hyperlink{lem:prop-short-time-conv}{Lemma \ref{lem:prop-short-time-conv}}. Then, in \protect\hyperlink{uniform-consistency-at-longer-times}{Section \ref{uniform-consistency-at-longer-times}} we extend this via a time splitting argument using the semigroup property for the propagator, to \(t \lesssim \epsilon^{-\frac{\mdim}{16}}\) in \protect\hyperlink{thm:halfwave-soln}{Theorem \ref{thm:halfwave-soln}}. We remark that while the proofs rely on \protect\hyperlink{lem:lap-symm}{Lemma \ref{lem:lap-symm}} establishing the spectral properties of \(\GAve_{\lambda,\epsilon,N}\) due to the connectedness of its graph and on \protect\hyperlink{lem:kern-consistency}{Lemma \ref{lem:kern-consistency}} for its pointwise consistency (used in Lemmas  \protect\hyperlink{lem:avgop-consistent}{\ref{lem:avgop-consistent}} and  \protect\hyperlink{lem:rwlap-conv}{\ref{lem:rwlap-conv}}), which involve probabilistic arguments (namely, Chernoff and Bernstein bounds, repsectively), the forthcoming arguments using these, until \protect\hyperlink{thm:halfwave-soln}{Theorem \ref{thm:halfwave-soln}} are themselves deterministic. Then, using the control on \(||U_{\lambda,\epsilon}[\psi_h]||_{\infty}\) we arrive at a Bernstein-type probabilistic bound for the consistency of the discrete propagation of coherent states in the first part of \protect\hyperlink{prop:max-observable-flow-consistency}{Proposition \ref{prop:max-observable-flow-consistency}}, which is again a deterministic argument. In the rest of that proposition we find that with \(\psi_h\) localized at \((x_0,\xi_0)\), \(\arg\max_{x \in \mathcal{X}_N} |U_{\lambda,\epsilon,N}[\psi_h]|\) is located within an \(O(\sqrt{h})\) ball of \(x_t := \pi_{\mathcal{M}} \Gamma^t(x_0, \xi_0)\), which is a discrete analogue of --- and uses --- part (4) of \protect\hyperlink{lem:prop-coherent-state-localized}{Lemma \ref{lem:prop-coherent-state-localized}}. This uses a simple Chernoff bound to establish a high probability of finding a sample point in a ball near \(\arg\max_{x \in \mathcal{M}} |U_{\lambda,\epsilon,N}[\psi_h]|\). Then \protect\hyperlink{interlude-l2-consistency}{Section \ref{interlude-l2-consistency}} gives a Bernstein-type bound for the approximation of \(u(x_t)\) with \(\sum_{j=1}^N u(x_j) |U_{\lambda,\epsilon,N}^t[\psi_h](x_j)|^2/||U_{\lambda,\epsilon,N}^t[\psi_h]||_{N,2}^2\) for \(u \in C^{\infty}\), which is a discrete realisation of the quantum-classical correspondence \(\eqref{eq:glap-prop-coherent-symbol}\). This follows again by deterministic reasoning, modulo the consistency between discrete and \(L^2\) inner products established in \protect\hyperlink{lem:l2-consistency}{Lemma \ref{lem:l2-consistency}}, which follows from a Bernstein inequality (similar to the proof of \protect\hyperlink{lem:kern-consistency}{Lemma \ref{lem:kern-consistency}}). Combining these, we arrive at a discrete analogue of \protect\hyperlink{prop:local-mean-geodesic-flow}{Proposition \ref{prop:local-mean-geodesic-flow}}: namely, in \protect\hyperlink{prop:mean-geodesic-recover-consistency}{Proposition \ref{prop:mean-geodesic-recover-consistency}} we have the recovery of geodesics from samples through the coordinate means of propagated coherent states. In \protect\hyperlink{summary-of-convergence-rates}{Section \ref{summary-of-convergence-rates}} we \emph{unroll} several of these bounds to give the dominant convergence rates for the various discrete objects under study.

\hypertarget{uniform-short-time-consistency}{%
\subsection{Uniform, short-time consistency}\label{uniform-short-time-consistency}}

We have seen from Lemmas  \protect\hyperlink{lem:avgop-consistent}{\ref{lem:avgop-consistent}} and  \protect\hyperlink{lem:rwlap-conv}{\ref{lem:rwlap-conv}} that for each \(\lambda \geq 0\) there are \(C_0, C_1, s,\eta\) so that
\begin{equation}\begin{aligned}
\Pr[|(A_{\lambda,\epsilon,N} - A_{\lambda,\epsilon})[u](x)| > \delta] \leq \pwb_N(\delta,1;s,\eta;C_0,C_1;u) ,
\end{aligned}  \nonumber  \end{equation}
wherein we define
\begin{equation}\begin{aligned}
\pwb_N(\delta,\sigma ; s, \eta ; C_0,C_1 ; u) := (C_0 + C_1 N) \exp\left( -\frac{(N - \eta) \epsilon^{\frac{n}{2}} \delta^2 \sigma}{2s||u||_{\infty}(\Kconst s ||u||_{\infty} + ||k||_{\infty} \delta/3)} \right) .
\end{aligned}  \nonumber  \end{equation}
We will often use a short-hand to denote this function with only the relevant arguments. The unspecified arguments will be understood to be set by the context, in view of the preceding lemmas and unless specified, we set by default \(\sigma, s = 1\). That is to say, in the following arguments we will use the probabilistic consistency rates of the averaging operators on graphs as a black-box, so for brevity we will use the,

\begin{notation} The symbol \(\pwb_{\lambda,N}(\delta ; u)\) denotes the minimal probability bound among those given in Lemmas  \protect\hyperlink{lem:avgop-consistent}{\ref{lem:avgop-consistent}} and  \protect\hyperlink{lem:rwlap-conv}{\ref{lem:rwlap-conv}} applicable to the data \((\mathcal{M}, k, p, \lambda, \epsilon, N, u, \delta)\): \emph{viz.}, this function gives the best bound when the parameters satisfy the statements of either Lemma and yields \(1\) otherwise. For brevity, when the dependence on \(u\) is clear from context, we will simply write \(\pwb_{\lambda,N}(\delta)\). Further, we denote by \(\eta_{\lambda}\), \(s_{\lambda}\), \(C_{0,\lambda}\) and \(C_{j,\lambda}\) the corresponding parameters such that \(\pwb_N(\delta,1; \eta_{\lambda}, s_{\lambda}; C_{0,\lambda}, C_{1,\lambda}; u) = \pwb_{\lambda,N}(\delta ; u)\).

\end{notation}

\begin{lemma}	\hypertarget{lem:avgop-uniform}{\label{lem:avgop-uniform}} Given \(\lambda \geq 0\), \(\epsilon > 0\), \(u \in L^{\infty}\) and \(\delta > 0\), there is a constant \(C > 0\) depending only on \(k, p, \lambda\) and the geometry of \(\mathcal{M}\) such that
\begin{equation}  \label{eq:avgop-uniform-consistent-bound}
\Pr[||(A_{\lambda,\epsilon,N} - A_{\lambda,\epsilon})[u]||_{\infty} > \delta] \leq C \varepsilon_{\mathcal{M}}^{\mdim} (\pwb_{\lambda,N}(\delta;u) + \pwb_{\lambda,N}(C ; 1)) =: \gamma_{\lambda,N}(\delta ; u)
\end{equation}
with \(\varepsilon_{\mathcal{M}} := \min \{ \operatorname{inj}_{\mathcal{M}} , \epsilon^{\frac{n+1}{2}}, \epsilon^{n + \frac{1}{2}}, \delta \epsilon^{\mdim + \frac{1}{2}}/||u||_{\infty} \}\). Thus, whenever \(0 < \epsilon \leq 1\) and \(0 < \delta \leq \operatorname{inj}_{\mathcal{M}} ||u||_{\infty}\), we have
\begin{equation}\begin{aligned}
\Pr[||(A_{\lambda,\epsilon,N} - A_{\lambda,\epsilon})[u]||_{\infty} > \delta] \leq C \left( \epsilon^{\mdim + \frac{1}{2}} \delta/||u||_{\infty} \right)^{\mdim} (\pwb_N(\delta;u) + \pwb_N(C ; 1)) .
\end{aligned}  \nonumber  \end{equation}

\end{lemma}

\begin{proof}

Let \(0 < \varepsilon < \operatorname{inj}(\mathcal{M})\) and take \(x_1^*, \ldots, x_M^* \in \mathcal{M}\) such that \(M = \mathcal{N}(\varepsilon)\) is the minimal covering number given by Bishop-Günther inequality and \(\cup_{j=1}^M B(x_j^*, \varepsilon) = \mathcal{M}\). For each \(j \in [M]\), let \(\mathcal{U}_j \subset B_{\mathcal{M}}(x_j^*, \operatorname{inj}(x_j^*)) \subset \mathcal{M}\) be a geodesically convex neighbourhood of \(x_j^*\) and let \(s^{-1} : \mathcal{U}_j \to V_j \subset \mathbb{R}^{\mdim}\) provide normal coordinates for this neighbourhood. Then, for any \(v \in C^{\infty}\) and \(x \in \mathcal{M}\), there is \(j \in [M]\) such that \(x \in B(x^*_j, \varepsilon)\) so by Taylor expansion centered at \(0 \in V_j\) and evaluated at \(s^{-1}(x) =: w \in V_j\), Taylor's theorem gives that
\begin{equation} \label{eq:taylor-expand-normal-coords}
|v(x)| \leq |v(x^*_j)| + \varepsilon \max_{|\alpha| = 1} \sup_{w \in \tilde{B}_j} |\partial^{\alpha}[v \circ s](w)|,
\end{equation}
with \(\tilde{B}_j := s(B(x_j^*, \varepsilon))\). We will use the following bound:
\begin{align} \label{eq:k-first-deriv-bound}
\begin{split}
|\partial_{w^{(i)}} k_{\epsilon}(s(w),y)|
    &= 2\epsilon^{-\frac{\mdim + 2}{2}} |\langle \iota \circ s(w) - \iota(y), \partial_{w^{(i)}}[\iota \circ s(\cdot) - \iota(y)](w) \rangle_{\mathbb{R}^{\mdim}} \, k'_{\epsilon}(s(w),y)| \\
    &\leq \epsilon^{-\frac{n + 2}{2}} ||\iota \circ s(w) - \iota(y)|| \, ||J_{\iota}[\partial_{w^{(i)}} s](w)|| \, |k'_{\epsilon}(s(w),y)|  \\
    &= \epsilon^{-\frac{n + 2}{2}} ||\iota \circ s(w) - \iota(y)|| \, ||\partial_{w^{(i)}} s(w)|| \, |k'_{\epsilon}(s(w),y)|    \\
    &\leq C_{\mathcal{M},k,1} \epsilon^{-\frac{n + 1}{2}},
\end{split}
\end{align}
wherein \(k' : r \mapsto \partial_r k\) is localized to \([0, R_k^2]\) and \(k'_{\epsilon}(x,y) := k'(||\iota(x) - \iota(y)||^2 / \epsilon)\), which gives that \(||\iota(x) - \iota(y)|| \leq R_k \sqrt{\epsilon}\) and \(C_{\mathcal{M},k,1} = C_{\mathcal{M},1} R_k ||k'||_{\infty}\) with \(C_{\mathcal{M},1} := \max_{i \in [\mdim]} ||\partial_{w^{(i)}} s(w)||_{\infty}\). Hence also, \(|\partial_{w^{(i)}} [p_{\epsilon} \circ s](w)|, |\partial_{w^{(i)}} [p_{\epsilon,N} \circ s](w)| \leq C_{\mathcal{M},k,1} \epsilon^{-\frac{n+1}{2}}\).

Now let \(\delta_1 > 0\) and assume we are in the event
\begin{align*}
\mathcal{A}(\varepsilon,\delta_1) : |p_{\lambda',\epsilon,N}(x^*_j) - p_{\lambda',\epsilon}(x^*_j)| \leq \delta_1/2 && \text{for all } j \in \{ 1, \ldots, M \}, \lambda' \in \{ 0, \lambda \} .
\end{align*}
Then for any \(x \in \mathcal{M}\), there is \(j \in [M]\) such that \(d_g(x^*_j,x) \leq \varepsilon\), so that by \(\eqref{eq:taylor-expand-normal-coords}\) and \(\eqref{eq:k-first-deriv-bound}\),
\begin{align*}
|p_{\epsilon,N}(x) - p_{\epsilon}(x)| &\leq \frac{\delta_1}{2} + 2\varepsilon \max_{|\alpha| = 1} \sup_{(w,y) \in \tilde{B}_j \times \mathcal{M}} |\partial^{\alpha}[k_{\epsilon}(s)](w,y)| \\
    &\leq \frac{\delta_1}{2} + 2 C_{\mathcal{M},k,1} \epsilon^{-\frac{n+1}{2}} \varepsilon .
\end{align*}
In the following, we will use the notation: \(\min_{\mathcal{M}} : \mathbb{R}^{S} \ni \vec{a} \mapsto \min\{ \operatorname{inj}(\mathcal{M}), a_1, \ldots, a_S \} \in \mathbb{R}\) for any \(S \geq 1\). Hence, with \(\varepsilon \leq \min\{ \operatorname{inj}(\mathcal{M}), \epsilon^{\frac{n + 1}{2}} \delta_1/(4 C_{\mathcal{M},k,1}) \} =: \min_{\mathcal{M}}(\epsilon^{\frac{n + 1}{2}} \delta_1/(4 C_{\mathcal{M},k,1}))\), we have that \(|p_{\epsilon,N}(x) - p_{\epsilon}(x)| \leq \delta_1\). Thus, assuming \(\epsilon \leq 1\) we have the bound,
\begin{equation}\begin{aligned}
|\partial_{w^{(i)}}[p_{\lambda,\epsilon,N} \circ s](w)| \leq \epsilon^{-n-\frac{1}{2}} C_{\mathcal{M},k,1} C_{p,\lambda,1,\delta_1} ,   \\
C_{p,\lambda,1,\delta_1} := \frac{\lambda (\overline{C}_p + \delta_1)^{2\lambda}}{(\underline{C}_{p} - \delta_1)^{4\lambda}}\left( \frac{1}{\lambda} + \frac{||k||_{\infty}}{\overline{C}_p + \delta_1} \right)
\end{aligned}  \nonumber  \end{equation}
and further assuming that \(\varepsilon \leq \min_{\mathcal{M}}\{ \delta_1(\epsilon^{\frac{n+1}{2}}, \epsilon^{n + \frac{1}{2}}/C_{p,\lambda,1,\delta_1})/(4C_{\mathcal{M},k,1}) \}\) gives on another application of \(\eqref{eq:taylor-expand-normal-coords}\) and \(\eqref{eq:k-first-deriv-bound}\),
\begin{equation}\begin{aligned}
|p_{\lambda,\epsilon,N}(x) - p_{\lambda,\epsilon}(x)| \leq \delta_1 .
\end{aligned}  \nonumber  \end{equation}
Now let \(x, y \in \mathcal{M}\) and \(s^{-1}\) be centered at \(x^*_j\) such that \(x \in B(x^*_j, \varepsilon)\). Then, combining the above bounds gives,
\begin{align*}
\left| \partial_{w^{(i)}}\left[ \frac{k_{\lambda,\epsilon,N}(s(\cdot), y)}{p_{\lambda,\epsilon,N} \circ s(\cdot)} \right] \right|_{\cdot = w}
    &= \left| \frac{(p_{\lambda,\epsilon,N} p_{\epsilon,N}^{\lambda})(x) \partial_{w^{(i)}}[k_{\epsilon}(s,y)](w) - k_{\epsilon}(s,y) \partial_{w^{(i)}}[p_{\lambda,\epsilon,N} \, p_{\epsilon,N}^{\lambda} \circ s](w)}{p_{\epsilon,N}(y)^{\lambda}(p_{\lambda,\epsilon,N}(x) \, p_{\epsilon,N}^{\lambda}(x))^2} \right|   \\
    &\leq \epsilon^{-(n + \frac{1}{2})} \frac{C_{\mathcal{M},k,1}}{(\underline{C}_{p,\lambda} - \delta_1)^3 (\underline{C}_p - \delta_1)^{3\lambda}}    \\
        &\quad\quad \times \left( \epsilon^{\frac{n}{2}} (\overline{C}_{p,\lambda} + \delta_1)(\overline{C}_p + \delta_1)^{\lambda} \right. \\
            &\quad\quad\quad\quad \left. + ||k||_{\infty}(C_{p,\lambda,1,\delta_1} (\overline{C}_p + \delta_1)^{\lambda} + \lambda(\overline{C}_{p,\lambda} + \delta_1)(\overline{C}_p + \delta_1)^{\lambda - 1}) \right)   \\
        &\leq \epsilon^{-(n + \frac{1}{2})} C_{\mathcal{M},k,1} C_{k,p,\lambda,1,\delta_1}
\end{align*}
with
\begin{align*}
C_{k,p,\lambda,1,\delta_1} &:= ((\underline{C}_{p,\lambda} - \delta_1) (\underline{C}_p - \delta_1)^{\lambda})^{-3} \left( (\overline{C}_{p,\lambda} + \delta_1)(\overline{C}_p + \delta_1)^{\lambda} \right.   \\
        &\quad\quad \left. + \, ||k||_{\infty}(C_{p,\lambda,1,\delta_1} (\overline{C}_p + \delta_1)^{\lambda} + \lambda(\overline{C}_{p,\lambda} + \delta_1)(\overline{C}_p + \delta_1)^{\lambda - 1} \right).
\end{align*}
In the continuum case, these bounds hold with \(\delta_1 = 0\), so letting \(\delta_1 = \min\{ \underline{C}_{p,\lambda}, \underline{C}_p \}/2\) we have \(C_{p,\lambda,1,\delta_1} \leq 36^{\lambda} C_{p,\lambda,1,0} =: C_{p,\lambda,1}\) and
\begin{equation}\begin{aligned}
\left| \partial_{w^{(i)}} \left[ \frac{k_{\lambda,\epsilon,N}(s(\cdot),y)}{p_{\lambda,\epsilon,N} \circ s(\cdot)} \right]_{\cdot = w} \right|, \left| \partial_{w^{(i)}} \left[ \frac{k_{\lambda,\epsilon}(s(\cdot),y)}{p_{\lambda,\epsilon} \circ s(\cdot)} \right]_{\cdot = w} \right| \leq \epsilon^{-(n + \frac{1}{2})} C_{\mathcal{M},k,1} C_{k,p,\lambda,1} , \\
C_{k,p,\lambda,1} := 12 (432^{\lambda} C_{k,p,\lambda,1,0}) .
\end{aligned}  \nonumber  \end{equation}
With this and the event
\begin{align*}
\mathcal{B}(\varepsilon, \delta) : |\GAve_{\lambda,\epsilon,N}[u](x_j^*) - \GAve_{\lambda,\epsilon}[u](x_j^*)| \leq \delta/2 && \text{for all } j \in \{ 1, \ldots, M \} ,
\end{align*}
gives upon Taylor expanding as in \(\eqref{eq:taylor-expand-normal-coords}\),
\begin{align*}
|(\GAve_{\lambda,\epsilon,N} - \GAve_{\lambda,\epsilon})[u](x)| &\leq \delta/2 + \varepsilon \max_{|\alpha| = 1} ||\partial^{\alpha} (\GAve_{\lambda,\epsilon,N} - \GAve_{\lambda,\epsilon})[u] \circ s||_{\infty}  \\
        &\leq \delta/2 + C'\varepsilon \epsilon^{-(n + \frac{1}{2})} ||u||_{\infty}
\end{align*}
with \(C' := 2 C_{\mathcal{M},k,1} C_{k,p,\lambda,1}\). Since the second inequality on the right-hand side is independent of \(x\) and the Bishop-Günther inequality tells that \(M = C \varepsilon^{\mdim}\) for some constant \(C > 0\), we have upon taking \(\varepsilon = \min_{\mathcal{M}}\{ [\delta_1 (\epsilon^{\frac{n+1}{2}}, \epsilon^{n + \frac{1}{2}}/C_{p,\lambda,1}), \delta \epsilon^{\mdim + \frac{1}{2}}/(||u||_{\infty} C_{k,p,\lambda,1})]/(4 C_{\mathcal{M},k,1}) \}\) and a union bound over the probability events \(\mathcal{A}(\varepsilon,\delta_1)\) and \(\mathcal{B}(\varepsilon,\delta)\) that
\begin{equation}\begin{aligned}
\Pr[||(A_{\lambda,\epsilon,N} - A_{\lambda,\epsilon})[u]||_{\infty} > \delta] \leq C \varepsilon^{\mdim} (\pwb_N(\delta ; u) + \pwb_N(\delta_1 ; 1)).
\end{aligned}  \nonumber  \end{equation}
\end{proof}

\begin{notation} We will denote
\begin{equation}\begin{aligned}
\gamma_{\lambda,N}(\cdot) := C \varepsilon_{\mathcal{M}}^{\mdim} (\pwb_{\lambda,N}(\delta;u) + \pwb_{\lambda,N}(C ; 1))
\end{aligned}  \nonumber  \end{equation}
with the notational simplifications for \(\pwb_{\lambda,N}\) being applied in the same way to \(\gamma_{\lambda,N}\). Additionally, to allow variability, we will denote
\begin{equation}\begin{aligned}
\gamma_{\lambda,N}^*(\delta,\sigma ; u) := C \varepsilon_{\mathcal{M}}^{\mdim} (\pwb_N(\delta, \sigma; \eta_{\lambda}, s_{\lambda}; C_{0, \lambda}, C_{1,\lambda} ; u) + \pwb_{\lambda,N}(C ; 1))
\end{aligned}  \nonumber  \end{equation}
and shorten this to \(\gamma_N^*(\delta,\sigma)\) when the dependence on \(u\) is clear from context.

\end{notation}

\begin{lemma}	\hypertarget{lem:avop-power-conv}{\label{lem:avop-power-conv}} Given \(\lambda \geq 0\), \(\epsilon > 0\), \(u \in L^{\infty}\) and \(m \in \mathbb{N}\), we have for all \(\delta > 0\),
\begin{equation}\begin{aligned}
\Pr[||A_{\lambda,\epsilon,N}^m[u] - A_{\lambda,\epsilon}^m[u]||_{\infty} > m \delta] \leq m \gamma_{\lambda,N}(\delta ; u).
\end{aligned}  \nonumber  \end{equation}

\end{lemma}

\begin{proof}

Since \(||A_{\lambda,\epsilon}[u]||_{\infty} \leq ||u||_{\infty}\), we have the same probabilistic bound for the event that \(||(A_{\lambda,\epsilon,N} - A_{\lambda,\epsilon})[A_{\epsilon} u]||_{\infty} \leq \delta\) as for the event \(||(A_{\lambda,\epsilon,N} - A_{\lambda,\epsilon})[u]||_{\infty} \leq \delta\). Moreover since \(A_{\lambda,\epsilon,N}[1] = 1\) and the matrix has non-negative entries,
\begin{align*}
||A_{\lambda,\epsilon,N}^2[u] - A_{\lambda,\epsilon}^2[u]||_{\infty} &\leq
    ||(A_{\lambda,\epsilon,N} - A_{\lambda,\epsilon})[A_{\lambda,\epsilon} u]||_{\infty} +  \\
    &\quad\quad + ||A_{\lambda,\epsilon,N}[(A_{\lambda,\epsilon,N} - A_{\lambda,\epsilon})[u]]||_{\infty}   \\
    &\leq ||(A_{\lambda,\epsilon,N} - A_{\lambda,\epsilon})[A_{\lambda,\epsilon} u]||_{\infty} +    \\
    &\quad\quad + ||(A_{\lambda,\epsilon,N} - A_{\lambda,\epsilon})[u]||_{\infty}
\end{align*}
Therefore, applying a union bound over both events, we have,
\begin{equation}\begin{aligned}
\Pr[||(A_{\lambda,\epsilon,N}^2 - A_{\lambda,\epsilon}^2)[u]||_{\infty} > 2\delta] \leq 2\gamma_{\lambda,N}(\delta ; u) .
\end{aligned}  \nonumber  \end{equation}
Now assume that for all \(2 \leq m \leq M-1\),
\begin{equation}\begin{aligned}
\Pr[||(A_{\lambda,\epsilon,N}^m - A_{\lambda,\epsilon}^m)[u]||_{\infty} \leq m\delta] > 1 - m\gamma_{\lambda,N}(\delta ; u).
\end{aligned}  \nonumber  \end{equation}
Write similar to before,
\begin{align*}
||A_{\lambda,\epsilon,N}^M[u] - A_{\lambda,\epsilon}^M[u]||_{\infty} &\leq
    ||(A_{\lambda,\epsilon,N}^{M-1} - A_{\lambda,\epsilon}^{M-1})[A_{\lambda,\epsilon} u]||_{\infty} +  \\
    &\quad\quad + ||A_{\lambda,\epsilon,N}^{M-1}[A_{\lambda,\epsilon,N}[u] - A_{\lambda,\epsilon}[u]]||_{\infty} .
\end{align*}
Since for any integer \(m \geq 1\), \(A_{\lambda,\epsilon,N}^m\) has non-negative entries and \(A_{\lambda,\epsilon,N}^m[1] = 1\), we see that for any \(v \in L^{\infty}\), \(|A_{\lambda,\epsilon,N}^m[v](x)| \leq ||v||_{\infty}\). Therefore, after a union bound we have that
\begin{equation}\begin{aligned}
\Pr[||(A_{\lambda,\epsilon,N}^M - A_{\lambda,\epsilon}^M)[u]||_{\infty} \leq M\delta] > 1 - M\gamma_{\lambda,N}(\delta ; u),
\end{aligned}  \nonumber  \end{equation}
whence having completed the induction step, we have the bound as given in the first part of the statement of the Lemma.
\end{proof}

\begin{theorem}	\hypertarget{thm:bdd-analytic-calc-conv}{\label{thm:bdd-analytic-calc-conv}} Let \(f : \mathbb{C} \to \mathbb{C}\) be an entire function such that for \(w \in \mathbb{C}\) there exists a constant \(K_w > 0\) such that for all \(m \geq 0\), \(|\partial^m f|_{z = w}| < K_w\). Then, given \(\lambda \geq 0\), \(\epsilon > 0\), \(u \in C^{\infty}\) and \(\delta > 0\),
\begin{align*}
\Pr[||f(A_{\lambda,\epsilon,N} + wI)[u] - f(A_{\lambda,\epsilon} + wI)[u]||_{\infty} > \delta] &\leq \tilde{\gamma}_{\lambda,N}(\delta/(2 e K_w)),    \\
\tilde{\gamma}_{\lambda,N} := \gamma_{\lambda,N}/(1 - \gamma_{\lambda,N})^2 .
\end{align*}

\end{theorem}

\begin{proof}

Let \(f : \mathbb{C} \to \mathbb{C}\) be an entire function with \(|\partial_z^m f|_{z = w}| < K_w\) for some constant \(K_w > 0\) and all \(m \geq 0\). By \protect\hyperlink{lem:avop-power-conv}{Lemma \ref{lem:avop-power-conv}} and a union bound, we have
\begin{equation}\begin{aligned}
\Pr[(\forall)m \in [M], ||(A_{\lambda,\epsilon,N}^m - A_{\lambda,\epsilon}^m)[u]||_{\infty} \leq m^2 \delta] > 1 - \sum_{m=1}^M m\gamma_{\lambda,N}(m\delta ; u)
\end{aligned}  \nonumber  \end{equation}
and since for all \(m \geq 1\), \(\gamma_{\lambda,N}(m\delta) \leq \gamma^*_{\lambda,N}(\delta, m)\), we can take all powers at once to have,
\begin{align*}
\Pr&[(\forall) m \in \mathbb{N}, ||(A_{\lambda,\epsilon,N}^m - A_{\lambda,\epsilon}^m)[u]||_{\infty} > m^2 \delta]  \\
    &\leq \sum_{m=1}^{\infty} m \gamma^*_{\lambda,N}(\delta, m) \leq \frac{\gamma_{\lambda,N}(\delta)}{(1 - \gamma_{\lambda,N}(\delta))^2} =: \tilde{\gamma}(\delta).
\end{align*}
In this event, upon taking a Taylor series expansion of \(f(z)\) at \(z = w\),
\begin{equation}\begin{aligned}
|f(A_{\lambda,\epsilon,N} + wI)[u] - f(A_{\lambda,\epsilon} + wI)[u]| \leq K_w \delta\sum_{m=1}^{\infty} \frac{m^2}{m!} < e K_w\delta
\end{aligned}  \nonumber  \end{equation}
hence,
\begin{equation}\begin{aligned}
\Pr[||f(A_{\lambda,\epsilon,N})[u] - f(A_{\lambda,\epsilon})[u]||_{\infty} > \delta] \leq \tilde{\gamma}(\delta/(2 e K_w)).
\end{aligned}  \nonumber  \end{equation}
\end{proof}

\begin{remark} The condition that all derivatives of \(f\) are bounded enforces that we take \(t \sim \sqrt{\epsilon}\). While the class of \(f\) can be generalized by way of Hadamard's multiplication theorem, the condition on the short time-scales is not artificial: if we directly apply this to the approximation of \(\mathcal{C}(t/\sqrt{\epsilon})\) from \(\mathcal{C}_N(t/\sqrt{\epsilon})\), then the resulting error is \(O(e^{\frac{t}{\sqrt{\epsilon}}} t^2/\epsilon)\), which is only practical for asymptotically short times, \(t \sim \sqrt{\epsilon}\) in any case.

\end{remark}

\begin{lemma}	\hypertarget{lem:sqrt-perturb-eps-conv}{\label{lem:sqrt-perturb-eps-conv}} Given \(\lambda \geq 0\), \(\epsilon > 0\), \(u \in L^{\infty}\) and \(\delta > 0\), if \(N \in \mathbb{N}\) is sufficiently large that \(\beta_{\gamma} := \log{[(C_{0,\lambda} + C_{1,\lambda} N)/\gamma_{\lambda,N}(\delta;u)]} > 1\), then with \(\upsilon := 8 \log(\beta_{\gamma}/2)/\beta_{\gamma}\) we have that for all \(0 < \varepsilon \leq 1\),
\begin{equation} \label{eq:lem-sqrt-perturb-eps-conv-bound}
\Pr[||B^{(\varepsilon)}_{\lambda,\epsilon,N}[u] - B^{(\varepsilon)}_{\lambda,\epsilon}[u]||_{\infty} > \delta] \leq \gamma_{\lambda,N}(\delta \varepsilon^{(\frac{1}{2} + \upsilon(\beta))}/(2 \sqrt{\pi}) ; u),
\end{equation}
wherein \(B_{\lambda,\epsilon,N}^{(\varepsilon)} := \sqrt{(1 + \varepsilon)I - A_{\lambda,\epsilon,N}}\) and \(B_{\lambda,\epsilon}^{(\varepsilon)} := \sqrt{(1 + \varepsilon)I - A_{\lambda,\epsilon}}\).

Thus,
\begin{equation}\begin{aligned}
\Pr[||B^{(\delta^2)}_{\lambda,\epsilon,N}[u] - B^{(\delta^2)}_{\lambda,\epsilon}[u]||_{\infty} > \delta] \leq \gamma_{\lambda,N}(\delta^{2(1 + \upsilon)}/(2 \sqrt{\pi}) ; u).
\end{aligned}  \nonumber  \end{equation}

\end{lemma}

\begin{proof}

Let \(\upsilon > 0\) be variable for the moment and proceed as in the proof of \protect\hyperlink{thm:bdd-analytic-calc-conv}{Theorem \ref{thm:bdd-analytic-calc-conv}}, with the change that now we take for each \(m \in \mathbb{N}\) an \(m^{\upsilon} \delta\) error in \(||A_{\lambda,\epsilon,N}[u] - A_{\lambda,\epsilon}[u]||_{\infty}\), so \emph{mutatis mutandis} we have,
\begin{equation}\begin{aligned}
\Pr[(\forall)m \in \mathbb{N}, ||(A_{\lambda,\epsilon,N}^m - A_{\lambda,\epsilon}^m)[u]||_{\infty} > m^{1+\upsilon} \delta] \leq \sum_{m=1}^{\infty} m\gamma^*_{\lambda,N}(\delta, m^{\upsilon} ; u).
\end{aligned}  \nonumber  \end{equation}
Suppressing all coefficients independent of \(m\), the right-hand side has the form
\begin{equation}\begin{aligned}
\sum_{m=1}^{\infty} m \gamma^*_{\lambda,N}(\delta,m^{\upsilon} ; u) = \sum_{m=1}^{\infty} m (\alpha e^{-\beta m^{\upsilon}})
\end{aligned}  \nonumber  \end{equation}
for some \(\alpha > 0\) and \(\beta = \beta_{\gamma} > 0\). We have,
\begin{equation}\begin{aligned}
\alpha\sum_{m=1}^{\infty} m e^{-\beta m^{\upsilon}} \leq \frac{1}{\upsilon} \int_1^{\infty} y^{\frac{2}{\upsilon}-1} e^{-\beta y} ~ dy .
\end{aligned}  \nonumber  \end{equation}
Letting \(r := \lceil \frac{1}{\upsilon} \rceil\) then bounding the right-hand side with the substitution of \(1/\upsilon\) with \(r\) and applying succesive integration by parts gives,
\begin{align*}
\sum_{m=1}^{\infty} m e^{-\beta m^{\upsilon}}
    &\leq \frac{e^{-\beta}}{\beta} r \sum_{j = 1}^{2(r - 1)} \frac{(2r - 1)!}{(2r - j)!} + r (2r)! \int_1^{\infty} y^{-1} e^{-\beta y} ~ dy \\
    &\leq \frac{2r (2r)!}{\beta} e^{-\beta} .
\end{align*}
By Stirling's approximation, \citep{robbins1955remark} gives the bound \((2r)! \leq e^{1 - 2r} (2r)^{2r + \frac{1}{2}}\) and therefore,
\begin{equation}\begin{aligned}
\sum_{m=1}^{\infty} m e^{-\beta m^{\upsilon}} \leq \frac{e (2r)^{\frac{3}{2}}}{\beta} e^{-\beta + 2r(\log(2r) - 1)} .
\end{aligned}  \nonumber  \end{equation}
Setting \(\upsilon(\beta) := 8 \log(\beta/2)/\beta\) and using that \(\beta \geq 1\) for sufficiently large \(N\) then gives,
\begin{align*}
2r(\log(2r) - 1)
    &= 2 \left\lceil \frac{\beta}{8 \log(\beta/2)} \right\rceil \left( \log\left( 2 \left\lceil \frac{\beta}{8 \log(\beta/2)} \right \rceil \right) - 1 \right) \\
    &\leq \frac{\beta}{2\log(\beta/2)} \left( \log\left( \frac{\beta/2}{ \log(\beta/2)} \right) - 1 \right),
\end{align*}
hence,
\begin{align*}
\sum_{m=1}^{\infty} m e^{-\beta m^{\upsilon(\beta)}}
    &\leq e \left( \frac{\beta}{2 \log(\beta/2)} \right)^{\frac{1}{2}} e^{-\frac{\beta}{2}\left(1 + \frac{1}{\log(\beta/2)} \right)}    \\
    &\leq e^{-\frac{\beta}{2}} .
\end{align*}
On recovering the original form of the probability upper bound, we have,
\begin{equation}\begin{aligned}
\Pr[(\forall)m \in \mathbb{N}, ||(A_{\lambda,\epsilon,N}^m - A_{\lambda,\epsilon}^m)[u]||_{\infty} > m^{1 + \upsilon(\beta)} \delta] \leq \gamma_{\lambda,N}(\delta/\sqrt{2}; u).
\end{aligned}  \nonumber  \end{equation}
Next, given the event that for all \(m \in \mathbb{N}\), \(|(A^m_{\lambda,\epsilon,N} - A^m_{\lambda,\epsilon})[u]| \leq m^{1 + \upsilon(\beta)} \delta\), we compare the applications of the Tayor series expansion of \(f := \sqrt{z}\) at \(1 + \varepsilon\) to \((1 + \varepsilon)I - A_{\lambda,\epsilon,N}\) and \((1 + \varepsilon)I - A_{\lambda,\epsilon}\), for a \emph{perturbation parameter} \(\varepsilon > 0\), with this term-wise absolute error rate: let \(B_{\lambda,\epsilon,N}^{(\varepsilon)} := \sqrt{(1 + \varepsilon)I - A_{\lambda,\epsilon,N}}\) and \(B_{\lambda,\epsilon}^{(\varepsilon)} := \sqrt{(1 + \varepsilon)I - A_{\lambda,\epsilon}}\), then assuming \(N\) is sufficiently large that \(\upsilon(\beta) \leq 1\) we have,
\begin{align*}
||B^{(\varepsilon)}_{\lambda,\epsilon,N}[u] - B^{(\varepsilon)}_{\lambda,\epsilon}[u]||_{\infty}
    &\leq \delta \sum_{m=1}^{\infty} \left| \binom{\frac{1}{2}}{m} \right| \frac{m^{1 + \upsilon(\beta)}}{(1 + \varepsilon)^{m - \frac{1}{2}}}  \\
    &\leq \delta \sum_{m=1}^{\infty} \frac{m^{\upsilon(\beta) - \frac{1}{2}}}{(1 + \varepsilon)^{m - \frac{1}{2}}} = \delta \sqrt{1 + \varepsilon} \, \operatorname{Li}\left( \frac{1}{2} - \upsilon(\beta), \frac{1}{1 + \varepsilon} \right)  \\ 
  &\leq \delta \Gamma\left[ \frac{1}{2} + \upsilon(\beta) \right] \varepsilon^{-(\frac{1}{2} + \upsilon(\beta))} \sqrt{1 + \varepsilon} \\
  &\leq \sqrt{2 \pi} \, \varepsilon^{-(\frac{1}{2} + \upsilon(\beta))} \delta .
\end{align*}
The second inequality comes from \(|\binom{1/2}{m}| \leq m^{\frac{3}{2}}\) and the following one is due to the observation that
\begin{align*}
\varepsilon^{s + 1/2} \operatorname{Li}\left( \frac{1}{2} - s, \frac{1}{1 + \varepsilon} \right)
    &= \frac{1}{\Gamma\left( \frac{1}{2} - s \right)} \int_0^{\infty} \frac{\varepsilon^{s + 1/2} t^{-(s + 1/2)}}{e^t (1 + \varepsilon) - 1} ~ dt   \\
    &= \frac{1}{\Gamma\left( \frac{1}{2} - s \right)} \int_0^{\infty} \frac{\varepsilon}{e^{\varepsilon t}(1 + \varepsilon) - 1} \, t^{-(s + 1/2)} ~ dt \\
    &\xrightarrow[\varepsilon \to 0]{} \frac{1}{\Gamma\left( \frac{1}{2} - s \right)} \int_0^{\infty} \frac{t^{-(s + 1/2)}}{1 + t} ~ dt \\
    &= \frac{\operatorname{B}\left( \frac{1}{2} - s, \frac{1}{2} + s \right)}{\Gamma\left( \frac{1}{2} - s \right)} = \Gamma\left( \frac{1}{2} + s \right),
\end{align*}
(wherein \(\operatorname{B}(x,y) = \Gamma(x) \Gamma(y)/\Gamma(x + y)\) is the Beta function) together with the monotonicity of \(\operatorname{Li}\left( \frac{1}{2} - \upsilon(\beta), \frac{1}{1 + \varepsilon} \right)\) in \(\varepsilon\). The fourth (final) inequality uses that \(\Gamma\left( \frac{1}{2} + s \right)\) is maximized at \(s = 0\) when \(0 \leq s \leq 1\), \emph{viz.}, we employ the assumption that \(N\) is sufficiently large that \(\upsilon(\beta) \leq 1\). Altogether,
\begin{equation}\begin{aligned}
\Pr[||B^{(\varepsilon)}_{\lambda,\epsilon,N}[u] - B^{(\varepsilon)}_{\lambda,\epsilon}[u]||_{\infty} > \delta] \leq \gamma_{\lambda,N}(\delta \varepsilon^{(\frac{1}{2} + \upsilon(\beta))}/(2 \sqrt{\pi}) ; u) ,
\end{aligned}  \nonumber  \end{equation}
which upon setting \(\varepsilon = \delta^2\) gives both of the probability bounds in the statement of the Theorem.
\end{proof}

\begin{notation} We will use the symbol \(\upsilon\) to denote the function given by \protect\hyperlink{lem:sqrt-perturb-eps-conv}{Lemma \ref{lem:sqrt-perturb-eps-conv}} and we define the functions
\begin{equation}\begin{aligned}
\gamma^*_{\lambda,N,\upsilon}(\delta, \varepsilon ; u) := \gamma_{\lambda,N}(\delta \varepsilon^{(\frac{1}{2} + \upsilon(\beta))}/(2 \sqrt{\pi}) ; u)
\end{aligned}  \nonumber  \end{equation}
and
\begin{equation}\begin{aligned}
\gamma_{\lambda,N,\upsilon}(\delta ; u) := \gamma_{\lambda,N}(\delta^{2(1 + \upsilon)}/(2 \sqrt{\pi}); u)
\end{aligned}  \nonumber  \end{equation}
to denote the corresponding probability bounds for the consistency of the square root of the perturbed operators.

\end{notation}

An immediate application of the relations in \protect\hyperlink{lem:lap-symm}{Lemma \ref{lem:lap-symm}} is to transfer consistency between the pair \(\tilde{B}_{\lambda,\epsilon,N}, \tilde{B}_{\lambda,\epsilon}\) to that between the pair of square roots of the unpertrubed operators: namely we have,

\begin{theorem}	\hypertarget{thm:sqrt-conv}{\label{thm:sqrt-conv}} Let \(\lambda \geq 0\) and \(u \in L^{\infty}\). Then, there are constants \(C_1, C_2 > 0\) depending only on \(k\) and the geometry of \(\mathcal{M}\) such that for all \(\epsilon \in (0, C_1]\) and \(\delta > 0\),
\begin{equation} \label{eq:thm-sqrt-conv-bound}
\Pr[||B_{\lambda,\epsilon,N}[u] - B_{\lambda,\epsilon}[u]||_{\infty} > \delta] \leq \gamma_{\lambda,N,\upsilon}(\delta/3 ; u) + e^{-\frac{N \epsilon^{\frac{\mdim}{2}} \delta^2}{C_2}},
\end{equation}
with \(B_{\lambda,\epsilon,N} := \sqrt{I - A_{\lambda,\epsilon,N}}\) and \(B_{\lambda,\epsilon} := \sqrt{I - A_{\lambda,\epsilon}}\).

\end{theorem}

\begin{proof}

Let \(\varepsilon \geq 0\) and set \(B^{(\varepsilon)}_{\lambda,\epsilon,N}\) and \(B^{(\varepsilon)}_{\lambda,\epsilon}\) as in the proof of \protect\hyperlink{lem:sqrt-perturb-eps-conv}{Lemma \ref{lem:sqrt-perturb-eps-conv}}. Assume throughout the event that \(A_{\lambda,\epsilon,N}\) gives a connected graph, as per part (2) of \protect\hyperlink{lem:lap-symm}{Lemma \ref{lem:lap-symm}}. Then, the matrix \((1 + \varepsilon)I - A_{\lambda,\epsilon,N}\) is an M-matrix and (in modulus) the largest eigenvalue of \(A_{\lambda,\epsilon,N}\) is one, which is also simple. Therefore, it follows from \citep[Theorem 4]{alefeld1982mmatsqrt} that \(B^{(\varepsilon)}_{\lambda,\epsilon,N} = (1 + \varepsilon)^{\frac{1}{2}}(I - \mathcal{B}_{\lambda,\epsilon,N}^{(\varepsilon)})\) is also an M-matrix for \(\mathcal{B}^{(\varepsilon)}_{\lambda,\epsilon,N}\) some matrix with non-negative entries and spectral radius \(r_{\varepsilon,N} < ((2 + \varepsilon)/(1 + \varepsilon))^{\frac{1}{2}} - 1 < 1\).

Let \(L^2_{\mathbb{R}}(\mathcal{M}) := L^2(\mathcal{M} \to \mathbb{R})\) be the real Hilbert space of real-valued functions on \(\mathcal{M}\). Then, \((B^{(\varepsilon)})^2_{\lambda,\epsilon} : L^2_{\mathbb{R}}(\mathcal{M}) \to L^2_{\mathbb{R}}(\mathcal{M})\) is an \emph{M-operator} (an infinite-dimensional generalization of M-matrices) with respect to the cone \(\mathscr{K}\) of non-negative functions and satisfies the conditions of \citep[Theorem 3]{marek1995mopsqrt}. Therefore, by that Theorem, \(B^{(\varepsilon)}_{\lambda,\epsilon} = (1 + \varepsilon)^{\frac{1}{2}}(I - \mathcal{B}^{(\varepsilon)}_{\lambda,\epsilon})\) is also an M-operator for some \(\mathcal{B}^{(\varepsilon)}_{\lambda,\epsilon} : L^2_{\mathbb{R}} \to L^2_{\mathbb{R}}\) with spectral radius \(r_{\varepsilon} < 1\) (in the same way as for the discrete counterpart) and such that \(\mathcal{B}^{(\varepsilon)}_{\lambda,\epsilon}[\mathscr{K}] \subset \mathscr{K}\).

As per \protect\hyperlink{lem:lap-symm}{Lemma \ref{lem:lap-symm}}, \(B^{(\varepsilon)}_{\lambda,\epsilon,N}\) and \(B^{(\varepsilon)}_{\lambda,\epsilon}\) are symmetrized via conjugation by \(\sqrt{p_{\lambda,\epsilon,N}}\) and \(\sqrt{p_{\lambda,\epsilon}}\), respectively, hence \(\mathcal{B}^{(\varepsilon)}_{\lambda,\epsilon,N}\) and \(\mathcal{B}^{(\varepsilon)}_{\lambda,\epsilon}\) are as well, so it follows from spectral theory that the symmetrizations of each of the latter operators have spectral radius less than one. Therefore, the spectral radius of \(\mathcal{B}^{(\varepsilon)}_{\lambda,\epsilon,N} + \mathcal{B}^{(0)}_{\lambda,\epsilon,N}\) and \(\mathcal{B}^{(\varepsilon)}_{\lambda,\epsilon} + \mathcal{B}^{(0)}_{\lambda,\epsilon}\) is less than two. We may write,
\begin{equation}\begin{aligned}
B^{(\varepsilon,0)}_{\lambda,\epsilon,N} := B_{\lambda,\epsilon,N}^{(\varepsilon)} + B_{\lambda,\epsilon,N}^{(0)} = 2(1 + \varepsilon)^{\frac{1}{2}}\left( I - \frac{1}{2}(\mathcal{B}^{(\varepsilon)}_{\lambda,\epsilon,N} + \mathcal{B}^{(0)}_{\lambda,\epsilon,N}) \right),  \\
B^{(\varepsilon,0)}_{\lambda,\epsilon} := B_{\lambda,\epsilon}^{(\varepsilon)} + B_{\lambda,\epsilon}^{(0)} = 2(1 + \varepsilon)^{\frac{1}{2}}\left( I - \frac{1}{2}(\mathcal{B}^{(\varepsilon)}_{\lambda,\epsilon} + \mathcal{B}^{(0)}_{\lambda,\epsilon}) \right),
\end{aligned}  \nonumber  \end{equation}
that is, each of these operators are in the form \(t(I - B)\) with \(t > 0\) and \(B\) having a spectral radius less than one and either having non-negative entries in the matrix case or preserving the non-negative cone \(\mathscr{K}\) in the infinite dimensional case. This means that the first sum is an M-matrix and the second is an M-operator. Moreover, whenever \(\varepsilon > 0\), both operators are non-singular. In the non-singular M-operator case, \citep[Theorem 1]{marek1995mopsqrt} tells that \((B^{(\varepsilon,0)}_{\lambda,\epsilon})^{-1}\) preserves \(\mathscr{K}\). We may express \((B_{\lambda,\epsilon}^{(\varepsilon,0)})^{-1}\) as the spectral application of the function \(f(z) := (\sqrt{1 + \varepsilon - z} + \sqrt{1 - z})^{-1}\) to \(z = A_{\lambda,\epsilon}\). Then, the decomposition \(\eqref{eq:derived-fun-calc-expansion}\) applies since \(f\) is analytic on the unit disc and has absolutely convergent Taylor series on \([-1, 1]\). Using this, we see that \((B_{\lambda,\epsilon}^{(\varepsilon, 0)})^{-1}(x,y)\) is smooth for \(x \neq y\) and due to the preservation of \(\mathscr{K}\) we also have that given any \((x, y) \in \mathcal{M}^2\), for all \(\epsilon_0 > 0\),
\begin{equation}\begin{aligned}
A_{\epsilon_0}[(B^{(\varepsilon,0)}_{\lambda,\epsilon})^{-1}(\cdot,y)](x) \geq 0.
\end{aligned}  \nonumber  \end{equation}
Combining these properties, we have by continuity off of the diagonal that that on letting \(\epsilon_0 \to 0\), we find that for all \(x \neq y\), \((B_{\lambda,\epsilon}^{(\varepsilon,0)})^{-1}(x,y) \geq 0\). Thus by \(\eqref{eq:derived-fun-calc-expansion}\), \((B_{\lambda,\epsilon}^{(\varepsilon,0)})^{-1} - f(0) I\) has a smooth kernel that is non-negative off of the diagonal, hence by continuity the kernel is also non-negative on the diagonal. Moreover, by \protect\hyperlink{lem:lap-symm}{Lemma \ref{lem:lap-symm}}, \(((B_{\lambda,\epsilon}^{(\varepsilon,0)})^{-1} - f(0)I)[1] = \varepsilon^{-\frac{1}{2}} - f(0) > 0\), so for any \(u \in C^{\infty}\) and \(x \in \mathcal{M}\),
\begin{align*}
|(B_{\lambda,\epsilon}^{(\varepsilon,0)})^{-1}[u](x)|
    &\leq f(0)|u(x)| + |((B_{\lambda,\epsilon}^{(\varepsilon,0)})^{-1} - f(0)I)[u](x)|  \\
    &\leq \varepsilon^{-\frac{1}{2}} ||u||_{\infty} .
\end{align*}
Further applications of \protect\hyperlink{lem:lap-symm}{Lemma \ref{lem:lap-symm}} give that \((B^{(\varepsilon,0)}_{\lambda,\epsilon,N})^{-1}\) also maps \(1 \mapsto \varepsilon^{-\frac{1}{2}}\) and due to commutativity under spectral mapping, we have,
\begin{equation}\begin{aligned}
B^{(\varepsilon)}_{\lambda,\epsilon,N} - B^{(0)}_{\lambda,\epsilon,N} = \varepsilon(B^{(\varepsilon,0)}_{\lambda,\epsilon,N})^{-1} ,    \\
B^{(\varepsilon)}_{\lambda,\epsilon} - B^{(0)}_{\lambda,\epsilon} = \varepsilon(B^{(\varepsilon,0)}_{\lambda,\epsilon})^{-1} .
\end{aligned}  \nonumber  \end{equation}
Therefore, in the event that \(|B^{(\varepsilon)}_{\lambda,\epsilon,N}[u](x) - B^{(\varepsilon)}_{\lambda,\epsilon}[u](x)| \leq ||u||_{\infty} \delta\), we get
\begin{align*}
|B_{\lambda,\epsilon,N}[u](x) - B_{\lambda,\epsilon}[u](x)|
    &\leq |B_{\lambda,\epsilon,N}[u](x) - B^{(\varepsilon)}_{\lambda,\epsilon,N}[u](x)| + |B^{(\varepsilon)}_{\lambda,\epsilon,N}[u](x) - B^{(\varepsilon)}_{\lambda,\epsilon}[u](x)|   \\
    &\quad\quad + |B_{\lambda,\epsilon}[u](x) - B^{(\varepsilon)}_{\lambda,\epsilon}[u](x)| \\
    &\leq \varepsilon \left( |(B^{(\varepsilon,0)})^{-1}_{\lambda,\epsilon,N}[u](x)| + |(B^{(\varepsilon,0)})^{-1}_{\lambda,\epsilon}[u](x)| \right) + ||u||_{\infty}\delta \\
    &\leq 2 \varepsilon^{\frac{1}{2}} ||u||_{\infty} + ||u||_{\infty} \delta ,
\end{align*}
whence upon setting \(\varepsilon = \delta^2\) and applying \protect\hyperlink{lem:sqrt-perturb-eps-conv}{Lemma \ref{lem:sqrt-perturb-eps-conv}} along with a union bound over the probability in part (2) of \protect\hyperlink{lem:lap-symm}{Lemma \ref{lem:lap-symm}} that \(A_{\lambda,\epsilon,N}\) gives a connected graph, we arrive at the probability bound in the statement of this Theorem.
\end{proof}

The consistency of the square root of the graph Laplacian and the quantum dynamics according to it are of interest in their own right, so we will carry out that theory below. Along with it, we will also state the consistency theorems for the dynamics given by the perturbed operator \(B_{\lambda,\epsilon,N}^{(\varepsilon)}\) because we have better probabilistic consistency bounds in terms of \(N\) for it than for its unperturbed counterpart, when \(\varepsilon\) is decoupled from \(\delta\). A second motivation for this is that for our applications, we see from \protect\hyperlink{thm:sym-cs-glap-psido}{Theorem \ref{thm:sym-cs-glap-psido}} that symbol information, and in particular the coordinates of geodesics propagations can only be recovered to order \(O(h)\) error. Furthermore, from \protect\hyperlink{symbol-of-a-graph-laplacian}{Section \ref{symbol-of-a-graph-laplacian}} we understand that the principal symbol of \(h^2 \GLap_{\lambda,\epsilon}\) and the propagation of coherent states that follow and form the key features for the applications, are unperturbed by order \(\varepsilon = O(h)\) perturbations to \(\GLap_{\lambda,\epsilon}\). Therefore, alongside our propagators we will also consider the perturbed ones, which we will write as,

\begin{notation}	\hypertarget{notation:perturb-glap}{\label{notation:perturb-glap}} Let \(\lambda \geq 0\) and \(\epsilon > 0\). Then, with \(\varepsilon \in [0, 1]\), we denote
\begin{gather*}
\GLap_{\lambda,\epsilon,N}^{(\varepsilon)} := (\GLap_{\lambda,\epsilon,N} + c_{2,0} \varepsilon/\epsilon)^{\frac{1}{2}}, \quad\quad \GLap_{\lambda,\epsilon}^{(\varepsilon)} := (\GLap_{\lambda,\epsilon} + c_{2,0} \varepsilon/\epsilon)^{\frac{1}{2}} , \\
\UepsN{\varepsilon}{t}:= \exp\left( -i t \GLap_{\lambda,\epsilon}^{(\varepsilon)} \right), \quad\quad \Ueps{\varepsilon}{t}:= \exp\left( -i t (\GLap_{\lambda,\epsilon} + c_{2,0} \varepsilon/\epsilon)^{\frac{1}{2}} \right) ,
\end{gather*}
which for \(\varepsilon > 0\) we call \(\GLap_{\lambda,\epsilon,N}^{(\varepsilon)}\) and \(\GLap_{\lambda,\epsilon}^{(\varepsilon)}\) the (\(\varepsilon\)-)\emph{perturbed} graph Laplacians and \(\UepsN{\varepsilon}{t}\) and \(\Ueps{\varepsilon}{t}\) the (\(\varepsilon\)-)\emph{perturbed propagators}, while for \(\varepsilon = 0\) we have the usual, \emph{unperturbed} graph Laplacians and propagators, respectively.

\end{notation}

Now we are ready to state the consistency of both, the unperturbed and perturbed propagators for short-times:

\begin{lemma}	\hypertarget{lem:prop-short-time-conv}{\label{lem:prop-short-time-conv}} Let \(\tau \in [-1/c_{2,0}, 1/c_{2,0}]\). Then, given \(\lambda \geq 0\), \(\epsilon \in (0, C_1]\), \(u \in L^{\infty}\) and \(\delta > 0\), we have,
\begin{equation}\begin{aligned}
\Pr[||U_{\lambda,\epsilon,N}^{\tau \sqrt{\epsilon}}[u] - U_{\lambda,\epsilon}^{\tau \sqrt{\epsilon}}[u]||_{\infty} > \delta] \leq \gamma_{\lambda,N,\upsilon}(\delta/(12e) ; u) + 2 \tilde{\gamma}_{\lambda,N}(\delta/(8e) ; u) + e^{-\frac{N \epsilon^{\frac{\mdim}{2}} \delta^2}{C_2}},
\end{aligned}  \nonumber  \end{equation}
wherein \(U_{\lambda,\epsilon,N}^{t} := e^{i t \sqrt{\GLap_{\lambda,\epsilon,N}}}\) and \(U^t_{\lambda,\epsilon} := e^{i t \sqrt{\GLap_{\lambda,\epsilon}}}\) for \(t \in \mathbb{R}\) and \(C_1, C_2 > 0\) are constants as in Lemmas  \protect\hyperlink{lem:lap-symm}{\ref{lem:lap-symm}} and  \protect\hyperlink{thm:sqrt-conv}{\ref{thm:sqrt-conv}}.

Moreover, for \(\varepsilon > 0\) we have for all \(\epsilon > 0\),
\begin{equation}\begin{aligned}
\Pr[||\UepsN{\varepsilon}{\tau \sqrt{\epsilon}}[u] - \Ueps{\varepsilon}{\tau \sqrt{\epsilon}}[u]||_{\infty} > \delta] \leq \gamma^*_{\lambda,N,\upsilon}(\delta/(4e), \varepsilon ; u) + 2 \tilde{\gamma}_{\lambda,N}(\delta/(8e) ; u) .
\end{aligned}  \nonumber  \end{equation}

\end{lemma}

\begin{proof}

We start with the application of \protect\hyperlink{thm:bdd-analytic-calc-conv}{Theorem \ref{thm:bdd-analytic-calc-conv}} to the pairs \((\mathcal{C}_N, \mathcal{C})\) and \((\mathcal{S}_N, \mathcal{S})\) of \(\eqref{eq:wave-ops}\) at time \(t = \tau \sqrt{\epsilon}\). More generally, Let \(f_{e,w} := f_e(z - w)\) and \(f_{o,w} := f_o(z - w)\), so we may write
\begin{equation}\begin{aligned}
f_{e,w}(z) = \sum_{j=0}^{\infty} \frac{(-1)^j}{(2j)!}\partial_z^{2j}[\cos(tc_{2,0}\epsilon^{-\frac{1}{2}} \cdot)]|_{z = 1 + w} z^j
\end{aligned}  \nonumber  \end{equation}
and similarly for \(f_{o,w}\). Hence, using the notation of that Theorem, we have that \(K_{1 + w} = 1\) for \(f_{e,w}\) and \(f_{o,w}\), so we can apply it to get that
\begin{equation} \label{eq:wave-ops-shorttime-consistency}
\Pr[(||\mathcal{C}_{w,N}[u] - \mathcal{C}_w[u]||_{\infty} > \delta) \wedge ||\mathcal{S}_{w,N}[u] - \mathcal{S}_w[u]||_{\infty} > \delta] \leq 2 \tilde{\gamma}_{\lambda,N}(\delta/(2 e);u) ,
\end{equation}
for \(\mathcal{C}_w := f_{e,w}(A_{\lambda,\epsilon})\), \(\mathcal{C}_{w,N} := f_{e,w}(A_{\lambda,\epsilon,N})\), \(\mathcal{S}_w := f_{o,w}(A_{\lambda,\epsilon})\) and \(\mathcal{S}_{w,N} := f_{o,w}(A_{\lambda,\epsilon,N})\).

Now suppose the event that \(||\mathcal{S}_N[u] - \mathcal{S}[u]||_{\infty} \leq \delta\) and \(||B_{\lambda,\epsilon,N}[\mathcal{S}[u]] - B_{\lambda,\epsilon}[\mathcal{S}[u]]||_{\infty} \leq \delta\) and note that through the spectral mapping properties we have \(\sin(t \sqrt{\GLap_{\lambda,\epsilon,N}}) = B_{\lambda,\epsilon,N} \mathcal{S}_N\) and \(\sin(t \sqrt{\GLap_{\lambda,\epsilon}}) = B_{\lambda,\epsilon} \mathcal{S}\). Further, recall from the proof of \protect\hyperlink{thm:sqrt-conv}{Theorem \ref{thm:sqrt-conv}} that \(B_{\lambda,\epsilon,N}\) is an M-matrix so that \(b_{\lambda,\epsilon,N} := I - B_{\lambda,\epsilon,N}\) has all entries non-negative and \(b_{\lambda,\epsilon,N}[1] = 1\). Then,
\begin{align*}
|| & \sin(t \sqrt{\GLap}_{\lambda,\epsilon,N})[u] - \sin(t \sqrt{\GLap}_{\lambda,\epsilon})[u] ||_{\infty}  \\
    &\leq ||b_{\lambda,\epsilon,N}[(\mathcal{S}_N - \mathcal{S})[u]]||_{\infty}
        + ||(\mathcal{S}_N - \mathcal{S})[u]||_{\infty}
        + ||(B_{\lambda,\epsilon,N} - B_{\lambda,\epsilon})[\mathcal{S}[u]]||_{\infty}  \\
    &\leq 3\delta .
\end{align*}
On including the event that \(|\mathcal{C}_N[u](x) - \mathcal{C}[u](x)| \leq \delta\), we have
\begin{equation}\begin{aligned}
||U^t_{\lambda,\epsilon,N}[u] - U^t_{\lambda,\epsilon}[u]||_{\infty} \leq 4\delta .
\end{aligned}  \nonumber  \end{equation}
By a Taylor series expansion we have that for any \(u \in L^{\infty}\),
\begin{equation}\begin{aligned}
|\mathcal{S}[u]| \leq \sum_{m=0}^{\infty} \frac{|A_{\lambda,\epsilon}^m[u]|}{m!} \leq e||u||_{\infty}.
\end{aligned}  \nonumber  \end{equation}
Therefore, taking a union bound over all events and applying \protect\hyperlink{thm:sqrt-conv}{Theorem \ref{thm:sqrt-conv}} we arrive at,
\begin{equation}\begin{aligned}
\Pr[||U_{\lambda,\epsilon,N}^t[u] - U_{\lambda,\epsilon}^t[u]||_{\infty} > \delta] \leq \gamma_{\lambda,N,\upsilon}(\delta/(12e) ; u) + 2 \tilde{\gamma}_{\lambda,N}(\delta/(8 e) ; u) + e^{-\frac{N \epsilon^{\frac{\mdim}{2}} \delta^2}{C_2}}.
\end{aligned}  \nonumber  \end{equation}
For the second part of the statement of the Theorem, \(\eqref{eq:wave-ops-shorttime-consistency}\) already gives the probabilistic consistency bound for \(\mathcal{C}_{\varepsilon,N}\) and \(\mathcal{S}_{\varepsilon,N}\). Then, the above statements go through essentially in the same way, modulo that in place of \(b_{\lambda,\epsilon,N}\) we use \(b_{\lambda,\epsilon,N}^{(\varepsilon)} := (1 + \varepsilon)^{\frac{1}{2}}I - B_{\lambda,\epsilon,N}^{(\varepsilon)}\), which also has all non-negative entries, since \(B_{\lambda,\epsilon,N}^{(\varepsilon)}\) is an M-matrix with spectral radius \((1 + \varepsilon)^{\frac{1}{2}}\). Then due to \(b_{\lambda,\epsilon,N}^{(\varepsilon)}[1] = \sqrt{1 + \varepsilon}\), we have under the analogous events,
\begin{align*}
&\left| \sin\left( t \sqrt{\GLap_{\lambda,\epsilon,N}^{(\varepsilon)}} \right)[u] - \sin\left( t \sqrt{\GLap_{\lambda,\epsilon}^{(\varepsilon)}} \right)[u] \right| \\
    &\quad \leq |b^{(\varepsilon)}_{\lambda,\epsilon,N}[(\mathcal{S}_{N,\varepsilon} - \mathcal{S}_{\varepsilon})[u]]|
  + (1 + \varepsilon)^{\frac{1}{2}}|(\mathcal{S}_{N,\varepsilon} - \mathcal{S}_{\varepsilon})[u]|   \\
    &\quad\quad + |(B_{\lambda,\epsilon,N}^{(\varepsilon)} - B_{\lambda,\epsilon}^{(\varepsilon)})[\mathcal{S}[u]]| \\
    &\quad \leq 4\delta .
\end{align*}
On invoking the probabilistic consistency bound from \protect\hyperlink{lem:sqrt-perturb-eps-conv}{Lemma \ref{lem:sqrt-perturb-eps-conv}} for the square root of the \(\varepsilon\)-perturbed graph Laplacian, we arrive at the bound in the second part of the Lemma.
\end{proof}

\begin{notation} We denote the pointwise convergence rate for the short-time solution to the half-wave equation by
\begin{equation}\begin{aligned}
\omega_{\lambda,N}(\delta,\upsilon;u) := \gamma_{\lambda,N,\upsilon}(\delta/(12e) ; u) + 2 \tilde{\gamma}_{\lambda,N}(\delta/(8e) ; u) + e^{-\frac{N \epsilon^{\frac{\mdim}{2}} \delta^2}{C_2}} ,
\end{aligned}  \nonumber  \end{equation}
while absorbing the condition \(\epsilon \in (0, C_1]\) in the notation; in the perturbed case we denote,
\begin{equation}\begin{aligned}
\omega^*_{\lambda,N}(\delta,\varepsilon,\upsilon;u) := \gamma^*_{\lambda,N,\upsilon}(\delta/(4e), \varepsilon ; u) + 2 \tilde{\gamma}_{\lambda,N}(\delta/(8e) ; u).
\end{aligned}  \nonumber  \end{equation}
The explicit dependence on \(\upsilon\) will often be suppressed in the usage of this notation and we will write simply, \(\omega_{\lambda,N}(\delta ; u)\) and \(\omega^*_{\lambda,N}(\delta, \varepsilon ; u)\) with the exact meaning understood to be as given above.

\end{notation}

\hypertarget{uniform-consistency-at-longer-times}{%
\subsection{Uniform consistency at longer times}\label{uniform-consistency-at-longer-times}}

We wish to extend the consistency between the discretized solution \(u_N^t := U_{\lambda,\epsilon,N}^t[u]\) and the continuum solution \(u^t := U_{\lambda,\epsilon}^t[u]\) to longer times. For this purpose, we see that \(u_N^t\) is defined on \(\mathcal{M}\) through the form of \emph{Nyström extension} via \(A_{\lambda,\epsilon,N}\) as discussed in the beginning of \protect\hyperlink{from-graphs-to-manifolds}{Section \ref{from-graphs-to-manifolds}}. The inequality \(\eqref{eq:fun-calc-derived-bound}\) gives for all \(u \in L^{\infty}\) and \(t \in \mathbb{R}\),
\begin{equation}    \label{eq:discrete-prop-sup-bound}
|U_{\lambda,\epsilon,N}^t[u](x)| \leq \left( 1 + c_{2,0} t \epsilon^{-\frac{1}{2}} \frac{||k_{\lambda,\epsilon,N}(x,\cdot)||_{N,\infty} ||p_{\lambda,\epsilon,N}||_{N,\infty}^{\frac{1}{2}}}{\min_{j \in [N]} p_{\lambda,\epsilon,N}^{\frac{3}{2}}(x_j)} \right) ||u||_{N,\infty}
\end{equation}
since in this case \(Df(z) = (\exp(-i \tau\sqrt{1 - z}) - \exp(-i \tau))/z\) for \(f = \exp(-i \tau\sqrt{1 - z})\) and \(\tau := c_{2,0} t \epsilon^{-\frac{1}{2}}\) and setting \(w := 1 - \sqrt{1 - z}\) then taking a Taylor series expansion at \(w = 0\) of \(|Df|(w)^2 = 2(1 - \cos(\tau w))/(w(2 - w))^2\) shows that its maximum is achieved near the inverse of the second order coefficient, \(w = \sqrt{48/(\tau^2 - 9)}/\tau\), which then applied to its first order Taylor approximation gives that \(|Df|\) gives a maximum less than \(\tau/2 + 1 < \tau\). Since \(\eqref{eq:derived-fun-calc-expansion}\) applies to the case with \(A_{\lambda,\epsilon}\) in place of \(A_{\lambda,\epsilon,N}\), we also have the bound for \(|f(A_{\lambda,\epsilon})[u](x)|\) given by the ultimate inequality in \(\eqref{eq:fun-calc-derived-bound}\) with \(||\cdot||_{q}\) in place of \(|| \cdot ||_{N,q}\). Hence,
\begin{equation}    \label{eq:cont-prop-sup-bound}
||U_{\lambda,\epsilon}^t[u]||_{\infty} \leq ||u||_{\infty} + c_{2,0} t \epsilon^{-\frac{1}{2}} ||k_{\lambda,\epsilon}||_{\infty}\frac{||p_{\lambda,\epsilon}||_{\infty}^{\frac{1}{2}}}{\inf p_{\lambda,\epsilon}^{\frac{3}{2}}} ||u||_2
\end{equation}
will be useful for bounding probabilities for pointwise consistency.

The bounds \(\eqref{eq:discrete-prop-sup-bound}\) and \(\eqref{eq:cont-prop-sup-bound}\) hold also for \(\UepsN{\varepsilon}{t}[u]\) and \(\Ueps{\varepsilon}{t}[u]\), respectively. This, along with \protect\hyperlink{lem:prop-short-time-conv}{Lemma \ref{lem:prop-short-time-conv}} is sufficient to establish consistency for the perturbed case through the same arguments as the \emph{unperturbed} \(\varepsilon = 0\) case. Indeed, in the following, we will use the short-time probabilistic consistency bounds given by \protect\hyperlink{lem:prop-short-time-conv}{Lemma \ref{lem:prop-short-time-conv}} as a black-box while extending to long-time (that is, \(O(1)\) in \(\epsilon\)) probabilistic consistency. We will work through the unperturbed case, mainly to keep the notation simple, but since the arguments essentially rely on \protect\hyperlink{lem:prop-short-time-conv}{Lemma \ref{lem:prop-short-time-conv}} in a simple way, they apply directly to the perturbed case with \(\omega_{\lambda,N}^*\) in place of \(\omega_{\lambda,N}\) as well, with the only thing to keep track of being that \(\delta\) (\emph{i.e.}, error) multipliers are decoupled from \(\varepsilon\), but the notation makes this clear.

We proceed by splitting time \(t \in \mathbb{R}\) into \emph{chunks} of length \(\sim \sqrt{\epsilon}\) and combining the short-time consistency \protect\hyperlink{lem:prop-short-time-conv}{Lemma \ref{lem:prop-short-time-conv}} with the semigroup structure of \(U_{\lambda,\epsilon}^t\). This leads to Bernstein-type exponential probabilistic bounds for long-time, uniform consistency, as in the following,

\begin{theorem}	\hypertarget{thm:halfwave-soln}{\label{thm:halfwave-soln}} Let \(\lambda \geq 0\), \(\epsilon > 0\), \(t \in \mathbb{R}\) and \(u \in L^{\infty}\) such that there exists \(K_u > 0\) so that \(||U_{\lambda,\epsilon}^s[u]||_{\infty} \leq K_u\) for all \(|s| \leq |t|\). Then with \(\kappa := \lceil |t| c_{2,0}/\sqrt{\epsilon} \rceil\), for all \(\delta > 0\),
\begin{equation}  \label{eq:thm-halfwave-soln-bound}
\begin{split}
\Pr[& ||(U_{\lambda,\epsilon,N}^t - U_{\lambda,\epsilon}^t)[u]||_{\infty} > \delta ]  \\
  &\leq \kappa \, \omega_{\lambda,N}(\delta \, (\epsilon^{\frac{n}{2}}/\kappa^2)/\tilde{C}_{p,\lambda} ; K_u) + \gamma_{\lambda,N}(\underline{K}_{p,\lambda} ; 1) + \gamma_{0,N}(\underline{K}_{p,\lambda} ; 1)   \\
  &=: \Omega_{\lambda,t,N}(\delta, \epsilon, K_u),
\end{split}
\end{equation}
wherein \(\underline{K}_{p,\lambda} := \min\{ \underline{C}_p, \underline{C}_{p,\lambda} \}/2\) and \(\tilde{C}_{p,\lambda} > 0\) is a constant depending only on \(\lambda\), \(\overline{C}_{p,\lambda}\), \(\underline{C}_{p,\lambda}\) and \(\underline{C}_{p}\).

In the perturbed case, let \(\tilde{\varepsilon} > 0\). Then, we have with a \(K_u > 0\) satisfying \(||\Ueps{\tilde{\varepsilon}}{s}[u]||_{\infty} \leq K_u\) for all \(|s| \leq |t|\),
\begin{equation} \label{eq:thm-halfwave-soln-perturb-op-bound}
\Pr[||(\UepsN{\tilde{\varepsilon}}{t} - \Ueps{\tilde{\varepsilon}}{t})[u]||_{\infty} > \delta] \leq \Omega^*_{\lambda,t,N}(\delta,\tilde{\varepsilon},\epsilon,K_u)
\end{equation}
with \(\Omega^*_{\lambda,t,N}\) defined in the same way as \(\Omega_{\lambda,t,N}\), except with the instance of \(\omega_{\lambda,N}(\cdot ; \cdot \cdot)\) replaced with \(\omega^*_{\lambda,N}(\cdot, \tilde{\varepsilon} ; \cdot \cdot)\).

\end{theorem}

\begin{proof}

We set \(\kappa := \lceil |t| c_{2,0}/\sqrt{\epsilon} \rceil\) and \(\tau := t/\kappa\) so that \(|\tau| \leq \sqrt{\epsilon}/c_{2,0}\). Let \(\varepsilon, \delta_0 > 0\) and let us begin in the following events:
\begin{alignat*}{4}
& \mathcal{A}_{\lambda'}(\varepsilon) &:&
    \quad ||p_{\lambda',\epsilon,N} - p_{\lambda',\epsilon}||_{\infty} \leq \varepsilon
    & & \quad (\forall)\lambda' \in \{ 0,\lambda \} \\
& \mathcal{A}(\delta_0) &:&
    \quad ||(U_{\lambda,\epsilon,N}^{\tau} - U_{\lambda,\epsilon}^{\tau}) U_{\lambda,\epsilon}^{m \tau}[u]||_{\infty} \leq \delta_0
    & & \quad (\forall) 0 \leq m \leq \kappa - 1 .
\end{alignat*}
Then we have due to \(\eqref{eq:discrete-prop-sup-bound}\),
\begin{equation}\begin{aligned}
||U_{\lambda,\epsilon,N}^{m\tau}[u]||_{\infty} \leq (1 + m C(\epsilon,\varepsilon)) ||u||_{N,\infty}
\end{aligned}  \nonumber  \end{equation}
with
\begin{equation}\begin{aligned}
C(\epsilon,\varepsilon) := \epsilon^{-\frac{\mdim}{2}} \frac{||k||_{\infty}}{(\underline{C}_p - \varepsilon)^{2\lambda}}\frac{(\overline{C}_{p,\lambda} + \varepsilon)^{\frac{1}{2}}}{(\underline{C}_{p,\lambda} - \varepsilon)^{\frac{3}{2}}} .
\end{aligned}  \nonumber  \end{equation}
Define for each \(m \in \mathbb{N}\), \(v_m := U_{\lambda,\epsilon}^{m\tau}[u]\) so that in the current event we have for all \(\kappa \geq 2\) and \(0 \leq m \leq \kappa - 2\),
\begin{align*}
||(U_{\lambda,\epsilon,N}^{2\tau} - U_{\lambda,\epsilon}^{2\tau})[v_m]||_{\infty} &\leq ||U_{\lambda,\epsilon,N}^{\tau}(U_{\lambda,\epsilon,N}^{\tau} - U_{\lambda,\epsilon}^{\tau})[v_m]||_{\infty}    \\
    &\quad\quad + ||(U_{\lambda,\epsilon,N}^{\tau} - U_{\lambda,\epsilon}^{\tau}) [v_{m+1}]||_{\infty}  \\
    &\leq (2 + C(\epsilon,\varepsilon))\delta_0 .
\end{align*}
Supposing that for all \(2 \leq \kappa' \leq \kappa - 1\), each \(s \in [\kappa']\) and \(0 \leq m \leq \kappa' - s\),
\begin{align*}
||(U_{\lambda,\epsilon,N}^{s\tau} - U_{\lambda,\epsilon}^{s\tau})[v_m]||_{\infty} &\leq s\left( \frac{s-1}{2}C(\epsilon,\varepsilon) + 1 \right) \delta_0
\end{align*}
leads to,
\begin{align*}
||(U_{\lambda,\epsilon,N}^{\kappa\tau} - U_{\lambda,\epsilon}^{\kappa\tau})[u]||_{\infty} &\leq ||U_{\lambda,\epsilon,N}^{(\kappa - 1)\tau}(U_{\lambda,\epsilon,N}^{\tau} - U_{\lambda,\epsilon}^{\tau})[u]||_{\infty}  \\
    &\quad\quad + ||(U_{\lambda,\epsilon,N}^{(\kappa - 1)\tau} - U_{\lambda,\epsilon}^{(\kappa - 1)\tau}) [v_1]||_{\infty}  \\
    &\leq \kappa \left( \frac{\kappa - 1}{2} C(\epsilon,\varepsilon) + 1 \right) \delta_0 .
\end{align*}
Therefore, having satisfied the induction hypothesis, we have that for each \(\kappa \in \mathbb{N}\), if for all \(0 \leq m \leq \kappa - 1\), \(||(U_{\lambda,\epsilon,N}^{\tau} - U_{\lambda,\epsilon}^{\tau})U_{\lambda,\epsilon}^{m \tau}[u]||_{\infty} \leq \delta_0\) then
\begin{equation}\begin{aligned}
||(U_{\lambda,\epsilon,N}^{\kappa \tau} - U_{\lambda,\epsilon}^{\kappa \tau})[u]||_{\infty} \leq \kappa^2 C(\epsilon,\varepsilon) \delta_0 .
\end{aligned}  \nonumber  \end{equation}
Now set \(\varepsilon = \min \{\underline{C}_p, \underline{C}_{p,\lambda} \}/2\) and \(\tilde{C}_{p,\lambda} := \sqrt{12} (2^{\lambda}) \overline{C}_{p,\lambda}^{\frac{1}{2}}/(\underline{C}_{p,\lambda}^{\frac{3}{2}} \underline{C}_p^{2\lambda})\) so that \(C(\epsilon,\varepsilon) \leq \epsilon^{-\frac{n}{2}} \tilde{C}_{p,\lambda}\). Moreover, by \protect\hyperlink{lem:avgop-uniform}{Lemma \ref{lem:avgop-uniform}}, the event \(\mathcal{A}_{\lambda'}(\varepsilon)\) happens with probability at least \(1 - \gamma_{\lambda',N}(\varepsilon ; 1)\) for each \(\lambda' \in \{ 0, \lambda \}\). Then, upon setting \(\delta_0 = \delta \, (\epsilon^{\frac{n}{2}}/\kappa^2)/\tilde{C}_{p,\lambda}\) and taking a union bound with respect to \protect\hyperlink{lem:prop-short-time-conv}{Lemma \ref{lem:prop-short-time-conv}}, we find that the event \(\mathcal{A}(\delta_0)\) happens with probability at least \(1 - \kappa \omega_{\lambda,N}(\delta \, (\epsilon^{\frac{n}{2}}/\kappa^2)/\tilde{C}_{p,\lambda} ; K_u)\). A union bound for the occurrence of all assumed events then gives the statement of the Lemma.
\end{proof}

A generic bound to be used in place of \(K_u\) in the conditions satisfying \protect\hyperlink{thm:halfwave-soln}{Theorem \ref{thm:halfwave-soln}} is given by \(\eqref{eq:cont-prop-sup-bound}\). We record this as,

\begin{corollary}	\hypertarget{cor:halfwave-soln-general-bound}{\label{cor:halfwave-soln-general-bound}} Let \(\lambda \geq 0\), \(\epsilon > 0\), \(t \in \mathbb{R}\) and \(u \in L^{\infty}\). Then, there exists a constant \(C > 0\) depending only on \(k, p, \lambda\) and geometry of \(\mathcal{M}\) such that for all \(\delta > 0\),
\begin{align*}
\Pr[||(U_{\lambda,\epsilon,N}^t - U_{\lambda,\epsilon}^t)[u]||_{\infty} > \delta]
  &\leq \Omega_{\lambda,t,N}(\delta,\epsilon,C_{\epsilon,t,u}),
\end{align*}
wherein \(C_{\epsilon,t,u} := ||u||_{\infty} + c_{2,0} t \epsilon^{-\frac{n + 1}{2}} ||k||_{\infty} \overline{C}_{p,\lambda}^{\frac{1}{2}} /(\underline{C}_p^{2 \lambda} \underline{C}_{p,\lambda}^{\frac{3}{2}}) ||u||_2\) and \(\Omega_{\lambda,t,N}\) as defined in \protect\hyperlink{thm:halfwave-soln}{Theorem \ref{thm:halfwave-soln}}.

\end{corollary}

The case for the application to coherent states is facilitated by \protect\hyperlink{prop:glap-prop-cs-localized}{Proposition \ref{prop:glap-prop-cs-localized}}, since it provides a supremum bound of the propagated coherent state for \(|t|\) below the injectivity radius at the point where the state is initially localized. In fact, it gives also the localization properties that allow us to place the geodesic in an \(O(\sqrt{h})\) radius ball about the maximum of the modulus of the propagated state. Therefore, we can readily approximate the geodesic flow using the \(\max\) observable; in doing so, we account for the fact that while the \emph{un-normalized} coherent state \(\tilde{\psi}_h := \psi_h/C_h(x_0,\xi_0)\) is readily interpolatable through the graph structure, its normalization is a matter of probabilistic \(L^2\) convergence and not necessary to handle the \(\max\) observable. For this reason, we use the un-normalized coherent state at present, to establish the consistency for this observation of the geodesic flow, in

\begin{proposition}	\hypertarget{prop:max-observable-flow-consistency}{\label{prop:max-observable-flow-consistency}} Let \(\lambda \geq 0\), \(\alpha \geq 1\), \(h \in (0, h_0]\) for \(h_0\) given by \protect\hyperlink{prop:glap-prop-cs-localized}{Proposition \ref{prop:glap-prop-cs-localized}} and \((x_0, \xi_0) \in T^*\mathcal{M}\). Then, for \(\tilde{\psi}_h := e^{-\frac{i}{h} \bar{\phi}(x_0, \xi_0; \cdot)}\) an \emph{un-normalized} coherent state localized at \((x_0, \xi_0)\) and \(|t| \leq \operatorname{inj}(x_0)\), we have with \(\epsilon := h^{2 + \alpha}\), for all \(\delta > 0\),
\begin{equation}\begin{aligned}
\Pr[|| \, |U_{\lambda,\epsilon,N}^t[\tilde{\psi}_h]|^2 - |U_{\lambda,\epsilon}^t[\tilde{\psi}_h]|^2 \,||_{\infty} > \delta] \leq \Omega_{\lambda,t,N}(\delta/(2K_{\psi} + 1 + \delta), \epsilon, K_{\psi})
\end{aligned}  \nonumber  \end{equation}
with a constant \(K_{\psi} > 0\) that satisfies \(||U_{\lambda,\epsilon}^t[\tilde{\psi}_h]||_{\infty} \leq K_{\psi}\) for all \(|t| \leq \operatorname{inj}(x_0)\).

Furthermore, denoting \(\hat{x}_{N,t} := \arg\max_{\mathcal{X}_N}|U_{\lambda,\epsilon,N}^t[\tilde{\psi}_h]|^2\) with \(\mathcal{X}_N := \{ x_1, \ldots, x_N \}\), there exist constants \(h_{\max}', C_{\max}' > 0\) such that for all \(h \in (0, h_{\max}']\),
\begin{equation} \label{eq:prop-max-observable-flow-consistency-bound}
\Pr[d_g(\hat{x}_{N,t}, x_t) > C_{\max}' h^{\frac{1}{2}}] \leq \Omega_{\lambda,t,N}(C h^{\frac{\mdim}{2}}, \epsilon, K_{\psi}) + e^{-2 N C' h^{\mdim}} ,
\end{equation}
with constants \(C, C' > 0\). In the perturbed case, we define \(\hat{x}^{(\varepsilon)}_{N,t} := \arg\max_{\mathcal{X}_N}|\UepsN{\varepsilon}{t}[\tilde{\psi}_h]|^2\) and set \(\varepsilon \in O(h^{1 + \alpha})\) to get with new constants \(h'_{\max,}, C'_{\max}, C, C' > 0\), and \(K^{(\varepsilon)}_{\psi} := \sup_{\{|t| \leq \operatorname{inj}(x_0) \}} \sup_{h \in (0, h'_{\max}]}||\Ueps{\varepsilon}{t}[\tilde{\psi}_h]||_{\infty}\),
\begin{equation} \label{eq:prop-max-observable-perturb-glap-flow-consistency-bound}
\Pr[d_g(\hat{x}^{(\varepsilon)}_{N,t}, x_t) > C'_{\max} h^{\frac{1}{2}}] \leq \Omega^*_{\lambda,t,N}(C h^{\frac{\mdim}{2}}, \varepsilon, \epsilon, K^{(\varepsilon)}_{\psi}) + e^{-2 N C' h^{\mdim}} .
\end{equation}

\end{proposition}

\begin{proof}

We may apply \protect\hyperlink{prop:glap-prop-cs-localized}{Proposition \ref{prop:glap-prop-cs-localized}} (we give the arguments for the unperturbed case, but they follow in the same way for the perturbed case), so in effect we have a constant \(K_{\psi} > 0\) such that \(||U_{\lambda,\epsilon}^t[\tilde{\psi}_h]||_{\infty} \leq K_{\psi}\) for all \(|t| < \operatorname{inj}(x_0)\) and \(h \in [0, h_0)\). Therefore, for all \(\delta_0 > 0\),
\begin{equation} \label{eq:halfwave-soln-consistency-cs}
\Pr[||U_{\lambda,\epsilon,N}^t[\tilde{\psi}_h] - U_{\lambda,\epsilon}^t[\tilde{\psi}_h]||_{\infty} > \delta_0] \leq \Omega_{\lambda,t,N}(\delta_0,\epsilon,K_{\psi}).
\end{equation}
Now assume the event that \(||U_{\lambda,\epsilon,N}^t[\tilde{\psi}_h] - U_{\lambda,\epsilon}^t[\psi_h]||_{\infty} \leq \delta/(2 K_{\psi} + 1 + \delta)\). Then,
\begin{align*}
| &\, |U_{\lambda,\epsilon,N}^t[\tilde{\psi}_h]|^2 - |U_{\lambda,\epsilon}^t[\tilde{\psi}_h]|^2 \,| \\
    &= | \, |U_{\lambda,\epsilon,N}^t[\tilde{\psi}_h]| - |U_{\lambda,\epsilon}^t[\tilde{\psi}_h]| \,| (|U_{\lambda,\epsilon,N}^t[\tilde{\psi}_h]| + |U_{\lambda,\epsilon}^t[\tilde{\psi}_h]|)   \\
    &\leq \frac{\delta}{2 K_{\psi} + 1 + \delta} \left( 2 K_{\psi} + \frac{\delta}{2 K_{\psi} + 1 + \delta} \right)
        \leq    \delta \frac{2K_{\psi} + 1}{2 K_{\psi} + 1 + \delta} \leq \delta .
\end{align*}
This establishes the first part of the statement of the Proposition.

To see the second part, start with noting that by \protect\hyperlink{prop:glap-prop-cs-localized}{Proposition \ref{prop:glap-prop-cs-localized}}, we also have a constant \(C_{\psi} > 0\) such that \(| \, |U_{\lambda,\epsilon}^t[\tilde{\psi}_h]|^2 - |\psi_h^t|^2/C_h(x_0, \xi_0)^2 \, | \leq C_{\psi} h^{\frac{\mdim}{2} + 1}\). Thus, \(| \, |U_{\lambda,\epsilon,N}^t[\tilde{\psi}_h]|^2 - |\psi_h^t|^2/C_h(x_0, \xi_0)^2 \, | \leq \delta + C_{\psi} h^{\frac{\mdim}{2} + 1}\). As per the proof of \protect\hyperlink{lem:prop-coherent-state-localized}{Lemma \ref{lem:prop-coherent-state-localized}} affecting that Proposition, there are constants \(h_{\max}, C_0, C_1, c_0, c_1 > 0\) such that for all \(h \in [0, h_{\max})\), given \(p_0 \in \overline{B}_0 := \overline{B}(x_t, g ; (C_0 h)^{\frac{1}{2}})\) and \(p_1 \in B_1^c := \mathcal{M} \setminus B(x_t, g ; (C_1 h)^{\frac{1}{2}})\), we have \(|\psi_h^t|^2(p_0) > c_0 > c_1 > |\psi_h^t|^2(p_1)\). Hence,
\begin{equation}\begin{aligned}
|U_{\lambda,\epsilon,N}^t[\tilde{\psi}_h]|^2(p_0) > c_0/C_h(x_0,\xi_0)^2 - \delta - C_{\psi} h^{\frac{\mdim}{2} + 1} ,  \\
|U_{\lambda,\epsilon,N}^t[\tilde{\psi}_h]|^2(p_1) < c_1/C_h(x_0,\xi_0)^2 + \delta + C_{\psi} h^{\frac{\mdim}{2} + 1} .
\end{aligned}  \nonumber  \end{equation}
Now let \(c > 0\) and \(\tilde{C}_{\psi} > 0\) such that \(C_h(x_0, \xi_0)^2 \leq \tilde{C}_{\psi} h^{-\frac{\mdim}{2}}\) and set \(\delta = c h^{\frac{\mdim}{2}}\). Then, \(C_h(x_0, \xi_0)^2(\delta + C_{\psi} h^{\frac{\mdim}{2} + 1}) \leq c \tilde{C}_{\psi} + \tilde{C}_{\psi} C_{\psi} h\). So with \(c = (c_0 - c_1)/(2\tilde{C}_{\psi})\), there is \(0 < h_0 \leq h_{\max}\) such that for all \(h \in (0, h_0]\), we have \(|U_{\lambda,\epsilon,N}^t[\tilde{\psi}_h]|^2(p_0) > |U_{\lambda,\epsilon,N}^t[\tilde{\psi}_h]|^2(p_1)\) and since these constants are independent of \(p_0, p_1\), this inequality holds uniformly with respect to their respective regions, so that \(\inf _{\overline{B}_0}|U_{\lambda,\epsilon,N}^t[\tilde{\psi}_h]|^2 > \sup_{B_1^c} |U_{\lambda,\epsilon,N}^t[\tilde{\psi}_h]|^2\). Therefore, to establish that \(d_g(\hat{x}_{N,t}, x_t) \leq (C_1 h)^{\frac{1}{2}}\), we need only find one point in \(\chi_N \cap \overline{B}_0\): let \(P_0\) be the probability that a random vector \(x\) lies in \(B_0\), which with the aid of \protect\hyperlink{lem:covering-number}{Lemma \ref{lem:covering-number}} is bounded below by,
\begin{equation}\begin{aligned}
P_0 = \int_{B_0} p(y) ~ d\nu_g(y) \geq (\inf p) \operatorname{vol}(B_0) \geq C_{\mathcal{M}} (\inf p) (C_0 h)^{\frac{\mdim}{2}} =: \underline{P}_0 .
\end{aligned}  \nonumber  \end{equation}
Then a \emph{sampler} \(Y_0 : \mathcal{M} \to \{ 0, 1 \}\) that maps \(Y_0[B_0] = 1\) and \(Y_0[B_0^c] = 0\) defines \(N\) \emph{i.i.d.} random variables governed by a Bernoulli process with probability \(P_0\), when applied to the random vectors \(x_1, \ldots, x_N\). This process has mean \(N P_0 \geq N \underline{P}_0\), so a Chernoff bound yields,
\begin{equation}\begin{aligned}
\Pr[\# (\chi_N \cap B_j) \geq 1] \geq 1 - e^{-2 N \underline{P}_0^2} ,
\end{aligned}  \nonumber  \end{equation}
hence upon taking a union bound with \(\eqref{eq:halfwave-soln-consistency-cs}\) using \(\delta_0 = c h^{\frac{\mdim}{2}}/(2 K_{\psi} + 1 + ch^{\frac{\mdim}{2}})\) and letting \(h \in (0, \min\{ h_0, c^{-\frac{2}{n}} \}]\), we have the second part of the statement of the Proposition.
\end{proof}

\hypertarget{interlude-l2-consistency}{%
\subsection{\texorpdfstring{Interlude: \(L^2\) consistency}{Interlude: L\^{}2 consistency}}\label{interlude-l2-consistency}}

We discuss now an intermediary step to recovering the geodesic flow through the graph structure using the mean position observables. These observables, as discussed in \protect\hyperlink{observing-geodesics}{Section \ref{observing-geodesics}}, can be expressed as \(L^2\) inner products, so to implement them we need a way to go from inner products on graphs to inner products on the manifold. The treatment in \citep[\(\S 2.3.5\)-\(6\)]{hein2005geometrical} shows that if a complex vector space over the graph structure is endowed with an inner product given by \(\langle \cdot, \cdot \rangle/N\) or \(\langle \cdot , p_{\lambda,\epsilon,N} \cdot \rangle/N\), then on the continuum side, these converge as \(N \to \infty\) and \(\epsilon \to 0\) to the inner producs \(\langle \cdot, \cdot \rangle_p\) and \(\langle \cdot, c_2^{1 - 2\lambda} p^{2 - 2\lambda} \cdot \rangle\), respectively. Hence, among these only the case \(\lambda = 1\) converges to \(\langle \cdot , \cdot \rangle_{L^2(\mathcal{M}, d\nu_g)}\), while by \protect\hyperlink{thm:sym-cs-glap-psido}{Theorem \ref{thm:sym-cs-glap-psido}}, for all other weights, the (classically) propagated continuum weight, namely \(c_2^{1 - 2\lambda} p^{2 - 2\lambda} \circ \Gamma^t\) will show up as a multiplier to top order in the expectation of a propagated quantum observable against a coherent state. Since this also means that the propagated weight must show up by itself when taking the norm of the propagated coherent state, it suggests that in order to recover the flow of symbols along geodesics through the coherent states, we must normalize their propagations. It so happens that we may do this with \emph{any} choice of the weights, so we proceed with the simplest choice that is sufficient for our needs: the corresponding picture on the continuum side is given by,

\begin{lemma}	\hypertarget{lem:prop-cs-mean-consistency}{\label{lem:prop-cs-mean-consistency}} Let \(\lambda \geq 0\), \(h \in (0, 1]\), \(\alpha \geq 1\) and \(\psi_h\) be a coherent state localized at \((x_0, \xi_0) \in T^*\mathcal{M}\). Then, for all \(|t| \leq \operatorname{inj}(x_0)\) and \(\varepsilon \in O(h^{1 + \alpha})\), with \(\epsilon := h^{2 + \alpha}\), we have,
\begin{equation}\begin{aligned}
||\Ueps{\varepsilon}{t}[\psi_h]||_p^2 = c_0^{2\lambda - 1} p_{\lambda,\epsilon}(x_0) p(x_t)^{2\lambda} + O(h) ,
\end{aligned}  \nonumber  \end{equation}
wherein \(x_t := \pi_{\mathcal{M}} \Gamma^t(x_0, \xi_0)\) and if \(v \in C^{\infty}\), then
\begin{equation}\begin{aligned}
\langle \CSeps{\varepsilon}{t} | (\Ueps{\varepsilon}{t})^* v \, \Ueps{\varepsilon}{t}|\CSeps{\varepsilon}{t} \rangle_p = v(x_t) + O(h)
\end{aligned}  \nonumber  \end{equation}
with \(\CSeps{\varepsilon}{t} := \psi_h/||\Ueps{\varepsilon}{t}[\psi_h]||_p\).

\end{lemma}

\begin{proof}

This is a straightforward application of \href{glaps-to-geodesic-flows\#thm:sym-cs-glap-psido}{Theorem}, which goes through for \(\Ueps{\varepsilon}{t}\) because \(h^2 \GLap_{\lambda,\epsilon}^{(\varepsilon)} = h^2 \GLap_{\lambda,\epsilon} + h^2 c_{2,0} \varepsilon/\epsilon I\) and \(h^2 \varepsilon/\epsilon = O(h)\) so \(\operatorname{Sym}[h^2 \GLap_{\lambda,\epsilon}^{(\varepsilon)}] = \operatorname{Sym}[h^2 \GLap_{\lambda,\epsilon}]\). Combined with \(\eqref{eq:glap-prop-adjoint-identity}\) and that the quantization of \(v p\) gives the multiplication operator by \(v p\), we have,
\begin{align*}
\langle \psi_h | (\Ueps{\varepsilon}{t})^* v p \, \Ueps{\varepsilon}{t} | \psi_h \rangle
    &= \langle \psi_h | p_{\lambda,\epsilon} (\Ueps{\varepsilon}{})^{-t} v \frac{p}{p_{\lambda,\epsilon}} \, \Ueps{\varepsilon}{t} | \psi_h \rangle \\
    &= \langle \psi_h | p_{\lambda,\epsilon} \operatorname{Op}_h\left( v \frac{p}{p_{\lambda,\epsilon}} \circ \Gamma^t \right) | \psi_h \rangle + O(h)  \\
    &= p_{\lambda,\epsilon}(x_0) \, v(x_t) \frac{p}{p_{\lambda,\epsilon}}(x_t) + O(h)   \\
    &= c_0^{2\lambda - 1} p_{\lambda,\epsilon}(x_0) p(x_t)^{2\lambda} v(x_t) + O(h),
\end{align*}
wherein the ultimate equality follows from \(\eqref{lem:taylor-expand-deg-func}\). Thus,
\begin{equation}\begin{aligned}
\langle \psi_{h,t,\varepsilon} | (\Ueps{\varepsilon}{t})^* v p \, \Ueps{\varepsilon}{t} | \psi_{h,t,\varepsilon} \rangle = v(x_t) + O(h).
\end{aligned}  \nonumber  \end{equation}
\end{proof}

We turn now to the consistency between the Hilbert space on the graph, given by the inner product \(\langle \cdot , \cdot \rangle_N := \langle \cdot , \cdot \rangle/N\) and the space \(L^2(\mathcal{M}, p \, d\nu_g)\). The primary step is to establish the error in
\begin{gather*}
|\langle \CSepsN{\varepsilon}{t} | (\UepsN{\varepsilon}{t})^* \, v \, \UepsN{\varepsilon}{t} | \CSepsN{\varepsilon}{t} \rangle_N - \langle \CSeps{\varepsilon}{t} | (\Ueps{\varepsilon}{t})^* \, v \, \Ueps{\varepsilon}{t} | \CSeps{\varepsilon}{t} \rangle_p| ,   \\
\CSepsN{\varepsilon}{t} := \psi_h/||\UepsN{\varepsilon}{t}[\psi_h]||_N ,
\end{gather*}
whereafter we can employ \protect\hyperlink{lem:prop-cs-mean-consistency}{Lemma \ref{lem:prop-cs-mean-consistency}} to give the consistency of flows of position-space observables. In anticipation of specifying to coherent states, we state the \(L^2\) inner product consistency bounds in terms that respect a class of \emph{local} functions. The point of this is that when we apply a Bernstein bound for the consistency of the inner product between two functions, at least one of which is localized to roughly an \(O(h^{\frac{\varsigma}{n}})\) radius ball with respect to \(h \in (0, h_0]\) for some \(h_0 > 0\), then we can use the dependence of the bound on the variance to get a better convergence rate by a factor of roughly \(h^{\varsigma}\) in the exponential. This trick was used in \citep{hein2005geometrical} to give a quadratic improvement of the rate of convergence in the bias term (\emph{viz.}, the density parameter \(\epsilon\)) for the graph Laplacian and we simply state it here in a somewhat more general and convenient form. More precisely, we have,

\begin{lemma}	\hypertarget{lem:l2-consistency}{\label{lem:l2-consistency}} Let \(\varsigma \in \mathbb{R}\) and \(0 < h_0 \leq 1\) be fixed and let \(u_{h,\varsigma} \in L^{\infty}\) be a family of functions over \(h \in (0, h_0]\) satisfying: there are constants \(\overline{K}_u, K_{u,2} > 0\) so that uniformly in \(h\), \(h^{\varsigma} ||u_{h,\varsigma}||_{\infty} \leq \overline{K}_u\) and \(||u_{h,\varsigma}||_{L^2}^2 \leq K_{u,2}\). Then, given any \(v \in L^{\infty}\), for all \(\delta > 0\), we have
\begin{align*}
\Pr[|\langle u_{h,\varsigma}, v \rangle_N - \langle u_{h,\varsigma}, v \rangle_p| > \delta]
  &\leq 4 \exp\left( -\frac{N h^{\varsigma} \delta^2}{2 ||v||_{\infty}(h^{\varsigma} K_{u,2} ||v||_{\infty} ||p||_{\infty} + \overline{K}_u \delta/3)} \right)    \\
  &=: \rho_N(\delta,h^{\varsigma}, \overline{K}_u, K_{u,2}, ||v||_{\infty}).
\end{align*}
In case that \(v = \tilde{u}_{h,\varsigma}\) is another such family, then
\begin{align*}
\Pr[|\langle u_{h,\varsigma}, \tilde{u}_{h,\varsigma} \rangle_N - \langle u_{h,\varsigma}, \tilde{u}_{h,\varsigma} \rangle_p| > \delta]
  &\leq 4 \exp\left( -\frac{N h^{2\varsigma} \delta^2}{2 \overline{K}_{\tilde{u}}(K_{u,2} \overline{K}_{\tilde{u}} ||p||_{\infty} + \overline{K}_u \delta/3)} \right) \\
  &= \rho_{N,2}(\delta,h^{2\varsigma}, \overline{K}_u, K_{u,2}, \overline{K}_{\tilde{u}}).
\end{align*}

\end{lemma}

\begin{proof}

We first compute the maximum and variance in order to apply Bernstein's inequality:
\begin{align*}
\max_{1 \leq j \leq N} & |u_{h,\varsigma}(x_j)\overline{v(x_j)}| \leq h^{-\varsigma} \overline{K}_u ||v||_{\infty} ,    \\
\operatorname{Var}(u_{h,\varsigma} \bar{v})
    &= \int_{\mathcal{M}} |u_{h,\varsigma} \bar{v}|^2 p ~ d\nu_g
    \leq ||v||_{\infty}^2 ||p||_{\infty} \int_{\mathcal{M}} |u_{h,\varsigma}|^2 ~ d\nu_g    \\
    &\leq K_{u,2} ||v||_{\infty}^2 ||p||_{\infty} .
\end{align*}
The probabilistic bounds in the statement of the Lemma now follows from Bernstein's inequality.
\end{proof}

\begin{remark} The coefficient \(K_{u,2}\) can be elaborated into more precise constants regarding the geometry of small balls in \(\mathcal{M}\), if the localization of \(u_{h,\varsigma}\) were explicitly utilised. The simpler form is sufficient for the essential purposes of the following applications.

\end{remark}

From here onwards, we will specify to the application to coherent states, simply to have explicit bounds and keep the discussion concrete. An immediate application is that we can recover the \(L^2(\mathcal{M}, p \, d\nu_g)\) norm of \(U_{\lambda,\epsilon}^t[\psi_h]\):

\begin{lemma}	\hypertarget{lem:prop-cs-norm-consistency}{\label{lem:prop-cs-norm-consistency}} Let \(\lambda \geq 0\), \(\epsilon, h \in (0, 1]\), \(t, \varepsilon \in \mathbb{R}\) and \((x_0, \xi_0) \in T^*\mathcal{M}\). Then, for \(\tilde{\psi}_h\) an \emph{un-normalized} coherent state localized about \((x_0, \xi_0)\), there are constants \(\overline{K}_{\psi}, K_{\psi,2} > 0\) depending on \(\psi_h\) and \(O(1)\) in \(h\) such that for all \(\delta > 0\),
\begin{align*}
\Pr&[|\, ||\UepsN{\varepsilon}{t}[\tilde{\psi}_h]||_N^2 - ||\Ueps{\varepsilon}{t}[\tilde{\psi}_h]||_p^2 \, | > \delta]    \\
  &\leq \rho_{N,2}(\delta/2, 1, \overline{K}_{\psi}, K_{\psi,2}, \overline{K}_{\psi}) \\
  &\quad\quad + 
      \begin{cases}
          \Omega_{\lambda,t,N}(\delta/(4 \overline{K}_{\psi} + 2 + 2\delta), \epsilon, \overline{K}_{\psi})   & \text{if } \varepsilon = 0,   \\
          \Omega^*_{\lambda,t,N}(\delta/(4 \overline{K}_{\psi} + 2 + 2\delta), \varepsilon,  \epsilon, \overline{K}_{\psi})   & \text{if } \varepsilon > 0
    \end{cases}   \\
  &=: \Xi_{\lambda,N}(\delta, \psi_h) .
\end{align*}

\end{lemma}

\begin{proof}

Assume the events \(\mathcal{A}_1\) and \(\mathcal{A}_2\) that \(|| \, |\Ueps{\varepsilon}{t}[\tilde{\psi}_h]|^2 - |\UepsN{\varepsilon}{t}[\tilde{\psi}_h]|^2 \, ||_{\infty} \leq \delta/2\) and \(| \, ||\Ueps{\varepsilon}{t}[\tilde{\psi}_h]||_N^2 - ||\UepsN{\varepsilon}{t}[\tilde{\psi}_h]||_N^2 \,| \leq \delta/2\), respectively. Then,
\begin{align*}
|&\, ||\UepsN{\varepsilon}{t}[\tilde{\psi}_h]||_N^2 - ||\Ueps{\varepsilon}{t}[\tilde{\psi}_h]||_p^2 \, |    \\
    &\leq |\, ||\UepsN{\varepsilon}{t}[\tilde{\psi}_h]||_N^2 - ||\Ueps{\varepsilon}{t}[\tilde{\psi}_h]||_N^2 \, | + |\, ||\Ueps{\varepsilon}{t}[\tilde{\psi}_h]||_N^2 - ||\Ueps{\varepsilon}{t}[\tilde{\psi}_h]||_p^2 \, |  \\
    &\leq \delta .
\end{align*}
Now, we use \protect\hyperlink{lem:l2-consistency}{Lemma \ref{lem:l2-consistency}} to establish a lower bound on the probability that \(\mathcal{A}_2\) occurs. Since \(\Ueps{\varepsilon}{t}\) is similar to a unitary operator by the diagonal \(p_{\lambda,\epsilon}^{\frac{1}{2}}\), we have that
\begin{equation} \label{eq:cs-prop-l2-bound}
||\Ueps{\varepsilon}{t}[\tilde{\psi}_h]||_p^2
    \leq \frac{||p_{\lambda,\epsilon}||_{\infty}}{\inf p_{\lambda,\epsilon}} ||\tilde{\psi}_h||_p^2
    \leq h^{\frac{\mdim}{2}} \overline{C}_{\psi,2} ||p||_{\infty} \frac{\overline{C}_{p,\lambda}}{\underline{C}_{p,\lambda}}
\end{equation}
for a constant \(\overline{C}_{\psi,2} > 0\). Likewise, there is a constant \(\overline{K}_{\psi} > 0\) such that \(||\Ueps{\varepsilon}{t}[\tilde{\psi}_h]||_{\infty} \leq \overline{K}_{\psi}\). Therefore, by \protect\hyperlink{lem:l2-consistency}{Lemma \ref{lem:l2-consistency}},
\begin{equation}\begin{aligned}
\Pr[\mathcal{A}_2] \geq 1 - \rho_{N,2}(\delta/2, 1, \overline{K}_{\psi}, K_{\psi,2}, \overline{K}_{\psi})
\end{aligned}  \nonumber  \end{equation}
with \(K_{\psi,2} := \overline{C}_{\psi,2} \overline{C}_{p,\lambda}/\underline{C}_{p,\lambda}\). As for the event \(\mathcal{A}_1\), \protect\hyperlink{prop:max-observable-flow-consistency}{Proposition \ref{prop:max-observable-flow-consistency}} gives a lower bound for the probability that this occurs; then, a union bound gives the probability as in the statement of the Lemma.
\end{proof}

The consistency of propagations given in \protect\hyperlink{thm:halfwave-soln}{Theorem \ref{thm:halfwave-soln}} applies readily to \(\psi_h\), as can be seen from a slight modification to account for the norm in the proof of \protect\hyperlink{prop:max-observable-flow-consistency}{Proposition \ref{prop:max-observable-flow-consistency}}. However, in practice, it is only reasonable to expect that we can construct \(\CSepsN{\varepsilon}{t} := \psi_h/||\UepsN{\varepsilon}{t}[\psi_h]||_N\). By combining \protect\hyperlink{thm:halfwave-soln}{Theorem \ref{thm:halfwave-soln}} and \protect\hyperlink{lem:prop-cs-norm-consistency}{Lemma \ref{lem:prop-cs-norm-consistency}}, we have the consistency between \(\CSepsN{\varepsilon}{t}\) and \(\CSeps{\varepsilon}{t}\). Now we use this to see the probabilistic consistency between the graph and manifold versions of the inner product of (the modulus squared of) a propagated, normalized coherent state with an arbitrary smooth function, in:

\begin{lemma}	\hypertarget{lem:prop-cs-funcexpect-consistency}{\label{lem:prop-cs-funcexpect-consistency}} Let \(\lambda \geq 0\), \(h_0 \in (0, 1]\) be given by \protect\hyperlink{prop:glap-prop-cs-localized}{Proposition \ref{prop:glap-prop-cs-localized}}, \(\alpha \geq 1\), and \((x_0, \xi_0) \in T^*\mathcal{M}\). Then, for \(\psi_h\) a coherent state localized about \((x_0, \xi_0)\), \(|t| \leq \operatorname{inj}(x_0)\) and given \(u \in C^{\infty}\), there are constants \(C_{\psi,\lambda,p,u}, C'_{\psi,\lambda,p,u} > 0\) and \(\overline{K}_{\psi}, K_{\psi,2}, \underline{\tilde{K}}_{\psi,2}, \tilde{K}_{\psi,2}, \overline{\tilde{K}}_{\psi} > 0\) depending on \(\psi_h\) and independent of \(h\) with \(C_{\psi,\lambda,p,u}, C'_{\psi,\lambda,p,u} > 0\) additionally depending on \(u\) such that for all \(h \in (0, h_0]\) and \(\delta \in (0, 4||u||_{\infty}]\),
\begin{align*}
\Pr&[| \langle |U_{\lambda,\epsilon,N}^{t}[\psi_{h,t,N}]|^2 , u \rangle_N - \langle |U_{\lambda,\epsilon}^{t}[\psi_{h,t}]|^2 , u \rangle_p | > \delta]    \\
  &\quad\quad \leq \Omega_{\lambda,t,N}(C'_{\psi,\lambda,p,u} h^{\frac{\mdim}{2}} \delta, \epsilon, \overline{\tilde{K}}_{\psi}) + \rho_N(\delta,h^{-\frac{\mdim}{4}}, \overline{K}_{\psi}^{\frac{1}{2}}, K_{\psi,2}, h^{-\frac{\mdim}{4}} \overline{K}_{\psi}^{\frac{1}{2}} ||u||_{\infty})  \\
  &\quad\quad\quad\quad + \rho_{N,2}(h^{\frac{\mdim}{2}} \tilde{\underline{K}}_{\psi,2}, 1, \overline{\tilde{K}}_{\psi}, h^{\frac{\mdim}{2}} \tilde{K}_{\psi,2}, \overline{\tilde{K}}_{\psi}) + \Xi_{\lambda,N}(C_{\psi,\lambda,p,u} h^{\frac{\mdim}{2}} \delta, \psi_h)  \\
  &\quad\quad =: \tilde{\Omega}_{\lambda,t,N}(\delta,\epsilon,\psi_h,u),
\end{align*}
wherein
\begin{equation}\begin{aligned}
\psi_{h,t,N} := \psi_h/||U_{\lambda,\epsilon,N}^t[\psi_h]||_N, \quad\quad \psi_{h,t} := \psi_h/||U_{\lambda,\epsilon}^t[\psi_h]||_p .
\end{aligned}  \nonumber  \end{equation}
This holds as well in the pertrubed case, wherein we have for \(\varepsilon > 0\),
\begin{align*}
\Pr&[| \langle |\UepsN{\varepsilon}{t}[\CSepsN{\varepsilon}{t}]|^2 , u \rangle_N - \langle |\Ueps{\varepsilon}{t}[\CSeps{\varepsilon}{t}]|^2 , u \rangle_p | > \delta]    \\
  &\quad\quad \leq \Omega^*_{\lambda,t,N}(C'_{\psi,\lambda,p,u} h^{\frac{\mdim}{2}} \delta,\varepsilon, \epsilon, \overline{\tilde{K}}_{\psi}) + \rho_N(\delta,h^{-\frac{\mdim}{4}}, \overline{K}_{\psi}^{\frac{1}{2}}, K_{\psi,2}, h^{-\frac{\mdim}{4}} \overline{K}_{\psi}^{\frac{1}{2}} ||u||_{\infty})    \\
  &\quad\quad\quad\quad + \rho_{N,2}(h^{\frac{\mdim}{2}} \tilde{\underline{K}}_{\psi,2}, 1, \overline{\tilde{K}}_{\psi}, h^{\frac{\mdim}{2}} \tilde{K}_{\psi,2}, \overline{\tilde{K}}_{\psi}) + \Xi_{\lambda,N}(C_{\psi,\lambda,p,u} h^{\frac{\mdim}{2}} \delta, \psi_h)  \\
  &\quad\quad =: \tilde{\Omega}^*_{\lambda,t,N}(\delta,\varepsilon,\epsilon,\psi_h,u)
\end{align*}
with new constants and
\begin{equation}\begin{aligned}
\CSepsN{\varepsilon}{t} := \psi_h/||\Ueps{\varepsilon}{t}[\psi_h]||_N, \quad\quad \CSeps{\varepsilon}{t} := \psi_h/||\Ueps{\varepsilon}{t}[\psi_h]||_p .
\end{aligned}  \nonumber  \end{equation}

\end{lemma}

\begin{proof}

We write,
\begin{align*}
| \langle |U_{\lambda,\epsilon,N}^{t}[\psi_{h,t,N}]|^2 & , u \rangle_N - \langle |U_{\lambda,\epsilon}^{t}[\psi_{h,t}]|^2 , u \rangle_p |   \\
    &\quad \leq | \langle |U_{\lambda,\epsilon,N}^{t}[\psi_{h,t,N}]|^2 , u \rangle_N - \langle |U_{\lambda,\epsilon}^{t}[\psi_{h,t}]|^2 , u \rangle_N | \\
    &\quad\quad + | \langle |U_{\lambda,\epsilon}^{t}[\psi_{h,t}]|^2 , u \rangle_N - \langle |U_{\lambda,\epsilon}^{t}[\psi_{h,t}]|^2 , u \rangle_p |   \\
    &\quad =: I + II
\end{align*}
and recalling that \(\tilde{\psi}_h = e^{\frac{i}{h} \phi}\) is the \emph{un-normalized} coherent state,
\begin{align*}
I &= | \langle |U_{\lambda,\epsilon,N}^{t}[\psi_{h,t,N}]|^2 - |U_{\lambda,\epsilon}^{t}[\psi_{h,t}]|^2 , u \rangle_N |  \\
    &= \frac{| \langle c |U_{\lambda,\epsilon,N}^{t}[\tilde{\psi}_{h}]|^2 - (c + c_N - c) |U_{\lambda,\epsilon}^{t}[\tilde{\psi}_h]|^2 , u \rangle_N |}{c_N c}  \\
    &\leq \frac{c | \langle |U_{\lambda,\epsilon,N}^{t}[\tilde{\psi}_{h}]|^2 - |U_{\lambda,\epsilon}^{t}[\tilde{\psi}_h]|^2 , u \rangle_N | + |c_N - c| \, | \langle |U_{\lambda,\epsilon}^{t}[\tilde{\psi}_h]|^2 , u \rangle_N |}{c_N c}   \\
    &\leq ||\iota||_{\infty} \left( c_{N}^{-1} \, ||(U_{\lambda,\epsilon,N}^t - U_{\lambda,\epsilon}^t)[\tilde{\psi}_h]||_{\infty} \, ||(U_{\lambda,\epsilon,N}^t + U_{\lambda,\epsilon}^t)[\tilde{\psi}_h]||_{N} + \frac{|c_N - c|}{c_N c} ||U_{\lambda,\epsilon}^t [\tilde{\psi}_h]||_N^2 \right),    \\
    &\quad \quad c := ||U_{\lambda,\epsilon}^t[\tilde{\psi}_h]||_p^2, \quad c_N := ||U_{\lambda,\epsilon,N}^t[\tilde{\psi}_h]||_N^2 .
\end{align*}
Suppose the occurrence of the following events:
\begin{align*}
&\mathcal{A}_1 : & & ||(U_{\lambda,\epsilon,N}^t - U_{\lambda,\epsilon}^t)[\tilde{\psi}_h]||_{\infty} \leq \delta_1,    \\
&\mathcal{A}_2 : & & | \, c_N - c \, | \leq \delta_1 ,  \\
&\mathcal{A}_3 : & & | \, ||U_{\lambda,\epsilon}^t[\tilde{\psi}_h]||_N^2 -c \, | \leq c' ,  \\
&\mathcal{A}_4 : & & |\langle | U_{\lambda,\epsilon}^t[\psi_{h,t}] |^2, u \rangle_N - \langle | U_{\lambda,\epsilon}^t[\psi_{h,t}] |^2, u \rangle_p| \leq \delta/2
\end{align*}
for some \(0 < \delta_1 \leq c/2\) and \(0 < c' \leq c\). Then,
\begin{equation}\begin{aligned}
I \leq \delta_1 ||u||_{\infty} c^{-1} (4 + 6\sqrt{c}) .
\end{aligned}  \nonumber  \end{equation}
and \(II \leq \delta/2\). So letting \(\delta_1 = \delta c [2||u||_{\infty} (4 + 6 \sqrt{c})]^{-1}\) then gives for \(\delta \leq ||u||_{\infty}(4 + 6\sqrt{c}) \in 4 ||u||_{\infty} + O(h^{\frac{\mdim}{4}})\),
\begin{equation} \label{eq:mean-consistency-err-bound}
I + II \leq \delta .
\end{equation}
We wish to apply the foregoing probabilistic bounds from the Lemmas of this section and from the first part of \protect\hyperlink{prop:max-observable-flow-consistency}{Proposition \ref{prop:max-observable-flow-consistency}} to bound the probabilities of the events \(\mathcal{A}_1, \ldots, \mathcal{A}_4\). In light of \protect\hyperlink{lem:l2-consistency}{Lemma \ref{lem:l2-consistency}}, we will have probabilistic bounds on the events \(\mathcal{A}_3\) and \(\mathcal{A}_4\) if we have \(L^2\) and \(L^{\infty}\) bounds on \(U_{\lambda,\epsilon}^t[\tilde{\psi}_h]\) and \(U_{\lambda,\epsilon}^t[\psi_{h,t}]\). These are as follows: by \(\eqref{eq:cs-prop-l2-bound}\), \(c \leq h^{\frac{\mdim}{2}} \overline{C}_{\psi,2} ||p||_{\infty} \overline{C}_{p,\lambda}/\underline{C}_{p,\lambda}\). Likewise, by the proof of \protect\hyperlink{thm:sym-cs-psido}{Theorem \ref{thm:sym-cs-psido}} we have a constant \(\underline{C}_{\psi,2} > 0\) also \(O(1)\) in \(h\) such that \(C_h(x_0,\xi_0)^{-2} \geq h^{\frac{\mdim}{2}} \underline{C}_{\psi,2}\), so
\begin{equation}\begin{aligned}
c \geq \frac{\inf p_{\lambda,\epsilon}}{||p_{\lambda,\epsilon}||_{\infty}} ||\tilde{\psi}_h||_p^2
    \geq h^{\frac{\mdim}{2}} \underline{C}_{\psi,2} \underline{C}_p \frac{\underline{C}_{p,\lambda}}{\overline{C}_{p,\lambda}} .
\end{aligned}  \nonumber  \end{equation}
Since \(|U_{\lambda,\epsilon}^t[\psi_{h,t}]|^2 = |U_{\lambda,\epsilon}^t[\tilde{\psi}_h]|^2/c\) and there is a constant \(\overline{C}_{\psi} > 0\) such that \(||U_{\lambda,\epsilon}^t[\tilde{\psi}_h]||^2_{\infty} \leq \overline{C}_{\psi}\), we have,
\begin{equation}\begin{aligned}
||U_{\lambda,\epsilon}^t[\psi_{h,t}]||_{\infty}^2
    \leq ||U_{\lambda,\epsilon}^t[\tilde{\psi}_h]||_{\infty}^2/c
        \leq h^{-\frac{\mdim}{2}} \overline{C}_{\psi} \overline{C}_{p,\lambda}/(\underline{C}_{\psi,2} \underline{C}_{p,\lambda} \underline{C}_p) =: h^{-\frac{\mdim}{2}} \overline{K}_{\psi}
\end{aligned}  \nonumber  \end{equation}
and there is a constant \(K_{\psi,2} > 0\) such that
\begin{align*}
||U_{\lambda,\epsilon}^t[\psi_{h,t}]||_{L^2}^2 &= ||U_{\lambda,\epsilon}^t[\tilde{\psi}_{h}]||_{L^2}^2/c    \\
    &= \langle \tilde{\psi}_h | p \, p_{\lambda,\epsilon} U_{\lambda,\epsilon}^{-t} \frac{1}{p \, p_{\lambda,\epsilon}} U_{\lambda,\epsilon}^t | \tilde{\psi}_h \rangle/c   \\
    &\leq \overline{C}_{\psi,2}(1 + O(h))\overline{C}_{p,\lambda}/(\underline{C}_{\psi,2} \underline{C}_{p,\lambda} \underline{C}_p)    \\
    &\leq K_{\psi,2} ,
\end{align*}
wherein the second equality follows from the fact that the adjoint of \(U_{\lambda,\epsilon}^t\) over \(L^2(\mathcal{M}, d\nu_g)\) is given by \(p \circ (U_{\lambda,\epsilon}^t)^* \circ 1/p\) with \((U_{\lambda,\epsilon}^t)^*\) the adjoint on \(L^2(\mathcal{M}, p d\nu_g)\). With this, we find on writing
\begin{equation}\begin{aligned}
\langle |U_{\lambda,\epsilon}^t[\psi_{h,t}]|^2, u \rangle_N - \langle |U_{\lambda,\epsilon}^t[\psi_{h,t}]|^2, u \rangle_p = \langle |U_{\lambda,\epsilon}^t[\psi_{h,t}]|, v \rangle_N - \langle |U_{\lambda,\epsilon}^t[\psi_{h,t}]|, v \rangle_p , \\
v := |U_{\lambda,\epsilon}^t[\psi_{h,t}]| u
\end{aligned}  \nonumber  \end{equation}
and applying \protect\hyperlink{lem:l2-consistency}{Lemma \ref{lem:l2-consistency}} that the event \(\mathcal{A}_4\) happens with probability at least \(1 - \rho_N(\delta/2,h^{-\frac{\mdim}{4}}, \overline{K}_{\psi}^{\frac{1}{2}}, K_{\psi,2}, h^{-\frac{\mdim}{4}} \overline{K}_{\psi}^{\frac{1}{2}} ||u||_{\infty})\) and the event \(\mathcal{A}_3\) with probability at least \(1 - \rho_{N,2}(h^{\frac{\mdim}{2}} \tilde{\underline{K}}_{\psi,2}, 1, \overline{C}_{\psi}^{\frac{1}{2}}, h^{\frac{\mdim}{2}} \tilde{K}_{\psi,2}, \overline{C}_{\psi}^{\frac{1}{2}})\) with \(\tilde{\underline{K}}_{\psi,2} := \underline{C}_{\psi,2} \underline{C}_p \underline{C}_{p,\lambda}/\overline{C}_{p,\lambda}\) and \(\tilde{K}_{\psi,2} := \overline{C}_{\psi,2} ||p||_{\infty} \overline{C}_{p,\lambda} /\underline{C}_{p,\lambda}\).

As for event \(\mathcal{A}_2\), we may apply \protect\hyperlink{lem:prop-cs-norm-consistency}{Lemma \ref{lem:prop-cs-norm-consistency}} while noting that in that notation, \(C_{\psi} := \overline{C}_{\psi,2}\) and
\begin{align*}
\delta_1 &\geq \delta h^{\frac{\mdim}{2}} \frac{\tilde{\underline{K}}_{\psi,2}}{2||u||_{\infty}(4 + 6 \sqrt{c})}    \\
    &\geq \delta h^{\frac{\mdim}{2}} \frac{\tilde{\underline{K}}_{\psi,2}}{2||u||_{\infty}(4 + 6 \tilde{K}_{\psi,2}^{\frac{1}{2}})} =: C_{\psi,\lambda,p,u} h^{\frac{\mdim}{2}} \delta .
\end{align*}
Hence, we have that this event happens with probability at least \(1 - \Xi_{\lambda,N}(h^{\frac{\mdim}{2}} C_{\psi,\lambda,p,u} \delta, \psi_h)\). The proof of the first part of \protect\hyperlink{prop:max-observable-flow-consistency}{Proposition \ref{prop:max-observable-flow-consistency}} shows that the event \(\mathcal{A}_1\) happens with probability at least \(1 - \Omega_{\lambda,t,N}(\delta_1/(2\overline{C}_{\psi} + 1 + \delta_1), \epsilon,\overline{C}_{\psi})\). Thus, after taking into account the condition that \(\delta \leq 4 ||u||_{\infty}\), which suffices to achieve \(\eqref{eq:mean-consistency-err-bound}\), we have a constant \(C'_{\psi,\lambda,p,u} > 0\) such that
\begin{equation}\begin{aligned}
\Pr[\mathcal{A}_1] \geq 1 - \Omega_{\lambda,t,N}(C_{\psi,\lambda,p,u}' h^{\frac{\mdim}{2}} \delta, \epsilon, \overline{C}_{\psi}).
\end{aligned}  \nonumber  \end{equation}
Hence, the probabilistic bounds in the statement of the present Lemma follow from taking a union bound and noting that the analogous statements follow for the perturbed case.
\end{proof}

Combined with \protect\hyperlink{lem:prop-cs-mean-consistency}{Lemma \ref{lem:prop-cs-mean-consistency}}, this tells that we may observe the propagation of classical configuration-space observables along the geodesic flow of \(\mathcal{M}\) through their means with propagated coherent states on the graph approximating its structure:

\begin{theorem}	\hypertarget{thm:mean-prop-geoflow-consistency}{\label{thm:mean-prop-geoflow-consistency}} Let \(\lambda \geq 0\), \(\alpha \geq 1\), \((x_0, \xi_0) \in T^*\mathcal{M}\) and \(|t| \leq \operatorname{inj}(x_0)\). Then, for \(\psi_h\) a coherent state localized about \((x_0, \xi_0)\) and given \(u \in C^{\infty}\), there is a constant \(C > 0\) such that we have for all \(h \in (0, \min\{ h_0, 4||u||_{\infty} \}]\) with \(h_0 \in (0, 1]\) from \protect\hyperlink{prop:glap-prop-cs-localized}{Proposition \ref{prop:glap-prop-cs-localized}} and \(\epsilon := h^{2 + \alpha}\),
\begin{equation} \label{eq:thm-mean-prop-geoflow-consistency-bound}
\Pr[|\langle |U_{\lambda,\epsilon,N}^t[\psi_{h,t,N}]|^2, u \rangle_N - u(x_t)| > Ch] \leq \tilde{\Omega}_{\lambda,t,N}(h, \epsilon,\psi_h, u) .
\end{equation}
Letting \(\varepsilon \in O(h^{1 + \alpha})\), we also have
\begin{equation} \label{eq:thm-mean-prop-perturbed-glap-geoflow-consistency-bound}
\Pr[|\langle |\UepsN{\varepsilon}{t}[\psi_{h,t,N}]|^2, u \rangle_N - u(x_t)| > Ch] \leq \tilde{\Omega}^*_{\lambda,t,N}(h, \varepsilon, \epsilon,\psi_h, u) .
\end{equation}

\end{theorem}

\begin{proof}

Suppose we are in the event that \(| \langle |U_{\lambda,\epsilon,N}^{t}[\psi_{h,t,N}]|^2 , u \rangle_N - \langle |U_{\lambda,\epsilon}^{t}[\psi_{h,t}]|^2 , u \rangle_p | \leq h\). Then by \protect\hyperlink{lem:prop-cs-mean-consistency}{Lemma \ref{lem:prop-cs-mean-consistency}} there is a constant \(C' > 0\) such that,
\begin{align*}
| \langle |U_{\lambda,\epsilon,N}^{t}[\psi_{h,t,N}]|^2 , u \rangle_N - u(x_t)|
    &\leq | \langle |U_{\lambda,\epsilon,N}^{t}[\psi_{h,t,N}]|^2 , u \rangle_N - \langle |U_{\lambda,\epsilon}^{t}[\psi_{h,t}]|^2 , u \rangle_p|    \\
    &\quad\quad + | \langle |U_{\lambda,\epsilon}^{t}[\psi_{h,t}]|^2 , u \rangle_p - u(x_t)|    \\
    &\leq (1 + C')h .
\end{align*}
Now, \protect\hyperlink{lem:prop-cs-funcexpect-consistency}{Lemma \ref{lem:prop-cs-funcexpect-consistency}} gives a lower bound on the occurrence of this event, under the condition that \(h \in (0, \min\{ h_0, 4||u||_{\infty} \}]\), as \(1 - \tilde{\Omega}_{\lambda,t,N}(h,\epsilon,\psi_h,u)\).

The perturbed case follows along the same lines upon applying Lemmas  \protect\hyperlink{lem:prop-cs-mean-consistency}{\ref{lem:prop-cs-mean-consistency}} and  \protect\hyperlink{lem:prop-cs-funcexpect-consistency}{\ref{lem:prop-cs-funcexpect-consistency}} with \(\varepsilon \in O(h^{1 + \alpha})\).
\end{proof}

\hypertarget{observing-geodesics-on-graphs}{%
\subsection{Observing Geodesics on Graphs}\label{observing-geodesics-on-graphs}}

We are now in the position to implement \protect\hyperlink{prop:local-mean-geodesic-flow}{Proposition \ref{prop:local-mean-geodesic-flow}} on the graph holding an approximative structure of \(\mathcal{M}\) and to give its probabilistic rate of consistency: to summarize the procedure, we give the example of an,

\begin{algorithm}	\hypertarget{alg:extrinsic-mean-geo-flow}{\label{alg:extrinsic-mean-geo-flow}} Having constructed \(\GLap_{\lambda,\epsilon,N}\) and \(\tilde{\psi}_h\) localized at \((x_0, \xi_0)\) on \(\embedM_N\) with \(\epsilon = h^{2 + \alpha}\) and \(\alpha \in [1,2]\),

\begin{enumerate}
\def\labelenumi{\arabic{enumi}.}
\tightlist
\item
  construct \(U_{\lambda,\epsilon,N}^t := e^{-i t \GLap_{\lambda,\epsilon,N}^{\frac{1}{2}}}\) by spectral methods,
\item
  identify \(\hat{x}_{N,t} = \arg\max_{\mathcal{X}_N} |U_{\lambda,\epsilon,N}^t[\tilde{\psi}_h]|\),
\item
  compute \(||U_{\lambda,\epsilon,N}^t[\tilde{\psi}_h]||_N^2\) to form \(\psi_{h,t,N}\), set \(\overline{\varepsilon}_t \sim \sqrt{h}\) and compute \(\bar{\iota}^t_{N,j} := \langle |U_{\lambda,\epsilon,N}^t[\psi_{h,t,N}]|^2, \chi_t(||\iota(\cdot) - \iota(\hat{x}_{N,t})||_{\mathbb{R}^{\rdim}}) \iota_j \rangle_N\) for each \(j \in [\rdim]\) with a smooth function \(\chi_t : \mathbb{R} \to \mathbb{R}\) satisfying: \(\operatorname{supp} \chi_t \subset [-R_t, R_t]\) such that \(R_t > \overline{\varepsilon}_t\) and \(\chi_t \equiv 1\) on \([-\overline{\varepsilon}_t, \overline{\varepsilon}_t]\),
\item
  identify \(\bar{x}_{N,\iota,t} := \arg\min_{x \in \mathcal{X}_N} ||\iota(x) - \bar{\iota}^t_{N,j}||_{\mathbb{R}^{\rdim}}\).
\end{enumerate}

The algorithm outputs \(\bar{x}_{N,\iota,t} \in \embedM_N\).

\end{algorithm}

\begin{remark} We can simply take \(\chi_t \equiv 1\) when working with extrinsic coordinates, since they are globally defined. This can be preferable in situations where using the cut-off \(\chi_t\) increases uncertainty without significant cost benefit.

\end{remark}

Anticipating applications such as \protect\hyperlink{alg:extrinsic-mean-geo-flow}{Algorithm \ref{alg:extrinsic-mean-geo-flow}}, we make the following,

\begin{definition} Let \(\psi_h\) be a coherent state localized at \((x_0, \xi_0) \in T^*\mathcal{M}\) and consider its propagation with respect to \(U_{\lambda,\epsilon,N}^t\). The \emph{intrinsic sample maximizer} of the propagated coherent state is,
\begin{equation}\begin{aligned}
\hat{x}_{N,t} := \arg\max_{\mathcal{X}_N} |U_{\lambda,\epsilon,N}^t[\tilde{\psi}_h]|^2 .
\end{aligned}  \nonumber  \end{equation}

\begin{enumerate}
\def\labelenumi{\arabic{enumi}.}
\item
  \emph{Extrinsic case}. The \emph{extrinsic sample mean} of the propagated coherent state is \(\bar{x}_{N,\chi_{\iota}, t} \in \embedM\), defined with respect to a cut-off \(\chi_{\iota} \in C_c^{\infty}(\mathbb{R}^{\rdim})\), as the closest point on \(\embedM_N\) to
  \begin{align*}
  &&\bar{\iota}^t_N(x_0, \xi_0) &:= (\bar{\iota}_{N,1}^t(x_0, \xi_0), \ldots, \bar{\iota}_{N,\rdim}^t(x_0, \xi_0)),  \\
  \quad\quad \text{with} && \bar{\iota}_{N,j}^t(x_0, \xi_0) &:= \langle \, |U_{\lambda,\epsilon,N}^t[\psi_{h,t,N}]|^2 \,, (\chi_{\iota} \circ \iota) \, \iota_j \rangle_N.
  \end{align*}
  That is, \(\bar{x}_{N,\chi_{\iota},t} := \arg\min_{X \in \Lambda_N} ||X - \bar{\iota}^t_N(x_0, \xi_0)||_{\mathbb{R}^{\rdim}}\).
\item
  \emph{Local coordinates}. Let \(\mathscr{O}_t \subset \mathcal{M}\) be an open neighbourhood of \(\hat{x}_t\). Then, given a diffeomorphic coordinate mapping \(u : \mathscr{O}_t \to V_t \subset \mathbb{R}^{\mdim}\) and \(\chi \in C_c^{\infty}(\mathbb{R}^{\mdim})\) a cut-off with \(\operatorname{supp} \chi \subset V_t\), the \(u\)-\emph{sample mean with respect to} \(\chi\) of the propagated coherent state is \(\bar{x}_{N,u,\chi,t} \in V_t\) that is the closest point on \(\mathscr{V}_N := u[\mathcal{X}_N \cap \mathscr{O}_t]\) to
  \begin{align*}
  && \bar{u}^t_N(x_0, \xi_0) &:= (\bar{u}_{N,1}^t(x_0, \xi_0), \ldots, \bar{u}_{N,\mdim}^t(x_0, \xi_0)), \\
  \quad\quad\text{with} && \bar{u}_{N,j}^t(x_0, \xi_0) &:= \langle \,|U_{\lambda,\epsilon,N}^t[\psi_{h,t,N}]|^2 , (\chi \circ u) \, u_j \, \rangle_N .
  \end{align*}
  That is, \(\bar{x}_{N,u,\chi,t} := \arg\min_{X \in \mathscr{V}_N} ||X - \bar{u}_N^t(x_0, \xi_0)||_{\mathbb{R}^{\mdim}}\).
\end{enumerate}

\end{definition}

That \protect\hyperlink{alg:extrinsic-mean-geo-flow}{Algorithm \ref{alg:extrinsic-mean-geo-flow}} outputs, with high probability, a point within an \emph{intrinsic} distance of \(O(h)\) from \(x_t\) is an application of the following,

\begin{proposition}	\hypertarget{prop:mean-geodesic-recover-consistency}{\label{prop:mean-geodesic-recover-consistency}} Let \(\lambda \geq 0\), \(\alpha \geq 1\), \((x_0, \xi_0) \in T^*\mathcal{M}\) and \(|t| \leq \operatorname{inj}(x_0)\). Then, for \(\psi_h\) a coherent state localized about \((x_0, \xi_0)\), given any open neighbourhood \(\mathscr{O}_t \subset \mathcal{M}\) about \(\hat{x}_{N,t}\) with its diffeomorphic coordinate mapping \(u : \mathscr{O}_t \to V_t \subset \mathbb{R}^{\mdim}\), there are constants \(h_{u,\max}, C_{u,\max} > 0\) such that if \(h \in (0, h_{u,\max})\) and \(\overline{\mathscr{B}}_t := \overline{B}_{C_{u,\max} \sqrt{h}}(u(\hat{x}_{N,t}), ||\cdot||_{\mathbb{R}^{\mdim}}) \subset V_t\), then for any smooth cut-off \(\chi \in C_c^{\infty}(\mathbb{R}^{\mdim}, [0, 1])\) with \(\operatorname{supp} \chi \subset V_t\) that is \(\chi \equiv 1\) on \(\overline{\mathscr{B}}_t\), we have
\begin{align}
\begin{split} \label{eq:prob-mean-local-geoflow}
\Pr&[d_g(u^{-1}(\bar{x}_{N,u,\chi,t}), x_t) > C_u h]  \\
  &\quad\quad\leq \mdim \, \tilde{\Omega}_{\lambda,t,N}(h,\epsilon,\psi_h,K_{\chi,u}) + \Omega_{\lambda,t,N}(C h^{\frac{\mdim}{2}}, \epsilon, K_{\psi}) + e^{-2 N C'' h^{2 \mdim}} + e^{-2 N C' h^{\mdim}} ,
\end{split}
\end{align}
for \(C_u, C'' > 0\) constants, \(K_{\psi}, C' > 0\) constants as in \protect\hyperlink{prop:max-observable-flow-consistency}{Proposition \ref{prop:max-observable-flow-consistency}} and \(K_{\chi,u} := \max_{j \in [\mdim]} || 1 + (\chi \circ u) \, u_j ||_{\infty}\).

Likewise, there are constants \(h_{\iota,\max}, C_{\iota,\max} > 0\) such that for all \(h \in [0, h_{\iota,\max})\), given any cut-off \(\chi_{\iota} \in C_c^{\infty}(\mathbb{R}^{\rdim}, [0, 1])\) such that \(\chi_{\iota} \equiv 1\) on \(\overline{\mathscr{B}}_{\iota,t} := \overline{B}_{C_{\iota,\max}\sqrt{h}}(\iota(\hat{x}_{N,t}), ||\cdot||_{\rdim})\), we have
\begin{align}
\begin{split} \label{eq:prob-mean-extrinsic-geoflow}
\Pr&[d_g(\iota^{-1}(\bar{x}_{N,\chi_{\iota},t}), x_t) > C_\iota h]    \\
  &\quad\quad\leq \rdim \, \tilde{\Omega}_{\lambda,t,N}(h,\epsilon,\psi_h,K_{\chi,\iota}) + \Omega_{\lambda,t,N}(C h^{\frac{\mdim}{2}}, \epsilon, K_{\psi}) + e^{-2 N C' h^{2 \mdim}} + e^{-2 N C' h^{\mdim}} ,
\end{split}
\end{align}
for \(C_{\iota} > 0\) a constant and \(K_{\chi,u} := \max_{j \in [\mdim]} || 1 + (\chi_{\iota} \circ \iota) \, \iota_j ||_{\infty}\).

\end{proposition}

\begin{proof}

We have by \protect\hyperlink{prop:max-observable-flow-consistency}{Proposition \ref{prop:max-observable-flow-consistency}} constants \(h_{\max}', C_{\max}' > 0\) such that for all \(h \in (0, h_{\max}')\), \(d_g(\hat{x}_{N,t}, x_t) \leq C_{\max}' \sqrt{h}\) with probability at least \(1 - \Omega_{\lambda,t,N}(C h^{\frac{\mdim}{2}}, \epsilon, K_{\psi}) + e^{-2 N C' h^{\mdim}}\) for some \(C, C' > 0\) and \(K_{\psi} := \sup_{\{|t| \leq \operatorname{inj}(x_0) \}} \sup_{h \in (0, h'_{\max}]}||U_{\lambda,\epsilon}^t[\tilde{\psi}_h]||_{\infty}\). Assume this event and call it \(\mathcal{A}_0\).

\begin{enumerate}
\def\labelenumi{\arabic{enumi}.}
\tightlist
\item
  Following the arguments in the proof of the first part of \protect\hyperlink{prop:local-mean-geodesic-flow}{Proposition \ref{prop:local-mean-geodesic-flow}}, we have constants \(h_{u,0} \in (0, h'_{\max}]\) and \(C_{u,\max} > 0\) such that for all \(h \in (0, h_{u,0})\), \(||u(\hat{x}_{N,t}) - u(x_t)||_{\mathbb{R}^{\mdim}} \leq C_{u,\max} \sqrt{h}\).
\item
  By \protect\hyperlink{thm:mean-prop-geoflow-consistency}{Theorem \ref{thm:mean-prop-geoflow-consistency}} we have for each \(j \in [\mdim]\), a constant \(C_j > 0\) such that for all \(h \in (0, \min\{ h_0, 4||u_{\chi,j}||_{\infty} \})\), with \(u_{\chi,j} := (\chi \circ u) u\), \(|\bar{u}^t_{N,j} - u_{\chi,j}(x_t)| \leq C_j h\) with probability at least \(1 - \tilde{\Omega}_{\lambda,t,N}(h, \epsilon, \psi_h, u_{\chi,j})\). Assuming also these events, say \(\mathcal{A}_1, \ldots, \mathcal{A}_{\mdim}\) and that for each \(h \in (0, \min\{h_{u,0}, 4(||u_{\chi,1}||_{\infty}, \ldots, ||u_{\chi,{\mdim}}||_{\infty})\})\), \(\overline{\mathscr{B}} := B(u(\hat{x}_{N,t}), ||\cdot||_{\mathbb{R}^{\mdim}} ; C_{u,\max} \sqrt{h}) \subset V_t\) and \(\chi \equiv 1\) on \(\overline{\mathscr{B}}\), we have that \(|\bar{u}_{N,j}^t - u_j(x_t)| \leq C_j h\) and hence, \(||\bar{u}^t_{N,\chi}(x_0, \xi_0) - u(x_t)||_{\mathbb{R}^{\mdim}} \leq (C_1^2 + \cdots + C_{\mdim}^2)^{\frac{1}{2}} h =: C_u' h\).
\item
  Then, there are \(C_u, h_{u,1} \in (0, 1]\) such that for all \(h \in [0, h_{u,1}]\) and any \(x^*_{u,\chi,N} \in B_{C'_u h}(\bar{u}^t_{N,\chi}(x_0, \xi_0), || \cdot ||_{\mathbb{R}^{\mdim}})\), \(d_g(u^{-1}(x^*_{u,\chi,N}), x_t) \leq (C_u/2) h\). Although the interval for \(h\) depends on \(\chi\), we may translate the coordinate mapping \(u\) to work with \(\tilde{u} := I_{\mdim} + u : \mathscr{O}_t \to \tilde{V}_t\) and \(\tilde{\chi} := \chi(I_{\mdim} - \cdot)\) so that \(\tilde{V}_t \subset \mathbb{R}^{\mdim} \setminus B_1(0, ||\cdot||_{\mathbb{R}^{\mdim}})\) and hence, \(||u_{\chi,1}||_{\infty} \geq 1\) in the events \(\mathcal{A}_j\). With this, we achieve the same result without affecting \(h_{u,0}, h_{u,1}\), \(C_{u,\max}\) and \(C_u\): \emph{viz.}, \(||u(\hat{x}_{N,t}) - u(x_t)||_{\mathbb{R}^{\mdim}} = ||\tilde{u}(\hat{x}_{N,t}) - \tilde{u}(x_t)||_{\mathbb{R}^{\mdim}} \leq C_{u,\max} \sqrt{h}\) and \(d_g(u^{-1}[\bar{u}_{N,\chi}^t(x_0, \xi_0)], x_t) = d_g(\tilde{u}^{-1}[\bar{\tilde{u}}_{N,\tilde{\chi}}^t(x_0, \xi_0)], x_t) \leq (C_u/2) h\) and these are now valid over \(h \in (0, \min\{ h_{u,0}, h_{u,1} \})\).
\item
  \(\sloppy\) The probability that at least one sample from \(\{ x_1, \ldots, x_N \}\) belongs to \(B_{\frac{C_u}{2} h}(u^{-1}[\bar{u}_{N,\chi}(x_0, \xi_0)], g)\) can be lower bounded in the same way as in the proof of \protect\hyperlink{prop:max-observable-flow-consistency}{Proposition \ref{prop:max-observable-flow-consistency}}, \emph{viz.}, this is at least \(1 - e^{-2N C'' h^{2 \mdim}}\) with \(C'' > 0\) a constant. Thus, possibly after translations and adjusting the probability for each \(\mathcal{A}_j\) to at least \(1 - \tilde{\Omega}_{\lambda,t,N}(h, \epsilon, \psi_h, \tilde{u}_{\tilde{\chi},j})\), then taking a union bound and \(h_{u,\max} := \min\{ h_{u,0}, h_{u,1}\}\), we have the probabilistic bound in the first part of the statement of the Proposition.
\end{enumerate}

The second part of the Proposition follows the same line of arguments.

\begin{enumerate}
\def\labelenumi{\arabic{enumi}.}
\tightlist
\item
  By \protect\hyperlink{thm:mean-prop-geoflow-consistency}{Theorem \ref{thm:mean-prop-geoflow-consistency}} we have for each \(j \in [\rdim]\), a constant \(C_j > 0\) such that for all \(h \in (0, \min\{ h_0, 4||\iota_{\chi,j}||_{\infty} \})\), with \(\iota_{\chi,j} := (\chi_{\iota} \circ \iota) \iota\), \(|\bar{\iota}^t_{N,j} - \iota_{\chi,j}(x_t)| \leq C_j h\) with probability at least \(1 - \tilde{\Omega}_{\lambda,t,N}(h, \epsilon, \psi_h, \iota_{\chi,j})\).
\item
  Assuming these events, say \(\mathcal{A}_1, \ldots, \mathcal{A}_{\rdim}\) along with the event \(\mathcal{A}_0\), then due to the \protect\hyperlink{assumptions}{Assumptions} we have that if for each \(h \in (0, \min\{(\kappa/C_{\max}')^2, 4(||\iota_{\chi,1}||_{\infty}, \ldots, ||\iota_{\chi,{\rdim}}||_{\infty})\})\), \(\overline{\mathscr{B}} := B(\iota(\hat{x}_{N,t}), ||\cdot||_{\mathbb{R}^{\rdim}} ; C_{\max}' \sqrt{h}) \subset V_t\) and \(\chi_{\iota} \equiv 1\) on \(\overline{\mathscr{B}}\), then \(|\bar{\iota}_{N,j}^t - \iota_j(x_t)| \leq C_j h\) and hence, \(||\bar{\iota}^t_N(x_0,\xi_0) - \iota(x_t)||_{\mathbb{R}^{\rdim}} \leq (C_1^2 + \cdots + C_{\rdim}^2)^{\frac{1}{2}} h =: C'_{\iota} h\).
\item
  The \protect\hyperlink{assumptions}{Assumptions} further imply that if \(h < \kappa/C'_{\iota}\) and \(x_{\iota,N}^* \in B(\bar{\iota}_N^t(x_0,\xi_0), || \cdot ||_{\mathbb{R}^{\rdim}} ; C'_{\iota} h) \cap \embedM\), then \(d_g(\iota^{-1}(x_{\iota,N}^*), x_t) \leq 2 C'_{\iota} h\). As before, we have that with probability at least \(1 - e^{-2N C'' h^{2 \mdim}}\), there is \(j \in [N]\) such that \(d_g(x_j, \iota^{-1}(x^*_{\iota,N})) \leq h\), hence \(d_g(x_j, x_t) \leq (1 + 2 C'_{\iota}) h\).
\item
  Now again we have the joint interval of \(h\) depending on \(\chi\), but as in the previous part, we can translate the isometry and \(\chi_{\iota}\), so that upon taking \(h_{\iota,\max} := \min\{ h_{\max}', (\kappa/C_{\max}')^2, \kappa/C' \}\) and a union bound (with respect to the shifted embedding), the probabilistic bound in the second part of the statement of the Proposition follows.
\end{enumerate}

\end{proof}

\begin{remark} The preceding bounds hold also for the perturbed case with \(\varepsilon \in O(h^{1 + \alpha})\).

\end{remark}

The \protect\hyperlink{alg:extrinsic-mean-geo-flow}{Algorithm \ref{alg:extrinsic-mean-geo-flow}} outlined above is expanded in \citep{qml} and put to practice on model examples as well as real-world datasets.

\hypertarget{summary-of-convergence-rates}{%
\subsection{Summary of Convergence Rates}\label{summary-of-convergence-rates}}

The foregoing discussions involve various inter-dependent probabilistic bounds. We now summarize them by giving their \emph{unwrapped} dominant terms as closed-form Bernstein-type exponential bounds:

\begin{longtable}[]{@{}
  >{\raggedright\arraybackslash}p{(\columnwidth - 6\tabcolsep) * \real{0.2500}}
  >{\raggedright\arraybackslash}p{(\columnwidth - 6\tabcolsep) * \real{0.2500}}
  >{\raggedright\arraybackslash}p{(\columnwidth - 6\tabcolsep) * \real{0.2500}}
  >{\raggedright\arraybackslash}p{(\columnwidth - 6\tabcolsep) * \real{0.2500}}@{}}
\toprule()
\begin{minipage}[b]{\linewidth}\raggedright
Approx.
\end{minipage} & \begin{minipage}[b]{\linewidth}\raggedright
Unperturbed
\end{minipage} & \begin{minipage}[b]{\linewidth}\raggedright
Perturbed
\end{minipage} & \begin{minipage}[b]{\linewidth}\raggedright
Bound functions
\end{minipage} \\
\midrule()
\endhead
\(\GAve_N \sim_{\delta} \GAve\): \protect\hyperlink{lem:avgop-uniform}{Lemma \ref{lem:avgop-uniform}} & \(\eqref{eq:avgop-uniform-consistent-bound}\) \(\leq e^{-\Omega\left( \frac{N \epsilon^{\frac{\mdim}{2}} \delta^2}{C} \right)}\) & & \(\gamma_{\lambda,N}(\delta ; u)\) \\
\midrule \(B_N \sim_{\delta} B\): \protect\hyperlink{thm:sqrt-conv}{Theorem \ref{thm:sqrt-conv}}, \protect\hyperlink{lem:sqrt-perturb-eps-conv}{Lemma \ref{lem:sqrt-perturb-eps-conv}} & \(\eqref{eq:thm-sqrt-conv-bound} \leq\) \(e^{-\Omega\left( \frac{N \epsilon^{\frac{\mdim}{2}} \delta^{4(1 + \upsilon)}}{C} \right)}\) & \(\eqref{eq:lem-sqrt-perturb-eps-conv-bound} \leq\) \(e^{-\Omega\left( \frac{N \epsilon^{\frac{\mdim}{2}} \delta^2 \varepsilon^{(1 + 2\upsilon)}}{C} \right)}\) & \begin{minipage}[t]{\linewidth}\raggedright
\(\gamma_{\lambda,N,\upsilon}(\delta ; u) :=\) \(\gamma_{\lambda,N}(C_0 \delta^{2(1 + \upsilon(\beta))} ; u)\)\\
\(\gamma^*_{\lambda,N,\upsilon}(\delta, \varepsilon ; u) :=\) \(\gamma_{\lambda,N}(C_0 \delta \varepsilon^{(\frac{1}{2} + \upsilon(\beta))} ; u)\)\strut
\end{minipage} \\
\midrule \(U_N^t \sim_{\delta} U^t\): \protect\hyperlink{thm:halfwave-soln}{Theorem \ref{thm:halfwave-soln}} & \begin{minipage}[t]{\linewidth}\raggedright
\(\eqref{eq:thm-halfwave-soln-bound}\) \(\leq |t| \cdot\) \(e^{-\Omega \left( \frac{N \delta^4 \epsilon^{{\frac{5}{2}\mdim + 4}}}{C K_u^2 \, |t|^8} \right)}\)\\
with \(\delta > \frac{K_u^{\frac{1}{2}} \epsilon^{-(\frac{5}{8} \mdim + 1)}}{C_0 N^{\frac{1}{4}}}\), \(|t| \lesssim K_u^{\frac{1}{4}} \epsilon^{-\frac{\mdim}{16}}\)\strut
\end{minipage} & \begin{minipage}[t]{\linewidth}\raggedright
\(\eqref{eq:thm-halfwave-soln-perturb-op-bound}\) \(\leq |t| \cdot\) \(e^{-\Omega\left( \frac{N \delta^2 \epsilon^{\frac{3}{2}\mdim + 2} \varepsilon}{K_u^2 |t|^4} \right) \varepsilon^{\tilde{O}\left( \frac{|t|^4}{N^{\sigma}} \right)}}\)\\
with \(\delta > \frac{K_u \epsilon^{-(\frac{3}{4} \mdim + 1)}}{C_0 N^{\frac{ 1 - \sigma}{2}}}\), \(0 < \sigma < 1\)\strut
\end{minipage} & \begin{minipage}[t]{\linewidth}\raggedright
\(\Omega_{\lambda,t,N}(\delta,\epsilon,K_u)\)\\
\(\Omega^*_{\lambda,t,N}(\delta,\varepsilon,\epsilon,K_u)\)\strut
\end{minipage} \\
\midrule \(\mathbb{E}_t u \sim_h u(x_t)\): \protect\hyperlink{thm:mean-prop-geoflow-consistency}{Theorem \ref{thm:mean-prop-geoflow-consistency}} & \begin{minipage}[t]{\linewidth}\raggedright
\(\eqref{eq:thm-mean-prop-geoflow-consistency-bound}\) \(\leq e^{-\Omega(N h^{2(\mdim + 2)} \epsilon^{\frac{5}{2}\mdim + 4})}\)\\
with \(h \gtrsim N^{-\frac{1}{(2 + \alpha)(\frac{5}{2} \mdim + 4) + 2(\mdim+ 2)}}\), \(\epsilon := h^{2 + \alpha}\)\strut
\end{minipage} & \begin{minipage}[t]{\linewidth}\raggedright
\(\eqref{eq:thm-mean-prop-perturbed-glap-geoflow-consistency-bound}\) \(\leq e^{-\Omega(N h^{\mdim + 1 + \beta} \epsilon^{3(\frac{\mdim}{2} + 1)})}\)\\
with \(h \gtrsim N^{-\frac{1}{\mdim + 3(2 + \alpha)(\frac{\mdim}{2} + 1)}}\), \(\epsilon := h^{2 + \alpha}\)\strut
\end{minipage} & \begin{minipage}[t]{\linewidth}\raggedright
\(\tilde{\Omega}_{\lambda,t,N}(h,\epsilon,\psi_h, u)\)\\
\(\tilde{\Omega}^*_{\lambda,t,N}(h,\varepsilon,\epsilon,\psi_h,u)\) with \(\varepsilon \in \Theta(h^{1 + \alpha + \beta})\), \(\beta \geq 0\)\strut
\end{minipage} \\
\midrule \(\hat{x}_{N,t} \sim_{\sqrt{h}} x_t\): \protect\hyperlink{prop:max-observable-flow-consistency}{Proposition \ref{prop:max-observable-flow-consistency}} & \begin{minipage}[t]{\linewidth}\raggedright
\(\eqref{eq:prop-max-observable-flow-consistency-bound}\) \(\leq e^{-\Omega(N h^{2 \mdim} \epsilon^{\frac{5}{2}\mdim + 4})}\)\\
with \(h \gtrsim N^{-\frac{1}{\mdim(\frac{5\alpha}{2} + 7) + 4(2 + \alpha)}}\), \(\epsilon := h^{2 + \alpha}\)\strut
\end{minipage} & \begin{minipage}[t]{\linewidth}\raggedright
\(\eqref{eq:prop-max-observable-perturb-glap-flow-consistency-bound}\) \(\leq e^{-\Omega(N h^{\mdim - 1 + \beta} \epsilon^{3(\frac{\mdim}{2} + 1)})}\)\\
with \(h \gtrsim N^{-\frac{1}{\mdim + (2 + \alpha)(\frac{3}{2}\mdim + 2) + 1}}\), \(\epsilon := h^{2 + \alpha}\)\strut
\end{minipage} & \begin{minipage}[t]{\linewidth}\raggedright
\(\Omega_{\lambda,t,N}(h^{\frac{\mdim}{2}}, \epsilon, K_{\psi})\)\\
\(\Omega^*_{\lambda,t,N}(h^{\frac{\mdim}{2}}, \varepsilon, \epsilon, K_{\psi}^{(\varepsilon)})\) with \(\varepsilon \in \Theta(h^{1 + \alpha + \beta})\), \(\beta \geq 0\)\strut
\end{minipage} \\
\midrule \(\bar{x}_{N,t} \sim_h x_t\) & Same as \(\eqref{eq:thm-mean-prop-geoflow-consistency-bound}\) & Same as \(\eqref{eq:thm-mean-prop-perturbed-glap-geoflow-consistency-bound}\) & (See the previous two rows) \\
\bottomrule()
\end{longtable}

The first column of the table indicates the approximation being considered, with a symbolic short-hand and the location of the precise statement. Here, \(\GAve_N, \GAve\) refer to the averaging operators, \(B_N, B\) the square roots, or \(\varepsilon\)-perturbed square roots of \(I - A_N, I - A\) resp., as in \protect\hyperlink{notation:perturb-glap}{Notation \ref{notation:perturb-glap}}, \(U_N^t, U^t\) refer to the corresponding propagators and \(\mathbb{E}_t u := \langle |U_N^t[\psi_{h,t,N}]|^2, u \rangle_N\) is short-hand for the \emph{discrete expectation} of \(u\) with the density of the sample propagation of the coherent state. The notation \(T_1 \sim_r T_2\) for \(T_1,T_2 : L^{\infty} \to L^{\infty}\) and \(r > 0\) means \(||(T_1 - T_2)[u]||_{\infty} \lesssim r\) for any \(u \in L^{\infty}\). The meaning when \(T_1, T_2\) are scalars is clear by analogy and \(x \sim_r x_t\) for \(x \in \mathcal{M}\) and \(r > 0\) refers to the condition \(x \in B_r(x_t ; g)\). The second and third columns give bounds for the corresponding probabilities in the \(\varepsilon = 0\) (\emph{unperturbed}) and \(\varepsilon > 0\) (\emph{perturbed}) cases, respectively. The fourth column gives, for reference, the dominant functions in the probability bounds for the unperturbed and perturbed cases that are being further bounded in the second and third columns, respectively.

We have borrowed from complexity theory the following,

\begin{notation} The following are asymptotic bounds of non-negative functions:

\begin{enumerate}
\def\labelenumi{\arabic{enumi}.}
\tightlist
\item
  \(f_0 = O(f(\Upsilon))\) means there is \(\Upsilon^* > 0\) such that for all \(\Upsilon \geq \Upsilon^*\), \(f_0(\Upsilon) \lesssim f(\Upsilon)\);
\item
  \(f_0 = \Omega(f(\Upsilon))\) means there is \(\Upsilon^* > 0\) such that for all \(\Upsilon \geq \Upsilon^*\), \(f(\Upsilon) \lesssim f_0(\Upsilon)\);
\item
  \(f_0 = \Theta(f(\Upsilon))\) means \(f_0 = \Omega(f(\Upsilon))\) and \(f_0 = O(f(\Upsilon))\) and
\item
  \(f_0 = \tilde{O}(f(\Upsilon))\) means there is \(\Upsilon^* > 0\) and a polynomial with non-negative coefficients \(q \in \mathbb{R}[\Upsilon]\) such that for all \(\Upsilon \geq \Upsilon^*\), \(f_0(\Upsilon) \lesssim q(\log \Upsilon) f(\Upsilon)\).
\end{enumerate}

\end{notation}

\noindent
\emph{Sketch of how these bounds come about}.

\begin{enumerate}
\def\labelenumi{\arabic{enumi}.}
\item
  By its definition in \protect\hyperlink{lem:sqrt-perturb-eps-conv}{Lemma \ref{lem:sqrt-perturb-eps-conv}}, \(\upsilon \lesssim (\log N + N \delta^2 \epsilon^{\frac{\mdim}{2}})^{-1} \log(N \delta^2 \epsilon^{\frac{\mdim}{2}} + \log N)\), hence \(\upsilon = \tilde{O}(N^{-1} \delta^{-2} \epsilon^{-\frac{\mdim}{2}})\).
\item
  The dominant exponential factor in \(\eqref{eq:thm-halfwave-soln-bound}\) is
  \begin{align*}
  e^{-\left( \frac{N \delta^4 \epsilon^{{\frac{5}{2}\mdim + 4}}}{C K_u^2 \, |t|^8} \right) \left( \frac{\delta^4 \epsilon^{2\mdim + 4}}{|t|^8} \right)^{\tilde{O}\left( \frac{\epsilon^{\frac{\mdim}{4}} |t|^4}{K_u\sqrt{N}} \right)}}
   &\leq e^{-\left( \frac{N \delta^4 \epsilon^{{\frac{5}{2}\mdim + 4}}}{C K_u^2 \, |t|^8} \right) \left( \frac{K_u^2 \epsilon^{- \frac{\mdim}{2}}}{N |t|^8} \right)^{\tilde{O}\left( \frac{\epsilon^{\frac{\mdim}{4}} |t|^4}{K_u\sqrt{N}} \right)}}    \\
  &= e^{-\left( \frac{N \delta^4 \epsilon^{{\frac{5}{2}\mdim + 4}}}{C K_u^2 \, |t|^8} \right) (Z/N)^{\tilde{O}\left( \frac{1}{\sqrt{N Z}} \right)}}
  \end{align*}
  with \(Z := K_u^2 \epsilon^{-\frac{\mdim}{2}}/|t|^8\). So assuming \(|t| = O(K_u^{\frac{1}{4}} \epsilon^{-\frac{\mdim}{16}})\) gives \(Z \gtrsim 1\) and \((Z/N)^{\tilde{O}(1/\sqrt{NZ})} \gtrsim 1\), hence that factor can be replaced by a constant in the exponential.
\item
  Examining \(\Omega_{\lambda,t,N}(h^\frac{\mdim}{2})\) and using that \(|t|\) is bounded gives for the dominating term in \(\eqref{eq:prop-max-observable-flow-consistency-bound} \leq e^{-\Omega(N h^{2 \mdim} \epsilon^{\frac{5}{2}\mdim + 4}) \, h^{\upsilon}}\) with \(\upsilon = \tilde{O}(N^{-1} h^{-\mdim} \epsilon^{-(\frac{3}{2} \mdim + 2)})\). Then, \(h \gtrsim N^{-\frac{1}{\mdim(\frac{5\alpha}{2} + 7) + 4(2 + \alpha)}}\) implies \(\upsilon = \tilde{O}(N^{-\frac{2}{5}})\) and there is a \(c > 0\) such that \(h^{\tilde{O}(N^{-\frac{2}{5}})} > c > 0\), hence \(\eqref{eq:prop-max-observable-flow-consistency-bound} \leq C e^{-\Omega(N h^{2 \mdim} \epsilon^{\frac{5}{2}\mdim + 4})}\).
\item
  The only significant difference in bounding \(\eqref{eq:thm-mean-prop-geoflow-consistency-bound}\) from bounding \(\eqref{eq:prop-max-observable-flow-consistency-bound}\) is that we have \(h^{\frac{\mdim}{2} + 1}\) in place of \(h^{\frac{\mdim}{2}}\) for the error rate argument to \(\Omega_{\lambda,t,N}\) and this gives the dominating term in \(\eqref{eq:thm-mean-prop-geoflow-consistency-bound} \leq C e^{-\Omega(N h^{2\mdim + 4} \epsilon^{\frac{5}{2}\mdim + 4}) h^{\upsilon}}\) with \(\upsilon = \tilde{O}(N^{-1} h^{-(\mdim + 2)} \epsilon^{-(\frac{3}{2} \mdim +2)})\). Then, the lower bound on \(h\) gives a lower bound, away from zero, for \(h^{\upsilon}\) in much the same way as for \(\eqref{eq:prop-max-observable-flow-consistency-bound}\) above.
\item
  We can bound \(\tilde{\Omega}_{\lambda,t,N}(h,\varepsilon)\) to give the dominant term in \(\eqref{eq:thm-mean-prop-perturbed-glap-geoflow-consistency-bound} \leq e^{-\Omega(N h^{\mdim + 2} \epsilon^{\frac{3}{2} \mdim + 2} \varepsilon) \varepsilon^{2 \upsilon}}\) with \(\upsilon = \tilde{O}(N^{-1} h^{-(\mdim + 2)} \epsilon^{-(\frac{3}{2}\mdim + 2)})\). Then requiring \(\varepsilon \in \Theta(h^{1 + \alpha + \beta})\) with \(\beta \geq 0\) we anyways have \(h \gtrsim N^{-\frac{1}{\mdim + 3(2 + \alpha)(\frac{\mdim}{2} + 1)}}\), which gives \(\varepsilon^{2\upsilon} = \Theta(1)\), hence it follows that \(\eqref{eq:thm-mean-prop-perturbed-glap-geoflow-consistency-bound} \leq e^{-\Omega(N h^{n + 1 + \beta} \epsilon^{3(\frac{\mdim}{2} + 1)})}\).
\item
  The function \(\Omega_{\lambda,t,N}(h^{\frac{\mdim}{2}},\varepsilon)\) can be bounded to give the dominant term in \(\eqref{eq:prop-max-observable-perturb-glap-flow-consistency-bound} \leq e^{-\Omega(N h^{\mdim} \epsilon^{\frac{3}{2}\mdim + 2} \varepsilon) \varepsilon^{2\upsilon}}\) with \(\upsilon = \tilde{O}(N^{-1} h^{-\mdim} \epsilon^{-(\frac{3}{2}\mdim + 2)})\). Then since \(\varepsilon \in \Theta(h^{1 + \alpha + \beta})\) with \(\beta \geq 0\), we require for exponential decay in \(N\) that \(h \gtrsim N^{-\frac{1}{\mdim + (2 + \alpha)(\frac{3}{2}\mdim + 2) + 1}}\) so \(\varepsilon^{2 \upsilon} = \Theta(1)\), hence \(\eqref{eq:prop-max-observable-perturb-glap-flow-consistency-bound} \leq e^{-\Omega(N h^{\mdim - 1 + \beta} \epsilon^{3(\frac{\mdim}{2} + 1)})}\).
\end{enumerate}

\bibliographystyle{amsalphaabbrv}\bibliography{refs.bib}
\appendix

\addtocontents{toc}{\SkipTocEntry}

\hypertarget{notation}{%
\section{Notation}\label{notation}}

We fix some general notation. Given a (complex) Hilbert space \(\mathcal{H}\), \(\langle \cdot | \cdot\cdot \rangle\) denotes its \emph{complex} inner product that is anti-linear in the \emph{first} component and we write the squared norm, \(||\cdot||_{\mathcal{H}}^2 := \langle \cdot | \cdot \rangle\), which we often shorten to simply \(||\cdot||^2\). Borrowing from physics, we use the following \emph{Dirac notation}: if \(A : \mathcal{H} \to \mathcal{H}\) is a linear operator and \(\psi, \phi \in \mathcal{H}\) are vectors, then we identify \(|\psi \rangle := \psi\) and \(A|\psi \rangle \in \mathcal{H}\) denotes the application of \(A\) on \(\psi\), while \(\langle \phi |\) denotes the \emph{dual} of \(| \phi \rangle\) and we write, \(\langle \phi | A | \psi \rangle := \langle \phi | A \psi \rangle\). The mathematically commonplace counterpart to this notation is denoted by \(\langle \psi, \phi \rangle := \langle \phi | \psi \rangle\). In \protect\hyperlink{semi-classical-measures-of-coherent-states}{Section \ref{semi-classical-measures-of-coherent-states}} we will also need a \emph{bilinear} inner product that is a \emph{complexification} of a real inner product: we denote this also by \(\langle \cdot , \cdot \rangle\) and mark specifically that this is the form being used. Correspondingly, if \(M\) is a real symmetric matrix with \(M > 0\), then we denote \(||\cdot||_M^2 := \langle M \cdot, \cdot \rangle\). In the Euclidean case, we sometimes denote simply \(|\cdot| := ||\cdot||\), when the meaning is clear by context.

The operator \(A^*\) then denotes the adjoint to \(A\) with respect to the inner product on \(\mathcal{H}\). Further, given \(z \in \mathbb{C}\), we denote by \(\Re z, \Im z\) its real and imaginary parts, respectively and when applied to \(z = Z\) a bounded operator, these are the usual, \(\Re Z := (Z + Z^*)/2\) and \(\Im Z := (Z - Z^*)/(2i)\).

We have used various order notations: given some functions \(f_1, f_2\), we write \(f_1 \sim f_2\) if there is some \(c > 0\) such that \(c f_1 \leq f_2 \leq f_1/c\) and we write \(f_1 \lesssim f_2\) if \(c f_1 \leq f_2\). Further, we have the asymptotic notations: for \(N \geq 0\), we write \(f = O_{L^2}(h^N)\) if there are constants \(h_0 > 0\) and \(C_N > 0\) such that for all \(h \in (0, h_0]\), \(||f||_{L^2} \leq C_N h^N\). Likewise, \(f = O(h^N)\) means that \(|f(h)| \leq C_N h^N\) and \(f = O_{\mathscr{S}}(h^N)\) means that for each \(\alpha\), there is a \(C_{\alpha,N}\) such that \(|\partial^{\alpha} f| \leq C_{\alpha,N} h^N\). When \(N = \infty\) in these notations, we mean that these inequalities hold for all \(N \geq 0\) with \(h_0\) fixed.

At various times, we write \(a \equiv b\) in different ways and hence, it is a blatant abuse of notation. Briefly, it means either that \(a, b\) belong to the same congruence class modulo some equivalence relation, or if \(a, b\) are functions then \(a = b\) at all points on (a prescribed subset of) their domain of definition; the appropriate definition is clarified by the context.

\pagebreak

\hypertarget{notation-of-fixed-objects}{%
\subsection{Notation of fixed objects}\label{notation-of-fixed-objects}}

Throughout this work, we have fixed a tuple \((\mathcal{M}, \iota, \embedM, P)\) as per the \protect\hyperlink{assumptions}{Assumptions}. These fixed objects and their relatives are denoted as follows:

\begin{longtable}[]{@{}
  >{\raggedright\arraybackslash}p{(\columnwidth - 2\tabcolsep) * \real{0.5000}}
  >{\raggedright\arraybackslash}p{(\columnwidth - 2\tabcolsep) * \real{0.5000}}@{}}
\toprule()
\begin{minipage}[b]{\linewidth}\raggedright
Notation
\end{minipage} & \begin{minipage}[b]{\linewidth}\raggedright
Object
\end{minipage} \\
\midrule()
\endhead
\(\mathcal{M}\) & manifold \\
\midrule \(g\) & Riemannian metric on \(\mathcal{M}\) \\
\midrule \(d\vol{y}\) & volume density on \(\mathcal{M}\) with respect to the metric \(g\) \\
\midrule \(n := \dim \mathcal{M}\) & dimension of \(\mathcal{M}\) \\
\midrule \(\embedM\) & Euclidean submanifold isometric to \(\mathcal{M}\) \\
\midrule \(\rdim\) & dimension of ambient space of \(\embedM\) \\
\midrule \(\iota : \mathcal{M} \to \embedM \subset \mathbb{R}^{\rdim}\) & isometry \\
\midrule \(P\) & sampling distribution with \(\supp P = \embedM\) \\
\midrule \(p \in C^{\infty}(\mathcal{M})\) & smooth density of \(P \circ \iota^{-1}\) with respect to \(d\nu_g\) \\
\midrule \(\embedM_N = \{ X_1, \ldots, X_N \} \subset \embedM\) & set of \(N\) random vectors with law \(P\) \\
\midrule \(\mathcal{X}_N = \{ x_1, \ldots, x_N \} \subset \mathcal{M}\) & pull-back of \(\embedM_N\) onto \(\mathcal{M}\) by \(\iota\) \\
\bottomrule()
\end{longtable}

\hypertarget{geometric-notation}{%
\subsection{Geometric notation}\label{geometric-notation}}

Generally, given a manifold \(\mathcal{N}\) along with its metric \(g\), we have used the following notation:

\begin{longtable}[]{@{}
  >{\raggedright\arraybackslash}p{(\columnwidth - 2\tabcolsep) * \real{0.5000}}
  >{\raggedright\arraybackslash}p{(\columnwidth - 2\tabcolsep) * \real{0.5000}}@{}}
\toprule()
\begin{minipage}[b]{\linewidth}\raggedright
Notation
\end{minipage} & \begin{minipage}[b]{\linewidth}\raggedright
Meaning
\end{minipage} \\
\midrule()
\endhead
\(T^*\mathcal{N}\) & cotangent bundle of \(\mathcal{N}\) \\
\midrule \(dx d\xi\) & short-hand for the measure induced by the canonical symplectic form on \(T^*\mathcal{N}\) \\
\midrule \(|\xi|_x, |\xi|_{g_x}\) (with \((x,\xi) \in T^*\mathcal{N}\)) & norm of \(\xi \in T_x^*\mathcal{N}\) according to the inner product given by \(g_x^{-1}\) on \(T_x^*\mathcal{N}\) \\
\midrule \(\langle \zeta, \xi \rangle_{x}, \langle \zeta, \xi \rangle_{g_x}\) (with \((x,\xi), (x, \zeta) \in T^*{\mathcal{N}}\)) & inner product between \(\zeta, \xi \in T_x^*\mathcal{N}\) with respect to \(g_x^{-1}\) \\
\midrule \(d_g(x,y)\) & geodesic distance between the points \(x,y \in \mathcal{N}\) w.r.t. \(g\) \\
\midrule \(\operatorname{inj}(x)\), \(\operatorname{inj}(\mathcal{N}) = \operatorname{inj}(\mathcal{N},g)\) & injectivity radius at \(x \in \mathcal{N}\) and its lower bound \(\operatorname{inj}(\mathcal{N}) := \inf_{x \in \mathcal{N}} \operatorname{inj}(x)\), resp. \\
\midrule \(s_{x_0}= \exp^{-1}_{x_0}\) & normal coordinates at \(x_0 \in \mathcal{N}\) \\
\midrule \(\gamma^*[dx]\) & \emph{pullback} of the density \(\sqrt{\det |g|} ~ dx_1 \wedge \cdots \wedge dx_n\) by the diffeomorphism \(\gamma : \mathcal{N}_1 \to \mathcal{N}_2\) between manifolds \(\mathcal{N}_1, \mathcal{N}_2\) of equal dimension \\
\midrule \(|dv| := \sqrt{\det |g \circ u^{-1}(v)|} ~ dv\) & the volume form in local coordinates \(u : \mathcal{M} \supset U \to V \subset \mathbb{R}^{n}_v\) \\
\midrule \(B_{r}(x_0, g) = B(x_0, g ; r)\) & ball of radius \(r\) centered at \(x_0 \in \mathcal{N}\) with respect to the metric \(g\) \\
\bottomrule()
\end{longtable}

\hypertarget{notation-related-to-pdos}{%
\subsection{\texorpdfstring{Notation related to \(\PDO\)s}{Notation related to \textbackslash PDOs}}\label{notation-related-to-pdos}}

In Sections \protect\hyperlink{quantization-and-symbol-classes}{\ref{quantization-and-symbol-classes}} and \protect\hyperlink{fbi-transform}{\ref{fbi-transform}} we have introduced the following notation regarding pseudodifferential operators (\(\PDO\)s) and the FBI transform:

\begin{longtable}[]{@{}
  >{\raggedright\arraybackslash}p{(\columnwidth - 2\tabcolsep) * \real{0.5000}}
  >{\raggedright\arraybackslash}p{(\columnwidth - 2\tabcolsep) * \real{0.5000}}@{}}
\toprule()
\begin{minipage}[b]{\linewidth}\raggedright
Notation
\end{minipage} & \begin{minipage}[b]{\linewidth}\raggedright
Meaning
\end{minipage} \\
\midrule()
\endhead
\(h^{\ell} S^m\), \(h^{\ell} S^m_{\text{phg}}\) & symbol classes of order \((\ell, m)\) and their polyhomogeneous subclasses, resp. \\
\midrule \(h^{\ell} \Psi^m\), \(h^{\ell} \Psi_{\operatorname{phg}}^m\) & classes of \(\PDO\)s of order \((\ell,m)\) and their subclasses of polyhomogeneous \(\PDO\)s, resp. \\
\midrule \(\Op_h(a)\) & \(\PDO\) that is the (semiclassical) \emph{right} (or \emph{adjoint}) quantization of a symbol \(a\) \\
\midrule \(T_h[u;\phi,a] = T[u; \phi/h, a]\) & \emph{FBI transform} of a smooth function \(u\) with \emph{admissible phase} \(\phi\) and symbol \(a \in h^{-\frac{3\mdim}{4}} S^{\frac{\mdim}{4}}\) \\
\midrule \(b \in h^{\frac{3\mdim}{2}} S^{-\frac{\mdim}{2}}\) & positive, elliptic symbol associated to \(T_h\), depending on \(\phi\), such that the properties of \(T_h[\cdot;\phi,b^{-\frac{1}{2}}]\) in \protect\hyperlink{thm:FBI-basic}{Theorem \ref{thm:FBI-basic}} hold \\
\bottomrule()
\end{longtable}

\hypertarget{notation-related-to-graph-laplacians}{%
\subsection{Notation related to graph Laplacians}\label{notation-related-to-graph-laplacians}}

In \protect\hyperlink{laplacians-from-graphs-to-manifolds}{Section \ref{laplacians-from-graphs-to-manifolds}} we have introduced the following objects and notations related to the discrete and continuum graph Laplacian constructions:

\begin{longtable}[]{@{}
  >{\raggedright\arraybackslash}p{(\columnwidth - 2\tabcolsep) * \real{0.5000}}
  >{\raggedright\arraybackslash}p{(\columnwidth - 2\tabcolsep) * \real{0.5000}}@{}}
\toprule()
\begin{minipage}[b]{\linewidth}\raggedright
Continuum / Discrete object
\end{minipage} & \begin{minipage}[b]{\linewidth}\raggedright
Name
\end{minipage} \\
\midrule()
\endhead
\(k : \mathbb{R}_+ \to \mathbb{R}_+\) & kernel function \\
\midrule \(p_{\epsilon}(x) := \int k_{\epsilon}(x,y) p(y) d\vol{g}\), \(p_{\epsilon,N}(x) := \frac{1}{N} \sum_{j=1}^N k_{\epsilon}(x,y)\) & degree functions \\
\midrule \(k_{\lambda,\epsilon}(x,y) := k_{\epsilon}(x,y)/(p_{\epsilon}(x) p_{\epsilon}(y))^{\lambda}\), \(k_{\lambda,\epsilon,N}(x,y) := k_{\epsilon}(x,y)/(p_{\epsilon,N}(x) p_{\epsilon,N}(y))^{\lambda}\) & sampling-dependent averaging kernels, or \(\lambda\)-\emph{averaging kernels} \\
\midrule \(\AvOp_{\lambda,\epsilon} : u(\cdot) \mapsto \int k_{\lambda,\epsilon}(\cdot,y) u(y) ~ d\vol{g}\), \((\AvOp_{\lambda,\epsilon,N})_{j,m} := \frac{1}{N} k_{\lambda,\epsilon,N}(x_j, x_m)\) & (sampling-dependent) \(\lambda\)-averaging operators \\
\midrule \(p_{\lambda,\epsilon}(x) := \AvOp_{\lambda,\epsilon}[1](x)\), \(p_{\lambda,\epsilon,N}(x) := \AvOp_{\lambda,\epsilon,N}[1](x)\) & \(\lambda\)-degree functions \\
\midrule \(\GAve_{\lambda,\epsilon} : u(\cdot) \mapsto \AvOp_{\lambda,\epsilon}[u]/p_{\lambda,\epsilon}(\cdot)\), \((\GAve_{\lambda,\epsilon,N})_{j,m} := (\AvOp_{\lambda,\epsilon,N})_{j,m}/p_{\lambda,\epsilon,N}(x_j)\) & normalized \(\lambda\)-averaging operators \\
\midrule \(\GLap_{\lambda,\epsilon} := \frac{2 c_0}{c_2} \frac{1}{\epsilon}(I - A_{\lambda,\epsilon})\), \(\GLap_{\lambda,\epsilon,N} := \frac{2 c_0}{c_2}\frac{1}{\epsilon}(I - A_{\lambda,\epsilon,N})\) & \((\lambda)\)-graph Laplacians \\
\bottomrule()
\end{longtable}

\addtocontents{toc}{\SkipTocEntry}

\hypertarget{acknowledgements}{%
\section*{Acknowledgements}\label{acknowledgements}}

The author thanks Mohan Sarovar for many discussions regarding physical interpretations throughout the course of this project, and Kurt Maier for graciously providing access to computing resources that greatly helped with simulations. This work was supported by the Laboratory Directed Research and Development program at Sandia National Laboratories, a multimission laboratory managed and operated by National Technology and Engineering Solutions of Sandia, LLC., a wholly owned subsidiary of Honeywell International, Inc., for the U.S. Department of Energy's National Nuclear Security Administration under contract DE-NA-0003525.

\end{document}